\documentclass[a4paper,10pt]{amsart}
\usepackage[a4paper, textwidth=15.5cm]{geometry}

\usepackage{marvosym}
\usepackage{amsmath,amsthm,amssymb,amsfonts,extarrows}
\usepackage{enumerate,amscd,amsxtra,MnSymbol}
\usepackage{mathrsfs}
\usepackage{color}
\usepackage[cmtip,all]{xy}
\usepackage{eucal}
\usepackage{bbold}
\usepackage{ upgreek }
\usepackage{paralist}
\usepackage{tikz-cd}
\usepackage{ wasysym }

\theoremstyle{plein}
\newtheorem{theorem}{Theorem}[section]
\newtheorem*{theorem*}{Theorem}
\newtheorem{lemma}[theorem]{Lemma}
\newtheorem{proposition}[theorem]{Proposition}
\newtheorem*{proposition*}{Proposition}
\newtheorem{corollary}[theorem]{Corollary}
\newtheorem*{corollary*}{Corollary}

\newtheorem{conjecture*}{Conjecture}

\theoremstyle{definition}
\newtheorem{construction}[theorem]{Construction}

\newtheorem{notation}[theorem]{Notation}

\newtheorem{example}[theorem]{Example}
\newtheorem{definition}[theorem]{Definition}
\newtheorem{definition*}{Definition}
\newtheorem{remark}[theorem]{Remark}
\newtheorem{remark*}{Remark}

\newcommand{\bE}{{\mathbb E}}

\renewcommand{\P}{{\mathbb P}}

\newcommand{\mA}{{\mathcal A}}
\newcommand{\mB}{{\mathcal B}}
\newcommand{\mC}{{\mathcal C}}
\newcommand{\mD}{{\mathcal D}}
\newcommand{\mE}{{\mathcal E}}
\newcommand{\mF}{{\mathcal F}}

\newcommand{\mH}{{\mathcal H}}
\newcommand{\mI}{{\mathcal I}}
\newcommand{\mJ}{{\mathcal J}}
\newcommand{\mK}{{\mathcal K}}

\newcommand{\mM}{{\mathcal M}}
\newcommand{\mN}{{\mathcal N}}
\newcommand{\mO}{{\mathcal O}}
\newcommand{\mP}{{\mathcal P}}
\newcommand{\mQ}{{\mathcal Q}}

\newcommand{\mS}{{\mathcal S}}
\newcommand{\mT}{{\mathcal T}}
\newcommand{\mU}{{\mathcal U}}
\newcommand{\mV}{{\mathcal V}}
\newcommand{\mW}{{\mathcal W}}
\newcommand{\mX}{{\mathcal X}}
\newcommand{\mY}{{\mathcal Y}}

\newcommand{\A}{{\mathrm A}}
\newcommand{\B}{{\mathrm B}}
\newcommand{\C}{{\mathrm C}}
\newcommand{\D}{{\mathrm D}}

\newcommand{\F}{{\mathrm F}}
\newcommand{\G}{{\mathrm G}}
\newcommand{\rH}{{\mathrm H}}

\newcommand{\K}{{\mathrm K}}
\renewcommand{\L}{{\mathrm L}}

\renewcommand{\P}{{\mathrm P}}
\newcommand{\Q}{{\mathrm Q}}
\newcommand{\R}{{\mathrm R}}
\newcommand{\rS}{{\mathrm S}}
\newcommand{\T}{{\mathrm T}}

\newcommand{\V}{{\mathrm V}}
\newcommand{\W}{{\mathrm W}}
\newcommand{\X}{{\mathrm X}}
\newcommand{\Y}{{\mathrm Y}}
\newcommand{\Z}{{\mathrm Z}}

\newcommand{\rc}{{\mathrm c}}

\newcommand{\bj}{{\mathrm j}}
\newcommand{\bi}{{\mathrm i}}
\newcommand{\m}{{\mathrm m}}
\newcommand{\bk}{{\mathrm k}}

\newcommand{\g}{{\mathrm g}}
\newcommand{\n}{{\mathrm n}}
\newcommand{\br}{{\mathrm r}}

\newcommand{\op}{\mathrm{op}}

\newcommand{\Ho}{\mathsf{Ho}}

\newcommand{\cc}{\mathsf{c}}

\newcommand{\Rex}{\mathrm{Rex}}

\newcommand{\Sp}{\mathrm{Sp}}

\newcommand{\y}{\mathsf{y}}
\newcommand{\f}{\mathsf{f}}

\newcommand{\colim}{\mathrm{colim}}

\newcommand{\Mod}{{\mathrm{Mod}}}
\newcommand{\LMod}{{\mathrm{LMod}}}
\newcommand{\RMod}{{\mathrm{RMod}}}

\newcommand{\ot}{\otimes}

\newcommand{\Ass}{  {\mathrm {   Ass  } }   }

\newcommand{\id}{\mathrm{id}}

\newcommand{\Cat}{\mathsf{Cat}}

\newcommand{\Set}{\mathsf{Set}}

\renewcommand{\Pr}{\mathrm{Pr}}

\newcommand{\Alg}{\mathrm{Alg}}

\newcommand{\Comm}{\mathrm{Comm}}

\newcommand{\Add}{\mathrm{Add}}
\newcommand{\Mon}{\mathrm{Mon}}
\newcommand{\Fun}{\mathrm{Fun}}
\newcommand{\tu}{{\mathbb 1}}
\newcommand{\Op}{{\mathrm{Op}}}

\newcommand{\ma}{{\mathrm{max}}}

\newcommand{\triv}{{\mathrm{triv}}}

\newcommand{\Fin}{{\mathrm{\mF in}}}

\newcommand{\St}{{\mathrm{St}}}
\newcommand{\rev}{{\mathrm{rev}}}

\newcommand{\lax}{{\mathrm{lax}}}
\newcommand{\oplax}{{\mathrm{oplax}}}

\newcommand{\Mul}{{\mathrm{Mul}}}

\newcommand{\ev}{{\mathrm{ev}}}
\newcommand{\Ind}{{\mathrm{Ind}}}

\newcommand{\LM}{{\mathrm{LM}}}

\newcommand{\Cmon}{{\mathrm{Cmon}}} 
 
\newcommand{\mi}{{\mathrm{min}}} 
 
\newcommand{\wc}{{\mathrm{w}}}

\newcommand{\Env}{{\mathrm{Env}}}

\newcommand{\Grp}{{\mathrm{Grp}}}  
  
\newcommand{\LinFun}{{\mathrm{LinFun}}}

\newcommand{\BMod}{{\mathrm{BMod}}}

\newcommand{\Enr}{{\mathrm{Enr}}} 
 
\newcommand{\Poset}{{\mathrm{Poset}}} 
 
\newcommand{\Mor}{{\mathrm{Mor}}}

\usepackage{pdfpages}

\title{The higher algebra of weighted colimits}
\author{Hadrian Heine, \\ University of Oslo, Norway, \\ hadriah@math.uio.no}

\begin{document}
	
\maketitle

\begin{abstract}

We develop a theory of weighted colimits in the framework of weakly bienriched
$\infty$-categories, an extension of Lurie's notion of enriched $\infty$-categories.
We prove an existence result for weighted colimits, study weighted colimits 
of diagrams of enriched functors, express weighted colimits via enriched coends, characterize the enriched $\infty$-category of enriched presheaves as the free cocompletion under weighted colimits, prove a Bousfield-Kan formula for weighted colimits and an enriched adjoint functor theorem and develop a theory of universally adjoining weighted colimits to an enriched $\infty$-category.
We use the latter technique to construct for every presentably $\bE_{\bk+1}$-monoidal $\infty$-category $\mV$ for $1 \leq \bk \leq \infty$ and set $\mH$ of $\mV$-weights, with respect to which weighted colimits are defined, a presentably $\bE_\bk$-monoidal structure on the $\infty$-category of $\mV$-enriched $\infty$-categories that admit $\mH$-weighted colimits. Varying the set of weights $\mH$ this $\bE_\bk$-monoidal structure interpolates between the tensor product for $\mV$-enriched $\infty$-categories and the relative tensor product for $\infty$-categories presentably left tensored over $\mV$.
Studying functoriality in the set of $\mV$-weights we deduce that the functor of $\mV$-enriched presheaves is $\bE_\bk$-monoidal with respect to the tensor product on small $\mV$-enriched $\infty$-categories and the relative tensor product on $\infty$-categories presentably left tensored over $\mV,$ which leads to a $\mV$-enriched version of Day-convolution
and a new description of the relative tensor product of $\infty$-categories presentably left tensored over $\mV$ from the enriched perspective.
As key applications we construct for every $\n \geq 1 $ %several presentably monoidal structures: 
and set $\mK$ of $(\infty,\n)$-categories %we construct %tensor product for $\mV$-presentable $\infty$-categories that we identify with the relative tensor product for $\infty$-categories presentably left tensored over $\mV$, 
%give a new construction of the tensor product for $\infty$-categories presentably left tensored over $\mV$ as a $\mV$-enriched localization of Day-convolution.
%Moreover we give an explicite construction of the internal hom of the tensor product on $\Cat_\infty^\mV(\mH)$.As key applications we 
a tensor product for $(\infty,\n)$-categories that admit $\mK$-indexed (op)lax colimits, %a tensor product for $(\infty,\n)$-categories that admit (op)lax pushouts, 
a tensor product for Cauchy-complete $\mV$-enriched $\infty$-categories and tensor products for (Cauchy complete) $\n$-stable, $\n$-additive and $\n$-preadditive $(\infty,\n)$-categories.
%Moreover we exhibit the tensor product of stable $\infty$-categories as a localization of the tensor product of spectral $\infty$-categories.

%nd express the relative tensor product of left $\mV$-modules as a $\mV$-enriched functor $\infty$-category, which leads to an enriched version of Day-convolution.
%We prove that the tensor product for $\mV$-enriched $\infty$-categories that admit $\mH$-weighted colimits is closed by giving explicte constructions forthe internal hom in the Lurie model and the Gepner-Haugseng model for enriched $\infty$-categories. 

% and a further construction of the internal hom in terms of the relative tensor product over $\mV$ with the enveloping closed $\infty$-category.

\end{abstract}

\tableofcontents

%\section{Motivation and overview}

%\begin{enumerate}\item A theory of weighted colimits/lax colimits

%\item A tensor product for enriched $\infty$-categories with weighted colimits

%\item bienriched $\infty$-categories

%\item Internal homs for enriched $\infty$-categories

%\item A relative tensor product formula

%\item An enriched end formula for morphism objects of enriched functor $\infty$-categories

%\end{enumerate}

\section{Introduction}

It was long time a problem in algebraic topology to construct a symmetric monoidal structure on the $\infty$-category of spectra. %structure on the stable homotopy category. % extending the smash product of pointed spaces along the infinite suspension.
Lurie \cite{lurie.higheralgebra} gives a remarkable elegant solution for this problem:
%in two steps: he axiomatizes the essential properties of the $\infty$-category of spectra $\Sp$ to obtain the notion of stable $\infty$-category.
%He constructs a tensor product for stable presentable $\infty$-categories whose tensor unit is the stable $\infty$-category of spectra $\Sp$.Playing the role of the tensor unit the $\infty$-category of spectra $\Sp$ refines to the initial $\bE_\infty$-algebra for the tensor product of stable presentable $\infty$-categories, which is identified with a closed symmetric monoidal structure on $\Sp.$
% working with $\infty$-categories: he constructed a symmetric monoidal $\infty$-category of spectra in the following two steps:

\begin{enumerate}
\item He axiomatizes the essential properties of the $\infty$-category of spectra $\Sp$ to obtain the notion of stable $\infty$-category \cite[Definition 1.1.1.9.]{lurie.higheralgebra}.
\item He constructs a tensor product for stable presentable $\infty$-categories whose tensor unit is the stable $\infty$-category of spectra $\Sp$ \cite[4.8.2.]{lurie.higheralgebra}.
Playing the role of the tensor unit the $\infty$-category of spectra 
$\Sp$ refines to the initial $\bE_\infty$-algebra for the tensor product of stable presentable $\infty$-categories, which is identified with a closed symmetric monoidal structure on $\Sp.$
\end{enumerate}
To achieve (2) Lurie %extends the theory of presentable categories to the $\infty$-categorical world and 
builds a tensor product for presentable $\infty$-categories \cite[Proposition 4.8.1.17.]{lurie.higheralgebra} %, which constitutes a closed symmetric monoidal structureon the $\infty$-category $\Pr^\L$ of presentable $\infty$-categories and left adjoint functors, % with a closed symmetric monoidal structure. %, for which the tensor unit is the $\infty$-category of spaces $\mS.$
that may be thought of as a higher categorical analogue of the tensor product of abelian groups: the higher algebra for this tensor product is a higher categorical analogon of commutative algebra, where the $\infty$-category of spaces plays the role of the integers,
rings and commutative rings are replaced by monoidal and symmetric monoidal presentable $\infty$-categories and left, right and bi-modules correspond to
left, right and bitensored presentable $\infty$-categories.
%This analogy is especially fruitful in the stable case, i.e. for stable presentable $\infty$-categories, and is omnipresent in the theory of higher-categorical invariants like higher derived Brauer groups [1] and secondary K-theory [16]. %Here stable presentably left tensored ∞-categories serve as a tractable substitute for spectral presentable ∞-categories.
%is the starting point of  
To construct a symmetric monoidal $\infty$-category of spectra 
%Lurie endows the $\infty$-category $\Cat_\infty^{\rc\rc}$ of $\infty$-categories that admit small colimits and functors preservingsmall colimits with a closed symmetric monoidal structure that restricts to the full subcategory $\Pr^\L$ of presentable $\infty$-categories and whose tensor unit is the $\infty$-category of spaces $\mS$. In a second step he 
Lurie extends the notion of spectrum to a notion of spectrum object in any presentable $\infty$-category and proves that forming spectrum objects is the universal recipe to turn
%universally turns a 
presentable $\infty$-categories to stable presentable $\infty$-categories:
he proves that forming spectrum objects defines a symmetric monoidal localization on the $\infty$-category $\Pr^\L$ of presentable $\infty$-categories and left adjoint functors, whose local objects are the stable presentable $\infty$-categories \cite[Proposition 4.8.2.18.]{lurie.higheralgebra}.
%Much of the power of Lurie's tensor product of presentable $\infty$-categories comes from the fact that it admits an explicite description: for any two presentable $\infty$-categories $\mC,\mD$ the tensor product $\mC\ot \mD$ is the $\infty$-category of right adjoint functors $\mC^\op \to \mD.$
%This way Lurie exhibits stable presentable $\infty$-categories as the local objects of a symmetric monoidal localization on the $\infty$-category $\Pr^\L$ of presentable $\infty$-categories that sends $\mS$ to the $\infty$-category of spectra $\Sp.$
%Lurie uses this description to prove that the localization on $\Pr^\L$ by forming spectrum objects is symmetric monoidal. 
This guarantees that the full subcategory of local objects, the full subcategory of stable presentable $\infty$-categories, %, the $\infty$-category of commutative algebra objects in $\Pr^\L$, 
inherits a symmetric monoidal structure from $\Pr^\L,$ for which $\Sp$ is the tensor unit \cite[Corollary 4.8.2.19]{lurie.higheralgebra}. 
%Besides the importance of 
This way Lurie constructs a tensor product for spectra giving rise to brave new algebra, a homotopical version of commutative algebra, and a tensor product of 
stable $\infty$-categories giving rise to a higher categorified version of commutative algebra, which is omnipresent in the theory of higher-categorical invariants like higher derived Brauer groups \cite{MR3190610} or secondary K-theory \cite{mazelgee2021universal},
and serves as a starting point for tensor triangular geometry \cite{Aoki2020TensorTG}. 
%The resulting tensor product for stable $\infty$-categories higher algebra for this tensor product 
%Here stable presentably left tensored ∞-categories serve as a tractable substitute for spectral presentable ∞-categories.This way Lurie constructs a tensor product for stable $\infty$-categorieswhose 
Lurie's approach to a tensor product of stable homotopy types via a tensor product of stable $\infty$-categories is very robust and works in various other situations:
studying multiplicative infinite loop space machines Gepner-Groth-Nikolaus \cite[Theorem 4.6.]{Gepner_2015} adapt Lurie's strategy to preadditive and additive presentable $\infty$-categories: the authors produce symmetric monoidal structures on the $\infty$-categories $\Mon_{\bE_\infty}(\mS), \Grp_{\bE_\infty}(\mS)$ 
of $\bE_\infty$-spaces and grouplike $\bE_\infty$-spaces via constructing
tensor products for preadditive and additive presentable $\infty$-categories, respectively, for which the $\infty$-categories $\Mon_{\bE_\infty}(\mS)$ and $\Grp_{\bE_\infty}(\mS)$ are the respective tensor units \cite[Theorem 5.1.]{Gepner_2015}. 
In a similar fashion one can produce tensor products for compact spectra, compact grouplike $\bE_\infty$-spaces and compact $\bE_\infty$-spaces by constructing
tensor products for small stable $\infty$-categories, additive $\infty$-categories, preadditive $\infty$-categories, respectively,
%, for which $\Sp^\omega$ is the tensor unit. 
%Besides this way of producing a smash product of spectra there is a slightly different way: the $\infty$-category of spectra is compactly generated by the full subcategory $\Sp^\omega$ of compact spectra and a symmetric monoidal structure on $\Sp$ is determined by a symmetric monoidal structure on $\Sp^\omega$ extending by filtered colimits.Since $\Sp^\omega$ is a small stable $\infty$-category, alternatively to (2)one can construct a tensor product for small stable $\infty$-categories, with respect to which $\Sp^\omega$ is the tensor unit. 
which arise via localization from closed symmetric monoidal structures
on the $\infty$-category $\Cat_\infty^{\mathrm{rex}}$ of small $\infty$-categories having finite colimits  \cite[Lemma 8.15.]{HEINE2023108941} and the $\infty$-category %functors preserving finite colimits and 
$ \Cat_\infty^{\coprod}$ of small $\infty$-categories having finite coproducts \cite[Proposition 3.6.]{heine2020infinity}, respectively.
Lurie's constructs the symmetric monoidal structure on $\Pr^\L$ as the restriction of a symmetric monoidal structure on the larger $\infty$-category $\Cat_\infty^{\rc\rc}$ of large $\infty$-categories having small colimits, into which $\Pr^\L$ embeds.
Consequently, the tensor products of (presentable) stable, additive and preadditive $\infty$-categories are all derived from symmetric monoidal structures on the $\infty$-categories $\Cat_\infty^{\rc\rc}$, $\Cat_\infty^{\mathrm{rex}}$ and $\Cat_\infty^{\coprod}$, respectively, which are instances of the following powerful construction:
for any set $\mK$ of small $\infty$-categories Lurie \cite[Proposition 4.8.1.3.] {lurie.higheralgebra} constructs a closed symmetric monoidal structure on the subcategory $\Cat_\infty(\mK) \subset\Cat_\infty $ of small $\infty$-categories having $\mK$-indexed colimits and functors preserving $\mK$-indexed colimits.
This symmetric monoidal structure %arises from the cartesian symmetric monoidal structure on the $\infty$-category $\Cat_\infty$ of small $\infty$-categories and
is characterized by the following key properties:
\begin{enumerate}
\item The symmetric monoidal structure on $\Cat_\infty(\mK)$ is closed. The internal hom between two $\infty$-categories $\mC,\mD$ having $\mK$-indexed colimits is the full subcategory $\Fun^\mK(\mC,\mD) \subset \Fun(\mC,\mD) $ of functors preserving $\mK$-indexed colimits.
%\item The inclusion $\Cat_\infty(\mK) \subset \Cat_\infty$admits a symmetric monoidal left adjoint, where $\Cat_\infty$ carries the cartesian structure, that sends an $\infty$-category $\mC$ to the full subcategory $\mP_\mK(\mC)$ of the $\infty$-category of presheaves on $\mC$ generated by $\mC$ under $\mK$-indexed colimits.

\item For $\mK=\emptyset$ the empty set the symmetric monoidal structure on 
$\Cat_\infty(\emptyset) = \Cat_\infty$ is the cartesian structure.
\item For $ \mK' \subset \mK$ an inclusion of subsets of small $\infty$-categories the subcategory inclusion
$\Cat_\infty(\mK) \subset \Cat_\infty(\mK')$ admits a symmetric monoidal left adjoint. %is lax symmetric monoidal.
%\item The symmetric monoidal left adjoint of the inclusion $\Cat_\infty(\mK) \subset \Cat_\infty$ sends an $\infty$-category $\mC$ to the full subcategory $\mP_\mK(\mC)$ of the $\infty$-category of presheaves on $\mC$ generated by $\mC$ under $\mK$-indexed colimits.
\end{enumerate}
%By (2) $\mK$-indexed colimits preserving functors $\mC \ot \mD \to \mE$ 
%correspond to functors $\mC \times \mD \to \mE$ preserving $\mK$-indexed colimits in each component.
Taking $\mK$ to be the set of finite $\infty$-categories gives the closed symmetric monoidal structure on $\Cat_\infty^{\mathrm{rex}}$
leading to the tensor product of small stable $\infty$-categories, choosing $\mK$ to be the set of finite sets gives the closed symmetric monoidal structure on $\Cat_\infty^{\coprod}$ leading to the tensor products of small additive and preadditive $\infty$-categories.
%To obtain more examples one enlarges the universe to produce for any large collection $\mK$ of small $\infty$-categories a symmetric monoidal structure on the $\infty$-category $\widehat{\Cat}_\infty(\mK) \subset\widehat{\Cat}_\infty $ of $\infty$-categories that admit $\mK$-indexed colimits and functors preserving $\mK$-indexed colimits.
For $\mK$ the large collection of all small $\infty$-categories and the corresponding construction $\widehat{\Cat}_\infty(\mK)$ for not necessarily small $\infty$-categories one obtains %(after enlarging the universe) 
the closed symmetric monoidal structure on $\Cat^{\rc\rc}_\infty$ leading to the tensor product of presentable stable $\infty$-categories.
For general set $\mK$ of small $\infty$-categories the tensor product on $\widehat{\Cat}_\infty(\mK)$ is understood best as interpolating between
the tensor product on $\Cat_\infty^{\rc\rc}$ and the cartesian structure on 
the $\infty$-category $\widehat{\Cat}_\infty$ 
of not necessarily small $\infty$-categories: %an inclusion $\mK\subset \mK'$ gives a right adjoint lax symmetric monoidal inclusion $\widehat{\Cat}_\infty(\mK')\subset \widehat{\Cat}_\infty(\mK)$ and so 
the factorization
$\emptyset \subset \mK \subset \Cat_\infty$ gives a factorization
\begin{equation}\label{ul}
\Cat^{\rc\rc}_\infty = \widehat{\Cat}_\infty(\Cat_\infty) \subset \widehat{\Cat}_\infty(\mK)
\subset \widehat{\Cat}_\infty=\widehat{\Cat}_\infty(\emptyset),
\end{equation}
where both inclusions admit symmetric monoidal left adjoints.
The composite of both left adjoints assigns to any small $\infty$-category $\mC$ the presentable $\infty$-category $\mP(\mC)$ of presheaves on $\mC$
\cite[Theorem 5.1.5.6.]{lurie.higheralgebra} %, which therefore becomes a symmetric monoidal functor. 
and so 
%This gives in particular that the composite of both left adjoints 
restricts to a symmetric monoidal functor $\Cat_\infty \to \Pr^\L$.
% assigning presheaves and 
This proves in particular that the functor of presheaves is symmetric monoidal. 

%right adjoint to a symmetric monoidal functor of adjoining colimits.For $\mK$ the empty set and $\mK'$ the collection of all small $\infty$-categories,the respective functor $\widehat{\Cat}_\infty \to \Cat_\infty^{\rc\rc}$ assigns the $\infty$-category of presheaves to any small $\infty$-category,so that the symmetric monoidal left adjoint of the inclusion $\Cat_\infty(\mK')\subset \Cat_\infty(\mK)$may be viewed as assigning a generalized version of $\infty$-category of presheaves.

% of large $\infty$-categories having small colimits. %, which restricts to the full subcategory $\Pr^\L.$
% and functors preserving small colimits.
%The resulting symmetric monoidal structure on $\Cat^{\rc\rc}_\infty$ is especially important since it restricts to the full subcategory $\Pr^\L$ of presentable $\infty$-categories and makes $\Pr^\L$ a closed symmetric monoidal $\infty$-category with many applications to topology.
%For any further collection $\mK' \subset \mK$ the inclusion $\theta: \widehat{\Cat}_\infty(\mK) \subset  \widehat{\Cat}_\infty(\mK')$ admits a symmetric monoidal left adjoint and is thus lax symmetric monoidal.If we choose $\mK'$ to be empty and $\mK$ the large collection of all small $\infty$-categories, the left adjoint of $\theta: \Cat_\infty^{\rc\rc} \subset \widehat{\Cat}_\infty$is a symmetric monoidal functor $\widehat{\Cat}_\infty \to \Cat_\infty^{\rc\rc}$that restricted to small $\infty$-categories assigns the $\infty$-category of presheaves. This proves that taking presheaves is symmetric monoidal.
It is goal of this work to build a theory of colimits and limits in enriched
$\infty$-category theory, which is an $\infty$-categorical extension of the theory of weighted colimits and limits in classical enriched category theory,
and to %extend the above situation from $\infty$-category theory to 
%construct an analogon of the tensor product on $\Cat_\infty(\mK)$for enriched $\mK$-indexed 
construct an analogoue of Lurie's tensor product on $\Cat_\infty(\mK)$ in $\mV$-enriched $\infty$-category theory, %\cite{HEINE2023108941},\cite{GEPNER2015575},\cite{HINICH2020107129},
where $\mV$ is any presentably $\bE_2$-monoidal $\infty$-category.
To draw this analogy we follow 
%of Lurie's tensor product on $\Cat_\infty(\mK)$ via the 
the subsequent table of analogies, which we explain in the following:
% to extend the situation above:
\begin{center}
\begin{tabular}{ | l | l | }
\hline
Non-enriched & Enriched  \\ \hline
$\mS$ & $\mV$\\ \hline 
$\infty$-categories & $\mV$-enriched $\infty$-categories
 \\ \hline
colimits & weighted colimits \\
\hline
the cartesian structure on $\Cat_\infty$ & the tensor product of $\mV$-enriched $\infty$-categories %on $\Cat_\infty^\mV$ 
\\
\hline
the tensor product on $\Pr^\L$ & the relative tensor product on $\LMod_\mV(\Pr^\L)$ \\
\hline
the $\infty$-category $\mP(\mC)$ of presheaves & the $\mV$-enriched $\infty$-category $\mP_\mV(\mC)$ of $\mV$-enriched presheaves \\
\hline
a set $\mK$ of small $\infty$-categories & a set $\mH$ of $\mV$-enriched presheaves (=$\mV$-weights) \\
%\hlinethe $\infty$-category $\Cat_\infty(\mK)$ & the $\infty$-category $\Cat^\mV_\infty(\mH)$ \\
\hline
\end{tabular}
\end{center}

\vspace{1mm}
To understand this table, we first need to explain the concept of weighted colimit.
%To make this analogies formal, we build a theory of weighted colimits and weighted limits in the $\infty$-categorical setting, the analogon of colimits and limits in enriched category theory.
For that we draw an analogy to colimits in the non-enriched setting.
A cocone on a functor $\F: \mJ \to \mD$ between $\infty$-categories
is an object $\X \in \mD$ equipped with a natural transformation from $\F$ to
the constant diagram on $\X,$ which is the same datum as a map of
presheaves on $\mJ$ from the final presheaf to the presheaf $\mD(\F(-),\X).$
The colimit of a functor $\F: \mJ \to \mD$, if it exists, is a cocone
$ \colim(\F)$ on $\F$ inducing for every object $\Y$ of $\mD$ an
equivalence $\mD(\colim(\F),\Y) \to \lim(\mD(\F(-),\Y)).$
The notion of weighted colimit arises by replacing the notion of cocone by the notion of weighted cocone: if $\mJ,\mD$ are $\infty$-categories enriched in a presentably monoidal $\infty$-category $\mV$ and $\F$ is a $\mV$-enriched functor $\mJ \to \mD$, it is natural to replace presheaves on $\mJ$ by $\mV$-enriched presheaves on $\mJ$
and the presheaf $\mD(\F(-),\X)$ on $\mJ$ by the $\mV$-enriched presheaf 
$\Mor_\mD(\F(-),\X)$, where $\Mor_\mD(-,-)$ is the morphism object in $\mV.$
Moreover it is natural to replace the final presheaf on $\mJ$ by the
constant $\mV$-enriched presheaf on the tensor unit of $\mV.$
But in this situation the following problem arises: there is generally no constant
$\mV$-enriched presheaf on $\mJ$ since there is generally no $\mV$-enriched functor
from $\mJ$ to the $\mV$-enriched $\infty$-category $B\tu_\mV$ classifying objects.
%that serves as an analogon for the final presheaf in the enriched world. So one could try 
The solution to this problem is to view any enriched presheaf as a possible analogon for the final presheaf and to consider colimits relative to a $\mV$-enriched presheaf on $\mJ$, which one calls a $\mV$-weight on $\mJ$ in this context. 
In analogy to the case of colimits a $\mV$-enriched cocone on a $\mV$-enriched functor $\F: \mJ \to \mD$ is an object $\X \in \mD$ equipped with a map $\rH \to \Mor_\mD(\F(-),\X)$ of $\mV$-enriched presheaves on $\mJ$. To refer to $\rH$ we call such a $\mV$-enriched cocone on $\F$ a $\rH$-weighted cocone on $\F.$
The $\rH$-weighted colimit of a $\mV$-enriched functor $\F: \mJ \to \mD$ if it exists, is a $\rH$-weighted cocone $\colim^\rH(\F)$ on $\F$ that induces for every object $\X$ of $\mD$ an equivalence
$\Mor_\mD(\colim^\rH(\F),\Y) \to \Mor_{\mP_\mV(\mJ)}(\rH,\Mor_\mD(\F(-),\X))$.
We call a $\mV$-weight on a $\mV$-enriched $\infty$-category $\mJ$ small if
$\mJ$ is a small $\mV$-enriched $\infty$-category meaning that the space of objects is small. If any $\mV$-enriched functor $\F:\mJ \to \mD$ admits a $\rH$-weighted colimit, one says that $\mD$ admits $\rH$-weighted colimits. For any collection $\mH$ of $\mV$-weights we say that $\mD$ admits $\mH$-weighted colimits if it admits $\rH$-weighted colimits for every $\rH \in \mH.$
%In analogy to the non-enriched situation one asks when a $\mV$-enriched functor preserves a given $\rH$-weighted colimit, and considers the subcategory $\Cat_\infty^\mV(\mH)\subset \Cat_\infty^\mV$ of the $\infty$-category $\Cat_\infty^\mV$ of small $\mV$-enriched $\infty$-categories whose objects are $\mV$-enriched $\infty$-categories that admit $\mH$-weighted colimits and whose morphisms are $\mV$-enriched functors preserving $\mH$-weighted colimits.

In nature weighted colimits are ubiquitious.
To get a feeling for this concept, it is reasonable to consider two important classes of weighted colimits, which are more elementary to understand, and to think of arbitrary weighted colimits as an interpolation between these two classes.
The first sort of weighted colimits are conical colimits,
which are the natural enriched analogon of colimits of diagrams in an enriched $\infty$-category indexed by a non-enriched $\infty$-category.
%An important example of interest of conical colimits in homotopy theory is the following one: 
%For example any stable $\infty$-category $\mC$ has a canonical enrichment in spectra, and the loops of any object $\X$ in $\mC$, i.e. the pullback $0 \times_\X 0$ in $\mC$, is a conical pullback in the spectral $\infty$-category underlying $\mC$and so the invertible adjunction $\Sigma: \mC \rightleftarrows\mC:\Omega$ of suspension and loops refines to a spectral adjunction.
Given a $\mV$-enriched $\infty$-category $\mD$ and a functor
$\F$ from an $\infty$-category $\K$ to the underlying $\infty$-category of $\mD$
%arising by forgetting the enrichment, 
one can consider the colimit $\colim(\F)$
of $\F$ satisfying the universal property that the natural map on mapping spaces
$$ \mD(\colim(\F),\Y) \to \lim(\mD(\F(-),\Y)$$ is an equivalence.
%, which corepresents the functor$\mD \to \mS, \X \mapsto \mP(\mJ)(\ast, \mD(\F(-),\X)) \simeq \lim(\mD(\F(-),\X))$.
But one can also use the enrichment and ask %Since $\mD$ is a $\mV$-enriched $\infty$-category, it seems reasonable to use themorphism objects of $\mD$ and ask 
for the stronger condition that the natural morphism on morphism objects 
$$ \Mor_\mD(\colim(\F),\Y) \to \lim(\Mor_\mD(\F(-),\Y))$$ is an equivalence.
If this holds, we call $\F$ the conical colimit of $\F.$
For example any stable $\infty$-category $\mC$ has a canonical enrichment in spectra \cite[Theorem 8.1.]{HEINE2023108941}. The loops of any object $\X$ in $\mC$, i.e. the pullback $0 \times_\X 0$ in $\mC$, is a conical pullback in the spectral $\infty$-category underlying $\mC$, which implies that the invertible adjunction $\Sigma: \mC \rightleftarrows\mC:\Omega$ of suspension and loops refines to a spectral adjunction.

A functor $\F$ from an $\infty$-category $\K$ to the underlying $\infty$-category of $\mD$ %from an 
%which do not suffer from the problem that there is no constant $\mV$-enriched presheaf.
%which are exclusively defined for functors $\F$ from an 
%$\infty$-category $\K$ to the underlying $\infty$-category of a $\mV$-enriched $\infty$-category $\mD$ 
corresponds to a 
%, which arises from $\mD$ by forgetting the enrichment.
%Such a functor $\F$ may be equivalently seen as a 
$\mV$-enriched functor $\F': \K_\mV \to \mD$ from the $\mV$-enriched $\infty$-category $\K_\mV$ arising from $\K$ by tensoring the mapping spaces of $\K$ with the tensor unit of $\mV.$
Since $\mV$-enriched presheaves on $\K_\mV$ are identified with non-enriched $\mV$-valued presheaves on $\K$, there is a constant $\mV$-enriched presheaf on any object of $\mV$ and the conical colimit of $\F$ is the 
%which gives rise to the definition of conical colimit:
%Since there is a unique functor from $\K \to *$ to the final $\infty$-category,there is a canonical $\mV$-enriched functor $\K_\mV \to *_\mV = B\tu_\mV$giving rise to a constant $\mV$-enriched presheaf.the conical colimit of $\F$ is the 
colimit of $\F'$ weighted with respect to the constant $\mV$-enriched presheaf on the tensor unit of $\mV$.
%By definition the conical colimit of $\F$ corepresents the $\mV$-enriched functor$\mD \to \mV, \X \mapsto \Mor_{\mV^{\K^\op}}(\tu_\mV,\Mor_\mD(\F(-),\X)) \simeq \lim(\Mor_\mD(\F(-),\X))$, and so arises from the non-enriched definition of colimit by replacing the mapping spaces of $\mD$ in the definition of colimit by the morphism objects of $\mD.$
%An important example of interest of conical colimits in homotopy theory is the following one: any stable $\infty$-category $\mC$ has a canonical enrichment in spectra.The loops of any object $\X$ in $\mC$, i.e. the pullback $0 \times_\X 0$ in $\mC$,is a conical pullback in the spectral $\infty$-category underlying $\mC$and so the invertible adjunction $\Sigma: \mC \rightleftarrows\mC:\Omega$ of suspension and loops refines to a spectral adjunction.
%and so the invertible loops functor $\Omega:\mC \to \mC$ refines to an invertible spectral functor.

The second sort of weighted colimits are tensors,
%which do not suffer from the problem that there is no constant $\mV$-enriched presheaf.
exclusively defined for $\mV$-enriched functors $B \tu_\mV \to \mD$,
which classify objects of $\mD$. Similarly, $\mV$-enriched presheaves on $B\tu_\mV$ classify objects of $\mV$ and the tensor of an object $\V \in \mV$ with an object $\Y \in \mD$ is the colimit of the $\mV$-enriched functor $B \tu_\mV \to \mD$ classifying $\Y$ weighted with respect to the $\mV$-enriched presheaf on $B\tu_\mV$ classifying $\V.$ So by definition the tensor of $\V$ and $\Y$ corepresents the $\mV$-enriched functor
$\mD\to \mV, \X \mapsto \Mor_{\mV}(\V,\Mor_\mD(\Y,\X))$.
If all tensors exist, the $\mV$-enrichment endows the underlying $\infty$-category of $\mD$ with a closed left $\mV$-action whose internal hom recovers the $\mV$-enrichment \cite[Theorem 6.7.]{HEINE2023108941}.

Another example of weighted (co)limits appears in higher category theory
%the study of $(\infty,\n)$-categories for $2 \leq \n \leq \infty$ 
under the names lax and oplax (co)limits and lax and oplax pushouts and pullbacks (Definition \ref{Laax}, Definition \ref{Puush}).
%Classically called lax and oplax (co)limits and studied $\infty$-categorically for diagrams in the $(\infty,2)$-category $\Cat_\infty$ indexed by a non-enriched $\infty$-category, lax and oplax (co)limits are the natural extension of (co)limitsof diagrams in a higher category, in which non-invertible morphisms exist for $\n >1.$
Given functors $\alpha: \mA \to \mC, \beta: \mB \to \mC$ of $\infty$-categories %in the $(\infty,\n+1)$-category $\Cat_{(\infty,\n)}$ of $(\infty,\n)$-categories for $1 \leq \n \leq \infty$ 
one can consider the pullback $\mA \times_\mC \times \mB$ whose objects are
triples $(\A\in \mA,\B\in \mB, \gamma: \alpha(\A) \simeq \beta(\B))$.
% which only uses the underlying $(\infty,1)$-category of the $(\infty,\n+1)$-category $\Cat_{(\infty,\n)}$.
But one can also consider the lax pullback $\mA \times^\lax_\mC \times \mB$ whose objects are triples $(\A\in \mA,\B\in \mB, \gamma: \alpha(\A) \to \beta(\B))$
as well as the oplax pullback $\mA \times^\oplax_\mC \times \mB$ whose objects are triples $(\A\in \mA,\B\in \mB, \beta(\A) \to \alpha(\B)).$ 
In higher category theory the lax and oplax versions of the pullback are drastically more useful since they make use
%Compared to the pullback the lax and oplax pullbacks make 
of non-invertible morphisms:
the pullback of two functors $* \to \mC$ classifying objects $\A,\B$ of $\mC$ is the space of
equivalences between $\A $ and $\B$ while the lax pullback is the space of morphisms $\A \to \B$ in $\mC$ and the oplax pullback is the space of morphisms $\B \to \A$ in $\mC$.

%for example morphisms $(\A,\B, \gamma) \to (\A',\B',\gamma') $ in the lax pullback $\mA \times^\lax_\mC \times \mB$ are triples consisting of morphisms $\f:\A \to \A'$ in $\mA, \g:\B \to \B'$ in $\mB$ and a 2-morphism $\gamma' \circ \alpha(\f) \to \beta(\g)\circ \gamma$ in $\mC.$
%Since higher morphisms of the lax pullback use higher morphisms of $\mC$, the lax pullback uses the
(Op)lax pullbacks can be described as weighted colimits for a simple weight:
the (op)lax pullback $\mA \times^\lax_\mC \times \mB$ is the limit of
the diagram $\mA \to \mC, \mB \to \mC$ of $\infty$-categories corresponding to a functor $\F: \{0,1\}^\triangleright \to \Cat_{\infty}$ weighted with respect to the functor $\rH: \{0,1\}^\triangleright \to \Cat_{\infty}$ corresponding to the diagram $\{0\} \subset [1] \supset \{1\},$ where we view $\F$ and $\rH$ as functors enriched in
$\Cat_\infty$. % $\mathrm{Gray}^{(\op)\lax}_\n$. % $\mathrm{Gray}^\oplax_\n,$
Similarly, 
%Other examples in these lines are lax and 
the oplax (co)limit of a functor $\mJ \to \mC$ of
$(\infty,\n)$-categories for any $\n \geq 2$, studied in \cite{articles} for $\mJ$ an $\infty$-category, is the colimit weighted with respect to a $\Cat_\infty$-enriched weight on $\mJ$ that assigns to any object $\Z$ of $\mJ$ an (op)lax version of the slice $(\infty, \n-1)$-category $\mJ_{\Z/}.$

%can be expressed as weighted (co)limits in our framework.
%Another motivating example of interest of weighted (co)limits appears in the study of $(\infty,\n)$-categories for $2 \leq \n \leq \infty$.
%Classically called lax and oplax (co)limits and studied $\infty$-categorically for diagrams in the $(\infty,2)$-category $\Cat_\infty$ indexed by a non-enriched $\infty$-category, lax and oplax (co)limits are the natural extension of (co)limitsof diagrams in a higher category, in which non-invertible morphisms exist for $\n >1.$
%Having these examples in mind, we can explain table ....
%For any presentably monoidal $\infty$-category $\mV$ let $\Cat_\infty^\mV$be the $\infty$-category of $\mV$-enriched $\infty$-categoriesand 

In this work we build a theory of weighted colimits and limits that extends the classical theory of weighted (colimits) from enriched category theory to enriched $\infty$-category theory. 
Among other we prove the following key results:
\begin{enumerate}

\item We prove an existence result of weighted colimits that splits the existence
of weighted colimits in the existence of conical colimits and tensors (Proposition \ref{weeei}).

\item We prove that enriched left adjoints preserve weighted colimits, enriched right adjoints preserve weighted limits (Proposition \ref{remqalo}) and prove a converse: an enriched analogon of the adjoint functor theorem (Theorem \ref{remwq}).

\item We characterize the $\mV$-enriched $\infty$-category of $\mV$-enriched presheaves on a small $\mV$-enriched $\infty$-category as the free cocompletion under small weighted colimits
(Corollary \ref{Yowei}).

\item We prove that the enriched Yoneda-embedding preserves weighted limits (Corollary \ref{yow}).

%\item We develop a theory of enriched kan extensions and prove that weighted colimits ensure the existence of enriched left kan extensions, while dually weighted limits ensure the existence of enriched right kan extensions.

\item We study enriched coends as examples of weighted colimits and prove a coend formula for 
weighted colimits % can be constructed from tensors and weighted coends
(Theorem \ref{heuu}).

\item We prove a Bousfield-Kan formula for weighted colimits that computes weighted colimits as a geometric realization of colimits indexed by spaces (Theorem \ref{BK}).

\item We study enriched functoriality of weighted colimits in the weight and in the enriched functor (Proposition \ref{weifun}, Corollary \ref{Funf}).

\item We prove that weighted colimits in enriched $\infty$-categories of enriched functors are formed object-wise (Theorem \ref{Enric}).

\item We characterize $\infty$-categories presentably left tensored over $\mV$ as $\mV$-presentable $\infty$-categories, an enriched analogon of presentability expressed via weighted colimits (Theorem \ref{Pree}).

\item We develop a framework of universally adjoining weighted colimits to an enriched $\infty$-category (Theorem \ref{wooond}).

\item We study absolute weighted colimits, those weighted colimits preserved by any $\mV$-enriched functor and construct the Cauchy-completion by universally adjoining absolute weighted colimits to an enriched $\infty$-category (Notation \ref{CC}, Proposition \ref{Univop}).

\end{enumerate}

To prove these results we use a theory of bienriched $\infty$-categories of \cite{heine2024bienriched} that extends enriched $\infty$-category theory and is especially useful to control enrichment on $\infty$-categories of enriched functors, which is crucial to prove (7). Consequently, we develop our theory of weighted colimits in the more general framework of bienriched $\infty$-categories.
However, our definition of weighted colimits specializes to the definition of weighted colimits for enriched $\infty$-categories of \cite[§ 6]{hinich2021colimits} (Remark \ref{sosos}).
%Feature (4) is especially important in homotopy theory, where structure on an object is typically encoded by a functor, and thus weighted colimits of diagrams of enriched functors become meaningful.

The theory of universally adjoining weighted colimits to an enriched $\infty$-category gives the following theorem,
which may be viewed as an enriched version of \cite[Proposition 5.3.6.2.]{lurie.HTT}:

\begin{theorem}\label{wooond007}(Theorem \ref{wooond}, Corollary \ref{wooond0000}, Proposition \ref{quirk}) Let $\mV$ be a presentably monoidal $\infty$-category, $\mC$ a small $\mV$-enriched $\infty$-category, $\mH$ a set of small $\mV$-weights and $\Lambda$ a set of $\mH$-weighted cocones in $\mC$.
There is a large but locally small $\mV$-enriched $\infty$-category $\mP\B\Env^{\mH}_{\Lambda}(\mC)$ that admits $\mH$-weighted colimits 
and a $\mV$-enriched functor $\mC \to \mP\B\Env^{\mH}_{\Lambda}(\mC)$
that induces for every $\mV $-enriched $\infty$-category $\mD$ that admits $\mH$-weighted colimits an equivalence $$\Enr\Fun_{\mV}^{\mH}(\mP\B\Env^{\mH}_{\Lambda}(\mC), \mD) \to \Enr\Fun_{\mV}^{\Lambda}(\mC, \mD)$$	
between the $\infty$-category of $\mV$-enriched functors preserving
$\mH$-weighted colimits and the $\infty$-category of $\mV$-enriched functors turning $\mH$-weighted cocones in $\mC$ to weighted colimits.
If $\Lambda$ consists of $\mH$-weighted colimiting cocones, the $\mV$-enriched functor $\mC \to \mP\B\Env^{\mH}_{\Lambda}(\mC)$ is fully faithful.
	
\end{theorem}

We use Theorem \ref{wooond007} to construct a tensor product for $\mV$-enriched $\infty$-categories that admit weighted colimits for a fixed set of weights:

\begin{theorem}\label{thh}(Theorem \ref{thqaz} (1))
Let $\mV$ be a presentably $\bE_{\bk+1}$-monoidal $\infty$-category for $0 \leq \bk \leq \infty$ and $\mH$ a set of small $\mV$-weights.
There is a presentably $\bE_\bk$-monoidal structure on the $\infty$-category $$\Enr_{\mV}(\mH)$$
of small $\mV$-enriched $\infty$-categories that admit $\mH$-weighted colimits and $\mV$-enriched functors preserving $\mH$-weighted colimits.
For every $\mC,\mD\in\Enr_{\mV}(\mH) $ the internal hom is the $\mV$-enriched $\infty$-category of $\mV$-enriched functors $\mC \to \mD$ preserving $\mH$-weighted colimits.

%For every inclusion $\mH' \subset \mH$ of sets of small $\mV$-weights the inclusion $$\Cat_\infty^{\mV}(\mH) \subset \Cat_\infty^{\mV}(\mH')$$ admits an $\bE_{\bk}$-monoidal left adjoint.
\end{theorem}
%construct for every $\bE_{\bk+1}$-monoidal $\infty$-category $\mV$ for $\bk \geq1$ a closed $\bE_{\bk}$-monoidal structure on the $\infty$-category $\Cat_\infty^\mV(\mH)$
%which is defined as the subcategory of the $\infty$-category $\Cat_\infty^\mV$ of $\mV$-enriched $\infty$-categories whose objects are of $\mV$-enriched $\infty$-categories having $\mH$-weighted colimits and $\mV$-enriched functors preserving $\mH$-weighted colimits.

Specializing Theorem \ref{thh} to absolute $\mV$-weights, i.e. $\mV$-weights preserved by any $\mV$-enriched functor, we obtain a tensor product for Cauchy-complete $\mV$-enriched $\infty$-categories:
\begin{corollary}\label{zzi}(Corollary \ref{CaC}) Let $\n \geq 1$ and $\mV^\boxtimes \to \bE_{\n+1}$ a presentably $\bE_{\n+1}$-monoidal $\infty$-category. % and $\mH$ a set of left enriched weights over $\mV.$
The $\infty$-category $ {_\mV\L\Enr^\wedge}$ of Cauchy-complete $\mV$-enriched $\infty$-categories carries a canonical presentably $\bE_\n$-monoidal structure
such that Cauchy-completion ${_\mV\L\Enr} \to {_\mV\L\Enr^\wedge}$ refines to an $\bE_\n$-monoidal functor.
%such that the embedding $${_\mV\L\Enr}(\mH)^\ot \subset {_\mV\L\Q\Enr}(\mH)^\ot$$ is $\bE_\n$-monoidal.	
	
\end{corollary}

Applying Corollary \ref{zzi} to the presentably symmetric monoidal $\infty$-category of spectra we obtain a presentably symmetric monoidal structure for Cauchy-complete spectral $\infty$-categories.

Specializing Theorem \ref{thh} to (op)lax colimits and (op)lax pushouts we obtain for any $\n \geq 1$ and set $\mK$ of small $(\infty,\n)$-categories a tensor product for $(\infty,\n)$-categories that admit $\mK$-indexed (op)lax colimits and a tensor product for $(\infty,\n)$-categories that admit (op)lax pushouts:
\begin{corollary}
Let $\n \geq 1$ and $\mK$ be a set of $(\infty,\n)$-categories.
There is a presentably symmetric monoidal structure on the $\infty$-category
$$ \Cat_{(\infty,\n)}^{\mathrm{(op)lax}}(\mK)$$	
of $(\infty,\n)$-categories that admit $\mK$-indexed (op)lax colimits and functors preserving $\mK$-indexed (op)lax colimits.
	
\end{corollary}	

\begin{corollary}Let $\n \geq 1$.
There is a presentably symmetric monoidal structure on the $\infty$-category
$$ \Cat_{(\infty,\n)}^{\mathrm{push}, \mathrm{(op)lax}}$$	
of $(\infty,\n)$-categories that admit (op)lax pushouts and functors preserving (op)lax pushouts.
	
\end{corollary}	

In Theorem \ref{thh} the condition of smallness on a $\mV$-enriched $\infty$-category is crucial: there is no tensor product of not necessarily small $\mV$-enriched $\infty$-categories
having $\mH$-weighted colimits. The reason is already visible for enrichment in the cartesian  monoidal $\infty$-category $\mS$ of small spaces: there is no tensor product on the $\infty$-category of not necessarily small $\mS$-enriched, i.e. locally small,
$\infty$-categories having small colimits and functors preserving small colimits
since any such tensor product destroys local smallness.
The remedy is to consider presentable $\infty$-categories, which are locally small and to which the tensor product restricts, or to consider locally large $\infty$-categories having small colimits whose very large $\infty$-category carries a closed symmetric monoidal structure.
A locally large $\infty$-category may be viewed as an $\infty$-category enriched in the $\infty$-category of large spaces, which arises from the $\infty$-category of small spaces $\mS$ by formally adding large filtered colimits.
In analogy we may think of a locally large $\mV$-enriched $\infty$-category as an $\infty$-category enriched in the formal completion of $\mV$ under large filtered colimits. 
Lurie \cite[Definition 4.2.1.25.]{lurie.higheralgebra} introduces for every monoidal $\infty$-category $\mV$ a notion of $\infty$-category pseudo-enriched in $\mV$, which we identify with an $\infty$-category enriched in the $\infty$-category of presheaves on $\mV$ equipped with Day-convolution \cite[Corollary 4.22.]{heine2024bienriched}.
Similarly, in \cite[Definition 3.29.]{heine2024bienriched} we introduce for every monoidal $\infty$-category $\mV$ compatible with small colimits a notion of $\infty$-category quasi-enriched in $\mV$, which we identify with an $\infty$-category enriched in the $\infty$-category of presheaves on $\mV$ preserving small limits, a model for the formal completion of $\mV$ under large filtered colimits. 
Consequently, we may view an  $\infty$-category quasi-enriched in $\mV$ as a locally large $\mV$-enriched $\infty$-category.
We prove analogues of Theorem \ref{thh} for quasi-enriched and pseudo-enriched $\infty$-categories:

\begin{theorem}\label{thhq} (Theorem \ref{thqaz} (2) and (3))
Let $0 \leq \bk \leq \infty$.	
	
\begin{enumerate} 
\item Let $\mV$ be an $\bE_{\bk+1}$-monoidal $\infty$-category compatible with small colimits and $\mH$ a collection of not necessarily small $\mV$-weights.
There is a closed $\bE_\bk$-monoidal structure on the $\infty$-category $$\Q\Enr_{\mV}(\mH)$$
of not necessarily small $\infty$-categories quasi-enriched in $\mV$ that admit $\mH$-weighted colimits and $\mV$-enriched functors preserving $\mH$-weighted colimits.

\item Let $\mV$ be a small $\bE_{\bk+1}$-monoidal $\infty$-category and $\mH$ a set of small $\mV$-weights.
There is a presentably $\bE_\bk$-monoidal structure on the $\infty$-category $$\P\Enr_{\mV}(\mH)$$
of small $\infty$-categories pseudo-enriched in $\mV$ that admit $\mH$-weighted colimits and $\mV$-enriched functors preserving $\mH$-weighted colimits.
%For every $\mC,\mD\in\Q\Enr_{\mV}(\mH) $ the internal hom is the $\mV$-enriched $\infty$-category of $\mV$-enriched functors $\mC \to \mD$ preserving $\mH$-weighted colimits.
\end{enumerate}		
%For every inclusion $\mH' \subset \mH$ of sets of small $\mV$-weights the inclusion $$\Cat_\infty^{\mV}(\mH) \subset \Cat_\infty^{\mV}(\mH')$$ admits an $\bE_{\bk}$-monoidal left adjoint.
\end{theorem}

For $\mH$ the collection of all small $\mV$-weights the $\infty$-category $\Q\Enr_\mV(\mH)$ is the $\infty$-category $\Q\Enr_\mV^{\rc\rc}$ of large $\infty$-categories quasi-enriched in $\mV$ that admit small weighted colimits, which by Theorem \ref{thhq} admits a closed $\bE_\bk$-monoidal structure.
%We identify with the $\infty$-category of $\infty$-categories with closed left $\mV$-action compatible with small colimits.
We identify $\Q\Enr_\mV^{\rc\rc} $ with the $\infty$-category $\LMod_\mV(\Cat_\infty^{\rc\rc}) $ of $\infty$-categories left tensored over $\mV$ compatible with small colimits (Example \ref{rfgp}).
%and prove that the $\bE_\bk$-monoidal structure on $\Q\Enr_\mV^{\rc\rc}$ restricts to the full subcategory $\LMod_\mV(\Pr^\L)$ of $\infty$-categories presentably left tensored over $\mV,$ an enriched analogon of the fact that the tensor product on $\Cat_\infty^{\rc\rc}$ restricts to $\Pr^\L:$
%There is a canonical equivalence $\Pr^\L_\mV \simeq \LMod_\mV(\Pr^\L)$and we identify the restricted $\bE_{\bk}$-monoidal structure on $\Pr_\mV^\L$ withthe relative tensor product on $\LMod_\mV(\Pr^\L).$
%On the other hand the relative tensor product on $\LMod_\mV(\Cat_\infty^{\rc\rc})$  also restricts to $\LMod_\mV(\Pr^\L)$.
%In fact there is a commutative square of $\bE_\bk$-monoidal $\infty$-categories and lax $\bE_\bk$-monoidal inclusions, where both vertical maps are $\bE_\bk$-monoidal embeddings:
%\begin{equation*} 
%\begin{xy}
%\xymatrix{
%\Pr^\L_\mV  \ar[d]^{} \ar[rr]^{ \simeq} &&  \LMod_\mV(\Pr^\L) \ar[d]
%\\ \Cat_\infty^{\mV, \rc\rc} \ar[rr] &&\LMod_\mV(\Cat^{\rc\rc}_\infty).
%}
%\end{xy} 
%\end{equation*}
%\LMod_\mV(\Pr^\L)\simeq \Pr^\L_\mV \subset 
For general set $\mH$ of $\mV$-weights the tensor product on $\Q\Enr_\mV(\mH)$ interpolates between the relative tensor product on $\LMod_\mV(\Cat_\infty^{\rc\rc}) \simeq \Q\Enr_\mV^{\rc\rc}$ and a tensor product of $\infty$-categories quasi-enriched in $\mV$: %an inclusion $\mH' \subset \mH$ of sets of small $\mV$-weights gives an inclusion $\Cat_\infty^{\mV}(\mH) \subset \Cat_\infty^{\mV}(\mH')$,which admits an $\bE_{\bk}$-monoidal left adjoint,
the inclusion $\emptyset \subset \mH \subset \{\mathrm{small} \hspace{1mm} \mV \mathrm{-weights} \} $ gives lax $\bE_\bk$-monoidal inclusions
$$\LMod_\mV(\Cat_\infty^{\rc\rc})\subset \Q\Enr_\mV(\mH) \subset \Q\Enr_\mV=\Q\Enr_\mV(\emptyset).$$
%that restrict to lax $\bE_\bk$-monoidal inclusions$$\LMod_\mV(\Pr^\L)\subset \widehat{\Enr}_\mV(\mH) \subset \widehat{\Enr}_\mV=\widehat{\Enr}_\mV(\emptyset).$$
The $\bE_{\bk}$-monoidal left adjoint of the inclusion $\LMod_\mV(\Cat_\infty^{\rc\rc}) \subset \Q\Enr_\mV$ sends any small $\mV$-enriched $\infty$-category $\mC$ to the presentably left tensored $\infty$-category $\mP_\mV(\mC)$ of $\mV$-enriched presheaves on $\mC$ and so restricts to an $\bE_\bk$-monoidal functor $\Enr_\mV \to \LMod_\mV(\Pr^\L)$
to the $\infty$-category of $\infty$-categories presentably left tensored over $\mV.$
This way we obtain the following corollary, which connects the higher algebra of
$\mV$-enriched $\infty$-categories with the considerably more tractable higher algebra in the $\infty$-category $\Pr^\L$ of presentable $\infty$-categories:

\begin{corollary}\label{Thy}(Corollary \ref{Preml})
Let $\mV$ be a presentably $\bE_{\bk+1}$-monoidal $\infty$-category for $1 \leq \bk \leq \infty$. The functor $$\mP_\mV: \Enr_\mV \to \LMod_\mV(\Pr^\L)$$ of $\mV$-enriched presheaves %sending a $\mV$-enriched $\infty$-category $\mC$ to the $\mV$-enriched $\infty$-category of $\mV$-enriched presheaves on $\mC$
is an $\bE_{\bk}$-monoidal functor, where the left hand side carries the tensor product of $\mV$-enriched $\infty$-categories and the right hand side carries the relative tensor product. 
	
\end{corollary}
By Corollary \ref{Thy} the functor of $\mV$-enriched presheaves is $\bE_\bk$-monoidal and so preserves $\bE_\bk$-algebras. % which in $\Enr_\mV$ are $\bE_\bk$-monoidal $\mV$-enriched $\infty$-categories and in $\LMod_\mV(\Pr^\L)$are presentably $\bE_\bk$-monoidal $\infty$-categories equipped with an $\bE_\bk$-monoidal left $\mV$-action. 
As a consequence the $\mV$-enriched $\infty$-category of $\mV$-enriched presheaves on any $\bE_{\bk}$-monoidal $\mV$-enriched $\infty$-category is an $\bE_\bk$-algebra in $\LMod_\mV(\Pr^\L)$ and the $\mV$-enriched Yoneda-embedding is $\bE_\bk$-monoidal.
In particular, for $\bk=\infty$ the $\mV$-enriched $\infty$-category of $\mV$-enriched presheaves on a symmetric monoidal $\mV$-enriched $\infty$-category is a $\mV$-algebra in $\Pr^\L$, i.e. a presentably symmetric monoidal $\infty$-category under $\mV$.
Hinich \cite[Proposition 8.4.3.]{hinich2021colimits} constructs for every $\bE_{\bk+1}$-monoidal $\infty$-category $\mV$ an $\bE_\bk$-monoidal $\mV$-enriched Yoneda-embedding, which satisfies the $\bE_\bk$-monoidal version of free cocompletion and therefore is equivalent to our version of $\bE_\bk$-monoidal $\mV$-enriched Yoneda-embedding
up to an $\bE_\bk$-monoidal automorphism that is possibly not the identity.
This gives an answer to a question of \cite[Question 1.8.]{BENMOSHE2024107625}
and proves a monoidal version of a conjecture of \cite[Conjecture 1.1.]{berman2020enriched}.

We combine Corollary \ref{Thy} with the following theorem to construct a $\mV$-enriched version of Day-convolution:
Lurie \cite[Proposition 4.8.1.17.]{lurie.higheralgebra} gives an explicite formula for the tensor product of presentable $\infty$-categories by proving that for every presentable $\infty$-categories $\mC,\mD$ there is a canonical equivalence
$$\mC \ot \mD \simeq \Fun^\R(\mC^\op,\mD),$$
where the right hand side is the $\infty$-category of right adjoint functors $\mC^\op $ to $\mD.$
We prove an enriched version of Lurie's result:
\begin{theorem}\label{dda}(Theorem \ref{prestensor})
Let $\mV$ be a presentably symmetric monoidal $\infty$-category and $\mC,\mD \in \LMod_\mV(\Pr^\L).$
There is a canonical $\mV$-enriched equivalence
$$\mC \ot_\mV \mD \simeq \Enr\Fun_\mV^{\R}(\mC^\op,\mD),$$
where the right hand side is the $\mV$-enriched $\infty$-category of right adjoint $\mV$-enriched functors $\mC^\op $ to $\mD.$
		
\end{theorem}

Specializing to the $\mV$-enriched $\infty$-category of $\mV$-enriched presheaves we obtain the following corollary:

\begin{corollary}\label{ddaa}(Corollary \ref{Dayso})
Let $\mV$ be a presentably symmetric monoidal $\infty$-category and $\mC$ a small $\mV$-enriched $\infty$-category. 
There is a canonical $\mV$-enriched equivalence
$$\mP_\mV(\mC) \ot_\mV \mD \simeq \Enr\Fun_\mV^{}(\mC^\op,\mD).$$
\end{corollary}

%Theorem \ref{dda} gives for every small $\mV$-enriched $\infty$-category $\mC$ and $\infty$-category $\mD$ presentably left tensored over $\mV$ a canonical $\mV$-enriched equivalence, where the third equivalence is the universal property of free cocompletion under weighted colimits: $$ \mP_\mV(\mC)\ot_\mV \mD \simeq \Enr\Fun_{\mV}^{\R}(\mP_\mV(\mC)^\op,\mD) \simeq \Enr\Fun_{\mV}^{\L}(\mP_\mV(\mC),\mD^\op)^\op$$$$\simeq \Enr\Fun_{\mV}(\mC,\mD^\op)^\op \simeq \Enr\Fun_\mV(\mC^\op,\mD).$$

Using that $\bE_\infty$-algebras in $\LMod_\mV(\Pr^\L)$ are $\mV$-algebras in $\Pr^\L$ we obtain the following corollary:

\begin{corollary}\label{C1} (Corollary \ref{Dayco2})
Let $\mV$ be a presentably symmetric monoidal $\infty$-category, $\mC$ a small symmetric monoidal $\mV$-enriched $\infty$-category and
$\mD$ a $\mV$-algebra in $\Pr^\L$.
Then $\Enr\Fun_\mV(\mC^\op,\mD)$ carries the structure of a $\mV$-algebra in $\Pr^\L.$

\end{corollary}

We use Corollary \ref{C1} to characterize the relative tensor product for $\infty$-categories presentably left tensored over $\mV$ as a $\mV$-enriched localization of the $\mV$-enriched $\infty$-category of $\mV$-enriched presheaves. For that we characterize $\infty$-categories presentably left tensored over $\mV$ as $\mV$-enriched presentable $\infty$-categories:
for any small regular cardinal $\kappa$ a $\mV$-enriched $\infty$-category $\mM$ is $\mV$-enriched $\kappa$-presentable (Definition \ref{presss}) if it admits small weighted colimits and every object of $\mM$ is a small $\kappa$-filtered conical colimit of $\mV$-enriched $\kappa$-compact objects of $\mM$, where an object $\X $ of $\mM$ is $\mV$-enriched $\kappa$-compact if the morphism object functor $\Mor_\mM(\X,-):\mM \to \mV$ preserves small $\kappa$-filtered conical colimits.
We characterize $\mV$-enriched $\kappa$-presentable $\infty$-categories:
\begin{theorem}\label{Thh3}(Theorem \ref{Pree})
Let $\tau\leq \kappa$ be small regular cardinals, $\mV $ a $\tau$-compactly generated monoidal $\infty$-category and $\mM$ a $\mV$-enriched $\infty$-category. The following conditions are equivalent:
	
\begin{enumerate}
\item The $\mV$-enriched $\infty$-category $\mM$ is $\mV$-enriched $\kappa$-presentable.

\item The $\mV$-enriched $\infty$-category $\mM$ is a $\kappa$-compactly generated left tensored $\infty$-category.
		
\item There is a small $\mV$-enriched $\infty$-category $\mN$ and a $\kappa$-accessible $\mV$-enriched localization $\mM \rightleftarrows \mP_\mV(\mN)$.
	
\end{enumerate}
	
\end{theorem}
We use Corollary \ref{Thy} and Theorem \ref{Thh3} to deduce the following corollary:

\begin{corollary} (Theorem \ref{prestensor}) Let $\mV$ be a presentably symmetric monoidal $\infty$-category and $\mM,\mN$ presentably left tensored $\infty$-categories. 
By Theorem \ref{Thh3} there are accessible $\mV$-enriched localizations
$\mP_\mV(\mC)\rightleftarrows \mM, \mP_\mV(\mD)\rightleftarrows \mN$
with respect to sets $\rS,\T$ closed under the left actions, respectively.
%There are small $\mV$-enriched $\infty$-categories $\mC \to \mD$, sets $\rS$ of morphisms of $\mP_\mV(\mC)$ and $\T$ of morphisms of $\mP_\mV(\mD)$ closed under the left actions and left $\mV$-enriched equivalences $\mM^\circledast \simeq \rS^{-1}\mP_\mV(\mC)^\circledast, \mN^\circledast \simeq \T^{-1}\mP_\mV(\mD)^\circledast.$ % compatible with small colimits.
% and $\mM^\circledast \to \mV^\ot $ a small right tensored $\infty$-category compatible with $\kappa$-small colimits.
%There is a canonical equivalence  The tensor product of $\Q\Enr_\mV^{\rc\rc}$ admits the following descriptions:
There is an accessible $\mV$-enriched localization $$\mP_\mV(\mC \ot \mD) \rightleftarrows \mM \ot_{\mV} \mN$$ with respect to the set $\rS \boxtimes\T:= \{\f \ot \W \ot \Y, \V \ot \X \ot \g \mid \f \in \rS, \g \in \T, \V, \W \in \mV, \X \in \mC, \Y \in \mD \}$.
		
\end{corollary}

Blumberg-Gepner-Tabuada \cite{article} develop algebraic $\K$-theory of stable $\infty$-categories
and characterize algebraic $\K$-theory by an universal property. % using a theory of non-commutative motives.
In analogy \cite{mazelgee2021universal} characterize secondary algebraic $\K$-theory by a universal property based on the notion of 2-stable $(\infty,2)$-category:
%, and characterize secondary algebraic $\K$-theory by an universal property using a theory of non-commutative motives.
secondary algebraic $\K$-theory is a categorification of algebraic $\K$-theory that instead of the stable $\infty$-category of module spectra $\Mod_\R$ over a ring $\R$ uses the $\infty$-category $ \Mod^2_{\R}$ of small (Cauchy-complete) stable $\infty$-categories left tensored over $\Mod_\R$,  thought of as secondary $\R$-modules. 
Like $\Mod_\R$ is the prime example of a stable $\infty$-category, $ \Mod^2_{\R}$ is the prime example of a 2-stable $(\infty,2)$-category.
An $\infty$-category  is stable if it admits a zero object, pushouts and pullbacks and a commutative square is a pushout square if and only if it is a pullback  square.
The $(\infty,2)$-category of small stable $\infty$-categories is a $2$-stable $(\infty,2)$-category in the following sense: it is enriched in stable
$\infty$-categories (via its canonical $\Cat_\infty$-enrichment) and %of small $\infty$-categories lies in the subcategory of %naturally enriched in stable $\infty$-categories %but has an additional feature: 
it is a preadditive $(\infty,2)$-category: %.e. conical $\bk$-fold coproducts and products agree for  every $\bk \geq 0$.meaning that 
it admits a zero object in the enriched sense and conical finite coproducts and products and such agree via the canonical map.
%the map from the conical $\bk$-fold coproduct to the conical $\bk$-fold product is an equivalence for any $\bk \geq0$. 
%By ... We call such $(\infty,2)$-categories 2-preadditive, 2-additive and 2-stable, respectively.In ... the authors study 2-stable $(\infty,2)$-categories as the natural domain for secondary $\K$-theory in an analogous way as
More generally, we inductively define (Cauchy-complete) $\n$-stable $(\infty,\n)$-categories for every $\n \geq1$ as those (Cauchy-complete) preadditive $(\infty,\n)$-categories that are
enriched in (Cauchy-complete) $\n$-1-stable $(\infty,\n-1)$-categories.
%whose enrichment in $\Cat_{(\infty,\n-1)}$ is in the subcategory of $\n-1$-stable $(\infty,\n-1)$-categories and that i
We prove that the $(\infty,\n+1)$-category $ \St(\n)$ of small stable $(\infty,\n)$-categories is $\n+1$-stable and the $(\infty,\n+1)$-category $ \St^\wedge(\n)$ of small Cauchy-complete stable $(\infty,\n)$-categories is Cauchy complete and $\n+1$-stable (Corollary \ref{jlmmozz}).
%denoted by $\Cat_{(\infty,\n)}^{\mathrm{preadd}},\Cat_{(\infty,\n)}^{\mathrm{add}}, \St(\n)$are $\n+1$-preadditive, $\n+1$-additive, $\n+1$-stable, respectively.Other examples are the $(\infty,\n+1)$-categories of small idempotent complete preadditive, additive, stable $(\infty,\n)$-categories
We use Theorem \ref{thh} to construct a tensor product for %$\n$-preadditive, $\n$-additive and 
$\n$-stable $(\infty,\n)$-categories generalizing the tensor product of stable $\infty$-categories:
%extending the case $\n=1:$
\begin{corollary}(Corollary \ref{Stab})
Let $\n \geq 1$.
There are canonical  presentably symmetric monoidal structures on 
the $\infty$-categories $$\St(\n), \St^\wedge(\n).$$
\end{corollary}
We use the same technique to construct a tensor product for $\n$-additive and
$\n$-preadditive $(\infty,\n)$-categories for any $\n \geq 1$ (Corollary \ref{Stab}).

\vspace{5mm}

\subsection{Notation and terminology}

We fix a hierarchy of Grothendieck universes whose objects we call small, large, very large, etc.
We call a space small, large, etc. if its set of path components and its homotopy groups are for any choice of base point. We call an $\infty$-category small, large, etc. if its maximal subspace and all its mapping spaces are.

\vspace{2mm}

We write 
\begin{itemize}
\item $\Set$ for the category of small sets.
\item $\Delta$ for (a skeleton of) the category of finite, non-empty, partially ordered sets and order preserving maps, whose objects we denote by $[\n] = \{0 < ... < \n\}$ for $\n \geq 0$.
\item $\mS$ for the $\infty$-category of small spaces.
\item $ \Cat_\infty$ for the $\infty$-category of small $\infty$-categories.
\item $\Cat_\infty^{\rc \rc} $ for the $\infty$-category of large $\infty$-categories with small colimits and small colimits preserving functors.
\end{itemize}

\vspace{2mm}

We often indicate $\infty$-categories of large objects by $\widehat{(-)}$, for example we write $\widehat{\mS}, \widehat{\Cat}_\infty$ for the $\infty$-categories of large spaces, $\infty$-categories.

For any $\infty$-category $\mC$ containing objects $\A, \B$ we write
\begin{itemize}
\item $\mC(\A,\B)$ for the space of maps $\A \to \B$ in $\mC$,
\item $\mC_{/\A}$ for the $\infty$-category of objects over $\A$,
\item $\Ho(\mC)$ for its homotopy category,
\item $\mC^{\triangleleft}, \mC^{\triangleright}$ for the $\infty$-category arising from $\mC$ by adding an initial, final object, respectively,
\item $\mC^\simeq $ for the maximal subspace in $\mC$.
\end{itemize}

Note that $\Ho(\Cat_\infty)$ is cartesian closed and for small $\infty$-categories $\mC,\mD$ we write $\Fun(\mC,\mD)$ for the internal hom, the $\infty$-category of functors $\mC \to \mD.$ 

%\subsubsection*{Inclusions and embeddings}
We often call a fully faithful functor $\mC \to \mD$ an embedding.
We call a functor $\phi: \mC \to \mD$ an inclusion (of a subcategory)
if one of the following equivalent conditions holds:
\begin{itemize}
\item For any $\infty$-category $\mB$ the induced map
$\Cat_\infty(\mB,\mC) \to \Cat_\infty(\mB,\mD)$ is an embedding.
\item The functor $\phi: \mC \to \mD$ induces an embedding on maximal subspaces and on all mapping spaces.
%\item The functor $\Ho(\phi):\Ho(\mC) \to \Ho(\mD)$ is an inclusion and the functor $\mC \to \Ho(\mC) \times_{\Ho(\mD)} \mD$ is an equivalence.

\end{itemize}

%or say that $\phi$ exhibits $\mC$ as a subcategory of $\mD$

%In this case $\phi$ is uniquely determined by $\mD$ and $\Ho(\phi): \Ho(\mC) \to \Ho(\mD).$\vspace{2mm}

%We call a monoidal $\infty$-category compatible with small colimits if it admits small colimits, which are preserved by the tensor product in each component.
%We call a monoidal $\infty$-category presentable if it is compatible with small colimits and its underlying $\infty$-category is presentable.
%\vspace{2mm}

%\subsubsection*{Relative cocartesian fibrations}

Let $\rS$ be an $\infty$-category and $\mE \subset \Fun([1],\rS)$ a full subcategory. We call a functor $\X \to \rS$ a cocartesian fibration relative to $\mE$ if for every morphism $[1] \to \rS$ that belongs to $\mE$ the pullback
$[1] \times_\rS \X \to [1]$ is a cocartesian fibration whose cocartesian morphisms are preserved by the projection $[1] \times_\rS \X \to \X.$
We call a functor $\X \to \Y$ over $\rS$ a map of cocartesian fibrations relative to $\mE$ if it preserves cocartesian lifts of morphisms of $\mE.$
We write $\Cat_{\infty/\rS}^\mE \subset \Cat_{\infty/\rS}$ for the subcategory of cocartesian fibrations relative to $\mE$ and maps of such.

\subsubsection*{Acknowledgements} We thank David Gepner and Markus Spitzweck for helpful discussions.

\label{rel}

\section{Enriched $\infty$-categories}

In this section we recall the theory of left, right and bienriched $\infty$-categories of \cite{heine2024bienriched} and further develope this theory for our needs.

%developed by \cite{MR3345192} and \cite{HINICH2020107129}.
%We work with an alternative but due to \cite{HEINE2023108941} equivalent approach to enriched $\infty$-categories based on the notion of weakly left enriched $\infty$-categories,originally defined by Lurie \cite[Definition 4.2.1.12.]{lurie.higheralgebra} in a slightly different model.

\subsection{$\infty$-operads}
To define enriched $\infty$-categories we first need to introduce non-symmetric $\infty$-operads \cite[Definition 2.2.6.]{GEPNER2015575}: 
%\cite[Definition 2.16.]{HEINE2023108941}.
%More generally, for later applications in section \ref{HW} we define $\bE_\bk$-operads for any $1 \leq \bk \leq \infty$, which interpolate between non-symmetric $\infty$-operads and symmetric $\infty$-operads.

\begin{notation}\label{ooop}
Let $\Ass:= \Delta^\op$ be the category of finite non-empty totally ordered sets and order preserving maps.
We call a map $[\n] \to [\m]$ in $\Ass$
\begin{itemize}
\item inert if it corresponds to a map of $\Delta$ of the form $[\m] \simeq \{\bi, \bi+1,..., \bi+\m \} \subset [\n]$ for some $\bi \geq 0.$
%\item standard inert if $\m=1$.
\item active if it corresponds to a map of $\Delta$, which preserves the minimum and maximum.
\end{itemize}
\end{notation}
For every $\n \geq 0$ there are $\n$ inert morphisms $[\n] \to [1]$, where the $\bi$-th inert morphism $[\n] \to [1]$ for $1 \leq \bi \leq \n$ corresponds to the map $[1] \simeq \{\bi-1, \bi\} \subset [\n]$.

%\begin{notation}We say that a morphism in $\Ass $ preserves the maximum if it corresponds to a morphism $[\m] \to [\n]$ in $\Delta$ that sends $\m$ to $\n.$\end{notation}
 
%\begin{notation}For $1 \leq \bk < \infty$ we write $\bE_\bk:= \prod_{\bi=1}^\bk \Ass $ and call a morphism in $\bE_\bk$ active, inert, respectively, if it is component-wise.\end{notation}

%For every $([\n_\bi])_{\bi=1}^{\bk} \in \bE_\bk$ there are$\prod_{\bi=1}^\bk \n_\bi $-many inert morphisms$([\n_\bi])_{\bi=1}^{\bk} \to ([1])_{\bi=1}^{\bk}$ in $\bE_\bk.$

%\begin{notation}We write $\bE_\infty$ for the category $\Comm$ of finite pointed sets.We write pointed finite sets as $\langle \n \rangle:= \{ \ast, 1, ..., \n\}$ for $\n \geq 0$, where $\ast$ is the base point.\end{notation}

%\begin{notation}
%We call a map $\theta$ of pointed finite sets $\langle \n \rangle \to \langle \m \rangle$ \begin{itemize}
%\item inert if for every $1 \leq \bi \leq \m$ the fiber of $\theta$ over $\bi$ consists precisely of one element.
%\item standard inert if $\m=1$.
%\item active if it sends only the base point to the base point.
%\end{itemize}\end{notation}
%For every $\n \geq 0$ there are $\n$ inert morphisms $\langle \n \rangle \to \langle 1 \rangle$, where the $\bi$-th inert morphism $ \langle \n \rangle \to \langle 1 \rangle$ for $1 \leq \bi \leq \n$ sends $\bi$ to 1.Now we are ready to define $\bE_\bk$-operads and $\bE_\bk$-monoidal $\infty$-categories:

\begin{definition}\label{ek}
A (non-symmetric) $\infty$-operad is a 
%Let $1 \leq \bk \leq \infty.$ A 
cocartesian fibration $\phi: \mV^\ot \to \Ass$ relative to the collection of inert morphisms such that the following conditions hold, where we set $\mV:=\mV^\ot_{[1]}:$

\begin{enumerate}
\item For every $\n \geq 0$ the family of $\n$ inert morphisms $[\n] \to [1]$ of $\Ass$ induces an equivalence $\mV^\ot_{[\n]} \to \mV^{\times \n}.$

\item For every $\Y,\X \in \mV^\ot $ lying over $[\m],[\n]\in \Ass $ the family
$\X \to \X_\bj$ for $1 \leq \bj \leq \n $ of $\phi$-cocartesian lifts of
the inert morphisms $[\n]\to [1]$ induces a pullback square
\begin{equation*} 
\begin{xy}
\xymatrix{
\mV^\ot(\Y,\X) \ar[d]^{} \ar[r]^{ }
& \prod_{1 \leq \bj \leq \n} \mV^\ot(\Y,\X_\bj) \ar[d]^{} 
\\  \Ass([\m],[\n]) 
\ar[r]^{} & \prod_{1 \leq \bj \leq \n} \Ass([\m], [1]). 
}
\end{xy} 
\end{equation*}

\end{enumerate}	
	
\end{definition}

\begin{definition} A monoidal $\infty$-category is an $\infty$-operad $\phi: \mV^\ot \to \Ass$ such that $\phi$ is a cocartesian fibration.
	
%\begin{itemize}
%\item 
%\item An $\bE_\bk$-operad $\phi: \mV^\boxtimes \to \bE_\bk$ is corepresentable if $\phi$ is a locally cocartesian fibration.
%\end{itemize}
\end{definition}

\begin{remark}\label{uuu}
A cocartesian fibration $\phi: \mV^\ot \to \Ass$ is a monoidal $\infty$-category if and only if the first condition
of Definition \ref{ek} holds. The second condition is automatic if $\phi$ is a cocartesian fibration \cite[Remark 2.17.]{HEINE2023108941}.
\end{remark}

%\begin{notation}For $\bk=\infty$ we call $\bE_\bk$-operads symmetric $\infty$-operads and$\bE_\bk$-monoidal $\infty$-categories symmetric monoidal $\infty$-categories.For $\bk=1$ we call $\bE_\bk$-operads $\infty$-operads and we call $\bE_\bk$-monoidal $\infty$-categories monoidal $\infty$-categories.\end{notation}

\begin{notation}
%Let $1 \leq \bk \leq \infty.$
For every $\infty$-operad $\phi: \mV^\ot \to \Ass$ we set
$ \mV:=\mV^\ot_{[1]}$ and call $\mV$ the underlying $\infty$-category.
\end{notation}

\begin{notation}
Let $\mV^\ot \to \Ass$ be an $\infty$-operad and $\V_1,..., \V_{\n},\W \in \mV$ for $\n \geq 0 $. Let
$$\Mul_{\mV}(\V_1,..., \V_\n;\W)$$ be the full subspace of $\mV^\ot(\V,\W)$ spanned by the morphisms $\V \to \W$ in $\mV^\ot$ lying over the active morphism $[1] \to [\n]$ in $\Delta,$ where $\V \in \mV^\ot_{[\n]} \simeq \mV^{\times \n}$ corresponds to $(\V_1,..., \V_\n) $.
	
\end{notation}

\begin{definition}
\begin{enumerate}
\item An $\infty$-operad $\phi: \mV^\ot \to \Ass$ is locally small if the mapping spaces of $\mV^\ot$ are small.

\item An $\infty$-operad $\phi: \mV^\ot \to \Ass$ is small if it is locally small and $\mV$ is small.

\end{enumerate}	
\end{definition}

\begin{remark}\label{oprer}
By the first axiom of Definition \ref{ek} an $\infty$-operad $ \mV^\ot \to \Ass$ is small if and only if $\mV^\ot$ is small.
An $\infty$-operad $\mV^\ot \to \Ass$ is locally small if and only if the multi-morphism spaces of $ \mV^\ot \to \Ass $ are small.	
So an $\infty$-operad $\mV^\ot \to \Ass$ is small if and only if $\mV$ is small and the multi-morphism spaces of $ \mV^\ot \to \Ass $ are small.	

\end{remark}

%\begin{notation}
%Let $1 \leq \bk \leq \infty$ and $\phi: \mV^\boxtimes \to \bE_\bk$ an  $\bE_\bk$-operad.
%W%e write $\mV^\ot \to \Ass$ for the pullback of $\phi$ along the functor
%$ \Ass \times \{[1]\}^{\times \bk-1} \to \Ass^{\times \bk}=\bE_\bk$
%and call $\mV^\ot \to \Ass $ the underlying $\infty$-operad.
%In particular, for $\bk=1$ we write $\mV^\ot$ for $\mV^\boxtimes.$\end{notation}

\begin{notation}
	
For every monoidal $\infty$-category $\mV^\ot \to \Ass$
we write $\ot: \mV \times \mV \simeq \mV^\ot_{[2]} \to \mV^\ot_{[1]}= \mV$ for the functor induced by the active map $[1] \to [2]$ in $\Delta.$

\end{notation}

%The next definition extends Definition \ref{compost} (1) for $\bk=1:$

\begin{definition}Let $\mK \subset \Cat_\infty$ be a full subcategory.
An $\infty$-operad $\mV^\ot \to \Ass$ is compatible with $\mK$-indexed colimits if $\mV$ admits $\mK$-indexed colimits and for every $\V_1,..., \V_\n, \V \in \mV$ for $\n \geq0$ and $0 \leq \bi \leq \n$ the presheaf $\Mul_\mV(\V_1,..,\V_\bi,-, \V_{\bi+1},..., \V_\n;\V)$ on $\mV$ preserves $\mK$-indexed limits.
	
\end{definition}

\begin{remark}
	
A monoidal $\infty$-category $\mV^\ot \to \Ass$ is compatible with $\mK$-indexed colimits if
it is compatible with $\mK$-indexed colimits as an $\infty$-operad, i.e. $\mV$ admits
$\mK$-indexed colimits and such are preserved by the tensor product component-wise.
\end{remark}

\begin{definition}\label{compost}
A monoidal $\infty$-category $\mV^\ot \to \Ass$ is
\begin{enumerate}
\item closed if the tensor product $\ot: \mV \times \mV \to \mV$ is left adjoint component-wise.
\item presentably if $\mV$ is presentable and (1) holds.
\item $\kappa$-compactly generated for a regular cardinal $\kappa$ if it is compatible with small colimits, $\mV$ is $\kappa$-compactly generated and the monoidal structure of $\mV$ restricts to the full subcategory $\mV^\kappa \subset \mV$ of $\kappa$-compact objects.
\end{enumerate}
\end{definition}
\begin{remark}
By \cite[Proposition 7.15.]{Rune} every presentably monoidal $\infty$-category is $\kappa$-compactly generated for some regular cardinal $\kappa.$
\end{remark}
We also consider maps of $\infty$-operads:
\begin{definition}
Let $\mV^\ot \to \Ass, \mW^\ot \to \Ass$ be $\infty$-operads.
A map of $\infty$-operads $\mV^\ot \to \mW^\ot$ is a map of cocartesian fibrations
$\mV^\ot \to \mW^\ot$ relative to the collection of inert morphisms of $\Ass$.
\end{definition}

\begin{definition}
	
A map of $\infty$-operads $\mV^\ot \to \mW^\ot$ is an embedding
if it is fully faithful.
\end{definition}
\begin{definition}
Let $\mV^\ot \to \bE_\bk, \mW^\ot \to \Ass$ be monoidal $\infty$-categories.
\begin{itemize}
\item A lax monoidal functor $\mV^\ot \to \mW^\ot$ is a map of $\infty$-operads $\mV^\ot \to \mW^\ot$.
\item A monoidal functor $\mV^\ot \to \mW^\ot$ is a map of cocartesian fibrations $\mV^\ot \to \mW^\ot$ over $\Ass.$	
\end{itemize}	
\end{definition}

\begin{notation}We fix the following notation:
\begin{itemize}
	
\item Let $\Op_\infty \subset \Cat_{\infty/\Ass}$ be the subcategory of $\infty$-operads and maps of $\infty$-operads.

\item Let $\Mon \subset\Op_\infty$ be the subcategory of monoidal $\infty$-categories and monoidal functors.

\item Let $\Pr\Mon \subset \widehat{\Mon}$ be the subcategory of monoidal $\infty$-categories compatible with small colimits and monoidal functors preserving small colimits.

\end{itemize}
	
\end{notation}

\begin{notation}\label{notilo}
For every $\infty$-operads $\mV^\ot \to \Ass, \mW^\ot \to \Ass$ let $$\Alg_\mV(\mW) \subset \Fun_\Ass(\mV^\ot,\mW^\ot)$$ be the full subcategory of maps of $\infty$-operads.
If $\mV^\ot \to \Ass, \mW^\ot \to \Ass$ are monoidal $\infty$-categories,
let $$\Fun^{\ot,\L}(\mV,\mW) \subset\Fun^\ot(\mV,\mW) \subset \Alg_\mV(\mW)$$ be the full subcategories of monoidal functors (that induce on underlying $\infty$-categories a left adjoint).
\end{notation}

\begin{definition} An $\infty$-operad $\phi: \mV^\ot \to \Ass$ admits a tensor unit if the unique morphism $[0]\to [1]$ in $\Ass$ admits a $\phi$-cocartesian lift $\emptyset \to \tu_\mV. $ We call $\tu_\mV$ the tensor unit of $\mV.$
	
\end{definition}

%\begin{warning}An $\infty$-operad $\mV^\ot \to \Ass$ does not necessarily admit a tensor unit if the functor $ \Mul_{\mV}(\emptyset;-): \mV \to \mS$ is corepresentable. The latter condition is equivalent to say that the unique morphism $[0]\to [1]$ in $\Ass$ admits a locally $\phi$-cocartesian lift $\emptyset \to \tu_\mV. $ The condition that the unique morphism $[0]\to [1]$ in $\Ass$ admits a $\phi$-cocartesian lift $\emptyset \to \tu_\mV$ is equivalent to say that the functor $ \Mul_{\mV}(\emptyset;-): \mV \to \mS$ is corepresentable by some object $\tu_\mV$ and the canonical map $$ \mV^\ot(\X,\Y) \to \Ass([1],[1]) \times \mV(\tu,\Y) $$ is an equivalence.\end{warning}

%\begin{definition}
%An $\infty$-operad $\mV^\ot \to \Ass$ admits a tensor unit if there is a morphism$\emptyset \to \tu_\mV$ in $\mV^\ot$ lying over the unique map $[0]\to [1]$ in $\Ass$ (corresponding to an object of $ \Mul_{\mV}(\emptyset;\tu_\mV)$) such that the induced map$\mV(\tu_\mV,\tu_\mV) \to \Mul_{\mV}(\emptyset;\tu_\mV)$ is an equivalence.We call $\tu_\mV$ the tensor unit of $\mV$.\end{definition}

%\begin{definition}Let $\mV^\ot \to \Ass, \mW^\ot \to \Ass$ be $\infty$-operads that admit a tensor unit.A map of $\infty$-operads $\F: \V^\ot \to \mW^\ot$ preserves the tensor unit if the induced map $\tu_\mW \to \F(\tu_\mV)$ is an equivalence.	\end{definition}

%\begin{notation}Let $\Mon \subset \Op_{\infty}$ be the subcategory whose objects are the $\infty$-operads that admit a tensor unit and whose morphisms are the mapsof $\infty$-operads preserving the tensor unit.\end{notation}

\begin{remark}\label{preps}
Let $\phi: \mV^\ot \to \Ass$ be an $\infty$-operad that admits a tensor unit
and for every $\n \geq 1, 1 \leq \bj \leq \n$ let $(\V_\bi^\bj)_{\bi=1}^{\m_\bj}$ be a family in $\mV^\ot,$ where $ \m_\bj \geq 0$.
%Let $\emptyset \to \tu$ be the $\phi$-cocartesian lift of the map $[0] \to [1]$ in $\Ass.$
By the axioms of an $\infty$-operad there is a morphism
$$(\V_\bi^\bj)_{\bi=1}^{\m_1}, \emptyset, (\V_\bi^\bj)_{\bi=1}^{\m_2}, \emptyset,(\V_\bi^\bj)_{\bi=1}^{\m_3}, ..., \emptyset, (\V_\bi^\bj)_{\bi=1}^{\m_\n} \to (\V_\bi^\bj)_{\bi=1}^{\m_1}, \tu_\mV, (\V_\bi^\bj)_{\bi=1}^{\m_2}, \tu_\mV, (\V_\bi^\bj)_{\bi=1}^{\m_3},..., \tu_\mV, (\V_\bi^\bj)_{\bi=1}^{\m_\n} $$
in $\mV^\ot$, which is $\phi$-cocartesian by \cite[Lemma 2.24.]{HEINE2023108941} since the morphism $\emptyset \to \tu$ is $\phi$-cocartesian.
%In particular, for every $\V \in \mV$ the following induced map is an equivalence:$$ \Mul_\mV((\V_\bi^\bj)_{\bi=1}^{\m_1}, \tu_\mV, (\V_\bi^\bj)_{\bi=1}^{\m_2}, \tu_\mV, ..., \tu_\mV, (\V_\bi^\bj)_{\bi=1}^{\m_\n} ;\V) \to \Mul_\mV((\V_\bi^\bj)_{\bi=1}^{\m_1}, \emptyset, (\V_\bi^\bj)_{\bi=1}^{\m_2}, \emptyset, ..., \emptyset, (\V_\bi^\bj)_{\bi=1}^{\m_\n};\V).$$

\end{remark}

%\begin{remark}\label{Remast}
%Let $\phi: \mV^\ot \to \Ass$ be an $\infty$-operad.The following conditions are equivalent.

%\begin{enumerate}\item For every $[\n]\in \Ass$ the unique morphism $[0] \to [\n]$ in $\Ass$admits a $\phi$-cocartesian lift.\item The unique morphism $[0] \to [1]$ in $\Ass$ admits a $\phi$-cocartesian lift.

%\item The morphism $[0] \to [1]$ in $\Ass$ admits a locally $\phi$-cocartesian lift: there is a morphism $\emptyset \to \tu_\mV$ in $\mV^\ot$, where %$\emptyset $ is the unique object of $\mV^\ot_{[0]}$ and $\tu_\mV \in \mV,$ that induces for any $\X \in \mV$ an equivalence$ \mV(\tu_\mV,\X) \to \Mul_{\mV}(\emptyset; \X)$.\item The functor $ \Mul_{\mV}(\emptyset;-): \mV \to \mS$ is corepresentable.

%\item For every $[\n]\in \Ass$ the unique morphism $[0] \to [\n]$ in $\Ass$admits a locally $\phi$-cocartesian lift: there is a morphism $\emptyset \to \tu_\mV,..., \tu_\mV$ in $\mV^\ot$, where $\tu_\mV \in \mV,$ such that for every $\X \in \mV$ the induced map+$ \Mul_\mV(\tu_\mV,...,\tu_\mV;\X) \to \Mul_{\mV}(\emptyset; \X)$ is an equivalence.\end{enumerate}

%In particular, by (3), (4) and (5) for every $\infty$-operad $\mV^\ot \to \Ass$ that admits a tensor unit the functor $ \Mul_\mV(\tu_\mV,...,\tu_\mV;-): \mV \to \mS$ is corepresented by $\tu_\mV$.\end{remark}

%\begin{notation}Let $\Op_{\infty}^\tu \subset \Op_{\infty}$ be the subcategory of $\infty$-operads that admit a tensor unit and morphisms of $\infty$-operads preserving the tensor unit.\end{notation}

\begin{lemma}\label{preten}

Let $\phi: \mV^\ot \to \Ass$ be an $\infty$-operad that admits a tensor unit.
The $\infty$-category $\Alg(\mV)$ admits an initial object.
An associative algebra in $\mV$ is initial if and only if it preserves the tensor unit as a map $\Ass \to \mV^\ot.$

\end{lemma}
\begin{proof}
A map $\A \to \B$ in $\Alg(\mV)$ is an equivalence if $\A, \B$ preserve the tensor unit since the induced map $\A([1]) \to \B([1])$ in $\mV^\ot$ is a map under $\emptyset \simeq \A([0]) \simeq \B([0])$. % and so corresponds to a map in $\mV$ under $\tu_\mV$. 		
Thus it is enough to prove that the functor $\rho: \Alg(\mV) \to \mV^\ot_{[0]}$ restricting along the embedding $[0] \subset \Ass$ admits a fully faithful left adjoint that sends $\emptyset$ to an algebra preserving the tensor unit.
By \cite[Lemma 4.3.2.13., Proposition 4.3.2.17.]{lurie.HTT} the functor
$\rho$ has a  fully faithful left adjoint if for any $\n \geq 0$ the functor $\lambda: \{[0]\}\times_\Ass \Ass_{/[\n]} \to \mV^\ot_{[0]} \subset \mV^\ot$ 
admits a $\phi$-colimit and the left adjoint sends $\emptyset \in \mV^\ot_{[0]}$ to the algebra $\Z $ in $ \mV$ that maps $[\n] \in \Ass$
to the $\phi$-colimit of $\lambda.$
Since $\{[0]\}\times_\Ass \Ass_{/[\n]} $ is contractible, a $\phi$-colimit of $\lambda$ is a $\phi$-cocartesian lift
of the map $[0]\to [\n]$ in $\Ass$, which exists by Remark \ref{preps}. The algebra $\Z $ in $ \mV$ sends $[1] \in \Ass$ to the $\phi$-cocartesian lift $\emptyset \to \tu_\mV$ of the map $[0]\to [1]$ in $\Ass$. So $\Z$ preserves the tensor unit.

\end{proof}

Via Lemma \ref{preten} we can make the following definition:

\begin{notation}Let $\n \geq 1$, $1 \leq \bi \leq \n$ and $\mV_\bi^\ot \to \Ass$ be $\infty$-operads such that for every $1 \leq \bj \leq \n, \bj \neq \bi$
the $\infty$-operad $\mV_\bj^\ot \to \Ass$ admits a tensor unit. 
Let $\tau_\bi: \mV_\bi^\ot \to \mV_1^\ot \times_\Ass ... \times_\Ass \mV_\n^\ot$
be the map of $\infty$-operads that is the identity on the $\bi$-th factor and the constant map $\mV_\bi^\ot \to \Ass \to \mV_\bj^\ot $ preserving the tensor unit on
the $\bj$-th factor for $\bj \neq \bi$, where we use Lemma \ref{preten}.

\end{notation}

%\subsection{Weak enrichment}

\begin{notation}\label{Ind} Let $\kappa$ be a small regular cardinal and $\mC$ a small $\infty$-category.
Let $\Ind_\kappa(\mC) \subset \mP(\mC)$ be the full subcategory generated by $\mC$ under small $\kappa$-filtered colimits.

\end{notation}

\begin{example}

For $\kappa=\emptyset$ we find that $\Ind_\emptyset(\mC)=\mP(\mC).$ 	

\end{example}

\begin{remark}\label{Indd} 
By \cite[Corollary 5.3.5.4.]{lurie.higheralgebra} for every small $\infty$-category $\mC$
that admits $\kappa$-small colimits the full subcategory $\Ind_\kappa(\mC) \subset \mP(\mC)$ precisely consists of the functors $\mC^\op \to \mS$ preserving $\kappa$-small limits. 
Thus $\Ind_\kappa(\mC)$ is a $\kappa$-accessible localization with respect to the set of maps $\{ \colim(\y \circ \rH) \to \y(\colim(\rH)) \mid \rH:\K \to \mC, \ \K \ \kappa\text{-small} \}$, where $\y: \mC \subset \mP(\mC)$ is the Yoneda-embedding. Hence $\Ind_\kappa(\mC)$ is a presentable $\infty$-category.

\end{remark}

The next proposition follows from \cite[Corollary 8.31.]{HEINE2023108941}:

\begin{proposition}\label{Day}Let $\kappa$ be a small regular cardinal
and $\mV^\ot \to \Ass$ a small monoidal $\infty$-category.

\begin{enumerate}
\item There is a monoidal $\infty$-category $\Ind_\kappa(\mV)^\ot \to \Ass $ compatible with small $\kappa$-filtered colimits and a monoidal embedding $\mV^\ot \to \Ind_\kappa(\mV)^\ot$ inducing the embedding $\mV \to \Ind_\kappa(\mV)$. % on underlying $\infty$-categories.

\item For every monoidal $\infty$-category $\mW^\ot \to \Ass$ compatible with small $\kappa$-filtered colimits the functor 
$$
\Alg_{\Ind_\kappa(\mV)}(\mW) \to \Alg_{\mV}(\mW)$$ admits a fully faithful left adjoint that lands in the full subcategory $\Alg^{\kappa-\mathrm{fil}}_{\Ind_\kappa(\mV)}(\mW)$ of maps of $\infty$-operads preserving small $\kappa$-filtered colimits.
%and sends maps of $\infty$-operads preserving $\kappa$-small colimits to maps of $\infty$-operads preserving small colimits.
Thus the following functor is an equivalence:
\begin{equation*}\label{ejjt}
\Alg^{\kappa-\mathrm{fil}}_{\Ind_\kappa(\mV)}(\mW) \to \Alg_{\mV}(\mW).
\end{equation*}

\item If $\mV$ admits $\kappa$-small colimits, the latter equivalence restricts to an equivalence $$\Alg^\L_{\Ind_\kappa(\mV)}(\mW) \to \Alg^\kappa_{\mV}(\mW).$$

%\item If $\mV^\ot \to \Ass$ is compatible with $\kappa$-small colimits, $\Ind_\kappa(\mV)^\ot \to \Ass$ is compatible with small colimits.

%\item If $\mV^\ot \to \Ass$ is a monoidal $\infty$-category, $\Ind_\kappa(\mV)^\ot \to \Ass$ is a monoidal $\infty$-category, the embedding $\mV^\ot \to \Ind_\kappa(\mV)^\ot$is monoidal and for every monoidal $\infty$-category $\mW^\ot \to \Ass$ compatible with small colimits the functor (\ref{ejjt}) reflects monoidal functors. 
%Thus the following functors are equivalences:
%\begin{equation*}\label{zubn}
%\Fun^{\ot,\kappa-\mathrm{fil}}(\Ind_\kappa(\mV), \mW) \to \Fun^{\ot}(\mV,\mW),
%\end{equation*}
%\begin{equation*}\label{zubn}
%\Fun^{\ot,\L}(\Ind_\kappa(\mV), \mW) \to \Fun^{\ot, \kappa}(\mV,\mW).
%\end{equation*}
\end{enumerate}

\end{proposition}

In the later sections we will also use symmetric $\infty$-operads.

\begin{notation}We write $\Comm$ for the category of finite pointed sets. We write finite pointed sets as $\langle \n \rangle:= \{ \ast, 1, ..., \n\}$ for $\n \geq 0$, where $\ast$ is the base point. \end{notation}

\begin{notation}
We call a map $\theta$ of pointed finite sets $\langle \n \rangle \to \langle \m \rangle$ \begin{itemize}
\item inert if for every $1 \leq \bi \leq \m$ the fiber of $\theta$ over $\bi$ consists precisely of one element.
%\item standard inert if $\m=1$.
\item active if it sends only the base point to the base point.
\end{itemize}\end{notation}

\begin{remark}
	
For every $\n \geq 0$ there are $\n$ inert morphisms $\langle \n \rangle \to \langle 1 \rangle$, where the $\bi$-th inert morphism $ \langle \n \rangle \to \langle 1 \rangle$ for $1 \leq \bi \leq \n$ sends $\bi$ to 1.
\end{remark}

\begin{definition}
A symmetric $\infty$-operad is a cocartesian fibration $\mV^\boxtimes \to \Comm$
relative to the collection of inert morphisms that satisfies the analogous conditions like in
Definition \ref{ek}, where $[\n], [\m]$ are replaced by $\langle \n\rangle, \langle \m\rangle.$
A symmetric monoidal $\infty$-category is a symmetric $\infty$-operad $\mV^\boxtimes \to \Comm$ that is also a cocartesian fibration.
\end{definition}

\begin{remark}
A cocartesian fibration $\mV^\boxtimes \to \Comm$ is a symmetric monoidal $\infty$-category if and only if it satisfies the analogue of condition (1) of Definition \ref{ek}, condition (2) is automatic for cocartesian fibrations like in Remark \ref{uuu}.
	
\end{remark}
%\begin{notation}We call $\mV:=\mV^\boxtimes_{\langle1\rangle}$ the underlying $\infty$-category. \end{notation}

\begin{notation}
	
By \cite[Construction 4.1.2.9.]{lurie.higheralgebra} there is a canonical functor $\theta: \Ass \to \Comm, [\n]\mapsto \langle\n\rangle $ and we write $\mV^\ot \to \Ass$ for the pullback of a symmetric monoidal $\infty$-category $\mV^\boxtimes \to \Comm$ along $\theta,$ which is a monoidal $\infty$-category that we call the underlying monoidal $\infty$-category.
We define maps of symmetric $\infty$-operads and (lax) symmetric monoidal functors analogously as in Notation \ref{notilo}.
\end{notation}
\begin{example}
By \cite[Proposition 5.1.0.3.]{lurie.higheralgebra} for every $\n \geq 0$ there is a symmetric $\infty$-operad $\bE_\n \to \Comm$, the $\n$-th little disk operad. %, whose algebras are $\bE_\n$-algebras.
For $\n=1$ the $\bE_1$-operad is the symmetrization of the non-symmetric $\infty$-operad $\Ass.$	
For $\n= \infty$ there is an equivalence between $\Comm$-algebras and $\Comm$-algebras
\cite[Theorem 5.1.2.2]{lurie.higheralgebra}.
$\bE_0$-algebras are precisely objects equipped with a map from the tensor unit.
	
\end{example}

\subsection{Weakly enriched $\infty$-categories}

To define enriched $\infty$-categories we first introduce weakly enriched $\infty$-categories following \cite{heine2024bienriched}. %, which there were called weakly tensored $\infty$-categories.

%\subsubsection{Weakly bienriched $\infty$-categories}

\begin{definition}
A morphism in $ \Ass$ preserves the minimum (maximum) if it corresponds to a map $[\m] \to [\n]$ in $\Delta$ sending $0$ to $0$
(sending $\m$ to $\n$).

\end{definition}

The next definition is \cite[Definition 2.25.]{heine2024bienriched}:

\begin{definition}\label{bla}
Let $\mV^\ot \to \Ass, \mW^\ot \to \Ass$ be $\infty$-operads.
An $\infty$-category weakly bienriched in $\mV, \mW$
is a map $\phi=(\phi_1,\phi_2): \mM^\circledast \to \mV^\ot \times \mW^\ot $ of cocartesian fibrations relative to the collection of inert morphisms of $\Ass \times \Ass$ whose first component preserves the maximum and whose second component preserves the minimum such that the following conditions hold:
\begin{enumerate}
\item for every $\n,\m \geq 0$ the map $[0]\simeq \{\n\} \subset [\n]$ in the first component and the map $[0]\simeq \{0\} \subset [\m]$ in the second component induce an equivalence
$$ \theta: \mM^\circledast_{[\n][\m]} \to \mV^\ot_{[\n]} \times  \mM^\circledast_{[0][0]} \times\mW^\ot_{[\m]},$$
\vspace{1mm}
\item for every $\X,\Y \in \mM^\circledast$ lying over $([\m'], [\n']), ([\m], [\n]) \in \Ass \times \Ass$ the cocartesian lift $\Y \to \Y'$ of the map $[0]\simeq \{\m\} \subset [\m]$ and $[0]\simeq \{0\} \subset [\n]$ induces a
pullback square
\begin{equation*} 
\begin{xy}
\xymatrix{
\mM^\circledast(\X,\Y)  \ar[d]^{} \ar[r]^{ }
& \mM^\circledast(\X,\Y') \ar[d]^{} 
\\ \mV^\ot(\phi_1(\X),\phi_1(\Y)) \times \mW^\ot(\phi_2(\X),\phi_2(\Y))  
\ar[r]^{}  & \mV^\ot(\phi_1(\X),\phi_1(\Y')) \times \mW^\ot(\phi_2(\X),\phi_2(\Y')). 
}
\end{xy} 
\end{equation*}
\end{enumerate}

\end{definition}

%\begin{remark}In \cite[Definition 3.3.]{HEINE2023108941} we define weakly bienriched $\infty$-categories under the name weakly bienriched $\infty$-categoriesand reserve the notion of weakly bienriched $\infty$-categories for a different structure of enrichment based on Gepner-Haugseng's model of enriched $\infty$-categories \cite{GEPNER2015575}.In \cite{HEINE2023108941} we prove an equivalence between weakly bienriched $\infty$-categories and weakly bienriched $\infty$-categoriesjustifying our choice of terminology. Moreover we prefer to use this terminology sincewe specialize the notion of weakly bienriched $\infty$-categories of Definition \ref{bla} to several notions of enrichment.	\end{remark}

\begin{notation}

For every weakly bienriched $\infty$-category $\phi: \mM^\circledast \to \mV^\ot \times \mW^\ot$ we call $\mM:= \mM^\circledast_{[0][0]} $ the underlying $\infty$-category of $\phi$ and say that $\phi$ exhibits $\mM$ as weakly bienriched in $\mV,\mW$.

\end{notation}

\begin{example}
Let $\mV^\ot \to \Ass$ be an $\infty$-operad.
We write $\mV^\circledast \to \Ass \times \Ass$ for the pullback of $\mV^\ot \to \Ass$ along the functor $\Ass \times \Ass \to \Ass, \ ([\n],[\m]) \mapsto [\n]\ast [\m]$.
The two functors $\Ass \times \Ass \times [1] \to \Ass$ corresponding to the natural transformations $(-)\ast \emptyset \to (-)\ast (-), \ \emptyset \ast (-) \to (-)\ast (-)$ send the morphism $\id_{[\n],[\m]}, 0 \to 1$ to
an inert one and so give rise to functors
$\mV^\circledast \to \mV^\ot \times \Ass, \ \mV^\circledast \to \Ass \times \mV^\ot $ over $\Ass\times \Ass$. The resulting functor $\mV^\circledast \to \mV^\ot \times \mV^\ot$ is an $\infty$-category weakly bienriched in $\mV,\mV.$

\end{example}

%\begin{example}\label{jhll}For every $\infty$-category $\K$ and $\infty$-operad $\mV^\ot \to \Ass$the projection $ \mV^\ot \times \K \times \mW^\ot \to \mV^\ot \times \mW^\ot $is an $\infty$-category weakly bienriched over $\mV,\mW.$\end{example}

\begin{example}
Let $\mM^\circledast \to \mV^\ot \times \mW^\ot$ be a weakly bienriched $\infty$-category and $\mN \subset \mM$ a full subcategory.
Let $\mN^\circledast \subset \mM^\circledast$ be the full subcategory spanned by all objects of $\mM^\circledast$ lying over some
$(\V,\W)\in \mV^\ot \times \mW^\ot$ corresponding some object of $\mN \subset \mM.$
The restriction $\mN^\circledast \subset \mM^\circledast \to \mV^\ot \times \mW^\ot$ is a weakly bienriched $\infty$-category, whose underlying $\infty$-category is $\mN$.
We call $\mN^\circledast \to \mV^\ot \times \mW^\ot$ the full weakly bienriched subcategory of $ \mM^\circledast $ spanned by $\mN.$

\end{example}

\begin{notation}\label{empp}
Let $\emptyset^\ot \subset \Ass$ be the full subcategory spanned by $[0] \in \Ass.$
Then $\emptyset^\ot$ is contractible and $\emptyset^\ot \subset \Ass$ is an $\infty$-operad that is the initial $\infty$-operad.
	
\end{notation}

\begin{definition}
An $\infty$-category weakly left enriched in $\mV$ is an $\infty$-category weakly bienriched in $\mV, \emptyset$.
An $\infty$-category weakly right enriched in $\mV$ is an $\infty$-category weakly bienriched in $ \emptyset, \mV$.

\end{definition}

%In this case Definition \ref{bla} simplifies to the following one:\begin{definition}\label{wla}Let $\mV^\ot \to \Ass$ be an $\infty$-operad and $\phi: \mM^\circledast \to \mV^\ot $ a map of cocartesian fibrations relative to the collection of inert morphisms of $\Ass$ that preserve the maximum.\vspace{1mm}	We call $\phi: \mM^\circledast \to \mV^\ot $ an $\infty$-category weakly left enriched over $\mV$ if the following conditions hold:\begin{enumerate}\item for every $\n \geq 0$ the map $[0]\simeq \{\n\} \subset [\n]$ in $\Delta$ induces an equivalence$$ \theta: \mM^\circledast_{[\n]} \to \mV^\ot_{[\n]} \times \mM^\circledast_{[0]}.$$\item for every $\X,\Y \in \mM^\circledast$ lying over $[\m], [\n] \in \Ass$ the cocartesian lift $\Y \to \Y'$ of the map $[0]\simeq \{\n\} \subset [\n]$ in $\Delta$ induces a pullback square
%\begin{equation*} 
%begin{xy}
%\xymatrix{\mM^\circledast(\X,\Y)  \ar[d]^{} \ar[r]^{ }& \mM^\circledast(\X,\Y') \ar[d]^{} \\ \mV^\ot(\phi(\X),\phi(\Y)) \ar[r]^{}  & \mV^\ot(\phi(\X),\phi(\Y')). }\end{xy} \end{equation*}\end{enumerate}\end{definition}There is a dual notion of weakly right enriched $\infty$-category.

\begin{notation}\label{mult}(Multi-morphism spaces)
Let $\mM^\circledast \to \mV^\ot \times \mW^\ot$ be a weakly bienriched $\infty$-category and $\V_1,..., \V_{\n} \in \mV, \ \X, \Y \in \mM, \W_1,...,\W_\m \in \mW$ for some $\n,\m \geq 0 $. Let
$$\Mul_{\mM}(\V_1,..., \V_\n,\X,\W_1,...,\W_\m; \Y)$$ be the full subspace of $\mM^\circledast(\Z,\Y)$ spanned by the morphisms $\Z \to \Y$ in $\mM^\circledast$ lying over the map $[0] \simeq \{0 \} \subset [\n],[0] \simeq \{\m \} \subset [\m]$ in $\Delta \times \Delta,$ where $\Z \in \mM_{[\n],[\m]}^\circledast \simeq \mV^{\times \n} \times \mM \times \mW^{\times\m}$ corresponds to $(\V_1,..., \V_\n,\X,\W_1,...,\W_\m) $.

\end{notation}

\begin{remark}For every $\infty$-operad $\mV^\ot \to \Ass$ the projection $\mM^\circledast:=\mV^\circledast \to \mV^\ot$ gives rise to an embedding
$\mV^\circledast(\Z,\Y) \to \mV^\ot(\Z,\Y) $
that restricts to an equivalence
$$\Mul_{\mM}(\V_1,..., \V_\n,\X,\W_1,...,\W_\m; \Y) \simeq \Mul_{\mV}(\V_1,..., \V_\n,\X,\W_1,...,\W_\m; \Y).$$
\end{remark}

\begin{definition}Let $\phi: \mM^\circledast \to \mV^\ot \times \mW^\ot $ be a weakly bienriched $\infty$-category.
\begin{enumerate}
\item We call $\phi$ locally small if the $\infty$-categories $\mV^\ot, \mW^\ot, \mM^\circledast$ are locally small.

\item We call $\phi$ small if $\phi$ is locally small and $\mM$ is small.

\item We call $\phi$ absolute small if $\phi$ is small and $\mV,\mW$ are small.

\end{enumerate}

\end{definition}

\begin{remark}Let $\phi: \mM^\circledast \to \mV^\ot \times \mW^\ot $ be a weakly bienriched $\infty$-category.
The first axiom of Definition \ref{bla} implies that $\phi$ is absolute small if and only if $ \mM^\circledast, \mV^\ot, \mW^\ot$ are small.
%an $\infty$-category $\phi: \mM^\circledast \to \mV^\ot \times \mW^\ot $ weakly bienriched in small $\infty$-operads is small if and only if $ \mM^\circledast$ is small.
Remark \ref{oprer} and the axioms of Definition \ref{bla} imply that $\phi$ is locally small if and only if the multi-morphism spaces (Notation \ref{mult}) of $\mV^\ot \to \Ass, \mW^\ot \to \Ass, \phi $ are small.	So $\phi$ is small if and only if $\mM$ is small and the multi-morphism spaces of $\mV^\ot \to \Ass, \mW^\ot \to \Ass, \phi  $ are small.		
\end{remark}

Next we define tensored $\infty$-categories.

\begin{definition}\label{leftten}
Let $\phi: \mM^\circledast \to \mV^\ot \times \mW^\ot$ be a weakly bienriched $\infty$-category.

\begin{enumerate}
\item We say that $\phi: \mM^\circledast \to \mV^\ot \times \mW^\ot$ exhibits $\mM$ as left tensored over $\mV$ if $\mV^\ot \to \Ass$ is a monoidal $\infty$-category and $\phi$ is a map of cocartesian fibrations over $\Ass$
via projection to the first factor.

%relative to the collection of morphisms of $\Ass \times \Ass$ whose second component is inert and preserves the minimum.

\item We say that $\phi: \mM^\circledast \to \mV^\ot \times \mW^\ot$ exhibits $\mM$ as right tensored over $\mW$ if $\mW^\ot \to \Ass$ is a monoidal $\infty$-category and $\phi$ is a map of cocartesian fibrations over $\Ass$
via projection to the second factor.

%$\phi$ is a map of cocartesian fibrations relative to the collection of morphisms of $\Ass \times \Ass$ whose first component is inert and preserves the maximum. 

\item We say that $\phi: \mM^\circledast \to \mV^\ot \times \mW^\ot$ exhibits $\mM$ as bitensored over $\mV,\mW$ if $\phi$ exhibits $\mM$ as left tensored over $\mV$ and right tensored over $\mW$.

% or equivalently if $\phi$ is a map of cocartesian fibrations over $\Ass \times \Ass$.

%\item We say that an $\infty$-category $\phi: \mM^\circledast \to \mV^\ot$weakly left enriched over $\mV$ exhibits $\mM$ as left tensored over $\mV$if $\phi$ is a map of cocartesian fibrations over $\Ass$.

\end{enumerate}

\end{definition} 

\begin{remark}

Let $\mV^\ot \to \Ass, \mW^\ot \to \Ass$ be monoidal $\infty$-categories and $\phi: \mM^\circledast \to \mV^\ot \times \mW^\ot$ a map of cocartesian fibrations over $\Ass \times \Ass.$
Then $\phi: \mM^\circledast \to \mV^\ot \times \mW^\ot$ exhibits $\mM$ as bitensored over $\mV,\mW$ if and only if condition (1) of Definition \ref{bla} holds. Condition (2) is then automatic. 
\end{remark}

%We apply Definition \ref{leftten} in particular to the case that $\mV^\ot=\emptyset^\ot$ or $\mW^\ot=\emptyset^\ot$.

\begin{example}
	
Let $\mM^\circledast \to \mV^\ot $ be a (weakly) left tensored $\infty$-category
and $\mN^\circledast \to \mW^\ot $ a (weakly) right tensored $\infty$-category.
The functor $\mM^\circledast \times \mN^\circledast \to \mV^\ot \times \mW^\ot $ is a (weakly) bitensored $\infty$-category.

%Let $\mM^\circledast \to \mV^\ot $ be a weakly left enriched $\infty$-categoryand $\mN^\circledast \to \mW^\ot $ a right tensored $\infty$-category.The functor $\mM^\circledast \times \mN^\circledast \to \mV^\ot \times \mW^\ot $ is a right tensored $\infty$-category.

%Let $\mM^\circledast \to \mV^\ot $ be a left tensored $\infty$-categoryand $\mN^\circledast \to \mW^\ot $ a weakly right enriched $\infty$-category.The functor $\mM^\circledast \times \mN^\circledast \to \mV^\ot \times \mW^\ot $ is a left tensored $\infty$-category.

%Let $\mM^\circledast \to \mV^\ot $ be a left tensored $\infty$-categoryand $\mN^\circledast \to \mW^\ot $ a right tensored $\infty$-category.The functor $\mM^\circledast \times \mN^\circledast \to \mV^\ot \times \mW^\ot $ is a bitensored $\infty$-category.
	
\end{example}

%\begin{definition}A left tensored, right tensored, bitensored $\infty$-category $\mM^\circledast \to \mV^\ot \times \mW^\ot$, respectively, is small if $\mM^\circledast, \mV^\ot, \mW^\ot$ are small.\end{definition}

\begin{definition}Let $\kappa$ be a small regular cardinal. %Let $\phi: \mM^\circledast \to \mV^\ot \times \mW^\ot$ be a bitensored $\infty$-category.
\begin{enumerate}
\item A left tensored $\infty$-category $\phi: \mM^\circledast \to \mV^\ot \times \mW^\ot$ is compatible with $\kappa$-small colimits if $\mV^\ot \to \Ass$ is compatible with $\kappa$-small colimits, $\mM$ admits $\kappa$-small colimits, for every $\V \in \mV, \X \in \mM$ the functors $(-) \ot \X: \mV \to \mM, \V \ot (-): \mM \to \mM$ preserve $\kappa$-small colimits and for every $\W_1,...,\W_\m \in \mW$
for $\m \geq0$ and $\Y \in \mM$ the functor $\Mul_\mM(-,\W_1,...,\W_\m;\Y): \mM^\op \to \mS$ preserves $\kappa$-small limits.

\item A right tensored $\infty$-category $\phi: \mM^\circledast \to \mV^\ot \times \mW^\ot$ is compatible with $\kappa$-small colimits if $\mW^\ot \to \Ass$ is compatible with $\kappa$-small colimits, $\mM$ admits $\kappa$-small colimits, for every $\W \in \mW, \X \in \mM$ the functors $\X \ot (-): \mW \to \mM, (-) \ot \W: \mM \to \mM$ preserve $\kappa$-small colimits and for every $\V_1,...,\V_\n \in \mV$
for $\n \geq0$ and $\Y \in \mM$ the functor $\Mul_\mM(\V_1,...,\V_\n,-;\Y): \mM^\op \to \mS$ preserves $\kappa$-small limits.

\item A bitensored $\infty$-category is compatible with $\kappa$-small colimits if it is a left and right tensored $\infty$-category compatible with $\kappa$-small colimits, i.e. $\mV^\ot \to \Ass, \mW^\ot \to \Ass$ are compatible with $\kappa$-small colimits, $\mM$ admits $\kappa$-small colimits and for every $\V \in \mV, \W \in \mW, \X \in \mM$ the functors $(-)\ot \X: \mV \to \mM, \X \ot (-): \mW \to \mM, \V \ot (-), (-) \ot \W: \mM \to \mM$ preserve $\kappa$-small colimits.

% and right tensored $\infty$-category compatible with small colimits.

%\item We say that $\phi$ is an $\infty$-category with closed left action if for any $\X \in \mM$ the functor$(-) \ot \X: \mV \to \mM$
%$\V \ot (-): \mM \to \mM$ admits a right adjoint.

\end{enumerate}

\end{definition}

\begin{definition}Let $\kappa$ be a small regular cardinal.

\begin{enumerate}
\item A presentably left tensored $\infty$-category is a left tensored $\infty$-category  $\phi: \mM^\circledast \to \mV^\ot \times \mW^\ot$ compatible with small colimits such that $\mV,\mM$ are presentable.

\item A presentably right tensored $\infty$-category is a right tensored $\infty$-category  $\phi: \mM^\circledast \to \mV^\ot \times \mW^\ot$ compatible with small colimits such that $\mM, \mW$ are presentable.

\item A presentably bitensored $\infty$-category is a presentably left tensored and  presentably right tensored $\infty$-category.

\item A left tensored $\infty$-category $\mM^\circledast \to \mV^\ot \times \mW^\ot$ is $\kappa$-compactly generated if $\mV^\ot \to \Ass$ and $\mM$ are $\kappa$-compactly generated and the left $\mV$-action on $\mM$ restricts to a left $\mV^\kappa$-action on $\mM^\kappa$. 

\item A right tensored $\infty$-category $\mM^\circledast \to \mV^\ot \times \mW^\ot$ is $\kappa$-compactly generated if $\mW^\ot \to \Ass$ and $\mM$ are $\kappa$-compactly generated and the right $\mW$-action on $\mM$ restricts to a right $\mW^\kappa$-action on $\mM^\kappa$. 

\item A bitensored $\infty$-category $\mM^\circledast \to \mV^\ot \times \mW^\ot$ is $\kappa$-compactly generated if the left $\mV$-action and right $\mW$-action are
$\kappa$-compactly generated.

\end{enumerate}	
\end{definition}

\begin{remark}

By \cite[Proposition 7.15.]{Rune} every presentably left tensored, right tensored, bitensored $\infty$-category, respectively, is $\kappa$-compactly generated for some regular cardinal $\kappa$.
\end{remark}
%\subsubsection{Tensors and cotensors}

\vspace{1mm}
In the following we generalize the notions of left, right and bitensored $\infty$-categories %(compatible with small colimits) 
by introducing the notions of tensors. %and conical colimits.

\begin{definition}\label{Defo}
	
Let $\mM^\circledast \to \mV^\ot\times \mW^\ot$ be a weakly bienriched $\infty$-category and $\V \in \mV, \W \in \mW, \X,\Y \in \mM$.
	
\begin{enumerate}
%\item A multi-morphism $\psi \in \Mul_\mM(\V,\X, \W;\Y)$ exhibits $\Y$ as the bitensor of $\V, \X, \W$ if for every $\Z \in \mM, \V_1,...,\V_\n \in \mV, \W_1,...,\W_\m \in \mW$ for $\n,\m \geq 0 $ the following map is an equivalence:\begin{equation*}\label{tenss}\Mul_\mM(\V_1,...,\V_\n, \Y, \W_1,...,\W_\m; \Z) \to \Mul_{\mM}(\V_1,...,\V_\n,\V,\X,\W, \W_1,...,\W_\m; \Z).\end{equation*}We write $\V \ot \X \ot \W$ for $\Y$.

\item A multi-morphism $\psi \in \Mul_\mM(\V,\X;\Y)$ exhibits $\Y$ as the left tensor of $\V, \X$, denoted by $\V \ot \X$, if for every $\Z \in \mM, \V_1,...,\V_\n \in \mV, \W_1,...,\W_\m \in \mW$ for $\n,\m \geq 0 $ the following map is an equivalence:
\begin{equation*}\label{tenss}
\Mul_\mM(\V_1,...,\V_\n, \Y, \W_1,...,\W_\m; \Z) \to \Mul_{\mM}(\V_1,...,\V_\n,\V,\X, \W_1,...,\W_\m; \Z).
\end{equation*}
%We write $\V \ot \X $ for $\Y$. 

\item A multi-morphism $\psi \in \Mul_\mM(\X,\W;\Y)$ exhibits $\Y$ as the right tensor of $\X, \W$, denoted by $\X \ot \W,$ if for any $\Z \in \mM, \V_1,...,\V_\n \in \mV, \W_1,...,\W_\m \in \mW$ for $\n,\m \geq 0 $ the map is an equivalence:
\begin{equation*}\label{tenss}
\Mul_\mM(\V_1,...,\V_\n, \Y, \W_1,...,\W_\m; \Z) \to \Mul_{\mM}(\V_1,...,\V_\n,\X,\W, \W_1,...,\W_\m; \Z).
\end{equation*}
%We write $\X\ot\W $ for $\Y$. 

%Similarly, we define when a multi-morphism $\psi \in \Mul_\mM(\W,\X;\Y)$ for $\W \in \mW, \X,\Y \in \mM$ exhibits $\Y$ as the tensor of $\W$ and $\X$.

\vspace{1mm}

%\item A multi-morphism $\tau \in \Mul_\mM(\V,\Y, \W;\X)$ exhibits $\Y$ as the bicotensor of $\V, \X, \W$ if for every $\Z \in \mM, \V_1,...,\V_\n \in \mV, \W_1,...,\W_\m\in \mW$ for $\n,\m \geq 0 $ the following map is an equivalence: $$\Mul_\mM(\V_1,...,\V_\n, \Z, \W_1,...,\W_\m; \Y) \to \Mul_{\mM}(\V,\V_1,...,\V_\n,\Z, \W_1,...,\W_\m,\W; \X).$$ We write $\X^{\V, \W}$ for $\Y$.

\item A multi-morphism $\tau \in \Mul_\mM(\V,\Y;\X)$ exhibits $\Y$ as the left cotensor of $\V, \X$, denoted by $^{\V}\X $, if for any $\Z \in \mM, \V_1,...,\V_\n \in \mV, \W_1,...,\W_\m\in \mW$ for $\n,\m \geq 0 $ the following map is an equivalence: $$\Mul_\mM(\V_1,...,\V_\n, \Z, \W_1,...,\W_\m; \Y) \to \Mul_{\mM}(\V,\V_1,...,\V_\n,\Z, \W_1,...,\W_\m; \X).
$$ 

\item A multi-morphism $\tau \in \Mul_\mM(\Y,\W;\X)$ exhibits $\Y$ as the right cotensor of $\X, \W$, denoted by $\X^\W$, if for every $\Z \in \mM, \V_1,...,\V_\n \in \mV, \W_1,...,\W_\m\in \mW$ for $\n,\m \geq 0 $ the map is an equivalence: $$\Mul_\mM(\V_1,...,\V_\n, \Z, \W_1,...,\W_\m; \Y) \to \Mul_{\mM}(\V_1,...,\V_\n,\Z, \W_1,...,\W_\m, \W; \X).
$$ 

% and $\X^\V$ if $\W=\tu_{\mP\Env(\mV)}$ and $\X^\W$ if $\V=\tu_{\mP\Env(\mV)}.$We call $\X^\V$ the left cotensor of $\V, \X$ and $\X^\W$ the right cotensor of $\X, \W.$	
\end{enumerate}
%Similarly, we define when a multi-morphism $\psi \in \Mul_\mM(\W,\Y;\X)$ for $\W \in \mW, \X,\Y \in \mM$ exhibits $\Y$ as the cotensor of $\W$ and $\X$.
\end{definition}
%By Proposition \ref{eqq} the $\infty$-category $\mM$ is canonically left tensored over $\mP\Env(\mV).$We also allow $\V \in \mP\Env(\mV)$ in the definition of (co)tensors.

\begin{remark}\label{reuj}

Let $\mM^\circledast \to \mV^\ot\times \mW^\ot$ be a weakly bienriched $\infty$-category, $\V \in \mV, \W \in \mW$ and $ \X \in \mM$. 
If the respective left and right tensors exist,
there is a canonical equivalence
$ (\V \ot \X) \ot \W \simeq \V \ot (\X\ot \W).$
In this case we refer to $ (\V \ot \X) \ot \W \simeq \V \ot (\X\ot \W)$ as the bitensor of $\V,\X,\W.$
%Let $\psi \in \Mul_\mM(\V,\X;\Y)$ be a multi-morphism that exhibits $\Y$ as the left tensor of $\V, \X$ and $\psi' \in \Mul_\mM(\Y, \W;\Z)$ a multi-morphism that exhibits $\Z$ as the right tensor of $\Y, \W.$The induced multi-morphism $\psi'' \in \Mul_\mM(\V, \X, \W;\Z)$ exhibits $\Z$ as the bitensor of $\V, \X, \W.$In other words, if $\V \ot \X$ exists and $(\V\ot \X) \ot \W$ exists, then $(\V\ot \X) \ot \W \simeq \V \ot \X \ot \W.$Similarly, if $\X \ot \W$ exists and $\V\ot (\X \ot \W)$ exists, then $\V\ot (\X \ot \W)\simeq\V \ot \X \ot \W.$The same is true for cotensors.
Similarly, if the respective left and right cotensors exist,
there is a canonical equivalence
$ (\X^\V)^\W \simeq (\X^\W)^\V$ and we refer to the latter object as the bicotensor of $\V,\X,\W.$
\end{remark}

\begin{definition}

A weakly bienriched $\infty$-category $\mM^\circledast \to \mV^\ot \times \mW^\ot$ admits
\begin{itemize}

%\item bi(co)tensors if for every object $\V \in \mV, \W \in \mW, \X \in \mM$ there is a bi(co)tensor of $\V, \X, \W.$

\item left (co)tensors if for every object $\V \in \mV, \X \in \mM$ there is a left (co)tensor of $\V$ and $\X$.

\item right (co)tensors if for every object $\W \in \mW, \X \in \mM$ there is a right (co)tensor of $\X$ and $\W.$

\end{itemize}

%Note that $\mM$ admits bi(co)tensors if and only if it has left and right (co)tensors using Remark \ref{reuj}.

\begin{example}
Every weakly bienriched $\infty$-category that is left, right, bitensored admits left tensors, right tensors, left and right tensors, respectively.	
	
\end{example}

\begin{remark}\label{locre}
Let $\phi: \mM^\circledast \to \mV^\ot \times \mW^\ot$ be a weakly bienriched $\infty$-category that admits left and right tensors.
Then $\phi$ is a locally cocartesian fibration. 
	
\end{remark}

\end{definition}
\subsection{Enriched functors}

Next we define maps of weakly bienriched $\infty$-categories.
The next definition is \cite[Definition 2.50.]{heine2024bienriched}:

\begin{definition}\label{linmapp}
Let $\alpha: \mV^\ot \to \mV'^\ot,\beta: \mW^\ot \to \mW'^\ot$ be maps of $\infty$-operads and $\phi: \mM^\circledast \to \mV^\ot \times \mW^\ot, \phi':\mM'^\circledast \to \mV'^\ot \times \mW'^\ot $ weakly bienriched $\infty$-categories.
An enriched functor $\phi \to \phi'$ is a commutative square of $\infty$-categories over $\Ass \times \Ass$ 
\begin{equation*} 
\begin{xy}
\xymatrix{
\mM^\circledast  \ar[d]^{\phi} \ar[rr]^{\gamma}
&&\mM'^\circledast \ar[d]^{\phi'} 
\\ \mV^\ot \times \mW^\ot
\ar[rr]^{\alpha \times \beta}  && \mV'^\ot \times \mW'^\ot
}
\end{xy} 
\end{equation*}
such that $\gamma$ preserves cocartesian lifts of inert morphisms of $\Ass \times \Ass$ whose first component preserves the maximum and whose second component preserves the minimum.

\begin{itemize}
\item An enriched functor $\phi \to \phi'$ is an embedding if $\alpha, \beta, \gamma$ are fully faithful.	
 
\item An enriched functor $\phi \to \phi'$ is left (right) linear if it preserves left (right) tensors.

\item An enriched functor $\phi \to \phi'$ is linear if it is left and right linear.

\item An enriched functor $\phi \to \phi'$ is left $\mV$-enriched if $\alpha$ is the identity.

\item An enriched functor $\phi \to \phi'$ is right $\mW$-enriched if $\beta$ is the identity.

\item An enriched functor $\phi \to \phi'$ is $\mV,\mW$-enriched if it is left $\mV$-enriched and right $\mW$-enriched.

\item An enriched functor $\phi \to \phi'$ is left $\mV$-linear if it is left linear
and left $\mV$-enriched.

\item An enriched functor $\phi \to \phi'$ is right $\mW$-linear if it is right linear
and right $\mW$-enriched.

\item An enriched functor $\phi \to \phi'$ is $\mV,\mW$-linear if it is
left $\mV$-linear and right $\mW$-linear (or equivalently $\mV,\mW$-enriched and linear).

%If $\phi,\phi'$ are bitensored $\infty$-categories

\end{itemize}
\end{definition}

\begin{notation}Let $\mM^\circledast \to \mV^\ot \times \mW^\ot, \mN^\circledast \to \mV'^\ot \times \mW'^\ot$ be weakly bienriched $\infty$-categories.

\begin{enumerate}
\item Let $$\Enr\Fun(\mM,\mN) \subset (\Fun(\mV^\ot, \mV'^\ot) \times \Fun(\mW^\ot, \mW'^\ot)) \times_{\Fun(\mM^\circledast,\mV'^\ot \times \mW'^\ot) } \Fun(\mM^\circledast,\mN^\circledast) $$ be the full subcategory of enriched functors.

\item Let $$\L\LinFun(\mM,\mN), \ \R\LinFun(\mM,\mN), \ \LinFun(\mM,\mN) \subset \Enr\Fun(\mM,\mN)$$ be the full subcategories of left linear, right linear, linear functors, respectively.

\end{enumerate}
%maps of weakly bienriched $\infty$-categories.

%\vspace{1mm}
%\item Let $\Enr\Fun(\mM,\mN)_\L \subset \Enr\Fun(\mM,\mN)$ be the full subcategory of maps of weakly bienriched $\infty$-categories lying over maps of $\infty$-operads that admit a right adjoint relative to $\Ass.$

%\item If $\mM^\circledast \to \mV^\ot \times \mW^\ot, \mN^\circledast \to \mV'^\ot \times \mW'^\ot$ admit bitensors,let $\LinFun(\mM,\mN) \subset \Enr\Fun(\mM,\mN)$ be the full subcategory of maps of weakly bienriched $\infty$-categories that preserve bitensors.

\end{notation} 

\begin{notation}Let $\mM^\circledast \to \mV^\ot \times \mW^\ot, \mN^\circledast \to \mV^\ot \times \mW^\ot$ be weakly bienriched $\infty$-categories.

\begin{enumerate}

\item Let $$ \Enr\Fun_{\mV,\mW}(\mM,\mN) \subset \Fun_{\mV^\ot \times \mW^\ot}(\mM^\circledast,\mN^\circledast)$$ be the full subcategory of $\mV,\mW$-enriched functors.

\item Let $$\L\LinFun_{\mV, \mW}(\mM,\mN) \subset \Enr\Fun_{\mV, \mW}(\mM,\mN),$$$$ \R\LinFun_{\mV, \mW}(\mM,\mN)\subset \Enr\Fun_{\mV, \mW}(\mM,\mN),$$$$ \LinFun_{\mV, \mW}(\mM,\mN) \subset \Enr\Fun_{\mV, \mW}(\mM,\mN)$$ be the full subcategories of left linear, right linear, linear $\mV,\mW$-enriched functors, respectively.
%$\mV$-linear and right $\mW$-enriched functors.

\item %If $\mM^\circledast \to \mV^\ot \times \mW^\ot, \mN^\circledast \to \mV^\ot \times \mW^\ot$ admit left and right tensors, 
Let $$ \LinFun^\L_{\mV, \mW}(\mM,\mN) \subset \LinFun_{\mV, \mW}(\mM,\mN)$$ be the full subcategory of $\mV,\mW$-linear functors whose underlying functor admits a right adjoint.

\end{enumerate}
\end{notation} 

%For the next remark we use the notation of Example \ref{euuz}:
\begin{example}\label{euuz}
Let $\mM^\circledast \to \mV^\ot \times \mW^\ot$ be a weakly bienriched $\infty$-category and $\K$ an $\infty$-category.
\begin{itemize}
\item 
The functor $\K \times \mM^\circledast \to \mV^\ot \times \mW^\ot$ 	
is a weakly bienriched $\infty$-category that exhibits $\K \times \mM$
as weakly bienriched in $\mV,\mW.$
	
\item The pullback $(\mM^\K)^\circledast:= \mV^\ot \times \mW^\ot\times_{\Fun(\K,\mV^\ot \times \mW^\ot)} \Fun(\K,\mM^\circledast) \to \mV^\ot \times \mW^\ot$	along the diagonal functor $\mV^\ot \times \mW^\ot\to \Fun(\K,\mV^\ot \times \mW^\ot)$ is a weakly bienriched $\infty$-category that exhibits $\Fun(\K,\mM)$
as weakly bienriched in $\mV,\mW.$
\end{itemize}
\end{example}

\begin{remark}\label{2-cat}
Let $\mM^\circledast \to \mV^\ot \times \mW^\ot, \mN^\circledast \to \mV^\ot \times \mW^\ot$ be weakly bienriched $\infty$-categories and $\K$ an $\infty$-category.
The canonical equivalences $$ \Fun(\K,\Fun_{\mV^\ot\times\mW^\ot}(\mM^\circledast,\mN^\circledast)) \simeq \Fun_{\mV^\ot\times\mW^\ot}(\K\times \mM^\circledast,\mN^\circledast)
\simeq\Fun_{\mV^\ot\times\mW^\ot}(\mM^\circledast,(\mN^\K)^\circledast)$$
restrict to equivalences
$$ \Fun(\K,\Enr\Fun_{\mV,\mW}(\mM,\mN)) \simeq \Enr\Fun_{\mV,\mW}(\K\times\mM,\mN)
\simeq\Enr\Fun_{\mV,\mW}(\mM,\mN^\K).$$

\end{remark}

We will often use the following proposition, which is \cite[Proposition 4.2.4.2.]{lurie.higheralgebra}:

\begin{proposition}\label{Line}
For every bitensored $\infty$-category $\mM^\circledast \to \mV^\ot \times \mW^\ot$
and $\infty$-category $\K$ the following forgetful functor is an equivalence:
$$\LinFun_{\mV, \mW}(\mV \times \K \times \mW,\mM) \to \Fun(\K,\mM).$$
	
\end{proposition}

%The following lemma is \cite[Lemma 3.74.]{HEINE2023108941}:

%\begin{lemma}\label{colas}

%Let $\mV^\ot \to \Ass, \mW^\ot \to \Ass$ be $\infty$-operads, $\mM^\circledast \to \mV^\ot \times \mW^\ot, \mN^\circledast \to \mV^\ot \times \mW^\ot$ weakly bienriched $\infty$-categories and $\mK \subset \Cat_\infty$ a full subcategory. 
%\begin{enumerate}
%\item If $\mN$ admits $\mK$-indexed conical colimits and left and right tensors, 
%$ \Enr\Fun_{\mV,\mW}(\mM, \mN) $ admits $\mK$-indexed colimits and the functor $ \Enr\Fun_{\mV,\mW}(\mM, \mN) \to \Fun(\mM, \mN)$ preserves $\mK$-indexed colimits.
%\item If for every $\V \in \mV, \W \in \mW$ the functors $\V \ot (-), (-) \ot \W: \mN \to \mN$ are accessible, $\mM^\circledast$ is small and $\mN$ is accessible, then $ \Enr\Fun_{\mV,\mW}(\mM, \mN)$ is accessible.\end{enumerate}\end{lemma}

\begin{notation}

Let $$\omega\B\Enr \subset (\Op_{\infty} \times \Op_{\infty}) \times_{ \Cat_{\infty / \Ass \times \Ass} } \Fun([1], \Cat_{\infty / \Ass \times \Ass}) $$ be the subcategory of weakly bienriched $\infty$-categories.
\end{notation}

\begin{notation}
Evaluation at the target restricts to a forgetful functor $\omega\B\Enr \to \Op_\infty \times \Op_\infty$ 
whose fibers over $\infty$-operads
$\mV^\ot \to \Ass, \mW^\ot \to \Ass$ we denote by $ _\mV\omega\B\Enr_{\mW}.$

\end{notation}

\begin{remark}

There is a canonical equivalence
$$ \omega\B\Enr_{\emptyset, \emptyset} \simeq \Cat_\infty, \ \mM^\circledast \to \emptyset^\ot \times \emptyset^\ot \mapsto \mM^\circledast$$
invserse to the functor $\K \mapsto \emptyset^\ot \times \K \times \emptyset^\ot \simeq \K.$
\end{remark}

%\begin{proposition}\label{pr} For every small $\infty$-operads $\mV^\ot \to \Ass, \mW^\ot \to \Ass$ the $\infty$-category ${_\mV\omega\B\Enr}_{\mW}$ is compactly generated.\end{proposition}

The next proposition is \cite[Proposition 2.62.]{heine2024bienriched}:

\begin{proposition}\label{bica}
The forgetful functor $$\gamma: \omega\B\Enr \to \Op_\infty \times \Op_\infty$$
% \BMod \to \Mon \times \Mon
is a cartesian fibration. %and the embedding $\BMod \subset \omega\B\Enr$ preserves cartesian morphisms.
Let $\psi: \mM^\circledast \to \mN^\circledast$ be an enriched functor lying over maps of $\infty$-operads $\alpha: \mV^\ot \to \mV'^\ot, \alpha': \mW^\ot \to \mW'^\ot$.
The map $\psi$ is $\gamma$-cartesian if and only if 
the commutative square
\begin{equation*} 
\begin{xy}
\xymatrix{
\mM^\circledast \ar[d] \ar[rr]^{\psi}
&&\mN^\circledast \ar[d] 
\\ \mV^\ot \times \mV'^\ot
\ar[rr]^{\alpha \times \alpha'}  && \mW^\ot \times \mW'^\ot
}
\end{xy} 
\end{equation*}
is a pullback square, in other words if $\psi$ induces an equivalence
$\mM \simeq \mN$ and for every $\V_1,...,\V_\n \in \mV,\X,\Y \in \mM, \W_1,...,\W_\m\in \mW$ for $\n,\m \geq 0 $ the following canonical map is an equivalence: $$\Mul_\mM(\V_1,..., \V_\n,\X,\W_1,...,\W_\m;\Y) \to \Mul_\mN(\alpha(\V_1),...,\alpha(\V_\n),\psi(\X), \alpha'(\W_1),...,\alpha'(\W_\m);\psi(\Y)).$$
\end{proposition}

\begin{notation}\label{invo}
The category $\Ass=\Delta^\op$ carries a canonical involution sending $[\n] $ to $[\n]$ and a map $\f:[\n] \to [\m]$ to the map $[\n] \to [\m], \bi \mapsto \m-\f(\n-\bi).$ The involution on $\Ass$ induces an involution on
$\Cat_{\infty/\Ass}$ that restricts to an involution $(-)^\rev$ on $\Op_{\infty} \subset \Cat_{\infty/\Ass}.$ 
Moreover the involution on $\Ass$ induces an involution on
$\Ass \times \Ass$ by applying the involution on $\Ass$ to each factor and switching the factors, which induces an involution on
$(\Cat_{\infty / \Ass} \times \Cat_{\infty / \Ass}) \times_{ \Cat_{\infty / \Ass \times \Ass} } \Fun([1], \Cat_{\infty / \Ass \times \Ass}) $ that restricts to
involutions $(-)^\rev$ on $\omega\B\Enr$ and $\BMod$, under which $ \LMod$ corresponds to $ \RMod$.

\end{notation}

%\begin{remark}\label{symmpu} By \cite[Construction 4.1.2.9.]{lurie.higheralgebra} there is a canonical functor $\chi: \bE_1=\Delta^\op \to \Comm$ sending $[\n] $ to $\langle\n\rangle$. If a monoidal $\infty$-category $\mV^\ot \to \Ass$ is the pullback along $\chi$ of a symmetric monoidal $\infty$-category $\mV^\boxtimes \to \Comm$, there is a canonical monoidal equivalence$(\mV^\rev)^\ot \simeq \mV^\ot$ since $\chi$ factors as $ \Ass \xrightarrow{(-)^\op} \Ass \xrightarrow{\chi}\Comm$.\end{remark}

\begin{remark}\label{2-catt}
	
There is a canonical left action of $\Cat_\infty$ on $\omega\B\Enr$
that sends $\K, \mM^\circledast \to \mV^\ot \times \mW^\ot$ to
$\K \times \mM^\circledast \to \mV^\ot \times \mW^\ot.$
The forgetful functor $\omega\B\Enr \to \Op_\infty \times \Op_\infty$ is $\Cat_\infty$-linear, where the target carries the trivial action.
Thus for any small $\infty$-operads $\mV^\ot \to \Ass, \mW^\ot \to \Ass$ also the fiber $_\mV \omega\B\Enr_\mW$ carries a left $\Cat_\infty$-action that acts the same. These actions are closed by Remark \ref{2-cat} and so exhibit $\omega\B\Enr$ and $_\mV \omega\B\Enr_\mW$ as left $\Cat_\infty$-enriched, i.e. as $(\infty,2)$-categories.
	
\end{remark}

\begin{notation}
Let $$ \LMod, \RMod, \BMod \subset \omega\B\Enr $$
be the subcategories of left tensored, right tensored, bitensored $\infty$-categories and enriched functors preserving left tensors, right tensors, left and right tensors, respectively.
	
\end{notation}

%For later applications we add the following notation:

\begin{notation}
Let $$\cc\cc\LMod \subset \widehat{\LMod},\ \cc\cc\RMod \subset \widehat{\RMod},\ \cc\cc\BMod \subset \widehat{\BMod} $$ be the subcategories of left tensored, right tensored, bitensored $\infty$-categories, respectively, compatible with small colimits and left linear, right linear, linear functors preserving small colimits, respectively.

Let $$\Pr\LMod \subset \cc\cc\LMod,\ \Pr\RMod \subset \cc\cc\RMod,\ \Pr\BMod \subset\cc\cc\BMod $$ be the full subcategories of presentably left tensored, right tensored, bitensored $\infty$-categories, respectively.

\end{notation}

The next proposition follows from \cite[Corollary 8.31.]{HEINE2023108941}:

\begin{proposition}\label{cool} Let $\kappa$ be a small regular cardinal and $\mM^\circledast \to \mV^\ot \times \mW^\ot$ an absolute small bitensored $\infty$-category. 

\begin{enumerate}

\item There is a bitensored $\infty$-category $\Ind_\kappa(\mM)^\circledast \to \Ind_\kappa(\mV)^\ot \times \Ind_\kappa(\mW)^\ot$ compatible with small $\kappa$-filtered colimits and a $\mV,\mW$-linear embedding $
\mM^\circledast \to \mV^\ot \times_{\Ind_\kappa(\mV)^\ot} \Ind_\kappa(\mM)^\circledast \times_{\Ind_\kappa(\mW)^\ot} \mW^\ot$ inducing the embedding $\mM \to \Ind_\kappa(\mM)$ on underlying $\infty$-categories.
\item For every bitensored $\infty$-category $\mN^\circledast \to \Ind_\kappa(\mV)^\ot \times \Ind_\kappa(\mW)^\ot$ compatible with small $\kappa$-filtered colimits the functor
$$
\Enr\Fun_{\Ind_\kappa(\mV),\Ind_\kappa(\mW)}(\Ind_\kappa(\mM), \mN) \to \Enr\Fun_{\mV,\mW}(\mM,\mN)$$
admits a fully faithful left adjoint that lands in the full subcategory $$\Enr\Fun^{\kappa-\mathrm{fil}}_{\Ind_\kappa(\mV),\Ind_\kappa(\mW)}(\Ind_\kappa(\mM), \mN) $$
of enriched functors preserving small $\kappa$-filtered colimits.
So the following is an equivalence:
\begin{equation*}\label{exol2}\Enr\Fun^{\kappa-\mathrm{fil}}_{\Ind_\kappa(\mV),\Ind_\kappa(\mW)}(\Ind_\kappa(\mM), \mN) \to \Enr\Fun_{\mV,\mW}(\mM,\mN).\end{equation*}

\item If $\mM$ admits $\kappa$-small colimits, the latter equivalence restricts to an equivalence
$$\Enr\Fun^{\L}_{\Ind_\kappa(\mV),\Ind_\kappa(\mW)}(\Ind_\kappa(\mM), \mN) \to \Enr\Fun^\kappa_{\mV,\mW}(\mM,\mN).$$

%\item If $\mM^\circledast \to \mV^\ot \times \mW^\ot$ is compatible with $\kappa$-small colimits, $\Ind_\kappa(\mM)^\circledast \to \Ind_\kappa(\mV)^\ot \times \Ind_\kappa(\mW)^\ot$ is compatible with small colimits.

%\item If $\mM^\circledast \to \mV^\ot \times \mW^\ot$ is a bitensored $\infty$-category, $\Ind_\kappa(\mM)^\circledast \to \Ind_\kappa(\mV)^\ot \times \Ind_\kappa(\mW)^\ot$ is a bitensored $\infty$-category, the embedding (\ref{jett}) is linearand for every bitensored $\infty$-category $\mN^\circledast \to \Ind_\kappa(\mV)^\ot \times \Ind_\kappa(\mW)^\ot$ compatible with small colimits the functor (\ref{exol2}) reflects linear functors.

%\item If $\mM^\circledast \to \mV^\ot \times \mW^\ot$ is a left tensored $\infty$-category, $\Ind_\kappa(\mM)^\circledast \to \Ind_\kappa(\mV)^\ot \times \Ind_\kappa(\mW)^\ot$ is a left tensored $\infty$-category, the embedding (\ref{jett}) is left linearand for every left tensored $\infty$-category $\mN^\circledast \to \Ind_\kappa(\mV)^\ot \times \Ind_\kappa(\mW)^\ot$ compatible with small colimits the functor (\ref{exol2}) reflects left linear functors.
%Thus the following functors are equivalences:$$\LinFun^{\kappa-\mathrm{fil}}_{\Ind_\kappa(\mV),\Ind_\kappa(\mW)}(\Ind_\kappa(\mM), \mN) \to \LinFun_{\mV,\mW}(\mM,\mN),$$$$\LinFun^{\L}_{\Ind_\kappa(\mV),\Ind_\kappa(\mW)}(\Ind_\kappa(\mM), \mN) \to \LinFun^\kappa_{\mV,\mW}(\mM,\mN).$$

\end{enumerate}

\end{proposition}

%\begin{remark}\label{Reim}Let $\kappa$ be a small regular cardinal and $\mM^\circledast \to \mV^\ot \times \mW^\ot$ an absolute small weakly bienriched $\infty$-category. 
%For every $\V_1, ..., \V_\n \in \mV, \W_1, ..., \W_\m \in \mW$ for $\n,\m \geq 0$ the functor
%$$\Mul_{\Ind_\kappa(\mM)}(\V_1, ..., \V_\n, \X, \W_1, ..., \W_\m;-) : \Ind_\kappa(\mM) \to \mS $$
%preserves small $\kappa$-filtered colimits. This follows from the fact that the embedding $\iota: \Ind_\kappa(\mM) \subset \mP(\B\Env(\mM))$ preserves small $\kappa$-filtered colimits
%so that the latter functor identifies with the functor
%$$\mP(\B\Env(\mM))(\V_1 \ot  ...\ot \V_\n \ot \X\ot \W_1\ot ...\ot \W_\m,-) \circ \iota : \Ind_\kappa(\mM) \to \mS. $$\end{remark}

\begin{notation}
Let  $\mM^\circledast \to \mV^\ot \times \mW^\ot$ be an absolute small weakly bienriched $\infty$-category. We set $$\mP(\mV)^\ot:=\Ind_\emptyset(\mV)^\ot, \
\mP(\mM)^\circledast:= \Ind_\emptyset(\mM)^\circledast \to \mP(\mV)^\ot \times \mP(\mW)^\ot.$$

\end{notation}

Next we define weakly enriched $\infty$-categories of enriched functors. The following lemma
is \cite[Lemma 5.4.]{heine2024bienriched}: % follows from \cite[Corollary 3.29.]{heine2024localglobalprincipleparametrizedinftycategories}:

\begin{lemma}\label{lemist}
Let $\mC \to \T, \T \to \rS$ be functors such that $\mC \to \T \to \rS$ is a cocartesian fibration. The functor 
$(-)\times_\rS \mC: \Cat_{\infty / \rS} \to \Cat_{\infty / \mC} \to \Cat_{\infty / \T}$ admits a right adjoint that we denote by
$ \Fun_\T^{\rS}(\mC,-).$ 
%If $\rS, \T $ are contractible, we drop $\rS$, $\T$ from the notation.	
\end{lemma}

The following remark is \cite[Remark 3.23.]{heine2024localglobalprincipleparametrizedinftycategories}:

\begin{remark}\label{puas}
For every functors $\mD \to \T$ and $\rS' \to \rS$ there is a canonical equivalence 
$ \rS' \times_\rS  \Fun_\T^{\rS}(\mC,\mD) \simeq \Fun_{\rS' \times_\rS \T}^{\rS'}( \rS' \times_\rS \mC, \rS' \times_\rS \mD)$ specifying the fibers of the functor $\Fun_\T^{\rS}(\mC,\mD) \to \rS.$
\end{remark}

%\begin{remark}\label{reuil}Let $\phi: \T \to \rS$ be a functor, $\mE \subset \Fun([1],\rS)$ a full subcategoryand $\mC \to \T$ a cartesian fibration relative to the collection of $\phi$-cocartesian lifts of morphisms of $\rS$ that belong to $\mE$ and $\mD \to \T$ a cocartesian fibration relative to the collection of $\phi$-cocartesian lifts of morphisms of $\rS$ that belong to $\mE$.By \cite[Proposition B.4.1.]{lurie.higheralgebra} the functor $\Fun_\T^{\rS}(\mC,\mD) \to \rS$ is a cocartesian fibration relative to $\mE.$\end{remark}

%\begin{lemma}\label{faith}Let $\mC \to \T, \T \to \rS$ be functors such that the composition $\mC \to \T \to \rS$ is a cocartesian fibration.Let $\mD \to \mE$ be a fully faithful functor over $\T$. Then $\Fun^\rS_\T(\mC,\mD) \to \Fun^\rS_\T(\mC,\mE) $ is fully faithful.\end{lemma}

The next proposition is \cite[Lemma 5.8.]{heine2024bienriched}:

\begin{proposition}\label{innerho}
Let $\mM^\circledast \to \mV^\ot, \mN^\circledast \to \mV^\ot \times \mW^\ot$ be weakly left, weakly bienriched $\infty$-categories, respectively.
The functor $\Fun^{\mW^\ot}_{\mV^\ot \times \mW^\ot}(\mM^\circledast \times \mW^\ot, \mN^\circledast) \to \mW^\ot $ is a	weakly right enriched $\infty$-category.

%\begin{enumerate}\item The functor $\alpha: \Fun^{\mW^\ot}_{\mV^\ot \times \mW^\ot}(\mM^\circledast \times \mW^\ot, \mN^\circledast) \to \mW^\ot $ is a	weakly right enriched $\infty$-category.\item The functor $\beta: \Fun^{\mV^\ot}_{\mV^\ot \times \mW^\ot}(\mV^\ot \times \mO^\circledast, \mN^\circledast) \to \mV^\ot $ is a weakly left enriched $\infty$-category.

%\vspace{1mm}

%\item If $\mN^\circledast \to\mV^\ot \times \mW^\ot$ exhibits $\mN$ as right tensored over $\mW$, the functor $\alpha$ is a right tensored $\infty$-category and for every $\X \in \mM$ the following canonical functor is right $\mW$-linear:$$ \Fun^{\mW^\ot}_{\mV^\ot \times \mW^\ot}(\mM^\circledast \times \mW^\ot, \mN^\circledast) \to \mN^\circledast.$$

%\item If $\mN^\circledast \to\mV^\ot \times \mW^\ot$ exhibits $\mN$ as left tensored over $\mV$, the functor $\beta$ is a left tensored $\infty$-category and for every $\X \in \mO$ the following canonical functor is left $\mV$-linear:$$ \Fun^{\mV^\ot}_{\mV^\ot \times \mW^\ot}(\mV^\ot \times \mO^\circledast, \mN^\circledast) \to \mN^\circledast.$$\end{enumerate}
\end{proposition}
\begin{notation}\label{bbbb}
Let $\mM^\circledast \to \mV^\ot, \mN^\circledast \to \mV^\ot \times \mW^\ot$ be weakly left, weakly bienriched $\infty$-categories, respectively.
Let $$\Enr\Fun_{\mV, \emptyset}(\mM,\mN)^\circledast \subset \Fun^{\mW^\ot}_{\mV^\ot \times \mW^\ot}(\mM^\circledast \times \mW^\ot, \mN^\circledast)\to \mW^\ot $$ be the full subcategory weakly right enriched in $\mW$ spanned by 
$$\Enr\Fun_{\mV, \emptyset}(\mM, \mN) \subset \Fun_{\mV^\ot}(\mM^\circledast, \mN_{[0]}^\circledast) \simeq \Fun^{\mW^\ot}_{\mV^\ot \times \mW^\ot}(\mM^\circledast \times \mW^\ot, \mN^\circledast)_{[0]}.$$

\end{notation}

%The next notation is \cite[Corollary 8.31.]{HEINE2023108941}:

\begin{notation}\label{bbbb}
Let $\mO^\circledast \to \mW^\ot, \mN^\circledast \to \mV^\ot \times \mW^\ot$ be weakly right, weakly bienriched $\infty$-categories, respectively.
Let $$\Enr\Fun_{\emptyset,\mW}(\mO,\mN)^\circledast:=(\Enr\Fun_{\mW, \emptyset}(\mO^\rev,\mN^\rev)^\rev)^\circledast \to \mV^\ot $$ 
be the weakly left $\mV$-enriched $\infty$-category.
\end{notation}

%\begin{remark}There is a canonical right $\mW$-enriched equivalence$$\Enr\Fun_{\mV, \emptyset}(\mM,\mN)^\circledast \simeq (\Enr\Fun_{\emptyset,\mV}(\mM^\rev,\mN^\rev)^\rev)^\circledast.$$\end{remark}

The following proposition is \cite[Proposition 3.78.]{HEINE2023108941}:

\begin{proposition}\label{lehmmm} 
Let $\mM^\circledast \to \mV^\ot, \mO^\circledast \to \mW^\ot, \mN^\circledast \to \mV^\ot \times \mW^\ot$ be weakly left, weakly right, weakly bienriched $\infty$-categories, respectively.
The canonical equivalence
$$ \Fun_{\mW^\ot}(\mO^\circledast, \Fun^{\mW^\ot}_{\mV^\ot \times \mW^\ot}(\mM^\circledast \times \mW^\ot, \mN^\circledast)) \simeq \Fun_{\mV^\ot \times \mW^\ot}(\mM^\circledast \times \mO^\ot, \mN^\circledast)$$
restricts to an equivalence
\begin{equation*}\label{equino}
\Enr\Fun_{\emptyset, \mW}(\mO, \Enr\Fun_{\mV, \emptyset}(\mM,\mN)) \simeq \Enr\Fun_{\mV, \mW}(\mM \times \mO, \mN).
\end{equation*} 

%\begin{enumerate}
%\item The canonical equivalence
%$$ \Fun_{\mW^\ot}(\mO^\circledast, \Fun^{\mW^\ot}_{\mV^\ot \times \mW^\ot}(\mM^\circledast \times \mW^\ot, \mN^\circledast)) \simeq \Fun_{\mV^\ot \times \mW^\ot}(\mM^\circledast \times \mO^\ot, \mN^\circledast)$$
%restricts to an equivalence
%\begin{equation*}\label{equino}
%\Enr\Fun_{\emptyset, \mW}(\mO, \Enr\Fun_{\mV, \emptyset}(\mM,\mN)) \simeq \Enr\Fun_{\mV, \mW}(\mM \times \mO, \mN).\end{equation*} 

%\item The canonical equivalence
%$$ \Fun_{\mV^\ot}(\mM^\circledast,\Fun^{\mV^\ot}_{\mV^\ot \times \mW^\ot}(\mV^\ot \times \mO^\circledast, \mN^\circledast)) \simeq \Fun_{\mV^\ot \times \mW^\ot}(\mM^\circledast \times \mO^\ot, \mN^\circledast)$$
%restricts to an equivalence$$ \Enr\Fun_{\mV, \emptyset}(\mM, \Enr\Fun_{\emptyset, \mW}(\mO,\mN)) \simeq \Enr\Fun_{\mV,\mW}(\mM \times \mO, \mN).$$\end{enumerate}
\end{proposition}

%\begin{remark}\label{lubi}\label{leman}
%Let $\mM^\circledast \to \mV^\ot, \mN^\circledast \to \mV^\ot \times \mW^\ot$ be weakly left, weakly bienriched $\infty$-categories, respectively. 
%There is a canonical right $\mW$-enriched equivalence$$\Enr\Fun_{\mV, \emptyset}(\mM,\mN) \simeq \Enr\Fun_{\emptyset, \mV}(\mM^\rev,\mN^\rev)^\rev.$$
%For every map $\mW'^\ot \to \mW^\ot$ of $\infty$-operadsthere is a canonical $\mW'$-enriched equivalence
%$$ \mW'^\ot \times_{\mW^\ot} \Enr\Fun_{\mV, \emptyset}(\mM,\mN)^\circledast \simeq \Enr\Fun_{\mV, \emptyset}(\mM,\mW' \times_\mW \mN)^\circledast.$$
%represented by the canonical equivalence for every weakly right enriched $\infty$-category $\mO^\circledast \to \mW^\ot:$$$ \Enr\Fun_{\emptyset, \mW}(\mO, \Enr\Fun_{\mV, \emptyset}(\mM,\mN)) \simeq \Enr\Fun_{\mV, \mW}(\mM \times \mO, \mN)\simeq \Enr\Fun_{\mW,\mV}((\mM \times \mO)^\rev, \mN^\rev)$$$$\simeq \Enr\Fun_{\mW,\mV}(\mO^\rev \times \mM^\rev, \mN^\rev) \simeq \Enr\Fun_{\mW, \emptyset}(\mO^\rev, \Enr\Fun_{\emptyset, \mV}(\mM^\rev,\mN^\rev)) $$$$\simeq \Enr\Fun_{\emptyset,\mW}(\mO, \Enr\Fun_{\emptyset, \mV}(\mM^\rev,\mN^\rev)^\rev).$$\end{remark}

The next remark is \cite[Remark 5.13.]{heine2024bienriched}.

\begin{remark}

Let $\mM^\circledast \to \mV^\ot, \mN^\circledast \to \mV^\ot \times \mW^\ot, \mN'^\circledast \to \mV^\ot \times \mW^\ot$ be weakly left, weakly bienriched $\infty$-categories, respectively.
For every $\mV,\mW$-enriched embedding $ \mN^\circledast \to \mN'^\circledast$
the induced right $\mW$-enriched functor
$\Enr\Fun_{\mV, \emptyset}(\mM,\mN)^\circledast \to \Enr\Fun_{\mV, \emptyset}(\mM,\mN')^\circledast$
is an embedding.

\end{remark}

%\begin{proposition}\label{innerhosty} Let $\mK \subset \Cat_\infty$ be a full subcategory and $\mM^\circledast \to \mV^\ot, \mO^\circledast \to \mW^\ot, \mN^\circledast \to \mV^\ot \times \mW^\ot$ weakly  left enriched, weakly right enriched, weakly bienriched $\infty$-categories, respectively.

%\begin{enumerate}\item If $\mN^\circledast \to \mV^\ot \times \mW^\ot$ admits right tensorsand forming right tensors preserves $\mK$-indexed colimits, the weakly right enriched $\infty$-category $\Enr\Fun_{\mV, \emptyset}(\mM,\mN)^\circledast \to \mW^\ot$ admits right tensors and $\mK$-indexed conical colimits and the right $\mW$-enriched functor$\Enr\Fun_{\mV, \emptyset}(\mM,\mN)^\circledast \to \Fun(\mM,\mN)^\circledast $ is right $\mW$-linear and preserves $\mK$-indexed colimits.

%\item If $\mN^\circledast \to \mV^\ot \times \mW^\ot$ admits left tensorsand forming left tensors preserves $\mK$-indexed colimits, the weakly left enriched $\infty$-category $\Enr\Fun_{\emptyset, \mW}(\mO,\mN)^\circledast \to \mV^\ot$ admits left tensors and $\mK$-indexed conical colimits and the left $\mV$-enriched functor$ \Enr\Fun_{\emptyset, \mW}(\mO,\mN)^\circledast \to \Fun(\mO,\mN)^\circledast $ is left $\mV$-linear and preserves $\mK$-indexed colimits.\end{enumerate}\end{proposition}

Next we consider trivial weakly bienriched $\infty$-categories.

\begin{definition}Let $\mV^\ot \to \Ass, \mW^\ot \to \Ass$ be $\infty$-operads.
	
\begin{enumerate}
\item A weakly bienriched $\infty$-category $\mM^\circledast \to \mV^\ot \times \mW^\ot$ 
is left trivial if for every $\n>0,\m \geq 0$ and $\V_1,...,\V_\n \in \mV, \W_1,...,\W_\m \in \mW,\X,\Y \in \mM$ the space $\Mul_{\mM}(\V_1,...,\V_\n,\X, \W_1,...,\W_\m;\Y)$ is empty.	
\item A weakly bienriched $\infty$-category $\mM^\circledast \to \mV^\ot \times \mW^\ot$ 
right trivial if for every $\n \geq 0,\m > 0$ and $\V_1,...,\V_\n \in \mV, \W_1,...,\W_\m \in \mW,\X,\Y \in \mM$ the space $\Mul_{\mM}(\V_1,...,\V_\n,\X, \W_1,...,\W_\m;\Y)$ is empty.	
\item A weakly bienriched $\infty$-category is trivial if it is left trivial and right trivial.
\end{enumerate}	
	
\end{definition}

\begin{remark}\label{tent}
A weakly bienriched $\infty$-category $\mM^\circledast \to \mV^\ot \times \mW^\ot$ is left trivial if and only if the functor $\mM^\circledast \to \mV^\ot$ factors through the subcategory $\triv_\mV^\circledast \subset \mV^\ot.$
And dually for right triviality.
	
\end{remark}	

%\begin{remark}\label{Poll}Let $\mV'^\ot \to \mV^\ot, \mW'^\ot \to \mW^\ot$ be maps of $\infty$-operads. For every left trivial, right trivial, trivial weakly bienriched $\infty$-category $\mM^\circledast \to \mV^\ot \times \mW^\ot$ the pullback$\mV'^\ot \times_{\mV^\ot} \mM^\circledast \times_{\mW^\ot}\mW'^\ot \to \mV'^\ot \times \mW'^\ot$ is a left trivial, right trivial, trivial weakly bienriched $\infty$-category, respectively.\end{remark}Next we consider examples:

\begin{notation}
	
Let $\Ass_\mi,\Ass_\ma \subset\Ass$ be the subcategories with the same objects and with morphisms the inert morphisms preserving the minimum, maximum, respectively.
\end{notation}

%\begin{example}\emph{ }\begin{enumerate}\item The inclusion $\Ass_\ma \subset \Ass$ exhibits $[0]$ as weakly left enriched in $[0]$, where the left enrichment is left trivial.
%\item The inclusion $\Ass_\mi \subset \Ass$ exhibits $[0]$ as weakly right enriched in $[0]$, where the right enrichment is right trivial.\end{enumerate}\end{example}

\begin{notation}Let $\mV^\ot \to \Ass, \mW^\ot \to \Ass$ be $\infty$-operads and $\K$ an $\infty$-category.

\begin{enumerate}
\item Let $$_\mV\triv^\circledast :=\mV^\ot \times_\Ass \Ass_\ma \to \mV^\ot.$$

\item Let $$\triv_\mW^\circledast :=\mW^\ot \times_\Ass \Ass_\mi \to \mW^\ot.$$

\item Let $$\K_{\mV,\mW}^\circledast :=\triv_\mV^\circledast \times \K \times \triv_\mW^\circledast \to \mV^\ot \times \mW^\ot.$$
\end{enumerate}
\end{notation}

%\begin{remark}Observe that $\triv_\mV, \triv_\mW, \triv_{\mV,\mW}$ are contractibleamd for every $\V_1,...,\V_\n \in \mV,\X,\Y \in \triv_\mV$ for $\n \geq0$the space $\Mul_{\triv_\mV}(\V_1,...,\V_\n,\X;\Y)$ is empty if $\n >0$ and contractible if $\n=0$.
%So $\ast_\mV$ is contractible.In particular, $\ast_\mV^\circledast \to \mV^\ot$ is different from the final weakly left enriched $\infty$-category given by $\id: \mV^\ot \to \mV^\ot$, whose underlying $\infty$-category is contractible, too, but whose multimorphism spaces are all contractible.\end{remark}

%\begin{remark}Note that $\ast_\mV^\circledast $ is the full weakly left enriched subcategoryof $ \mV^\ot \times_{\mP\Env(\mV)^\ot} \mP\Env(\mV)^\circledast $ spanned by the tensor unit and so identifies with the pullback $\mV^\ot \times_{\mP\Env(\mV)^\ot}B\tu_{\mP\Env(\mV)}^\circledast \to \mV^\ot.$ \end{remark}

%\begin{warning}Note that the weakly left enriched $\infty$-category $\ast_\mV^\circledast \to \mV^\ot$ is different from the final weakly left enriched $\infty$-categorygiven by $\id: \mV^\ot \to \mV^\ot$, whose underlying $\infty$-category is contractible, too, but whose multimorphism spaces are all contractible.\end{warning}Remark \ref{Poll} gives the following example:

\begin{example}\label{Triv}

Let $\mV^\ot \to \Ass, \mW^\ot \to \Ass$ be $\infty$-operads
and $\K$ a small $\infty$-category.
Then $_\mV\triv^\circledast \to \mV^\ot, \triv_\mW^\circledast \to \mW^\ot, \K_{\mV, \mW}^\circledast \to \mV^\ot \times \mW^\ot$
are trivial weakly left, weakly right enriched, weakly bienriched $\infty$-categories, respectively.
We call $\K_{\mV, \mW}^\circledast \to \mV^\ot \times \mW^\ot$ the trivial weakly bienriched $\infty$-category on $\K$.
\end{example}

%\begin{remark}\label{aii}The object $[0]$ is a final object of the subcategories $\Ass_\ma, \Ass_\mi \subset \Ass$ since for every $\n \geq0$ there is a unique inert order preserving map $[0]\to[\n]$ preserving the maximum, minimum, respectively.	The functors $_\mV\triv^\circledast \to \Ass_\ma, \triv_\mW^\circledast \to \Ass_\mi$ are cocartesian fibrations whose fiber over the final object is contractible. Thus the unique object of ${_\mV \triv^\circledast}, \triv_\mW^\circledast $ lying over the final object is a final object of $_\mV\triv^\circledast, \triv_\mW^\circledast,$ respectively.\end{remark} 

The next Proposition is \cite[Proposition 2.87.]{heine2024bienriched}:

\begin{proposition}\label{ljnbfg}Let $\mN^\circledast \to \mV^\ot \times \mW^\ot$ be a weakly bienriched $\infty$-category.
\begin{enumerate}\item Let $\mM^\circledast \to \mV^\ot $ be a weakly left enriched $\infty$-category. The following functor is an equivalence: $$\Enr\Fun_{\mV, \mW}(\mM \times \triv_\mW, \mN) \to \Enr\Fun_{\mV,\emptyset}(\mM, \mN).$$

\item Let $\mM^\circledast \to \mW^\ot $ be a weakly right enriched $\infty$-category. The following functor is an equivalence: $$\Enr\Fun_{\mV, \mW}(_\mV\triv \times \mM, \mN) \to \Enr\Fun_{\emptyset, \mW}(\mM, \mN).$$
\item Let $\mM^\circledast \to \mV^\ot \times \mW^\ot $ be a weakly bienriched $\infty$-category. The next functor is an equivalence: $$\Enr\Fun_{\mV, \mW}(_\mV\triv \times \triv_\mW, \mN) \to \mN.$$
\end{enumerate}
\end{proposition}

%\begin{corollary}\label{asio}Let $\mV^\ot \to \Ass, \mW^\ot \to \Ass$ be small $\infty$-operads.\begin{enumerate}\item The functor ${_\mV\omega\B\Enr_\mW} \to {_\mV\omega\B\Enr_\emptyset}$taking pullback along $\emptyset^\circledast \subset \mW^\circledast$admits a fully faithful left adjoint whose essential image precisely consists of theright trivial weakly bienriched $\infty$-categories.\item The functor ${_\mV\omega\B\Enr_\mW} \to {_\emptyset\omega\B\Enr_\mW}$taking pullback along $\emptyset^\circledast \subset \mV^\circledast$admits a fully faithful left adjoint whose essential image precisely consists of theleft trivial weakly bienriched $\infty$-categories.
%\item The functor ${_\mV\omega\B\Enr_\mW} \to \Cat_\infty, \mM^\circledast \to \mV^\ot \times \mW^\ot \mapsto \mM$admits a fully faithful left adjoint whose essential image precisely consists of the trivial weakly bienriched $\infty$-categories.\end{enumerate}	

% $\mM^\circledast \to \mV^\ot \times \mW^\ot$ such that for every $\n,\m > 0$ and $\V_1,...,\V_\n \in \mV, \W_1,...,\W_\m \in \mW,\X,\Y \in \mM$ the space $\Mul_{\mM}(\V_1,...,\V_\n,\X, \W_1,...,\W_\m;\Y)$ is empty.\end{corollary}

%The next propositon is \cite[Proposition 3.29.]{HEINE2023108941}:

Next we define enriched adjunctions following \cite[§ 2.4.]{heine2024bienriched}.
The next definitions are \cite[Definitions 2.66., 2.67.]{heine2024bienriched}:

\begin{definition}
An enriched adjunction is an adjunction in the $(\infty,2)$-category $\omega\B\Enr$ of Remark \ref{2-catt}.	
	
\end{definition}

\begin{definition}Let $\mV^\ot \to \Ass, \mW^\ot \to \Ass$ be small $\infty$-operads. An $\mV,\mW$-enriched adjunction is an adjunction in the $(\infty,2)$-category $_\mV\omega\B\Enr_\mW$ of Remark \ref{2-catt}.	
	
\end{definition}

\begin{remark}
A $\mV,\mW$-enriched adjunction is an adjunction relative to $\mV^\ot \times \mW^\ot.$
	
\end{remark}

\begin{remark}
Enriched adjunctions were also discussed in \cite[6.2.6.]{heine2019restricted}
and for Gepner-Haugseng's model of enrichment in \cite[§4.1., §4.2.]{stefanich2020presentable}.
	
\end{remark}

The next remark is \cite[Remark 2.71.]{heine2024bienriched}:
\begin{remark}Let $\mM^\circledast \to \mV^\ot \times \mW^\ot, \mN^\circledast \to \mV'^\ot \times \mW'^\ot$ be weakly bienriched $\infty$-categories. 
	
\begin{enumerate}
\item An enriched functor $ \G: \mN^\circledast \to \mM^\circledast$ lying over maps of $\infty$-operads $\gamma: \mV'^\ot \to \mV^\ot, \delta: \mW'^\ot \to \mW^\ot$
admits an enriched left adjoint if $\gamma, \delta$ admit left adjoints $\alpha, \beta$ relative to $\Ass$, respectively, and for every $\X \in \mM$ there is an $\Y \in \mN$
and a morphism $\X \to \G(\Y) $ in $\mM$ such that for any $\V_1,...,\V_\n \in \mV, \W_1,..., \W_\m \in \mW$ for $\n, \m \geq 0$ and $\Z \in \mN$ the following map is an equivalence:
$$ \Mul_\mN(\alpha(\V_1),...,\alpha(\V_\n),\Y, \beta(\W_1),..., \beta(\W_\m);\Z) \to$$$$ \Mul_\mM(\gamma(\alpha(\V_1)),...,\gamma(\alpha(\V_\n)),\G(\Y), \delta(\beta(\W_1)),..., \delta(\beta(\W_\m));\G(\Z))\to $$$$ \Mul_\mM(\V_1,...,\V_\n,\X, \W_1,..., \W_\m;\G(\Z)).$$

\item An enriched functor $ \F: \mM^\circledast \to \mN^\circledast$ lying over maps of $\infty$-operads $\alpha: \mV^\ot \to \mV'^\ot, \beta: \mW^\ot \to \mW'^\ot$
admits an enriched right adjoint if $\alpha, \beta$ admit right adjoints $\gamma, \delta$ relative to $\Ass$, respectively, and for every $\Y \in \mN$ there is an $\X \in \mM$
and a morphism $\F(\X) \to \Y $ in $\mN$ such that for any $\V_1,...,\V_\n \in \mV, \W_1,..., \W_\m \in \mW$ for $\n, \m \geq 0$ and $\Z \in \mM$ the following map is an equivalence:
$$\hspace{12mm}\Mul_\mM(\V_1,...,\V_\n,\Z, \W_1,..., \W_\m;\X) \to \Mul_\mN(\alpha(\V_1),...,\alpha(\V_\n),\F(\Z), \beta(\W_1),..., \beta(\W_\m);\F(\X))$$$$ \to \Mul_\mN(\alpha(\V_1),...,\alpha(\V_\n),\F(\Z), \beta(\W_1),..., \beta(\W_\m);\Y).$$

\item The same holds for $\mV,\mW$-enriched left and right adjoint, where $\alpha, \beta$ are the identities.
\end{enumerate}
\end{remark}

We will often use the following lemma, which is \cite[Lemmas 2.72, 2.73.]{heine2024bienriched}:

\begin{lemma}\label{Adj}\label{Adj2} 
\begin{enumerate}
	
\item Let $\F: \mM^\circledast \to \mN^\circledast$ be an enriched functor between weakly bienriched $\infty$-categories that lies over maps of $\infty$-operads that admit right adjoints relative to $\Ass$ such that $\mM$ admits left and right tensors.
Then $\F: \mM^\circledast \to \mN^\circledast$ admits an enriched right adjoint if and only if
$\F$ is linear and the underlying functor $\mM \to \mN$ of $\F$ admits a right adjoint $\G:\mN \to \mM$.

\vspace{1mm}
\item Let $\G: \mN^\circledast \to \mM^\circledast$ be an enriched
functor between weakly bienriched $\infty$-categories that lies over maps of $\infty$-operads that admit left adjoints relative to $\Ass$ such that $\mM, \mN$ admits left and right tensors.
Then $\G: \mN^\circledast \to \mM^\circledast$ admits an enriched left adjoint if and only if the underlying functor $\mN \to \mM$ of $\G$ admits a left adjoint $\F:\mM \to \mN$ that is linear.

\end{enumerate}\end{lemma}

%\begin{corollary}Let $\mM^\circledast \to \mV^\ot \times \mW^\ot, \mN^\circledast \to \mV^\ot \times \mW^\ot$ be weakly bienriched $\infty$-categories that admit left and right tensors.Then $$\Enr\Fun^\L_{\mV, \mW}(\mM,\mN) =\LinFun^\L_{\mV, \mW}(\mM,\mN).$$\end{corollary}

%\begin{notation}Let $\mV^\ot \to \Ass,\mW^\ot\to\Ass$ be small $\infty$-operads.Let $${_\mV\omega\B\Enr^\L_\mW}, {_\mV \omega\B\Enr^\R_\mW} \subset {_\mV\omega\B\Enr_\mW} $$ be the subcategories with the same objects and with morphisms the $\mV,\mW$-enriched functors that admit a $\mV,\mW$-enriched left adjoint, right adjoint, respectively.\end{notation}

%Lemma \ref{ohai} implies the following:\begin{corollary}\label{kurt}Let $\mV^\ot \to \Ass,\mW^\ot\to\Ass$ be small $\infty$-operads.There is a canonical equivalence $${_\mV\omega\B\Enr^\L_\mW} \simeq ({_\mV\omega\B\Enr^\R_\mW})^\op .$$\end{corollary}

\begin{notation}Let $\mM^\circledast \to \mV^\ot \times \mW^\ot, \mN^\circledast \to \mV^\ot \times \mW^\ot$ be weakly bienriched $\infty$-categories.
Let
$$\Enr\Fun^\L_{\mV, \mW}(\mM,\mN) \subset \Enr\Fun_{\mV, \mW}(\mM,\mN),\ \Enr\Fun^\R_{\mV, \mW}(\mM,\mN) \subset \Enr\Fun_{\mV, \mW}(\mM,\mN)$$ be the full subcategories of $\mV, \mW$-enriched functors that admit a $\mV,\mW$-enriched left adjoint, a $\mV,\mW$-enriched right adjoint, respectively.	

\end{notation}

The next proposition is \cite[Proposition 2.115.]{heine2024bienriched}:

\begin{proposition}\label{adjeq}
Let $\mM^\circledast \to \mV^\ot \times\mW^\ot, \mN^\circledast \to \mV^\ot \times\mW^\ot$ be weakly bienriched $\infty$-categories.
There is a canonical equivalence
$$\Enr\Fun^\L_{\mV, \mW}(\mM,\mN)^\op \simeq \Enr\Fun_{\mV,\mW}^\R(\mN,\mM).$$	
\end{proposition}

%Remark \ref{reyyt} and Remark \ref{2-cat} imply the following:

%\begin{remark}\label{indadj}Let $\mO^\circledast \to \mV^\ot \times \mW^\ot$ be a weakly bienriched $\infty$-category, $\F: \mM^\circledast \to \mN^\circledast, \G: \mN^\circledast \to \mM^\circledast$ be $\mV,\mW$-enriched functors and $\eta: \id \to \G \circ \F$ a morphism in $\Enr\Fun_{\mV,\mW}(\mM,\mM)$ that exhibits $\F$ as a $\mV,\mW$-enriched left adjoint of $\G$. Then $$\Enr\Fun_{\mV,\mW}(\eta,\mO) : \Enr\Fun_{\mV,\mW}(\F,\mO) \circ \Enr\Fun_{\mV,\mW}(\G,\mO) \to \id$$exhibits $ \Enr\Fun_{\mV,\mW}(\G,\mO)$ as a left adjoint of $\Enr\Fun_{\mV,\mW}(\F,\mO)$ and $$\Enr\Fun_{\mV,\mW}(\mO,\eta) : \id \to \Enr\Fun_{\mV,\mW}(\mO,\G) \circ \Enr\Fun_{\mV,\mW}(\mO,\F) $$exhibits $ \Enr\Fun_{\mV,\mW}(\mO,\F)$ as a left adjoint of $\Enr\Fun_{\mV,\mW}(\mO,\G)$. \end{remark}

\subsection{Tensored envelopes}

In the following we introduce a tool that reduces questions about weakly enriched $\infty$-categories to tensored $\infty$-categories. The next propositions follow from \cite[Proposition 3.92, Proposition 3.101.]{HEINE2023108941}:

\begin{proposition}\label{envv} Let $\mV^\ot \to \Ass$ be an $\infty$-operad. There is a monoidal $\infty$-category $\Env(\mV)^\ot \to \Ass$ and an embedding of $\infty$-operads $ \mV^\ot \subset \Env(\mV)^\ot$ such that
\begin{enumerate}
\item Every object of $\Env(\mV)$ is equivalent to $\V_1 \ot ...\ot \V_\n$ for
$\n \geq 0$ and $\V_1,...,\V_\n \in \mV \subset \Env(\mV).$ 
		
\item For every monoidal $\infty$-category $\mW^\ot \to \Ass$ 
the induced functor $$\Alg_{\Env(\mV)}(\mW) \to \Alg_{ \mV}(\mW)$$
admits a fully faithful left adjoint that lands in the full subcategory 
of monoidal functors. In particular, the following induced functor is an equivalence:
$$\Fun^{\ot}(\Env(\mV), \mW) \to \Alg_{ \mV}(\mW).$$
\end{enumerate}
	
\end{proposition}

\begin{definition}Let $\mV^\ot \to \Ass$ be an $\infty$-operad.
We call $\Env(\mV)^\ot \to \Ass $ the monoidal envelope. % of $\mV^\ot \to \Ass$.
		
\end{definition}

\begin{proposition}\label{bitte}
Let $\mM^\circledast \to \mV^\ot \times \mW^\ot$ be a weakly bienriched $\infty$-category.
There is a bitensored $\infty$-category $\B\Env(\mM)^\circledast \to \Env(\mV)^\ot \times \Env(\mW)^\ot$ and an enriched embedding $\mM^\circledast \subset \B\Env(\mM)^\circledast$ lying over the embeddings $\mV^\ot \subset \Env(\mV)^\ot, \mW^\ot \subset \Env(\mW)^\ot$ such that
\begin{enumerate}
\item Every object of $\B\Env(\mM)$ is equivalent to $\V_1 \ot ...\ot \V_\n \ot \X \ot \W_1 \ot ...\ot\W_\m$ for
$\n,\m \geq 0$ and $\V_1,...,\V_\n \in \mV \subset \Env(\mV), \X \in \mM \subset \B\Env(\mM), \W_1,...,\W_\m \in \mW \subset \Env(\mW).$ 
		
\item For every bitensored $\infty$-category $\mN^\circledast \to \Env(\mV)^\ot \times \Env(\mW)^\ot$ the functor $$\Enr\Fun_{\Env(\mV),\Env(\mW)}(\B\Env(\mM), \mN) \to \Enr\Fun_{\mV,\mW}(\mM, \mN)$$
admits a fully faithful left adjoint that lands in the full subcategory 
of $\Env(\mV), \Env(\mW)$-linear functors. In particular, the following induced functor is an equivalence: $$\LinFun_{\Env(\mV),\Env(\mW)}(\B\Env(\mM), \mN) \to \Enr\Fun_{\mV,\mW}(\mM, \mN).$$
\end{enumerate}
	
\end{proposition}

\begin{definition}
Let $\mM^\circledast \to \mV^\ot \times \mW^\ot$ be a weakly bienriched $\infty$-category.
We call $\B\Env(\mM)^\circledast \to \Env(\mV)^\ot\times \Env(\mW)^\ot $ the bitensored envelope. %of $\mM^\circledast \to \mV^\ot \times \mW^\ot$.
	
\end{definition}

The next lemma is \cite[Lemma 3.93.]{HEINE2023108941}:

\begin{lemma}\label{looocx}\label{loccyy}
Let % $\mV^\ot \to \Ass$ be a monoidal $\infty$-category and 
$\mM^\circledast \to \mV^\ot \times \mW^\ot$ be a bitensored $\infty$-category.
The embedding $\mV^\ot \subset \Env(\mV)^\ot$ admits a left adjoint relative to $\Ass$.
%that sends $(\V_1,...,\V_\n)$ for $\n \geq 0$ and $\V_1,...,\V_\n \in \mV$ to $\V_1 \ot ...\ot \V_\n$.
The embedding $\mM^\circledast \subset \B\Env(\mM)^\circledast$ admits an enriched left adjoint covering the left adjoints of the embeddings $\mV^\ot \subset \Env(\mV)^\ot, \mW^\ot \subset \Env(\mW)^\ot.$
\end{lemma}

For the next definition we use Propositions \ref{Day} and \ref{cool}:

\begin{definition}\label{zgbbl}Let $\mV^\ot \to \Ass, \mW^\ot \to \Ass$ be $\infty$-operads and $\mM^\circledast \to \mV^\ot \times \mW^\ot$ a weakly bienriched $\infty$-category.

\begin{enumerate}
\item The closed monoidal envelope is $$\mP\Env(\mV)^\ot:=\mP(\Env(\mV))^\ot \to \Ass$$  %of $\mV^\ot \to \Ass$.

%\item We call $$\mP\L\Env(\mM)^\circledast := \mP(\L\Env(\mM))^\circledast \to \mP\Env(\mV)^\ot \times \mP(\mW)^\ot$$ the closed left tensored envelope.

%\item We call $$\mP\R\Env(\mM)^\circledast := \mP(\R\Env(\mM))^\circledast \to \mP(\mV)^\ot \times \mP\Env(\mW)^\ot$$ the closed right tensored envelope.

\item The closed bitensored envelope is $$\mP\B\Env(\mM)^\circledast := \mP(\B\Env(\mM))^\circledast \to \mP\Env(\mV)^\ot \times \mP\Env(\mW)^\ot.$$

\item The closed left tensored envelope of $\mM^\circledast \to \mV^\ot \times \mW^\ot$ is the full left tensored subcategory $$\mP\L\Env(\mM)^\circledast \subset \mP\B\Env(\mM)^\circledast \times_{\mP\Env(\mW)^\ot} \mW^\ot \to \mP\Env(\mV)^\ot \times \mW^\ot$$ generated under small colimits by the objects $ \V_1 \ot ... \ot \V_\n \ot \X$ for $\n \geq 0$ and $\V_1,...,\V_\n \in \mV$ and $\X \in \mM.$
		
\item The closed right tensored envelope of $\mM^\circledast \to \mV^\ot \times \mW^\ot$ is the full weakly right enriched subcategory $$ \mP\R\Env(\mM)^\circledast \subset \mV^\ot \times_{\mP\Env(\mV)^\ot}\mP\B\Env(\mM)^\circledast \to \mV^\ot \times \mP\Env(\mW)^\ot$$ generated under small colimits by the objects $ \X \ot \W_1 \ot ... \ot \W_\m$ for $\m \geq 0$ and $\W_1,...,\W_\m \in \mW, \X \in \mM.$
		
\end{enumerate}	

\end{definition}

The embedding $\mM^\circledast \to \mP\B\Env(\mM)^\circledast$ induces embeddings $\mM^\circledast \to \mP\L\Env(\mM)^\circledast, \mM^\circledast \to \mP\R\Env(\mM)^\circledast.$

\vspace{1mm}
%\begin{remark}The left action of $\mP\Env(\mV)$ on $\mP\B\Env(\mM)$ restricts to$\mP\L\Env(\mM)$ and therefore $\mP\L\Env(\mM)^\circledast \to \mP\Env(\mV)^\ot \times \mW^\ot$exhibits $\mP\L\Env(\mM)$ as left tensored over $\mP\Env(\mV)$.Similarly, $\mP\R\Env(\mM)^\circledast \to \mV^\ot \times \mP\Env(\mW)^\ot$exhibits $\mP\R\Env(\mM)$ as right tensored over $\mP\Env(\mW)$.\end{remark}

Propositions \ref{Day}, \ref{cool}, \ref{envv} and \ref{bitte} imply the following corollary:

\begin{corollary}\label{envvcor} 

\begin{enumerate}
\item Let $\mV^\ot \to \Ass$ be a small $\infty$-operad. For every monoidal $\infty$-category $\mW^\ot \to \Ass$ compatible with small colimits
the induced functor $$\Alg_{\mP\Env(\mV)}(\mW) \to \Alg_{\mV}(\mW)$$
admits a fully faithful left adjoint that lands in the full subcategory 
of monoidal functors that admit a right adjoint. In particular, the following induced functor is an equivalence:
$$\Fun^{\ot,\L}(\mP\Env(\mV), \mW) \to \Alg_{\mV}(\mW).$$

\item Let $\mM^\circledast \to \mV^\ot \times \mW^\ot$ be a weakly bienriched $\infty$-category. For every bitensored $\infty$-category $\mN^\circledast \to\mP\Env(\mV)^\ot \times \mP\Env(\mW)^\ot $ compatible with small colimits the induced functor $$\Enr\Fun_{\mP\Env(\mV),\mP\Env(\mW)}(\mP\B\Env(\mM), \mN) \to \Enr\Fun_{\mV,\mW}(\mM, \mN)$$
admits a fully faithful left adjoint that lands in the full subcategory 
of $\mP\Env(\mV), \mP\Env(\mW)$-linear functors that admit a right adjoint. In particular, the following functor is an equivalence: $$\LinFun^\L_{\mP\Env(\mV),\mP\Env(\mW)}(\mP\B\Env(\mM), \mN) \to \Enr\Fun_{\mV,\mW}(\mM, \mN).$$

\end{enumerate}	

\end{corollary}

\begin{remark}

By Corollary \ref{envvcor} (2) every enriched functor
$\F:\mM^\circledast \to \mN^\circledast$ starting at an absolute small weakly bienriched $\infty$-category uniquely extends to a linear functor
$\mP\B\Env(\mM)^\circledast \to \mN^\circledast$ that lies over left adjoint monoidal functors and admits an enriched right adjoint, which factors as $\mN^\circledast \subset \mP\B\Env(\mN)^\circledast \xrightarrow{\B\Env(\F)^*} \mP\B\Env(\mM)^\circledast$. 

\end{remark}

\begin{remark}\label{switch}
By the universal property of Corollary \ref{envvcor} (2) there is a canonical equivalence
$ \mP\B\Env(\mM^\rev)^\circledast \simeq (\mP\B\Env(\mM)^\rev)^\circledast$
that restricts to an equivalence
$ \mP\L\Env(\mM^\rev)^\circledast \simeq (\mP\R\Env(\mM)^\rev)^\circledast.$
	
\end{remark}

%Corollary \ref{envdecom} gives the following one:
%\begin{corollary}\label{envdecom2}
%Let $\mM^\circledast \to \mV^\ot \times \mW^\ot$ be a weakly bienriched $\infty$-category. There is an equivalence
%$$\mP\L\Env(\R\Env(\mM))^\circledast \simeq \mP\B\Env(\mM)^\circledast$$
%of $\infty$-categories presentably bitensored over $\mP\Env(\mV),\mP\Env(\mW).$	\end{corollary}

\begin{lemma}\label{alor}
%Let $\mN^\circledast \to \mV'^\ot \times \mW'^\ot $ be a bitensored $\infty$-category compatible with small colimits.
Let $\F: \mM^\circledast \to \mM'^\circledast$ be an enriched functor of absolute small weakly bienriched $\infty$-categories lying over maps of $\infty$-operads $\alpha:  \mV^\ot \to \mV'^\ot,\beta:\mW^\ot \to \mW'^\ot$.
For every $\V_1,...,\V_\n \in \mV, \W_1,...,\W_\m \in \mW$ for $\n,\m \geq 0$ and $\Z \in \mM' $ there is a canonical equivalence of presheaves on $\mP(\mM):$
$$ \kappa: \mP(\mM)(-, \Mul_{\mM'}(\alpha(\V_1),...,\alpha(\V_\n),\F(-), \beta(\W_1),...,\beta(\W_\m);\Z)) \simeq $$$$ \Mul_{\mP\B\Env(\mM)}(\V_1,...,\V_\n,\bj_!(-), \W_1,...,\W_\m; \F^*(\Z)).$$
	
\end{lemma}

\begin{proof}

%By Corollary \ref{envvcor} (2) every enriched functor$\mM^\circledast \to \mN^\circledast$ starting at an absolute small weakly bienriched $\infty$-category uniquely extends to a linear functor$\mP\B\Env(\mM)^\circledast \to \mN^\circledast$ that lies over left adjoint monoidal functors and admits a right adjoint $\beta$ relative to $\Ass \times \Ass.$In particular, every enriched functor$\F: \mM^\circledast \to \mM'^\circledast$ of absolute small weakly bienriched $\infty$-categories lying over maps of $\infty$-operads $\alpha:  \mV^\ot \to \mV'^\ot,\beta:\mW^\ot \to \mW'^\ot$ uniquely extends to a linear functor$\F_!: \mP\B\Env(\mM)^\circledast \to \mP\B\Env(\mN)^\circledast$ that lies over the left adjoint monoidal functors $\alpha_!, \beta_!$ and that admits a right adjoint $\F^*$ relative to $\Ass \times \Ass.$The right adjoint $\F^*$ is the functor $\B\Env(\F)^*$ and the right adjoint $\beta$ factors as $\mN^\circledast \subset \mP\B\Env(\mN)^\circledast \xrightarrow{\F^*} \mP\B\Env(\mM)^\circledast$ as a consequence of Remark \ref{ula}.For every $\V_1,...,\V_\n \in \mV, \W_1,...,\W_\m \in \mW$ for $\n,\m \geq 0$ and $\Z \in \mM' $ 
The equivalence $\kappa$ is uniquely determined by its restriction to $\mM$
since source and target of $\kappa$ preserve small limits.
There is a canonical equivalence of presheaves on $\mM$,
where $\bj: \mM \subset \B\Env(\mM)$ is the canonical embedding:
$$ \mP(\mM)(\y_\mM(-), \Mul_{\mM'}(\alpha(\V_1),...,\alpha(\V_\n),\F(-), \beta(\W_1),...,\beta(\W_\m);\Z)) \simeq$$$$ \Mul_{\mM'}(\alpha(\V_1),...,\alpha(\V_\n),\F(-), \beta(\W_1),...,\beta(\W_\m);\Z) \simeq$$$$ \B\Env(\mM')(\alpha(\V_1) \ot ... \ot \alpha(\V_\n) \ot \bj \circ \F(-) \ot \beta(\W_1) \ot... \ot \beta(\W_\m),\Z) \simeq$$$$
\F^*(\Z) \circ (\V_1 \ot ... \ot \V_\n \ot \bj(-) \ot \W_1 \ot... \ot \W_\m) \simeq $$
$$\mP\B\Env(\mM)(\V_1 \ot ... \ot \V_\n \ot (\y_{\B\Env(\mM)} \circ \bj(-)) \ot \W_1 \ot... \ot \W_\m, \F^*(\Z)) \simeq $$$$\Mul_{\mP\B\Env(\mM)}(\V_1,...,\V_\n,(\bj_! \circ \y_{\mM}(-)), \W_1,...,\W_\m; \F^*(\Z)).$$
%The latter equivalence uniquely extends to an equivalenceof presheaves on $\mP(\mM)$ preserving small limits:$$ \mP(\mM)(-, \Mul_{\mM'}(\alpha(\V_1),...,\alpha(\V_\n),\F(-), \beta(\W_1),...,\beta(\W_\m);\Z)) \simeq $$$$ \Mul_{\mP\B\Env(\mM)}(\V_1,...,\V_\n,\bj_!(-), \W_1,...,\W_\m; \F^*(\Z)).$$	
	
\end{proof}

Proposition \ref{bitte}, \ref{ljnbfg} and \ref{cool} imply the following corollary:

\begin{corollary}\label{leiik}
Let $\mV^\ot \to \Ass, \mW^\ot \to \Ass$ be small $\infty$-operads and $\K$ an $\infty$-category.
\begin{enumerate} 
\item 
The unique $\Env(\mV), \Env(\mW)$-linear functor
$$\B\Env(\K_{\mV,\mW})^\circledast \to \Env(\mV)^\circledast \times \K\times \Env(\mW)^\circledast $$
extending the functor $\K \to \Env(\mV) \times\K\times\Env(\mW)$
%sending $\K_{\mV,\mW}$ to the tensor unit 
is an equivalence, which is inverse to the unique $\Env(\mV), \Env(\mW)$-linear functor
$\Env(\mV)^\circledast \times \K\times \Env(\mW)^\circledast \to \B\Env(\K_{\mV,\mW})^\circledast $
extending the embedding $\K \subset \B\Env(\K_{\mV,\mW}).$

%The unique left adjoint $\mP\Env(\mV)$-linear functor
%$$\kappa: \mP\L\Env(\K_\mV)^\circledast \to \mP\Env(\mV)^\circledast$$sending $\triv_\mV$ to the tensor unit %extending the embedding $ \triv_\mV^\circledast \subset \mP\Env(\mV)^\circledast$is an equivalence, which is inverse to the unique left adjoint $\mP\Env(\mV)$-linear functor$\mP\Env(\mV)^\circledast \to \mP\L\Env(\triv_\mV)^\circledast $sending the tensor unit to $\triv_\mV.$
\item 
The unique left adjoint $\mP\Env(\mV), \mP\Env(\mW)$-linear functor
$$\mP\B\Env(\K_{\mV,\mW})^\circledast \to (\mP\Env(\mV) \otimes \mP(\K)\ot \mP\Env(\mW))^\circledast $$
extending the functor $\K \to \mP(\K)\to \mP\Env(\mV) \otimes \mP(\K)\ot \mP\Env(\mW)$
%sending $\K_{\mV,\mW}$ to the tensor unit 
is an equivalence, which is inverse to the unique left adjoint $\mP\Env(\mV), \mP\Env(\mW)$-linear functor
$$(\mP\Env(\mV)\otimes \mP(\K) \ot \mP\Env(\mW))^\circledast \to \mP\B\Env(\K_{\mV,\mW})^\circledast $$
extending the embedding $\K \subset \mP\B\Env(\K_{\mV,\mW}).$
%sending the tensor unit to $\triv_{\mV,\mW}.$	
	
\end{enumerate}
	
\end{corollary}
\subsection{Enrichment}

Before defining enriched $\infty$-categories we define the more general class of pseudo-enriched $\infty$-categories that contains all tensored $\infty$-categories.
We follow \cite[§ 3]{heine2024bienriched}.
The next definition is \cite[Definition 3.1.]{heine2024bienriched}:

\begin{definition}\label{Lu}Let $\phi: \mM^\circledast\to \mV^\ot \times \mW^\ot$ be a weakly bienriched $\infty$-category.
	
\begin{enumerate}
\item We say that $\phi$ exhibits $\mM$ as left pseudo-enriched in $\mV$ if $\mV^\ot \to\Ass$ is a monoidal $\infty$-category and for every $\X,\Y \in \mM$ and $\V_1, ..., \V_\n \in \mV, \W_1, ..., \W_\m \in \mW$ for $\n,\m \geq 0$ the map
$$ \Mul_{\mM}(\V_1 \ot ... \ot \V_\n, \X, \W_1, ..., \W_\m; \Y) \to \Mul_{\mM}(\V_1, ..., \V_\n, \X, \W_1, ..., \W_\m ; \Y)$$ 
induced by the active morphism $\V_1, ..., \V_\n \to \V_1 \ot  ... \ot \V_\n $
in $\mV^\ot$ is an equivalence.
\vspace{1mm}

\item We say that $\phi$ exhibits $\mM$ as right pseudo-enriched in $\mW$ if $\mW^\ot \to\Ass$ is a monoidal $\infty$-category and for every $\X,\Y \in \mM$ and $\V_1, ..., \V_\n \in \mV, \W_1, ..., \W_\m \in \mW$ for $\n,\m \geq 0$ the map
$$ \Mul_{\mM}(\V_1, ..., \V_\n, \X, \W_1 \ot ... \ot \W_\m; \Y) \to \Mul_{\mM}(\V_1, ...,\V_\n, \X, \W_1, ..., \W_\m ; \Y)$$ 
induced by the active morphism $\W_1, ..., \W_\m \to \W_1 \ot  ... \ot \W_\m $
in $\mW^\ot$ is an equivalence.

\item We say that $\phi$ exhibits $\mM$ as bipseudo-enriched in $\mV, \mW$ if $\mV^\ot \to\Ass, \mW^\ot \to\Ass $ are monoidal $\infty$-categories and for any $\X,\Y \in \mM$ and $\V_1, ..., \V_\n \in \mV, \W_1, ..., \W_\m \in \mW$ for $\n,\m \geq 0$ the following map is an equivalence:
$$ \Mul_{\mM}(\V_1 \ot ... \ot \V_\n, \X, \W_1 \ot ...\ot \W_\m; \Y) \to \Mul_{\mM}(\V_1, ..., \V_\n, \X, \W_1, ..., \W_\m ; \Y).$$

\end{enumerate}

\end{definition}

\begin{remark}\label{both}
A weakly bienriched $\infty$-category $\mM^\circledast\to \mV^\ot \times \mW^\ot$ exhibits $\mM$ as bipseudo-enriched in $\mV, \mW$ if and only if it exhibits $\mM$ as left pseudo-enriched in $\mV$ and right pseudo-enriched in $\mW$.
	
\end{remark}

%\begin{example}Let $\phi: \mM^\circledast\to \mV^\ot \times \mW^\ot$ be a weakly bienriched $\infty$-category.If $\phi$ exhibits $\mM$ as left tensored in $\mV$, right tensored in $\mW$, bitensored in $\mV, \mW$, respectively, then $\phi$ exhibits $\mM$ as left pseudo-enriched in $\mV$, right pseudo-enriched in $\mW$, bipseudo-enriched in $\mV, \mW.$
	
%Any bitensored $\infty$-category $\mM^\circledast\to \mV^\ot \times \mW^\ot$ exhibits $\mM$ as pseudo-enriched in $\mV, \mW$since for any $\X,\Y \in \mM$ and $\V_1, ..., \V_\n \in \mV, \W_1, ..., \W_\m \in \mW$ for $\n,\m \geq 0$ the map$$ \Mul_{\mM}(\V_1 \ot ... \ot \V_\n, \X, \W_1 \ot ...\ot \W_\m; \Y) \to \Mul_{\mM}(\V_1, ..., \V_\n, \X, \W_1, ..., \W_\m ; \Y)$$ identifies with the equivalence$$ \Mul_{\mM}(\V_1 \ot ... \ot \V_\n, \X, \W_1 \ot ...\ot \W_\m; \Y) \simeq \mM(\V_1 \ot ... \ot \V_\n \ot \X \ot \W_1 \ot ...\ot \W_\m, \Y) $$$$\simeq \Mul_{\mM}(\V_1, ..., \V_\n, \X, \W_1, ..., \W_\m ; \Y). $$ Similarly any weakly bienriched $\infty$-category that exhibits $\mM$ as\end{example}

The next lemma is \cite[Lemma 3.3.]{heine2024bienriched}:

\begin{lemma}\label{zzzz} 
	
\begin{enumerate}
\item A weakly bienriched $\infty$-category $\mM^\circledast \to \mV^\ot \times \mW^\ot$ is a left tensored $\infty$-category if and only if it is a left pseudo-enriched $\infty$-category and admits left tensors.	

\item A weakly bienriched $\infty$-category $\mM^\circledast \to \mV^\ot \times \mW^\ot$ is a right tensored $\infty$-category if and only if it is a right pseudo-enriched $\infty$-category and admits right tensors.	
	
\item A weakly bienriched $\infty$-category $\mM^\circledast \to \mV^\ot \times \mW^\ot$ is a bitensored $\infty$-category if and only if it is a bipseudo-enriched $\infty$-category and admits bitensors.
\end{enumerate}	
\end{lemma}

\begin{notation}
Let $$ \L\P\Enr, \R\P\Enr, \B\P\Enr \subset \omega\B\Enr $$ be the full subcategories of left pseudo-enriched, right pseudo-enriched, bipseudo-enriched $\infty$-categories, respectively.
	
\end{notation}

%The next example is:

%\begin{example}
%Let $\mM^\circledast \to \mV^\ot$ be a weakly left enriched $\infty$-category and $\mN^\circledast \to \mW^\ot$ a weakly right enriched $\infty$-category.
%\begin{enumerate}
%\item If $\mM^\circledast \to \mV^\ot$ exhibits $\mM$ as left pseudo-enriched in $\mV$, then
%$\mM^\circledast \times \mN^\circledast \to \mV^\ot \times \mW^\ot$ exhibits
%$\mM \times \mN$ as left pseudo-enriched in $\mV$.
%\item If $\mN^\circledast \to \mW^\ot$ exhibits $\mN$ as right pseudo-enriched in $\mW$, then
%$\mM^\circledast \times \mN^\circledast \to \mV^\ot \times \mW^\ot$ exhibits$\mM \times \mN$ as right pseudo-enriched in $\mW$.
%\item If $\mM^\circledast \to \mV^\ot$ exhibits $\mM$ as left pseudo-enriched in $\mV$ and $\mN^\circledast \to \mW^\ot$ exhibits $\mN$ as right pseudo-enriched in $\mW$, then$\mM^\circledast \times \mN^\circledast \to \mV^\ot \times \mW^\ot$ exhibits$\mM \times \mN$ as bipseudo-enriched. % in $\mV, \mW$.\end{enumerate}\end{example}

\begin{notation}\label{Enros}
Let $\mM^\circledast \to \mV^\ot \times \mW^\ot$ be a weakly bienriched $\infty$-category.
%Since $\mP\B\Env(\mM)^\circledast \to \mP\Env(\mV)^\ot \times \mP\Env(\mW)^\ot$is an $\infty$-category with closed biaction, by Proposition \ref{Line} evaluation at the tensor units gives an equivalence$$\Enr\Fun^\L_{\mP\Env(\mV),\mP\Env(\mW)}(\mP\Env(\mV)\ot \mP\B\Env(\mW),\mP\B\Env(\mM)) \simeq \mP\B\Env(\mM).$$By Proposition \ref{adjeq} there is a canonical equivalence$$\Enr\Fun^\L_{\mP\Env(\mV),\mP\Env(\mW)}(\mP\Env(\mV)\ot \mP\B\Env(\mW),\mP\B\Env(\mM))^\op$$$$\simeq \Enr\Fun^\R_{\mP\Env(\mV),\mP\Env(\mW)}(\mP\B\Env(\mM), \mP\Env(\mV)\ot \mP\B\Env(\mW)).$$The graph of $\mM$ is the $\mV,\mW$-enriched functor $\Gamma_\mM: \mM^\circledast \times \mM^\op \to \mP\Env(\mV)\ot \mP\B\Env(\mW)$ corresponding to the composition $$\mM^\op \subset \mP\B\Env(\mM)^\op \simeq \Enr\Fun^\L_{\mP\Env(\mV),\mP\Env(\mW)}(\mP\Env(\mV)\ot \mP\B\Env(\mW),\mP\B\Env(\mM))^\op$$$$\simeq \Enr\Fun^\R_{\mP\Env(\mV),\mP\Env(\mW)}(\mP\B\Env(\mM), \mP\Env(\mV)\ot \mP\B\Env(\mW))$$$$\to \Enr\Fun_{\mV,\mW}(\mM, \mP\Env(\mV)\ot \mP\B\Env(\mW)),$$where the last functor is restriction.
For every $\X\in \mM$ let $$\Gamma_\mM(\X,-): \mM^\circledast \to (\mP\Env(\mV)\ot \mP\B\Env(\mW))^\circledast \simeq \mP(\Env(\mV) \times \Env(\mW))^\circledast$$ be the restriction of the enriched right adjoint of the left adjoint $\mP\Env(\mV),\mP\Env(\mW)$-linear functor $$(\mP\Env(\mV)\ot \mP\B\Env(\mW))^\circledast \to \mP\B\Env(\mM)^\circledast, (\V,\W) \mapsto \V\ot\X\ot\W.$$ 
We get a $\mV,\mW$-enriched functor $\Gamma_\mM: \mM^\op \times \mM^\circledast \to \mP(\Env(\mV) \times \Env(\mW))^\circledast$ that we call graph of $\mM.$

\end{notation}
\begin{remark}
	
For every $\X, \Y \in \mM$ and $\V_1 ,...,\V_\n \in \mV, \W_1,...,\W_\m \in \mW$ for $\n, \m \geq 0$ there is a canonical equivalence
$$ \Gamma_\mM(\X,\Y)(\V_1 \ot ... \ot \V_\n, \W_1 \ot ...\ot \W_\m) \simeq \mP\B\Env(\mM)(\V_1 \ot ... \ot \V_\n \ot \X \ot \W_1,...,\W_\m, \Y) $$$$ \simeq \Mul_\mM(\V_1, ..., \V_\n, \X, \W_1,...,\W_\m; \Y).$$

\end{remark}

The next remark is \cite[Remark 3.8.]{heine2024bienriched}:

\begin{remark}\label{rhhhj} An absolute small weakly bienriched $\infty$-category $\mM^\circledast \to \mV^\ot \times \mW^\ot$ exhibits $\mM$ as 
\begin{enumerate}
		
\item left pseudo-enriched in $\mV$ if and only if $\mV^\ot \to \Ass$ is a monoidal $\infty$-category and for every $\X,\Y \in \mM$ the object $$\Gamma_\mM(\X,\Y)\in \mP(\Env(\mV) \times \Env(\mW)) \simeq \Fun(\Env(\mW)^\op, \mP\Env(\mV))$$ lies in $\Fun(\Env(\mW)^\op, \mP(\mV)).$
		
\item right pseudo-enriched in $\mW$ if and only if $\mW^\ot \to \Ass$ is a monoidal $\infty$-category and for every $\X,\Y \in \mM$ the object $$\Gamma_\mM(\X,\Y)\in \mP(\Env(\mV) \times \Env(\mW)) \simeq \Fun(\Env(\mV)^\op, \mP\Env(\mW))$$ 
lies in $\Fun(\Env(\mV)^\op, \mP(\mW)).$
			
\item bipseudo-enriched in $\mV, \mW$ if and only if $\mV^\ot \to \Ass, \mW^\ot \to \Ass$ are monoidal $\infty$-categories and for every $\X,\Y \in \mM$ the object $$\Gamma_\mM(\X,\Y)\in \mP(\Env(\mV) \times \Env(\mW))$$ lies in $\mP(\mV \times \mW).$
\end{enumerate}
	
\end{remark}

Now we are ready to define enriched $\infty$-categories. The next definition is \cite[Definition 3.9.]{heine2024bienriched}:

\begin{definition}
Let $\mM^\circledast\to \mV^\ot \times \mW^\ot $ be a weakly bienriched $\infty$-category.

\begin{enumerate}
\item A left multi-morphism object of $\W_1,..., \W_\m, \X, \Y \in \mM$ for $\m \geq 0$ is an object $$\L\Mul\Mor_{\mM}(\X, \W_1,..., \W_\n; \Y) \in \mV $$ such that there is a multi-morphism $\beta \in \Mul_{\mM}(\L\Mul\Mor_{\mM}(\X, \W_1,..., \W_\n; \Y), \X, \W_1,..., \W_\m; \Y) $ that induces for every objects $\V_1, ..., \V_\n \in \mV$ for $\n \geq 0$ an equivalence
$$\hspace{3mm} \Mul_{\mV}(\V_1, ..., \V_\n;  \L\Mul\Mor_{\mM}(\X, \W_1,..., \W_\m; \Y)) \simeq \Mul_{\mM}(\V_1, ..., \V_\n, \X, \W_1, ..., \W_\m; \Y).$$	
	
\item A right multi-morphism object of $\V_1,..., \V_\n, \X, \Y \in \mM$ for $\n \geq 0$ is an object $$\R\Mul\Mor_{\mM}(\V_1,..., \V_\n, \X; \Y) \in \mW $$ such that there is a multi-morphism $\alpha \in \Mul_{\mM}(\V_1,..., \V_\n, \X, \R\Mul\Mor_{\mM}(\V_1,..., \V_\n, \X; \Y); \Y) $ that induces for every objects $\W_1, ..., \W_\m \in \mW$ for $\m \geq 0$ an equivalence
$$\hspace{3mm}\Mul_{\mW}(\W_1, ..., \W_\m;  \R\Mul\Mor_{\mM}(\V_1,..., \V_\n, \X; \Y)) \simeq \Mul_{\mM}(\V_1, ..., \V_\n, \X, \W_1, ..., \W_\m; \Y).$$

%\item A bi-multi-morphism object of $\X, \Y \in \mM$ for $\m \geq 0$ is an object $\Mor_{\mM}(\X, \Y) \in \mV $ together with a multi-morphism $\beta \in \Mul_{\mM}(\Mul\Mor_{\mM}(\W_1,..., \W_\n, \X, \Y), \X, \W_1,..., \W_\m; \Y) $ that induces for every objects $\V_1, ..., \V_\n \in \mV$ an equivalence$$\Mul_{\mV}(\V_1, ..., \V_\n;  \Mul\Mor_{\mM}(\W_1,..., \W_\m, \X, \Y)) \simeq \Mul_{\mM}(\V_1, ..., \V_\n, \X, \W_1, ..., \W_\m; \Y).$$

\end{enumerate}

\end{definition}

The next definition is \cite[Definition 3.10.]{heine2024bienriched}:

\begin{definition}
Let $\mM^\circledast\to \mV^\ot \times \mW^\ot $ be a weakly bienriched $\infty$-category and $\X, \Y \in \mM$.	

\begin{enumerate}
\item A left morphism object of $\X,\Y $ in $\mM$ is a factorization
$\L\Mor_\mM(\X,\Y): \Env(\mW)^\op \to \mV$ of the functor 
$\Env(\mW)^\op \to \mP\Env(\mV)$ corresponding to $\Gamma_\mM(\X,\Y) \in \mP(\Env(\mV) \times \Env(\mW))$.

\item A right morphism object of $\X,\Y $ in $\mM$ is a factorization
$\R\Mor_\mM(\X,\Y): \Env(\mV)^\op \to \mW$ of the functor 
$\Env(\mV)^\op \to \mP\Env(\mW)$ corresponding to $\Gamma_\mM(\X,\Y) \in \mP(\Env(\mV) \times \Env(\mW))$.

\end{enumerate}

\end{definition}

\begin{remark}Let $\mM^\circledast\to \mV^\ot \times \mW^\ot $ be a weakly bienriched $\infty$-category and $\X, \Y \in \mM, \W_1, ..., \W_\m \in \mW$ for $\m \geq 0$.
A left multi-morphism object $\L\Mul\Mor_{\mM}(\W_1,..., \W_\m, \X; \Y) \in \mV$
represents the presheaf $\Gamma_\mM(\X,\Y)(\W_1 \ot ... \ot \W_\m,-) \in \mP\Env(\mV).$ 
Consequently, there is a left morphism object $\L\Mor_\mM(\X,\Y): \Env(\mW)^\op \to \mV$
if and only if for every $ \W_1, ..., \W_\m \in \mW$ for $\m \geq 0$ there is a left multi-morphism object $\L\Mul\Mor_{\mM}(\W_1,..., \W_\m, \X; \Y) \in \mV$.
In this case there is a canonical equivalence
$$ \L\Mor_\mM(\X,\Y)(\W_1 \ot ... \ot \W_\m) \simeq \L\Mul\Mor_{\mM}(\X, \W_1,..., \W_\m; \Y).$$
The similar holds for right (multi-) morphism objects. 
\end{remark}

\begin{remark}Let $\mM^\circledast\to \mV^\ot \times \mW^\ot $ be a weakly bienriched $\infty$-category and $\X, \Y \in \mM.$ If the left and right morphism object of $\X,\Y \in \mM$ exist, for every $\V_1, ..., \V_\n \in \mV, \W_1,..., \W_\m \in \mW$ for $\n,\m \geq 0$ there is a canonical equivalence $$ \Mul_{\mW}( \W_1,..., \W_\m; \R\Mor_\mM(\X,\Y)(\V_1 \ot ... \ot \V_\n)) \simeq$$$$ \Mul_{\mW}( \W_1,..., \W_\m;  \R\Mul\Mor_{\mM}(\V_1, ..., \V_\n, \X; \Y)) \simeq\Mul_{\mM}(\V_1, ..., \V_\n, \X, \W_1, ..., \W_\m; \Y) \simeq$$$$ \Mul_{\mV}(\V_1, ..., \V_\n;  \L\Mul\Mor_{\mM}(\X, \W_1,..., \W_\m; \Y)) \simeq \Mul_{\mV}(\V_1, ..., \V_\n; \L\Mor_\mM(\X,\Y)(\W_1 \ot ... \ot \W_\m)).$$
This gives an adjunction 
$$\L\Mor_{\mM}(\X,\Y)^\op: \Env(\mW) \rightleftarrows \Env(\mV)^\op: \R\Mor_{\mM}(\X,\Y),$$
where the left adjoint lands in $\mV^\op$ and the right adjoint lands in $\mW$
\cite[Remark 3.20.]{heine2024bienriched}.

\end{remark}

%\begin{lemma}\label{ato} Let $\mM^\circledast \to \mV^\ot \times \mW^\ot$ be a small bi-tensored $\infty$-category.The linear embedding $\mM^\circledast \subset \mP(\mM)^\circledast$ preserves left and right multi-morphism objects.\end{lemma}

The next proposition is \cite[Proposition 3.13.]{heine2024bienriched}:

\begin{proposition}\label{morpre}
Let $\mM^\circledast \to \mV^\ot \times \mW^\ot$ be a weakly bienriched $\infty$-category.
The embeddings $$\mM^\circledast \subset \B\Env(\mM)^\circledast, \B\Env(\mM)^\circledast \subset \mP\B\Env(\mM)^\circledast$$ of weakly bienriched $\infty$-categories preserve left and right multi-morphism objects.
%and so preserves left and right morphism objects.

\end{proposition}

The next definition is \cite[Definition 3.14.]{heine2024bienriched}:

\begin{definition}
	
\begin{enumerate}
\item A weakly bienriched $\infty$-category $\mM^\circledast \to \mV^\ot \times \mW^\ot$ exhibits $\mM$ as left enriched in $\mV$ if for every $\X, \Y \in \mM$ there is a left morphism object $\L\Mor_{\mM}( \X, \Y) \in \mV$.

\item A weakly bienriched $\infty$-category $\mM^\circledast \to \mV^\ot \times \mW^\ot$ exhibits $\mM$ as right enriched in $\mW$ if for every $\X, \Y \in \mM$ there is a right morphism object $\R\Mor_{\mM}(\X, \Y) \in \mW.$

\item A weakly bienriched $\infty$-category $\mM^\circledast \to \mV^\ot \times \mW^\ot$ exhibits $\mM$ as bienriched in $\mV, \mW$ if it exhibits $\mM$ as left enriched in $\mV$ and right enriched in $\mW$.

\end{enumerate}

\end{definition}

\begin{notation}
Let $$\L\Enr, \R\Enr, \B\Enr \subset \omega\B\Enr$$ be the full subcategories of left enriched, right enriched, bienriched $\infty$-categories.

\end{notation}

The next lemma is \cite[Lemma 3.18.]{heine2024bienriched}:

\begin{lemma}
A weakly bienriched $\infty$-category $\mM^\circledast \to \mV^\ot \times \mW^\ot$ that exhibits $\mM$ as left enriched in a monoidal $\infty$-category, right enriched in a monoidal $\infty$-category, bienriched in monoidal $\infty$-categories, respectively, exhibits $\mM$ as left pseudo-enriched, right pseudo-enriched, bipseudo-enriched, respectively.
	
\end{lemma}

%\begin{remark}Let $\mC, \mD$ be small $\infty$-categories.There is a canonical equivalence $$\Fun^\R(\mC^\op,\mD)\simeq \Fun^\R(\mD^\op,\mC).$$

%Let $\mC, \mD$ be presentable $\infty$-categories. There is a canonical equivalence$$\mC \ot \mD \simeq \Fun^\R((\mC \otimes \mD)^\op, \mS) \simeq \Fun^\R(\mC^\op,\mD) \simeq \mD \ot \mC. $$ The universal functor $\mC \times \mD \to \mC \ot \mD \simeq \Fun^\R(\mC^\op,\mD)$preserving small colimits component-wise sends $\C,\D$ to the image of $\D$ under the left adjoint $\lan_\C$ of the functor $\Fun^\R(\mC^\op,\mD) \to \mD$ evaluating at $\C.$\end{remark}

%\begin{remark}\label{EEnr}Let $\mM^\circledast \to \mV^\ot \times \mW^\ot$ be a weakly bienriched $\infty$-category and $\X,\Y \in \mM$ such that there is a left and a right morphism object $\L\Mor_\mM(\X,\Y), \R\Mor_\mM(\X,\Y).$ Then $$ \Gamma_\mM(\X,\Y)\in \mP(\Env(\mV) \times \Env(\mW)) \simeq \Fun(\Env(\mV)^\op, \mP\Env(\mW))$$is the image of $\R\Mor_\mM(\X,\Y) \in \Fun(\Env(\mV)^\op, \mW)$ and $$\Gamma_\mM(\X,\Y)\in \mP(\Env(\mV) \times \Env(\mW)) \simeq \Fun(\Env(\mW)^\op, \mP\Env(\mV)) $$is the image of $ \L\Mor_\mM(\X,\Y) \in \Fun(\Env(\mW)^\op, \mV).$ Thus $ \R\Mor_\mM(\X,\Y), \L\Mor_\mM(\X,\Y)$ lie in $$\Fun^\R(\Env(\mV)^\op, \Env(\mW)) \simeq \Fun^\R(\Env(\mW)^\op, \Env(\mV))$$so that there is an adjunction $$\L\Mor_{\mM}(\X,\Y)^\op: \Env(\mW) \rightleftarrows \Env(\mV)^\op: \R\Mor_{\mM}(\X,\Y). $$\end{remark}

\begin{remark}Let $\mV^\ot \to \Ass, \mW^\ot \to \Ass$ be presentably monoidal $\infty$-categories, $\mM^\circledast \to \mV^\ot \times \mW^\ot$ a weakly bienriched $\infty$-category and $\alpha: \mV \times \mW \to \mV \ot \mW$ the universal functor preserving small colimits component-wise.
By the adjoint functor theorem the canonical embedding $$\alpha^*: \mV \ot \mW \subset \widehat{\mP}(\mV \ot \mW) \to \widehat{\mP}(\mV \times \mW) \simeq \Fun(\mV^\op, \widehat{\mP}(\mW))$$ 
induces an equivalence $\mV \ot \mW \to \Fun^\R(\mV^\op,\mW).$
%to the full subcategory of  
Under this equivalence
%$\mV \ot \mW \simeq \Fun^\R(\mV^\op,\mW)$ 
the left morphism object $\L\Mor_\mM(\X,\Y)$ of $\X,\Y \in \mM$ is an object of $\mV \ot \mW$, whose image under the embedding $$\alpha^*: \mV \ot \mW \subset \widehat{\mP}(\mV \ot \mW) \to \widehat{\mP}(\mV \times \mW)$$ is $\Mul_{\mM}(-, \X,-; \Y).$	
In particular, for every $\V \in \mV, \W \in \mW$ there is a canonical equivalence
\begin{equation}\label{uop}
\mV \ot \mW(\alpha(\V,\W), \L\Mor_{\mM}(\X, \Y)) \simeq \mM(\V \ot \X \ot \W,\Y) \simeq \Mul_{\mM}(\V,\X,\W;\Y).\end{equation}

\end{remark}

\begin{example}\label{Biolk}

Every presentably bitensored $\infty$-category $\mM^\circledast \to \mV^\ot \times \mW^\ot$ exhibits $\mM$ as bienriched in $\mV,\mW$.
By Example \ref{Line} for any $\X \in \mM$ there is a unique left adjoint $\mV, \mW$-linear functor $\mV\ot \mW \to \mM$ sending the tensor unit to $\X$.
The right adjoint sends $\Y \in \mM$ to a left morphism object $\L\Mor_{\mM}(\X, \Y) \in \mV \ot \mW $ by equivalence (\ref{uop}).
%since for every $\V \in \mV, \W \in \mW$ there is a canonical equivalence$$\mV \ot \mW(\alpha(\V,\W), \L\Mor_{\mM}(\X, \Y)) \simeq \mM(\V \ot \X \ot \W,\Y) \simeq \Mul_{\mM}(\V,\X,\W;\Y).$$
In particular, for every weakly bienriched $\infty$-category $\mM^\circledast \to \mV^\ot \times \mW^\ot$ and every $\X, \Y \in \mM$ the object
$\Gamma_\mM(\X, \Y) \in \mP(\Env(\mV)\times \Env(\mW)) \simeq \mP\Env(\mV) \otimes \mP\Env(\mW) $ is the left morphism object $\L\Mor_{\bar{\mM}}(\X,\Y)$ of $\X$ and $\Y$,
where $\bar{\mM}^\circledast \subset \mP\B\Env(\mM)^\circledast \to \mP\Env(\mV)^\ot \times \mP\Env(\mW)^\ot$ is the full weakly bienriched subcategory spanned by $\mM$.

\end{example}

Moreover we use the following terminology:

\vspace{1mm}
\begin{definition}\label{Luel}Let $\kappa, \tau$ be regular cardinals
and $\phi: \mM^\circledast\to \mV^\ot \times \mW^\ot$ a weakly bienriched $\infty$-category.
	
\begin{enumerate}
\item We say that $\phi$ exhibits $\mM$ as left $\kappa$-enriched in $\mV$ if $\mV^\ot \to\Ass$ is a monoidal $\infty$-category compatible with $\kappa$-small colimits, $\phi$ is a left pseudo-enriched $\infty$-category and for every $\X,\Y \in \mM$ and $\W_1,...,\W_\m \in \mW$ for $\m \geq0$ the presheaf
$\Mul_\mM(-,\X, \W_1,...,\W_\m;\Y)$ on $\mM$ preserves $\kappa$-small limits.
	
\vspace{1mm}
		
\item We say that $\phi$ exhibits $\mM$ as right $\tau$-enriched in $\mW$ if $\mW^\ot \to\Ass$ is a monoidal $\infty$-category compatible with $\tau$-small colimits, $\phi$ is a right pseudo-enriched $\infty$-category and for every $\X,\Y \in \mM$ and $\V_1,...,\V_\n \in \mV$ for $\n \geq0$ the presheaf
$\Mul_\mM(\V_1,...,\V_\n,\X,-;\Y)$ on $\mM$ preserves $\tau$-small limits.
		
\vspace{1mm}
		
\item We say that $\phi$ exhibits $\mM$ as $\kappa, \tau$-bienriched in $\mV, \mW$ if it exhibits $\mM$ as left $\kappa$-enriched in $\mV$ and right $\tau$-enriched in $\mW.$
		
\end{enumerate}
	
\end{definition}

\begin{notation}Let $\kappa, \tau$ be small regular cardinals.
Let $$ ^\kappa\L\Enr, \R\Enr^\tau, {^\kappa\B\Enr}^{\tau} \subset \omega\B\Enr $$ be the full subcategories of left $\kappa$-enriched, right $\tau$-enriched, $\kappa, \tau$-bienriched $\infty$-categories, respectively.
	
\end{notation}

\begin{example}
	
A weakly bienriched $\infty$-category $\mM^\circledast\to \mV^\ot \times \mW^\ot$ exhibits $\mM$ as left (right) $\emptyset$-enriched in $\mV$ if and only if it exhibits $\mM$ as left (right) pseudo-enriched.
	
\end{example}

%\begin{example}\label{laulo2}
%Let $\kappa$ be a small regular cardinal.  weakly bienriched $\infty$-category $\mM^\circledast \to \mV^\ot \times \mW^\ot$ is a left (right) tensored $\infty$-category compatible with $\kappa$-small colimits if and only if $\mM^\circledast \to \mV^\ot \times \mW^\ot$ is a left (right) $\kappa$-enriched $\infty$-category that admits left (right) tensors and $\kappa$-small conical colimits.\end{example}

\begin{definition}Let $\sigma$ be the strongly inaccessible cardinal corresponding to the small universe and $\phi: \mM^\circledast\to \mV^\ot \times \mW^\ot$ a weakly bienriched $\infty$-category.
	
\begin{enumerate}
\item We say that $\phi$ exhibits $\mM$ as left quasi-enriched in $\mV$
if $\phi$ exhibits $\mM$ as left $\sigma$--enriched. % in $\mV$.
		
\item We say that $\phi$ exhibits $\mM$ as right quasi-enriched in $\mW$
if $\phi$ exhibits $\mM$ as right $\sigma$-enriched. %in $\mW$.
		
\item We say that $\phi$ exhibits $\mM$ as bi-quasi-enriched in $\mV, \mW$
if $\phi$ exhibits $\mM$ as $\sigma, \sigma$-bienriched. % in $\mV, \mW$.
\end{enumerate}
	
\end{definition}

\begin{notation}
Let $$ \L\Q\Enr, \R\Q\Enr, \B\Q\Enr \subset \omega\widehat{\B\Enr} $$ be the full subcategories of left quasi-enriched, right quasi-enriched, bi-quasi-enriched $\infty$-categories, respectively, whose underlying $\infty$-category is small.
	
\end{notation}

The next theorem is \cite[Theorem 5.51.]{heine2024bienriched}:

\begin{theorem}\label{psinho}Let $\kappa$ be a small regular cardinal,
\begin{enumerate}
\item Let $\mV^\otimes \to \Ass$ be a small $\infty$-operad, $ \mW^\ot \to \Ass$ a small monoidal $\infty$-category compatible with $\kappa$-small colimits, $\mM^\circledast \to \mV^\ot$ a small weakly left enriched $\infty$-category and $\mN^\circledast \to \mV^\ot \times \mW^\ot$ a small 
right $\kappa$-enriched $\infty$-category. Then $\Enr\Fun_{\mV, \emptyset}(\mM, {\mN})^\circledast \to \mW^\ot $ is right $\kappa$-enriched. % $\infty$-category.
	
\vspace{1mm}
\item Let $\mV^\ot \to \Ass$ be an $\infty$-operad, $\mM^\circledast \to \mV^\ot$ a small weakly left enriched $\infty$-category, $\mN^\circledast \to \mV^\ot \times \mW^\ot $ a right enriched $\infty$-category and $ \mW^\ot \to \Ass$ an $\infty$-operad that exhibits $\mW$ as right enriched in $\mW$ such that $\mW$ admits small limits and for every $\W \in \mW$ the functor $\R\Mor_\mW(\W,-):\mW \to \mW$ preserves small limits. Then $\Enr\Fun_{\mV, \emptyset}(\mM, {\mN})^\circledast \to \mW^\ot $ is a right enriched $\infty$-category.
	
\end{enumerate}

\end{theorem}

\begin{remark}
Let $\mV^\ot \to \Ass, \mW^\ot \to \Ass$ be presentably monoidal $\infty$-categories.
An $\infty$-category left quasi-enriched in $\mV$ is an $\infty$-category left enriched in $\mV$ if and only if it is locally small.
An $\infty$-category right quasi-enriched in $\mW$ is an $\infty$-category right enriched in $\mW$ if and only if it is locally small.

%An $\infty$-category bi-quasi-enriched in $\mV, \mW$ is an $\infty$-category bienriched in $\mV, \mW$ if and only if it is locally small.
	
\end{remark}

\begin{notation}Let $\mM^\circledast \to \mU^\ot \times \mV^\ot$, $\mN^\circledast \to \mV^\ot \times \mW^\ot$, $\mO^\circledast \to \mW^\ot \times \mQ^\ot$ be presentably bitensored $\infty$-categories.
We write $ (\mM \ot_\mV \mN \ot_\mW \mO)^\circledast \to \mU^\ot \times \mQ^\ot $ for the relative tensor product of \cite[Definition 4.4.2.10.]{lurie.higheralgebra}.
\end{notation}

The next theorem is \cite[Theorem 4.67.]{heine2024bienriched}:

\begin{theorem}\label{bica}
The forgetful functor $\gamma: \omega\B\Enr \to \Op_\infty \times \Op_\infty$ 
is a cocartesian fibration.
Let $\psi: \mM^\circledast \to \mN^\circledast$ be an enriched functor lying over maps of $\infty$-operads $\alpha: \mV^\ot \to \mV'^\ot, \alpha': \mW^\ot \to \mW'^\ot$.
The following conditions are equivalent:

\begin{enumerate}
\item The enriched functor $\psi: \mM^\circledast \to \mN^\circledast$ is $\gamma$-cocartesian.

%\item The induced map $\B\Env(\mM)^\circledast \to \B\Env(\mN)^\circledast$of bitensored $\infty$-categories exhibits $\B\Env(\mN)$as $$\Env(\mV') \ot_{\Env(\mV)} \B\Env(\mM) \ot_{\Env(\mW)} \Env(\mW').$$

\item The underlying functor $\mM \to \mN$ is essentially surjective and %the map 
$\mP\B\Env(\mM)^\circledast \to \mP\B\Env(\mN)^\circledast$
%of bitensored $\infty$-categories
induces an equivalence $$\mP\B\Env(\mN) \simeq \mP\Env(\mV') \ot_{\mP\Env(\mV)} \mP\B\Env(\mM) \ot_{\mP\Env(\mW)} \mP\Env(\mW').$$

\item The functor $\mM \to \mN$ is essentially surjective and for every $\X, \Y \in \mM$ the induced morphism 
$$ (\alpha, \beta)_!(\Gamma_{\mM}(\X,\Y)) \to \Gamma_{\mN}(\psi(\X),\psi(\Y))$$ in $\mP(\B\Env(\mV') \times \B\Env(\mW')) $ is an equivalence.		

\end{enumerate}

\end{theorem}

The next proposition is \cite[Proposition 5.45.]{heine2024bienriched}:

\begin{proposition}\label{oppoen}
	
There is an involution $$(-)^\op: {_\mV\omega\B\Enr_\mW} \simeq {_\mW\omega\B\Enr_\mV}$$
fitting into a commutative square 
\begin{equation*} 
\begin{xy}
\xymatrix{
{_\mV\omega\B\Enr_\mW} \ar[d] \ar[rr]^{(-)^\op}
&&{_\mW\omega\B\Enr_\mV} \ar[d] 
\\ \Cat_\infty
\ar[rr]^{(-)^\op}  &&  \Cat_\infty,
}
\end{xy} 
\end{equation*}
such that for every weakly bienriched $\infty$-category $\mM^\circledast \to \mV^\ot \times \mW^\ot $ and $\X,\Y \in \mM$ there is an equivalence $$\Gamma_{\mM^\op}(\X,\Y) \simeq \Gamma_\mM(\Y,\X).$$	
\end{proposition}

%\subsection{Enriched presheaves via tensored envelopes}
\subsection{Generalized enrichment}

%In the following we 

%In this subsection we construct the $\infty$-category ofenriched presheaves as a localization of the enveloping $\infty$-category with closed biaction (Proposition \ref{pseuso}, Proposition \ref{unipor}, Corollary \ref{unipor3}).
%Next we study the relationship between enriched, pseudo-enriched and weakly enriched $\infty$-categories.For that we introduce the following concept that treats weak enrichment, pseudo-enrichment and enrichment on one footage.

The next definition is \cite[Definition 4.1.]{heine2024bienriched}:

\begin{definition}\label{locpa}
	
A small localization pair is a pair $(\mV^\ot \to \Ass, \rS)$,
where $\mV^\ot \to \Ass$ is a small $\infty$-operad and $\rS$ is a set of morphisms
of $\mP\Env(\mV)$ %closed under the tensor product of $\mP\Env(\mV)$
%and containing all identities of objects of $\mV \subset \mP\Env(\mV)$ 
such that the saturated closure $\bar{\rS}$ of $\rS$ is closed under the tensor product %every object of $\mV \subset \mP\Env(\mV)$ is local with respect to the localization $\rS^{-1}\mP\Env(\mV)$ and 
and for every $\V \in \mV$ and $\f \in \rS$ %the morphisms $\V \ot \f, \f \ot \V$ belong to $\rS$and 
the map $\mP\Env(\mV)(\f,\V)$ is an equivalence.
\end{definition}
\begin{remark}
	
Let $(\mV^\ot \to \Ass, \rS)$ be a small localization pair.
Since $\mP\Env(\mV)^\ot \to \Ass$ is a presentably monoidal $\infty$-category, $\rS$ is a set and $\bar{\rS}$ is closed under the tensor product, the
embedding $\rS^{-1}\mP\Env(\mV)^\ot \subset \mP\Env(\mV)^\ot$ of the full suboperad spanned by the $\rS$-local objects admits a left adjoint relative to $\Ass.$
%re is a monoidal localization$\mP\Env(\mV)^\ot \rightleftarrows \rS^{-1}\mP\Env(\mV)^\ot,$ where $\rS^{-1}\mP\Env(\mV)^\ot$ is with respect to $\rS.$
Moreover by definition the embedding of $\infty$-operads $\mV^\ot \subset \mP\Env(\mV)^\ot $ induces an embedding of $\infty$-operads $\mV^\ot \subset \rS^{-1}\mP\Env(\mV)^\ot.$
\end{remark}

The next definition is \cite[Definition 4.3.]{heine2024bienriched}:

\begin{definition}
Let $(\mV^\ot \to \Ass, \rS)$, $(\mW^\ot \to \Ass, \T)$	be small localization pairs.

\begin{enumerate}
\item A weakly bienriched $\infty$-category $\mM^\circledast \to \mV^\ot \times \mW^\ot$ exhibits $\mM$ as left $\rS$-enriched in $\mV$
if for every $\f \in \rS, \X, \Y \in \mM, \W \in \mP\Env(\mW)$
the induced map $\mP\B\Env(\mM)(\f \ot \X \ot \W, \Y)$ is an equivalence. 
	
\item A weakly bienriched $\infty$-category $\mM^\circledast \to \mV^\ot \times \mW^\ot$ exhibits $\mM$ as right $\T$-enriched in $\mW$
if for every $\g \in \T, \X, \Y \in \mM, \V \in \mP\Env(\mV)$
the map $\mP\B\Env(\mM)(\V \ot \X \ot \g, \Y)$ is an equivalence. 

\item A weakly bienriched $\infty$-category $\mM^\circledast \to \mV^\ot \times \mW^\ot$ exhibits $\mM$ as $\rS, \T$-bienriched in $\mV, \mW$ if it exhibits $\mM$ as left $\rS$-enriched in $\mV$ and right $\T$-enriched in $\mW$.
\end{enumerate}	
\end{definition}

%\begin{remark}A weakly bienriched $\infty$-category $\mM^\circledast \to \mV^\ot \times \mW^\ot$ exhibits $\mM$ as $\rS, \T$-bienriched in $\mV, \mW$ if and only if for every $\f \in \bar{\rS}, \g \in \bar{\T}, \X, \Y \in \mM$ the map $\mP\B\Env(\mM)(\f \ot \X \ot \g, \Y)$ is an equivalence.\end{remark}

\begin{notation}Let $(\mV^\ot \to \Ass, \rS)$, $(\mW^\ot \to \Ass, \T)$	be small localization pairs.
Let $$_\mV ^\rS\L\Enr_{\mW}, \ {_\mV\R\Enr}^{\T}_{\mW}, {^\rS_\mV\B\Enr^\T_{\mW}} \subset {_\mV\omega\B\Enr}_{\mW} $$ the full subcategories of left $\rS$-enriched, right $\T$-enriched, $\rS, \T$-bienriched $\infty$-categories, respectively.
	
\end{notation}

\begin{example}Let $\mU^\ot \to \Ass$ be a small $\infty$-operad.
The pair $(\mU^\ot \to \Ass, \emptyset)$ is a small localization pair,
where $\emptyset^{-1}\mP\Env(\mU)^\ot= \mP\Env(\mU)^\ot.$
Let $(\mV^\ot \to \Ass, \rS)$, $(\mW^\ot \to \Ass, \T)$	be small localization pairs.
Then $$_\mV ^\rS\L\Enr_{\mW}=  {^\rS_\mV\B\Enr^\emptyset_{\mW}}, \ {_\mV\R\Enr}^{\T}_{\mW}= {^\emptyset_\mV\B\Enr^\T_{\mW}}.$$

%A weakly bienriched $\infty$-category $\mM^\circledast \to \mV^\ot \times \mW^\ot$ exhibits $\mM$ as $\emptyset, \emptyset$-bienriched in $\mV, \mW$.

\end{example}

\begin{notation}\label{enr}
Let $\kappa$ be a small regular cardinal and $\mV^\ot \to \Ass$ a small monoidal $\infty$-category compatible with $\kappa$-small colimits. 
Let $\Enr_\kappa$ be the set of morphisms
$$\V'_1, ... , \V'_\ell, \V_1, ..., \V_\n, \V''_1, ... , \V''_\bk \to \V'_1, ... , \V'_\ell,\V_1 \ot ...\ot \V_\n, \V''_1, ... , \V''_\bk,$$ 
$$ \V'_1, ... , \V'_\ell,\colim(\iota \circ \F), \V''_1, ... , \V''_\bk \to \V'_1, ... , \V'_\ell, \iota(\colim(\F)),\V''_1, ... , \V''_\bk$$
for $\V'_1, ... , \V'_\ell, \V_1,...,\V_\n, \V''_1, ... , \V''_\bk \in \mV$ and $\ell,\n,\bk \geq0, $ where $\iota: \mV \subset \mP\Env(\mV)$ is the embedding, $\F: \K \to \mV$ is a functor and $\K$ is $\kappa$-small.
\end{notation}

%Let $\mV^\ot \to \Ass, \mW^\ot \to \Ass$ be small $\infty$-operads.

\begin{notation}Let $\kappa, \tau$ be small regular cardinals.

\begin{enumerate}
\item If $\mV^\ot \to \Ass$ is a small monoidal $\infty$-category compatible with $\kappa$-small colimits and $\mW^\ot \to \Ass$ is a small $\infty$-operad, let $\L\Enr_\kappa:=(\Enr_\kappa,\emptyset)$.

\item If $\mV^\ot \to \Ass$ is a small $\infty$-operad and $\mW^\ot \to \Ass$ is a small monoidal $\infty$-category compatible with $\kappa$-small colimits, let $\R\Enr_\kappa:=(\emptyset,\Enr_\kappa)$.

\item If $\mV^\ot \to \Ass, \mW^\ot \to \Ass$ are small monoidal $\infty$-categories  compatible with $\kappa$-small colimits, $\tau$-small colimits, respectively, let $\B\Enr_{\kappa,\tau}:=(\Enr_\kappa,\Enr_\tau)$.
If $\kappa,\tau$ are the the strongly inaccessible cardinal corresponding to the small universe, we drop $\kappa, \tau$, respectively.

%If $\kappa=\emptyset$, we drop $\kappa.$If $\tau=\emptyset$, we drop $\tau.$

%set $\L\Enr:=\L\Enr_\kappa, \R\Enr:=\R\Enr_\kappa, \B\Enr:=\B\Enr_\kappa.$

\end{enumerate}	

\end{notation}

\begin{example}\label{Exaso}
Let $\kappa$ be a small regular cardinal.

\begin{enumerate}
\item For every small monoidal $\infty$-category $\mV^\ot \to \Ass$ compatible with $\kappa$-small colimits the pair $(\mV^\ot \to \Ass, \Enr_\kappa)$ is a small localization pair, 
where $\Enr_\kappa^{-1}\mP\Env(\mV)^\ot= \Ind_\kappa(\mV)^\ot$,
which is a monoidal localization of $\mP\Env(\mV)^\ot \to \Ass.$
A weakly bienriched $\infty$-category $\mM^\circledast \to \mV^\ot \times \mW^\ot$ exhibits $\mM$ as left $\Enr_\kappa$-enriched in $\mV$ if and only if
it exhibits $\mM$ as left $\kappa$-enriched in $\mV$.

\item If $\mV^\ot \to\Ass$ is a presentably monoidal $\infty$-category,
a locally small weakly bienriched $\infty$-category $\mM^\circledast \to \mV^\ot \times \mW^\ot$ exhibits $\mM$ as left $\Enr$-enriched in $\mV$ if and only if it exhibits $\mM$ as left enriched in $\mV$.
\end{enumerate}
\end{example}

%\begin{lemma}\label{locL}Let $(\mV^\ot \to \Ass, \rS)$, $(\mW^\ot \to \Ass, \T)$ be small localization pairssuch that $\rS$ is a set of morphisms of $\Env(\mV)$ and $\T$ is a set of morphisms of $\Env(\mW).$Let $\mM^\circledast \to \mV^\ot $ be a small left $\rS$-enriched $\infty$-categoryand $\mN^\circledast \to \mW^\ot $ a small right $\T$-enriched $\infty$-category.Then $\mM^\circledast \times \mN^\circledast \to \mV^\ot \times \mW^\ot$ is a $\rS,\T$-bienriched $\infty$-category.\end{lemma}

\begin{definition}
Let $(\mV^\ot \to \Ass, \rS)$, $(\mV'^\ot \to \Ass, \rS')$ be small localization pairs.	
A map of localization pairs $(\mV^\ot \to \Ass, \rS) \to (\mV'^\ot \to \Ass, \rS')$
is a map of $\infty$-operads $\mV^\ot \to \mV'^\ot$ such that the induced left adjoint monoidal functor $\mP\Env(\mV) \to \mP\Env(\mV')$
sends morphisms of $\rS$ to morphisms of $\rS'.$

\end{definition}

\begin{notation}
Let $\mM^\circledast \to \mV^\ot \times \mW^\ot, \mN^\circledast \to \mV'^\ot \times \mW'^\ot$ be weakly bienriched $\infty$-categories, $\rS, \rS', \T,\T'$ collections of morphisms of $\mP\Env(\mV),\mP\Env(\mV'),\mP\Env(\mW), \mP\Env(\mW')$, respectively.
Let $$\Enr\Fun_{\rS, \rS',\T, \T'}(\mM, \mN)\subset \Enr\Fun(\mM, \mN) $$ be the full subcategory of enriched functors lying over maps of $\infty$-operads whose induced left adjoint monoidal functors on closed monoidal envelopes send $\rS$ to $\rS'$ and $\T$ to $\T'$, respectively.
\end{notation}

\begin{example}
	
Let $\mV^\ot \to \Ass, \mV'^\ot \to \Ass$ be small monoidal $\infty$-categories compatible with $\kappa$-small colimits, compatible with $\kappa'$-small colimits,
respectively, for small regular cardinals $\kappa < \kappa'$.
A map of $\infty$-operads $\mV^\ot \to \mV'^\ot$ is a map of localization pairs
$(\mV^\ot \to \Ass, \Enr_\kappa) \to (\mV'^\ot \to \Ass, \Enr_{\kappa'})$
if and only if it is a monoidal functor preserving $\kappa$-small colimits.

\end{example}

%\begin{corollary}\label{corg} Let $(\mV^\ot \to \Ass, \rS)$, $(\mW^\ot \to \Ass, \T)$	be small localization pairs.

%\begin{enumerate}\item A weakly bienriched $\infty$-category $\mM^\circledast \to \mV^\ot \times \mW^\ot$ exhibits $\mM$ as left $\rS$-enriched if for every $ \X, \Y \in \mM, \W_1,...,\W_\m \in \mW$ for $\m \geq 0$ the left multi-morphism object $$ \L\Mul\Mor_{\bar{\mM}}(\W_1,...,\W_\m,\X;\Y)$$ belongs to $\rS^{-1}\mP\Env(\mV).$

%\item A weakly bienriched $\infty$-category $\mM^\circledast \to \mV^\ot \times \mW^\ot$ exhibits $\mM$ as right $\T$-enriched if for every $ \X, \Y \in \mM, \V_1,...,\V_\n \in \mV$ for $\n \geq 0$ the right multi-morphism object $$ \R\Mul\Mor_{\bar{\mM}}(\V_1,...,\V_\n,\X;\Y)$$ belongs to $\T^{-1}\mP\Env(\mW).$

%\item A weakly bienriched $\infty$-category $\mM^\circledast \to \mV^\ot \times \mW^\ot$ exhibits $\mM$ as $\rS, \T$-bienriched if for every $ \X, \Y \in \mM, \V_1,...,\V_\n \in \mV$ for $\n \geq 0$ the morphism object $ \Mor_{\bar{\mM}}(\X,\Y)$ belongs to $\rS^{-1}\mP\Env(\mV) \ot \T^{-1}\mP\Env(\mW).$\end{enumerate}\end{corollary}

The next proposition is \cite[Proposition 4.21.]{heine2024bienriched}:

\begin{proposition}\label{eqq}\label{coronn} Let $(\mV^\ot \to \Ass, \rS), (\mW^\ot \to \Ass, \T)$ be small localization pairs.

\begin{enumerate}

\item The functor 
$${_{\rS^{-1}\mP\Env(\mV)}\omega\B\Enr}_{\mW} \to {_\mV\omega\B\Enr}_{\mW},$$ which takes pullback along the embedding $\mV^\ot \subset \rS^{-1}\mP\Env(\mV)^\ot$, restricts to an equivalence
$$ {_{\rS^{-1}\mP\Env(\mV)}\L\Enr}_{\mW} \to {^\rS_\mV\L\Enr}_{\mW}.$$

\item The functor 
$${_\mV\omega\B\Enr}_{\T^{-1}\mP\Env(\mW)} \to {_\mV\omega\B\Enr}_{\mW},$$ which takes pullback along the embedding $\mW^\ot \subset \T^{-1}\mP\Env(\mW)^\ot$, restricts to an equivalence
$$ {_\mV\R\Enr}_{\T^{-1}\mP\Env(\mW)} \to {_\mV\R\Enr}^\T_{\mW}. $$

\item The functor 
$$\rho: {_{\rS^{-1}\mP\Env(\mV)}\omega\B\Enr}_{\T^{-1}\mP\Env(\mW)} \to {_\mV\omega\B\Enr}_{\mW},$$ which takes pullback along the embeddings $\mV^\ot \subset \rS^{-1}\mP\Env(\mV)^\ot, \mW^\ot \subset \T^{-1}\mP\Env(\mW)^\ot$, restricts to an equivalence
$$ {_{\rS^{-1}\mP\Env(\mV)}\B\Enr}_{\T^{-1}\mP\Env(\mW)}  \to {^\rS_\mV\B\Enr^ \T_{\mW}}.$$
\end{enumerate}	
\end{proposition}
Proposition \ref{eqq} gives the following corollary:	
	
\begin{corollary}\label{coronn}
Let $\mV^\ot \to \Ass, \mW^\ot \to \Ass$ be small $\infty$-operads. 
\begin{enumerate}
	
\item The functor $$ {_{\mP\Env(\mV)}\L\Enr}_{\mW} \to {_\mV\omega\B\Enr}_{\mW}$$ taking pullback along the embedding $\mV^\ot \subset \mP\Env(\mV)^\ot$ is an equivalence.\vspace{1mm}

\item The functor $$ {_\mV\R\Enr}_{\mP\Env(\mW)} \to {_\mV\omega\B\Enr}_{\mW}$$ taking pullback along the embedding $ \mW^\ot \subset \mP\Env(\mW)^\ot$ is an equivalence.

\item The functor $$ {_{\mP\Env(\mV)}\B\Enr}_{\mP\Env(\mW)} \to {_\mV\omega\B\Enr}_{\mW}$$ taking pullback along the embeddings $\mV^\ot \subset \mP\Env(\mV)^\ot, \mW^\ot \subset \mP\Env(\mW)^\ot$ is an equivalence.
\end{enumerate}
\end{corollary}

\begin{corollary}\label{cosqa}
Let $\kappa, \tau$ be small regular cardinals.

\begin{enumerate}
\item Let $\mV^\ot \to \Ass$ be a small monoidal $\infty$-category compatible with $\kappa$-small colimits and $ \mW^\ot \to \Ass $ a small $\infty$-operad.
The functor $${_{\Ind_\kappa(\mV)}\L\Enr}_{\mW} \to {_{\mV}^\kappa\L\Enr_\mW} $$ taking pullback along the monoidal embedding $\mV^\ot \subset \Ind_\kappa(\mV)^\ot$ 
is an equivalence.
%$$ {_{\mP\Env(\mV)^\kappa}\L\P\Enr^\kappa_\mW} \to {_{\mV}\omega\L\Enr}_{\mW}.$$ 
%is fully faithful and the essential image precisely consists of the left pseudo-enriched $\infty$-categories$\mM^\circledast \to (\mV^\kappa)^\ot \times \mW^\ot $ such thatfor any $\X, \Y \in \mM, \W_1,..., \W_\m$ for $\m \geq 0$ the presheaf $\Mul_\mM(-,\X,\W_1,..., \W_\m;\Y): (\mV^\kappa)^\op \to \mS$ preserves $\kappa$-small limits.

\item Let $\mW^\ot \to \Ass$ be a small monoidal $\infty$-category compatible with $\tau$-small colimits and $\mV^\ot \to \Ass $ a small $\infty$-operad.
The functor $$ {_\mV\R\Enr}_{\Ind_\tau(\mW)} \to {_\mV\R\Enr^\tau_{\mW}}$$ taking pullback along the monoidal embedding $\mW^\ot \subset \Ind_\tau(\mW)^\ot$ is an equivalence.
%fully faithful and the essential image precisely consists of the right pseudo-enriched $\infty$-categories$\mM^\circledast \to \mV^\ot \times (\mW^\tau)^\ot $ such thatfor any $\X, \Y \in \mM, \V_1,..., \V_\n$ for $\n \geq 0$ the presheaf $\Mul_\mM(\V_1,...,\V_\n,\X,-;\Y): (\mW^\tau)^\op \to \mS$ preserves $\tau$-small limits.

\item Let $\mV^\ot \to \Ass, \mW^\ot \to \Ass$ be small monoidal $\infty$-categories compatible with $\kappa$-small colimits, $\tau$-small colimits, respectively.
The functor $$ {_{\Ind_\kappa(\mV)}\B\Enr}_{\Ind_\tau(\mW)} \to {^\kappa_{\mV}\B\Enr^{\tau}_{\mW}}$$ restricting along the monoidal embeddings $\mV^\ot \subset \Ind_\kappa(\mV)^\ot, \mW^\ot \subset \Ind_\tau(\mW)^\ot$ is an equivalence.
%fully faithful and the essential image precisely consists of the bipseudo-enriched $\infty$-categories$\mM^\circledast \to (\mV^\kappa)^\ot \times (\mW^\tau)^\ot $ such thatfor any $\X, \Y \in \mM$ the presheaf $\Mul_\mM(-,\X,-;\Y): (\mV^\kappa)^\op \times (\mW^\tau)^\op \to \mS$ preserves $\kappa$-small limits in the first component and $\tau$-small limits in the second component.

\end{enumerate}

\end{corollary}

%\begin{corollary}\label{sma}\label{presy}
	
%\begin{enumerate}\item Let $\mV^\ot \to \Ass$ be a presentably monoidal $\infty$-category,$\mW^\ot \to \Ass$ a small $\infty$-operad and $\mM^\circledast \to \mV^\ot \times \mW^\ot, \mN^\circledast \to \mV^\ot \times \mW^\ot$ left enriched $\infty$-categories such that $\mM,\mN$ are small.The $\infty$-category $\Enr\Fun_{\mV, \mW}(\mM,\mN)$ is small.

%\item Let $\mV^\ot \to \Ass$ be a small $\infty$-operad, $\mW^\ot \to \Ass$a presentably monoidal $\infty$-category and $\mM^\circledast \to \mV^\ot \times \mW^\ot, \mN^\circledast \to \mV^\ot \times \mW^\ot$ right enriched $\infty$-categories such that $\mM,\mN$ are small.The $\infty$-category $\Enr\Fun_{\mV, \mW}(\mM,\mN)$ is small.

%\item Let $\mM^\circledast \to \mV^\ot \times \mW^\ot, \mN^\circledast \to \mV^\ot \times \mW^\ot$ be $\infty$-categories bienriched in presentably monoidal $\infty$-categories such that $\mM,\mN$ are small. The $\infty$-category $\Enr\Fun_{\mV, \mW}(\mM,\mN)$ is small.\end{enumerate}\end{corollary}
%\subsubsection{Generalized enriched presheaves}

The next notation is \cite[Notation 4.14.]{heine2024bienriched}:

\begin{notation}Let $(\mV^\ot \to \Ass, \rS), (\mW^\ot \to \Ass, \T)$ be small localization pairs (Definition \ref{locpa}) and $\mM^\circledast \to \mV^\ot \times \mW^\ot$ an absolute small weakly bienriched $\infty$-category.
	
\begin{enumerate}
\item
Let $$\mP\B\Env(\mM)^\circledast_{\rS,\T} \subset \mP\B\Env(\mM)^\circledast$$
be the full weakly bienriched subcategory spanned by the presheaves on
$\B\Env(\mM)$ that are local with respect to the set of morphisms of the form $\f \ot \X \ot \W_1 ... \ot \W_\m $ and $\V_1 \ot ... \ot \V_\n \ot \X \ot \g$ 
for $\f \in \rS, \g \in \T, \X \in \mM$, $\V_1,...,\V_\n \in \mV, \W_1,...,\W_\m \in \mW$ for $\n, \m \geq0$.

\item Let $$\mP\widetilde{\B\Env}(\mM)^\circledast_{\rS,\T}:= \mV^\ot \times_{\rS^{-1}\mP\Env(\mV)^\ot}\mP\B\Env(\mM)^\circledast_{\rS,\T}\times_{\T^{-1}\mP\Env(\mW)^\ot} \mW^\ot.$$
\end{enumerate}

If $\rS=\emptyset$ or $\T=\emptyset$, we drop $\rS,\T$ from the notation, respectively.

\end{notation}

\begin{remark}
As a consequence of Remark \ref{switch} there is a canonical equivalence
$$(\mP\B\Env(\mM)_{\rS,\T}^\rev)^\circledast \simeq \mP\B\Env(\mM^\rev)_{\T,\rS}^\circledast.$$
	
\end{remark}

\begin{lemma}\label{Lauf}
Let $(\mV^\ot \to \Ass, \rS), (\mW^\ot \to \Ass, \T)$ be small localization pairs and $\mM^\circledast \to \mV^\ot \times \mW^\ot$ a small weakly bienriched $\infty$-category.
The embedding $\mP\B\Env(\mM)^\circledast_{\rS, \T} \subset \mP\B\Env(\mM)^\circledast$ admits an enriched left adjoint such that the unit lies over the units of the localizations $\rS^{-1}\mP\Env(\mV)^\ot \subset \mP\Env(\mV)^\ot, \T^{-1}\mP\Env(\mW)^\ot \subset \mP\Env(\mW)^\ot$. % relative to $\Ass$.
In particular, $\mP\B\Env(\mM)^\circledast_{\rS, \T} \to \rS^{-1}\mP\Env(\mV)^\ot \times \T^{-1}\mP\Env(\mW)^\ot$ is a presentably bitensored $\infty$-category.
	
\end{lemma}

\begin{remark}
Since $\mP\B\Env(\mM)$ is presentable, the embedding $\mP\B\Env(\mM)_{\rS,\T} \subset \mP\B\Env(\mM)$ admits a left adjoint and $\mP\B\Env(\mM)_{\rS,\T} \subset \mP\B\Env(\mM)$ precisely consists of the presheaves on $\B\Env(\mM)$ that are local with respect to the collection $\mQ$ of morphisms of the form $\f \ot \X \ot \g $ for $\f \in \bar{\rS}, \g \in \bar{\T}$ and $\X \in \mM.$ The biaction of $\mP\Env(\mV), \mP\Env(\mW)$ on $\mP\B\Env(\mM)$ sends morphisms of $\bar{\rS}, \mQ, \bar{\T}$ to $\mQ$.
This implies that the embedding $\mP\B\Env(\mM)^\circledast_{\rS, \T} \subset \mP\B\Env(\mM)^\circledast$ admits an enriched left adjoint such that the unit lies over the units of the localizations $\rS^{-1}\mP\Env(\mV)^\ot \subset \mP\Env(\mV)^\ot, \T^{-1}\mP\Env(\mW)^\ot \subset \mP\Env(\mW)^\ot$. % relative to $\Ass$.
\end{remark}

\begin{remark}\label{embil}
	
A weakly bienriched $\infty$-category $\mM^\circledast \to \mV^\ot \times \mW^\ot$ is a $\rS, \T$-bienriched $\infty$-category if and only if
$\mM^\circledast \subset \mP\B\Env(\mM)^\circledast$ lands in
$\mP\B\Env(\mM)^\circledast_{\rS,\T}$.
\end{remark}

\begin{remark}\label{Functol}
Let $\mM^\circledast \to \mN^\circledast$ be an enriched functor between absolute small weakly bienriched $\infty$-categories lying over maps of small localization pairs
$(\mV^\ot \to \Ass, \rS)\to (\mV'^\ot \to \Ass, \rS'), (\mW^\ot \to \Ass, \T)\to (\mW'^\ot \to \Ass, \T')$.
The induced left adjoint map of bitensored $\infty$-categories $\mP\B\Env(\mM)^\circledast \to \mP\B\Env(\mN)^\circledast$ preserves local equivalences for the corresponding localizations and so descends to a left adjoint map of bitensored $\infty$-categories $\mP\B\Env(\mM)_{\rS,\T}^\circledast \to \mP\B\Env(\mN)_{\rS',\T'}^\circledast$.
	
\end{remark}

The next proposition is \cite[Proposition 4.20.]{heine2024bienriched}:

\begin{proposition}\label{pseuso} Let $(\mV^\ot \to \Ass, \rS), (\mW^\ot \to \Ass, \T)$ be small localization pairs, $\mM^\circledast \to \mV^\ot \times \mW^\ot$ a small $\rS, \T$-bienriched $\infty$-category and $\rho: \mN^\circledast \to \rS^{-1}\mP\Env(\mV)^\ot \times \T^{-1}\mP\Env(\mW)^\ot$ a bitensored $\infty$-category compatible with small colimits. 

\begin{enumerate}\item The enriched functor $\mM^\circledast \subset \mP\B\Env(\mM)^\circledast \to \mP\B\Env(\mM)^\circledast_{\rS,\T}$ induces a functor $$\Enr\Fun_{\rS^{-1}\mP\Env(\mV),\T^{-1}\mP\Env(\mW)}(\mP\B\Env(\mM)_{\rS,\T},\mN) \to \Enr\Fun_{\mV,\mW}(\mM,\mN)$$that admits a fully faithful left adjoint. The left adjoint lands in the full subcategory $$\LinFun^\L_{\rS^{-1}\mP\Env(\mV),\T^{-1}\mP\Env(\mW)}(\mP\B\Env(\mM)_{\rS,\T},\mN).$$

\item The following induced functor is an equivalence: $$\LinFun^\L_{\rS^{-1}\mP\Env(\mV),\T^{-1}\mP\Env(\mW)}(\mP\B\Env(\mM)_{\rS,\T},\mN) \to \Enr\Fun_{\mV,\mW}(\mM,\mN).$$
\end{enumerate}\end{proposition}

%We add the following proposition, which we need in the next section toprove Proposition \ref{wooond}.
%\subsection{Enriched presheaves}

%Next we construct for every small regular cardinal $\kappa$ a $\kappa$-enriched $\infty$-category of $\kappa$-enriched presheaves.

%\begin{notation}\label{nity}Let $\kappa$ be a small regular cardinal and $\mV^\ot \to \Ass$ a small monoidal $\infty$-category.
%Let $\mP_\kappa(\mV)^\ot \subset \mP(\mV)^\ot$ be the full suboperad spanned by the full subcategory of $\mP(\mV)$ generated by $\mV$ under $\kappa$-small colimits.Let $\mP^\kappa(\mV)^\ot \subset \mP(\mV)^\ot$ be the full suboperad spanned by the $\kappa$-compact objects. % of $\mP(\mV)$.\end{notation}

%\begin{remark}\label{comge}Let $\mV^\ot \to \Ass$ be a small monoidal $\infty$-category and $\kappa$ a small regular cardinal.Since the Yoneda-embedding is monoidal and $ \mP(\mV)^\ot \to \Ass$ is compatible with small colimits, $\mP_\kappa(\mV)^\ot \to \Ass$ is a monoidal full subcategory of $\mP(\mV)^\ot\to\Ass$ compatible with $\kappa$-small colimits.The full subcategory $\mP^\kappa(\mV)$ is the smallest full subcategory of $\mP(\mV)$ containing $\mP_\kappa(\mV)$ and closed under retracts \cite[Proposition 5.3.4.17.]{lurie.HTT}. So $\mP^\kappa(\mV)$ is a monoidal full subcategory of $\mP(\mV)$ closed under $\kappa$-small colimits.\end{remark}

\begin{notation}
Let $\kappa$ be a small regular cardinal and $\mC$ a small $\infty$-category.
Let $\mP_\kappa(\mC)^\ot \subset \mP(\mC)^\ot$ be the full suboperad spanned by the full subcategory of $\mP(\mC)$ generated by $\mC$ under $\kappa$-small colimits.
	
\end{notation}

The next notation is \cite[Notation 4.37.]{heine2024bienriched}:

%\begin{remark}
%Observe that $\widehat{\mP}_\sigma(\mC)=\mP(\mC).$	\end{remark}

\begin{notation}\label{enrprrr}Let $\kappa$ be a small regular cardinal.
	
\begin{enumerate}\item Let $\mV^\ot \to \Ass$ be a small monoidal $\infty$-category compatible with $\kappa$-small colimits, $\mW^\ot \to \Ass$ a small $\infty$-operad and $\mM^\circledast \to \mV^\ot \times \mW^\ot$ a small left $\kappa$-enriched $\infty$-category.
Let $$\mP\B\Env_\kappa(\mM)_{\L\Enr}^\circledast \subset \mV^\ot \times_{\Ind_\kappa(\mV)^\ot} \mP\B\Env(\mM)_{\L\Enr_\kappa}^\circledast \times_{\mP\Env(\mW)^\ot} \mP_\kappa\Env(\mW)^\ot $$ be the full bitensored subcategory generated by $\mM$ under $\kappa$-small colimits. % and the $\mV, \mP_\kappa\Env(\mW)$-biaction.

%Let $$\mP\widetilde{\B\Env}_\kappa(\mM)_{\L\Enr}^\circledast := \mP\B\Env_\kappa(\mM)_{\L\Enr}^\circledast \times_{\mP_\kappa\Env(\mW)^\ot} \mW^\ot \to \mV^\ot \times \mW^\ot.$$
%Let$$\mP\L\Env_\kappa(\mM)_{\L\Enr}^\circledast \subset \mP\B\Env_\kappa(\mM)_{\L\Enr_\kappa}^\circledast $$be the full weakly bienriched subcategory generated by $\mM$ under $\kappa$-small colimits and the left $\mV$-action.

%\item Let $$\mP\widetilde{\L\Env}_\kappa(\mM)_{\L\Enr}^\circledast := \mP\L\Env_\kappa(\mM)_{\L\Enr}^\circledast \times_{\mP_\kappa\Env(\mW)^\ot} \mW^\ot \to \mV^\ot \times \mW^\ot.$$

\item Let $\mV^\ot \to \Ass$ be a small $\infty$-operad, $\mW^\ot \to \Ass$ a small monoidal $\infty$-category compatible with $\kappa$-small colimits and $\mM^\circledast \to \mV^\ot \times \mW^\ot$ a small right $\kappa$-enriched $\infty$-category.
%Let$$ \mP\B\Env_\kappa(\mM)_{\R\Enr}^\circledast := (\mP\B\Env_\kappa(\mM^\rev)^\rev_{\L\Enr})^\circledast.$$
Let $$ \mP\B\Env_\kappa(\mM)_{\R\Enr}^\circledast \subset \mP_\kappa\Env(\mV)^\ot \times_{\mP\Env(\mV)^\ot} \mP\B\Env(\mM)_{\R\Enr_\kappa}^\circledast \times_{\Ind_\kappa(\mW)^\ot} \mW^\ot$$
be the full bitensored subcategory generated by $\mM$ under $\kappa$-small colimits.
% and the $\mP_\kappa\Env(\mV), \mW$-biaction.

%Let $$\mP\widetilde{\B\Env}_\kappa(\mM)_{\R\Enr}^\circledast := \mV^\ot \times_{\mP_\kappa\Env(\mV)^\ot} \mP\B\Env_\kappa(\mM)_{\R\Enr}^\circledast \to \mV^\ot \times \mW^\ot.$$

%\item Let $\mV^\ot \to \Ass$ be a small $\infty$-operad, $\mW^\ot \to \Ass$ a small monoidal $\infty$-category compatible with $\kappa$-small colimits and $\mM^\circledast \to \mV^\ot \times \mW^\ot$ a small right $\kappa$-enriched $\infty$-category.Let$$ \mP\R\Env_\kappa(\mM)_{\R\Enr}^\circledast \subset \mP_\kappa\Env(\mV)^\ot \times_{\mP\Env(\mV)^\ot} \mP\B\Env(\mM)_{\R\Enr_\kappa}^\circledast \times_{\Ind_\kappa(\mW)^\ot} \mW^\ot$$be the full weakly bienriched subcategory generated by $\mM$ under $\kappa$-small colimits and the right $\mW$-action.Let$$\mP\widetilde{\R\Env}_\kappa(\mM)_{\R\Enr}^\circledast := \mV^\ot \times_{\mP_\kappa\Env(\mV)^\ot} \mP\R\Env_\kappa(\mM)_{\R\Enr}^\circledast \to \mV^\ot \times \mW^\ot.$$
	
\item Let $\mV^\ot \to \Ass, \mW^\ot \to \Ass$ be small monoidal $\infty$-categories compatible with $\kappa$-small colimits and $\mM^\circledast \to \mV^\ot \times \mW^\ot$ a small $\kappa, \kappa$-bienriched $\infty$-category. Let$$\mP\B\Env_\kappa(\mM)_{\B\Enr}^\circledast \subset \mV^\ot \times_{\Ind_\kappa(\mV)^\ot} \mP\B\Env(\mM)_{\B\Enr_\kappa}^\circledast\times_{\Ind_\kappa(\mW)^\ot} \mW^\ot$$be the full bitensored subcategory generated by $\mM$ under $\kappa$-small colimits.

%Applying the notation in a larger universe and for $\kappa=\sigma$ the strongly inacessible cardinal corresponding to the small universe we remove $\kappa$ from the notation.
% and the $\mV,\mW$-biaction.

%\item Let $\mV^\ot \to \Ass$ be a small monoidal $\infty$-category compatible with $\kappa$-small colimits, $\mW^\ot \to \Ass$ a small $\infty$-operad and $\mM^\circledast \to \mV^\ot \times \mW^\ot$ a small left $\kappa$-enriched $\infty$-category.Let$$\mP\widetilde{\B\Env}_\kappa(\mM)_{\L\Enr}^\circledast := \mP\B\Env_\kappa(\mM)_{\L\Enr}^\circledast \times_{\mP_\kappa\Env(\mW)^\ot} \mW^\ot \to \mV^\ot \times \mW^\ot.$$

%\itemLet $\mV^\ot \to \Ass$ be a small $\infty$-operad, $\mW^\ot \to \Ass$ a small monoidal $\infty$-category compatible with $\kappa$-small colimits and $\mM^\circledast \to \mV^\ot \times \mW^\ot$ a small right $\kappa$-enriched $\infty$-category.Let $$\mP\widetilde{\B\Env}_\kappa(\mM)_{\R\Enr}^\circledast := \mV^\ot \times_{\mP_\kappa\Env(\mV)^\ot} \mP\B\Env_\kappa(\mM)_{\R\Enr}^\circledast \to \mV^\ot \times \mW^\ot.$$
	
%\item Let $\mV^\ot \to \Ass$ be a small monoidal $\infty$-category compatible with $\kappa$-small colimits, $\mW^\ot \to \Ass$ a small $\infty$-operad and $\mM^\circledast \to \mV^\ot \times \mW^\ot$ a small left $\kappa$-pseudo-enriched $\infty$-category.

%\item Let $\mV^\ot \to \Ass$ be a small $\infty$-operad, $\mW^\ot \to \Ass$ a small monoidal $\infty$-category compatible with $\kappa$-small colimits and $\mM^\circledast \to \mV^\ot \times \mW^\ot$ a small right $\kappa$-pseudo-enriched $\infty$-category.
\end{enumerate}\end{notation}

Next we apply Notation \ref{enrprrr} in a larger universe:

\begin{notation}\label{enrprrrt}Let $\sigma$ be the large regular cardinal corresponding to the small universe.
	
\begin{enumerate}\item Let $\mV^\ot \to \Ass$ be a monoidal $\infty$-category compatible with small colimits, $\mW^\ot \to \Ass$ an $\infty$-operad and $\mM^\circledast \to \mV^\ot \times \mW^\ot$ a left quasi-enriched $\infty$-category.
Let $$ \mP\B\Env(\mM)_{\L\Enr}^\circledast := \mP\B\Env_\sigma(\mM)_{\L\Enr}^\circledast \to \mV^\ot \times \mP\Env(\mW)^\ot.$$

%Let $$\mP\widetilde{\B\Env}_\kappa(\mM)_{\L\Enr}^\circledast := \mP\B\Env_\kappa(\mM)_{\L\Enr}^\circledast \times_{\mP_\kappa\Env(\mW)^\ot} \mW^\ot \to \mV^\ot \times \mW^\ot.$$
%Let$$\mP\L\Env_\kappa(\mM)_{\L\Enr}^\circledast \subset \mP\B\Env_\kappa(\mM)_{\L\Enr_\kappa}^\circledast $$be the full weakly bienriched subcategory generated by $\mM$ under $\kappa$-small colimits and the left $\mV$-action.
		
%\item Let $$\mP\widetilde{\L\Env}_\kappa(\mM)_{\L\Enr}^\circledast := \mP\L\Env_\kappa(\mM)_{\L\Enr}^\circledast \times_{\mP_\kappa\Env(\mW)^\ot} \mW^\ot \to \mV^\ot \times \mW^\ot.$$
		
%\item Let $\mV^\ot \to \Ass$ be a small $\infty$-operad, $\mW^\ot \to \Ass$ a small monoidal $\infty$-category compatible with $\kappa$-small colimits and $\mM^\circledast \to \mV^\ot \times \mW^\ot$ a small right $\kappa$-enriched $\infty$-category.
%Let$$ \mP\B\Env_\kappa(\mM)_{\R\Enr}^\circledast := (\mP\B\Env_\kappa(\mM^\rev)^\rev_{\L\Enr})^\circledast.$$
%Let $$ \mP\B\Env_\kappa(\mM)_{\R\Enr}^\circledast \subset \mP_\kappa\Env(\mV)^\ot \times_{\mP\Env(\mV)^\ot} \mP\B\Env(\mM)_{\R\Enr_\kappa}^\circledast \times_{\Ind_\kappa(\mW)^\ot} \mW^\ot$$
%	be the full bitensored subcategory generated by $\mM$ under $\kappa$-small colimits.
% and the $\mP_\kappa\Env(\mV), \mW$-biaction.
		
%Let $$\mP\widetilde{\B\Env}_\kappa(\mM)_{\R\Enr}^\circledast := \mV^\ot \times_{\mP_\kappa\Env(\mV)^\ot} \mP\B\Env_\kappa(\mM)_{\R\Enr}^\circledast \to \mV^\ot \times \mW^\ot.$$
		
\item Let $\mV^\ot \to \Ass$ be an $\infty$-operad, $\mW^\ot \to \Ass$ a monoidal $\infty$-category compatible with small colimits and $\mM^\circledast \to \mV^\ot \times \mW^\ot$ a right quasi-enriched $\infty$-category.
Let $$\mP\B\Env(\mM)_{\R\Enr}^\circledast := \mP\B\Env_\sigma(\mM)_{\R\Enr}^\circledast \to \mP\Env(\mV)^\ot \times \mW^\ot.$$
		
\item Let $\mV^\ot \to \Ass, \mW^\ot \to \Ass$ be monoidal $\infty$-categories compatible with small colimits and $\mM^\circledast \to \mV^\ot \times \mW^\ot$ a bi-quasi-enriched $\infty$-category. Let $$\mP\B\Env(\mM)_{\B\Enr}^\circledast:= \mP\B\Env_\sigma(\mM)_{\B\Enr}^\circledast \to \mV^\ot \times \mW^\ot.$$
\end{enumerate}
	
\end{notation}

\begin{notation}
	
Let $\mV^\ot \to \Ass$ be a monoidal $\infty$-category compatible with small colimits and $\mM^\circledast \to \mV^\ot$ a left quasi-enriched $\infty$-category.
Let $\mP_\mV(\mM)^\circledast:= \mP\B\Env(\mM)_{\L\Enr}^\circledast.$	
	
\end{notation}

The next proposition is \cite[Corollary 4.46.]{heine2024bienriched}:
\begin{proposition}\label{rembrako}
\begin{enumerate}\item Let $\mV^\ot \to \Ass$ be a presentably monoidal $\infty$-category, $\mW^\ot \to \Ass$ a small $\infty$-operad and $\mM^\circledast \to \mV^\ot \times \mW^\ot$ a small left enriched $\infty$-category.
Then $$\mP\B\Env(\mM)_{\L\Enr}^\circledast \to \mV^\ot \times \mP\Env(\mW)^\ot$$ is a presentably left tensored $\infty$-category.
		
\item Let $\mV^\ot \to \Ass$ be a small $\infty$-operad, $\mW^\ot \to \Ass$ a presentably monoidal $\infty$-category and $\mM^\circledast \to \mV^\ot \times \mW^\ot$ a small right enriched $\infty$-category.
Then $$\mP\B\Env(\mM)_{\R\Enr}^\circledast \to \mP\Env(\mV)^\ot \times \mW^\ot$$
is a presentably right tensored $\infty$-category.

\item Let $\mV^\ot \to \Ass, \mW^\ot \to \Ass$ be presentably monoidal $\infty$-categories and $\mM^\circledast \to \mV^\ot \times \mW^\ot$ a small bienriched $\infty$-category. Then $$ \mP\B\Env(\mM)_{\B\Enr}^\circledast:= \mP\B\Env_\sigma(\mM)_{\B\Enr}^\circledast \to \mV^\ot \times \mW^\ot$$ is a presentably bitensored $\infty$-category.

\end{enumerate}
	
\end{proposition}

The next theorem is \cite[5.49.]{heine2024bienriched}, which gives rise to the enriched Yoneda-embedding.

\begin{theorem}\label{unitol}
Let $\mM^\circledast \to \mV^\ot $ be a left enriched $\infty$-category. There is a left $\mV$-linear equivalence $$\mid-\mid: \mP_\mV(\mM)^\circledast\simeq\Enr\Fun_{\emptyset, \mV}(\mM^\op,\mV)^\circledast$$	
sending $\X$ to $\L\Mor_{\mP_\mV(\mM)}((-)_{\mid\mM},\X).$
In particular, there is a left $\mV$-enriched embedding $$\mM^\circledast \to \Enr\Fun_{\emptyset, \mV}(\mM^\op,\mV)^\circledast$$	
that sends $\X$ to $\L\Mor_{\mM}(-,\X)$.
	
%If $\mV^\ot \to \Ass$ is a presentably monoidal $\infty$-category, the equivalence $\theta$ corresponds to the $\mV,\mV$-enriched functor $$\hspace{10mm}\mP\L\Env(\mM)_{\L\Enr}^\circledast\times (\mM^\op)^\circledast \subset\mP\L\Env(\mM)_{\L\Enr}^\circledast \times (\mP\L\Env(\mM)_{\L\Enr}^\op)^\circledast\xrightarrow{\L\Mor_{\mP\L\Env(\mM)_{\L\Enr}}(-,-)}\mV^\circledast.$$
\end{theorem}

Corollary \ref{coronn} gives the following corollary:

\begin{corollary}
	
Let $\mV^\ot \to \Ass$ be a small $\infty$-operad and $\mM^\circledast \to \mV^\ot $ a weakly left enriched $\infty$-category. There is a left $\mP\Env(\mV)$-linear equivalence $$\mid-\mid: \mP\L\Env(\mM)^\circledast\simeq\Enr\Fun_{\emptyset, \mV}(\mM^\op,\mP\Env(\mV))^\circledast$$	
sending $\X$ to $\L\Mor_{\mP\L\Env(\mM)}((-)_{\mid\mM},\X).$
In particular, there is a left $\mP\Env(\mV)$-enriched embedding $$\mM^\circledast \to \Enr\Fun_{\emptyset, \mV}(\mM^\op,\mP\Env(\mV))^\circledast$$	
that sends $\X$ to $\L\Mor_{\bar{\mM}}(-,\X)$.	
	
\end{corollary}

\section{Weighted colimits}\label{HW}

\subsection{Weighted colimits for weakly enriched $\infty$-categories}

In this section we define weighted colimits and study their properties. 
Weighted colimits are a generalization of conical colimits. % associated to enriched diagram $\infty$-categories.
For every weakly bienriched $\infty$-category $\mM^\circledast \to \mV^\ot\times \mW^\ot$ a 
functor $\bar{\F}:\K^\triangleleft \to \mM$ is a conical colimit diagram if for every $\V_1,...,\V_\n \in \mV, \W_1,...,\W_\m \in \mW, \Z \in \mM$ for $\n,\m \geq 0$ the canonical map 
\begin{equation}\label{eqvbp}
\Mul_\mM(\V_1,...,\V_\n, \bar{\F}(\infty),\W_1,...,\W_\m;\Z) \to \lim \Mul_\mM(\V_1,...,\V_\n,\F(-),\W_1,...,\W_\m;\Z)\end{equation} is an equivalence.
By Proposition \ref{ljnbfg} the functor $\F:= \bar{\F}_{\mid\K}:\K \to \mM$ underlies a unique $\mV,\mW$-enriched functor $ \K_{\mV,\mW}^\circledast \to \mM^\circledast$,
%and the canonical transformation $\F \to \bar{\F}(\infty)$ to the constant diagramcorresponds to a map of presheaves $* \to \F^*(\mM(-,\bar{\F}(\infty)))$ on $\K $ starting at the final presheaf on $\K$.Via 
which gives rise to a $\mP\Env(\mV), \mP\Env(\mW)$-enriched adjunction $$\F_!: \mP\B\Env(\K_{\mV,\mW})^\circledast \rightleftarrows \mP\B\Env(\mM)^\circledast:\F^*.$$
%the latter map of presheaves on $\K$ corresponds to a map$\F_!(*) \to \bar{\F}(\infty)$ of presheaves on $\mM.$
%of $\F:\mJ \to \mM$ exhibits $\bar{\F}(\infty)$ as the conical colimit of $\F$ if for every $\mV_1,...,\mV_\n \in \mV, \W_1,...,\W_\m \in \mW,\Z \in \mM$ for $\n,\m \geq 0$ the following map is an equivalence: \begin{equation}\label{eqvbp}\Mul_\mM(\V_1,...,\V_\n, \bar{\F}(\infty), \W_1,...,\W_\m;\Z) \to \lim\Mul_\mM(\V_1,...,\V_\n,\F(-), \W_1,...,\W_\m;\Z).\end{equation}
%The map $\F:\mJ^\circledast \to \mM^\circledast$ gives rise to an adjunction $\label{eqb}\F_!: \mP\B\Env(\mJ)^\circledast \rightleftarrows \mP\B\Env(\mM)^\circledast :\F^\ast.$Let $*\in \mP(\mJ)$ be the final presheaf.
Let $*$ be the final presheaf on $\K_{\mV,\mW}$. Lemma \ref{alor} provides a canonical equivalence
\begin{equation*}\label{eqccl}
\lim\Mul_{\mM}(\V_1,...,\V_\n,\F(-), \W_1,...,\W_\m;\Z) \simeq \mP(\K_{\mV,\mW})(*, \Mul_{\mM}(\V_1,...,\V_\n,\F(-), \W_1,...,\W_\m;\Z) \simeq $$
$$ \Mul_{\mP\B\Env(\K_{\mV,\mW})}(\V_1,...,\V_\n, *, \W_1,...,\W_\m; \F^*(\Z)) \simeq \Mul_{\mP\B\Env(\mM)}(\V_1,...,\V_\n, \F_!(*),  \W_1,...,\W_\m; \Z),
\end{equation*}
%where $*$ is the final presheaf on $\K_{\mV,\mW}$ and the middle equivalence is by Lemma \ref{alor}, 
under which the map (\ref{eqvbp}) identifies with the following map:
$$\Mul_\mM(\V_1,...,\V_\n, \bar{\F}(\infty), \W_1,...,\W_\m;\Z) \to \Mul_{\mP\B\Env(\mM)}(\V_1,...,\V_\n, \F_!(*), \W_1,...,\W_\m; \Z).$$

This motivates the following definition of weighted colimit,
where we replace the weakly bienriched $\infty$-category $\K_{\V,\W}^\circledast \to \mV^\ot \times \mW^\ot$ by any weakly bienriched $\infty$-category $\mJ^\circledast \to \mV^\ot \times \mW^\ot$ and the final presheaf $\ast \in \mP(\mJ) \subset \mP\B\Env(\mJ)$ by any object of $\mP\B\Env(\mJ)$: 

\begin{definition}\label{weight}
Let $\mM^\circledast \to \mV^\ot \times \mW^\ot, \mJ^\circledast \to \mV'^\ot \times \mW'^\ot$ be absolute small weakly  bienriched $\infty$-categories, $\F: \mJ^\circledast \to \mM^\circledast$ an enriched functor, $\Y\in\mM$ and $\rH\in\mP\B\Env(\mJ).$
A morphism $\psi: \F_!(\rH) \to \Y $ in $\mP\B\Env(\mM)$ exhibits $\Y$ as the $\rH$-weighted colimit of $\F$ if for every $\Z \in \mM$, $\V_1,...,\V_\n \in \mV, \W_1,...,\W_\m \in \mW$ for $\n,\m \geq 0$ the following map is an equivalence:
$$\Mul_\mM(\V_1,...,\V_\n, \Y, \W_1,...,\W_\m; \Z) \to \Mul_{\mP\B\Env(\mM)}(\V_1,...,\V_\n, \F_!(\rH), \W_1,...,\W_\m; \Z).$$
\end{definition}

\begin{notation}
	
If a morphism $\psi$ exists like in Definition \ref{weight}, it is unique, and we say that $\mM$ admits the $\rH$-weighted colimit of $ \F$ and write $\colim^\rH(\F)$ for $\Y.$ %We call $\rH$ a weight on $\mJ.$
%We say that $\mM$ admits small weighted colimits 
\end{notation}

Dually, we define weighted limits:

\begin{definition}\label{weightl}
Let $\mM^\circledast \to \mV^\ot \times \mW^\ot, \mJ^\circledast \to \mV'^\ot \times \mW'^\ot$ be absolute small weakly  bienriched $\infty$-categories, $\F: \mJ^\circledast \to \mM^\circledast$ an enriched functor, $\Y\in\mM$ and $\rH\in\mP\B\Env(\mJ^\op).$	
The $\rH$-weighted limit of $\F$ denoted by $\lim^\rH(\F)$ if it exists, is the $\rH$-weighted colimit of $\F^\op: (\mJ^\op)^\circledast \to (\mM^\op)^\circledast$.	
	
\end{definition}

We often prove results about weighted colimits that dually give results about weighted limits
and viceversa.

\begin{lemma}\label{sasa}
Let $\rho: \mI^\circledast \to \mJ^\circledast, \F: \mJ^\circledast \to \mM^\circledast$ be enriched functors and $\rH \in \mP\B\Env(\mI), \Y \in \mM.$
A morphism $(\F\circ \rho)_!(\rH) \simeq \F_!(\rho_!(\rH)) \to \Y$ in $\mP\B\Env(\mM)$
exhibits $\Y$ as the $\rho_!(\rH)$-weighted colimit of $\F$ if and only if
it exhibits $\Y$ as the $\rH$-weighted colimit of $\F \circ \rho.$	
\end{lemma}

\begin{proof}
This follows immediately from the definition.
\end{proof}

\begin{remark}\label{rest}Let the assumptions like in Lemma \ref{sasa} and let $\F: \mJ^\circledast \to \mM^\circledast$ lie over maps of $\infty$-operads $\alpha,\beta$.
Then $\F$ factors as the canonical enriched functor $\tau: \mJ^\circledast \to (\alpha,\beta)_!(\mJ)^\circledast$ followed by a unique $\mV,\mW$-enriched functor $\F': (\alpha,\beta)_!(\mJ)^\circledast \to \mM^\circledast$.
Thus by Lemma \ref{sasa} a morphism $\F'_!(\tau_!(\rH)) \simeq \F_!(\rH) \to\Y$ exhibits $\Y$ as the $\rH$-weighted colimit of $\F$ if and only if it exhibits $\Y$ as the $\tau_!(\rH)$-weighted colimit of $\F'.$
%In particular, $\mM$ admits $(\alpha,\beta, \mJ^\circledast \to \mV'^\ot \times \mW'^\ot, \rH)$-weighted colimits if and only if$\mM$ admits $\tau_!(\rH)$-weighted colimits.
In other words by changing the weight it is enough to define weighted colimits for
enriched functors lying over identity maps of $\infty$-operads.
On the other hand, the extra freedom in Definition \ref{weight} is more convenient when studying preservation of weighted colimits by general enriched functors.
\end{remark}

\begin{lemma}\label{lif} Let the assumptions like in Definition \ref{weight} assume that $\F$ lies over essentially surjective maps $\alpha:\mV'^\ot \to \mV^\ot, \beta: \mW'^\ot \to \mW^\ot $ of $\infty$-operads. The morphism $\psi: \F_!(\rH) \to \Y $ in $\mP\B\Env(\mM)$ exhibits $\Y$ as the $\rH$-weighted colimit of $\F$ if and only if the corresponding morphism
$\psi': \rH \to \F^*(\Y) $ in $\mP\B\Env(\mJ)$ induces for every $\Z \in \mM$, $\V'_1,...,\V'_\n \in \mV', \W'_1,...,\W'_\m \in \mW'$ for $\n,\m \geq 0$ an equivalence:
$$\Mul_{\mM}(\alpha(\V'_1),...,\alpha(\V'_\n), \Y, \beta(\W'_1),...,\beta(\W'_\m); \Z) \to \Mul_{\mP\B\Env(\mJ)}(\V'_1,...,\V'_\n, \rH, \W'_1,...,\W'_\m;\F^*(\Z)).$$
\end{lemma}

\begin{proof}
By adjointness %for every $\Z \in \mM$, $\V'_1,...,\V'_\n \in \mV', \W'_1,...,\W'_\m \in \mW'$ for $\n,\m \geq 0$ 
there is an equivalence 
$$\Mul_{\mP\B\Env(\mM)}(\alpha(\V'_1),...,\alpha(\V'_\n), \F_!(\rH), \beta(\W'_1),...,\beta(\W'_\m); \Z) \simeq $$$$ \Mul_{\mP\B\Env(\mJ)}(\V'_1,...,\V'_\n, \rH, \W'_1,...,\W'_\m; \F^*(\Z)).$$		
	
\end{proof}

\begin{remark}Let the assumptions like in Lemma \ref{lif} but let $\alpha, \beta$ not be essentially surjective.
Then the condition of Lemma \ref{lif} is weaker than the one of Definition \ref{weight} and amounts to say that $\Y $ is the $\rH$-weighted colimit of the induced $\mV',\mW'$-enriched functor $\mJ^\circledast \to \mV'^\ot \times_{\mV^\ot} \mM^\circledast \times_{\mW^\ot} \mW'^\ot.$

\end{remark}

We have the following functoriality of weighted colimits:
\begin{remark}\label{zujj}
Let $\alpha:\F\to \G: \mJ^\circledast \to \mM^\circledast$ be an enriched functor and $\tau: \rH \to \rH'$ a morphism in $\mP\Env(\mJ).$
The morphism $\F_!(\rH) \xrightarrow{\alpha_!(\rH)} \G_!(\rH) \xrightarrow{\G_!(\tau)} \G_!(\rH') \to \colim^{\rH'}(\G)$
in $\mP\Env(\mM)$
factors as the morphism $\F_!(\rH) \to \colim^\rH(\F)$ followed by a unique morphism $\colim^{\rH}(\F) \to \colim^{\rH'}(\G)$ in $\mM.$

\end{remark}

%In the following we consider some remarks about Definition \ref{weight}:

\begin{definition}\label{ddd}
%Let $\mV^\ot \to \Ass, \mW^\ot \to \Ass$ be $\infty$-operads. 

An absolute small weight %over $\mV,\mW$ 
is a quadruple $(\alpha: \mV'^\ot \to \mV^\ot,
\beta: \mW'^\ot \to \mW^\ot, \mJ^\circledast \to \mV'^\ot \times \mW'^\ot, \rH),$ where $\alpha, \beta$ are maps of small $\infty$-operads,
$\mJ^\circledast \to  \mV'^\ot \times \mW'^\ot$ is an absolute small weakly  bienriched $\infty$-category and $\rH \in \mP\B\Env(\mJ).$
We also call $(\alpha,\beta,\mJ^\circledast \to \mV'^\ot \times \mW'^\ot, \rH)$ an absolute small $(\alpha,\beta)$-weight over $\mV,\mW$ on $\mJ$.
\end{definition}

\begin{definition}Let $\mV^\ot \to \Ass, \mW^\ot \to \Ass$ small $\infty$-operads.

\begin{enumerate}
\item A left weight over $\mV$ is a weight over $\mV,\emptyset$.
% such that $\mW'^\ot \simeq \emptyset^\circledast$.

\item A right weight over $\mW$ is a weight over $\emptyset,\mW$.
% such that $\mV'^\ot \simeq \emptyset^\circledast$.

\item A $\mV,\mW$-weight is a weight over $\mV,\mW$
such that $\alpha, \beta$ are the identities.
% we identify a weight  $(\alpha, \beta, \mJ^\circledast \to \mV'^\ot \times \mW'^\ot, \rH)$ with the pair $(\mJ^\circledast \to \mV'^\ot \times \mW'^\ot, \rH)$, which we call a $\mV,\mW$-weight on $\mJ.$

\item A left $\mV$-weight is a left weight over $\mV,\mW$ such that $\alpha$ is the identity. %-weight such that $\mW^\ot \simeq \emptyset^\circledast.$

\item A right $\mW$-weight is a right weight over $\mV,\mW$ such that $\beta$ is the identity. %-weight such that $\mW^\ot \simeq \emptyset^\circledast.$

%is a $\kappa$-compact object of $\mP\B\Env(\mM).$
\end{enumerate}
\end{definition}

\begin{remark}\label{labo}
A left weight is a weight $(\alpha: \mV'^\ot \to \mV^\ot,
\beta: \mW'^\ot \to \mW^\ot, \mJ^\circledast \to \mV'^\ot \times \mW'^\ot, \rH \in \mP\B\Env(\mJ))$
such that $\mW^\ot\simeq \emptyset^\ot.$ So also $\mW'^\ot\simeq \emptyset^\ot.$
Consequently, a left weight is equivalently a triple $(\alpha: \mV'^\ot \to \mV^\ot, \mJ^\circledast \to \mV'^\ot, \rH \in \mP\L\Env(\mJ))$.
%In particular, small left weights over $\mV,\mW$ are precisely small left weights over$\mV, \emptyset.$
In particular, we can identify a left weight $(\alpha: \mV'^\ot \to \mV^\ot, \mJ^\circledast \to \mV'^\ot, \rH \in \mP\L\Env(\mJ))$ over $\mV$ with the 
weight $(\alpha: \mV'^\ot \to \mV^\ot, \beta: \emptyset^\ot \to \mW^\ot, \mJ^\circledast \to \mV'^\ot, \rH \in \mP\L\Env(\mJ))$ over $\mV, \mW$.
The same remark holds dually for right weights.	
	
\end{remark}

\begin{definition} Let $\kappa$ be a small regular cardinal.
An absolute small weight $(\alpha,\beta,\mJ^\circledast \to \mV'^\ot \times \mW'^\ot, \rH)$ is $\kappa$-small if $\rH$ belongs to the full subcategory of $\mP\B\Env(\mM)$
generated by $\B\Env(\mM)$ under $\kappa$-small colimits.	
	
\end{definition}

\begin{definition}\label{Dewe}
Let $\mM^\circledast \to \mV^\ot \times \mW^\ot$ be absolute small weakly  bienriched $\infty$-category and $\T=(\alpha,\beta,\mJ^\circledast \to \mV'^\ot \times \mW'^\ot, \rH)$ an absolute small weight over $\mV,\mW$ and $\kappa$ a small regular cardinal. 
\begin{enumerate}
\item We say that $\mM$ admits $\T$-weighted colimits if $\mM$ admits the $\rH$-weighted colimit of every enriched functor $\mJ^\circledast \to \mM^\circledast$ lying over $\alpha, \beta.$

\item Let $\mH$ be a collection of absolute small weights over $\mV,\mW$.
We say that $\mM$ admits $\mH$-weighted colimits if $\mM$ admits $\T$-weighted colimits for every weight $\T \in \mH.$

\item We say that $\mM$ admits $\kappa$-small weighted colimits if it admits $\T$-weighted colimits for every $\kappa$-small weight $\T$ over $\mV,\mW.$

\item We say that $\mM$ admits $\kappa$-small left weighted colimits if it admits $\T$-weighted colimits for every $\kappa$-small left weight $\T$ over $\mV.$

\item We say that $\mM$ admits $\kappa$-small right weighted colimits if it admits $\T$-weighted colimits for every $\kappa$-small right weight $\T$ over $\mW.$

\item We say that $\mM$ admits small weighted colimits if it admits $\T$-weighted colimits for every absolute small weight $\T$ over $\mV,\mW.$

\item We say that $\mM$ admits small left weighted colimits if it admits $\T$-weighted colimits for every small left weight $\T$ over $\mV.$

\item We say that $\mM$ admits small right weighted colimits if it admits $\T$-weighted colimits for every small right weight $\T$ over $\mW.$

%\item Let $\mH$ be a collection of small left weights over $\mV$.We say that $\mM$ admits $\mH$-weighted colimits if $\mM$ admits $\T$-weighted colimits for every left weight $\T \in \mH.\item Let $\mH$ be a collection of small right weights over $\mW$.We say that $\mM$ admits $\mH$-weighted colimits if $\mM$ admits $\T$-weighted colimits for every right weight $\T \in \mH.$

\end{enumerate}

\end{definition}

We apply Definition \ref{Dewe} in particular to the case that $\mH$ is a collection of left weights over $\mV$ or a collection of right weights over $\mW$ viewed as weights over $\mV,\mW$ in the canonical way (Remark \ref{labo}).

\subsection{Weighted colimits for pseudo-enriched $\infty$-categories}

Next we consider weighted colimits for quasi-enriched and pseudo-enriched $\infty$-categories.
To describe quasi-enriched and pseudo-enriched $\infty$-categories on one footage we use the notion of $\rS,\T$-bienriched $\infty$-categories for localization pairs $(\mV^\ot \to \Ass, \rS), (\mW^\ot \to \Ass, \T)$ (Definition \ref{locpa}), which specializes to quasi-enriched
and pseudo-enriched $\infty$-categories and the notion of $\kappa$-enriched $\infty$-categories for any small regular cardinal $\kappa$.

Let $(\mV^\ot \to \Ass, \rS), (\mW^\ot \to \Ass, \T)$ be small localization pairs and $\mM^\circledast \to \mV^\ot \times \mW^\ot$ a small $\rS,\T$-bienriched $\infty$-category.
By Lemma \ref{Lauf} there is an enriched localization $\L : \mP\B\Env(\mM)^\circledast \rightleftarrows \mP\B\Env(\mM)_{\rS,\T}^\circledast$ to the respective $\infty$-category of "presheaves" for $\rS,\T$-enriched $\infty$-categories (Proposition \ref{pseuso}).
We have the following description of weighted colimit in this context:

\begin{lemma}\label{Enristo}%Let $(\mV^\ot \to \Ass, \rS), (\mW^\ot \to \Ass, \T)$ be small localization pairs, $\mM^\circledast \to \mV^\ot \times \mW^\ot$ a small bi-$\rS,\T$-enriched $\infty$-category, % lying over $\alpha, \beta$
%Let $(\mV^\ot \to \Ass, \rS), (\mW^\ot \to \Ass, \T), (\mV'^\ot \to \Ass, \rS'), (\mW'^\ot \to \Ass, \T')$ be small localization pairs and $\F: \mJ^\circledast \to \mM^\circledast$ an enriched functor from a small $\rS',\T'$-bienriched $\infty$-category to a small $\rS,\T$-bienriched $\infty$-category % lying over $\alpha, \beta$
% $\mM^\circledast \to \mV^\ot \times \mW^\ot$ a small bi-$\rS,\T$-enriched $\infty$-category, $\mJ^\circledast \to \mV'^\ot \times \mW'^\ot$ a small bi-$\rS',\T'$-enriched $\infty$-category and that lies over maps of localization pairs $(\mV'^\ot \to \Ass, \rS) \to (\mV^\ot \to \Ass, \rS')$ and $ (\mW'^\ot \to \Ass, \T) \to (\mW^\ot \to \Ass, \T')$ and $\Y \in \mM, \rH \in \mP\B\Env(\mJ)_{\rS',\T'}.$
Let $(\mV^\ot \to \Ass, \rS), (\mW^\ot \to \Ass, \T)$ be small localization pairs and $\mM^\circledast \to \mV^\ot \times \mW^\ot$ a small $\rS,\T$-bienriched $\infty$-category, $\F: \mJ^\circledast \to \mM^\circledast$ an enriched functor, $ \rH \in \mP\B\Env(\mJ), \Y \in \mM$ and $\F_!(\rH) \to \Y $ a morphism in $\mP\B\Env(\mM)$. The following conditions are equivalent:
%The enriched functor $\F_!: \mP\B\Env(\mJ)^\circledast \to \mP\B\Env(\mM)^\circledast$ descends to an enriched functor $\widetilde{\F_!}: \mP\B\Env(\mJ)_{\rS',\T'}^\circledast \to \mP\B\Env(\mM)_{\rS,\T}^\circledast$ (Remark \ref{Functol}).

\begin{enumerate}
\item The morphism $\F_!(\rH) \to \Y $ in $\mP\B\Env(\mM)$
%corresponds to a morphism $\psi': \widetilde{\F_!}(\rH) \to \Y $ in $\mP\B\Env(\mM)_{\rS,\T}$. The morphism $\psi$ 
exhibits $\Y$ as the $\rH$-weighted colimit of $\F$.

\item For every $\Z \in \mM$, $\V_1,...,\V_\n \in \mV, \W_1,...,\W_\m \in \mW$ for $\n,\m \geq 0$ the induced morphism $\L(\F_!(\rH)) \to \Y$ in $\mP\B\Env(\mM)_{\rS,\T}$
%where $\L$ is the left adjoint of the embedding $\mP\B\Env(\mM)_{\rS,\T} \subset \mP\B\Env(\mM),$ 
induces an equivalence:
$$\hspace{10mm}\Mul_\mM(\V_1,...,\V_\n, \Y, \W_1,...,\W_\m; \Z) \to \Mul_{\mP\B\Env(\mM)_{\rS,\T}}(\V_1,...,\V_\n, \L(\F_!(\rH)), \W_1,...,\W_\m; \Z).$$
\end{enumerate}

\hspace{6mm} If $\rH \in \mP\B\Env(\mM)_{\rS,\T}$ conditions (1),(2) are equivalent to the following one: the corresponding morphism $\psi': \rH \to \F^*(\Y) $ in $\mP\B\Env(\mJ)_{\rS,\T}$ induces for any $\Z \in \mM$, $\V_1,...,\V_\n \in \mV, \W_1,...,\W_\m \in \mW$ for $\n,\m \geq 0$ an equivalence:
$$\hspace{11mm}\Mul_{\mM}(\V_1,..., \V_\n, \Y, \W_1,...,\W_\m; \Z) \to \Mul_{\mP\B\Env(\mJ)_{\rS,\T}}(\V_1,...,\V_\n, \rH, \W_1,...,\W_\m;\F^*(\Z)).$$

\end{lemma}

\begin{proof}
The equivalence between (1) and (2) follows by adjointness via the enriched localization $\L : \mP\B\Env(\mM)^\circledast \rightleftarrows \mP\B\Env(\mM)_{\rS,\T}^\circledast$.
The second part of the lemma follows from the enriched adjunction $\L \circ \F_!: \mP\B\Env(\mJ)_{\rS,\T}^\circledast \rightleftarrows \mP\B\Env(\mM)_{\rS,\T}^\circledast: \F^*.$
	
\end{proof}

%Lemma \ref{Enristo} applies to left, right and bi-quasi-enriched, left, right and bipseudo-enriched and left, right and bienriched

\begin{notation}Let $(\mV^\ot \to \Ass, \rS), (\mW^\ot \to \Ass, \T)$ be small localization pairs and $\mM^\circledast \to \mV^\ot \times \mW^\ot$ a small $\rS,\T$-bienriched $\infty$-category.
%Let $\mM^\circledast \to \mV^\ot \times \mW^\ot$ be an absolute small weakly  bienriched $\infty$-category.
Let $\mM_{\rS,\T}^\circledast \subset \mP\B\Env(\mM)_{\rS,\T}^\circledast $ be the full weakly bienriched subcategory spanned by the objects $\T \in \mP\B\Env(\mM)_{\rS,\T}$ such that there is a morphism $\T \to \T'$ in $\mP\B\Env(\mM)_{\rS,\T}$, where $\T' \in \mM$, such that for every $\Z \in \mM$, $\V_1,...,\V_\n \in \mV, \W_1,...,\W_\m \in \mW$ for $\n,\m \geq 0$ the following map is an equivalence:
\begin{equation}\label{eqaso}
\Mul_{\mP\B\Env(\mM)_{\rS,\T}}(\V_1,...,\V_\n, \T', \W_1,...,\W_\m; \Z) \to \Mul_{\mP\B\Env(\mM)_{\rS,\T}}(\V_1,...,\V_\n, \T, \W_1,...,\W_\m; \Z).\end{equation}
\end{notation}
\begin{remark}
Let $\bar{\mM}^\circledast \subset \mP\B\Env(\mM)_{\rS,\T}^\circledast $ be the full weakly bienriched subcategory spanned by $\mM.$  The embedding $\bar{\mM}^\circledast \subset \mP\B\Env(\mM)_{\rS,\T}^\circledast$ lands in 
$\mM_{\rS,\T}^\circledast$ (taking the identity of $\T$).
\end{remark}

\begin{lemma}
Let $(\mV^\ot \to \Ass, \rS), (\mW^\ot \to \Ass, \T)$ be small localization pairs and $\mM^\circledast \to \mV^\ot \times \mW^\ot$ a small $\rS,\T$-bienriched $\infty$-category.	
The embedding $\bar{\mM}^\circledast \subset\mM_{\rS,\T}^\circledast$ admits a $\rS^{-1}\mP\Env(\mV), \T^{-1}\mP\Env(\mW)$-enriched left adjoint $\phi$.
\end{lemma}

\begin{proof}Let $\T \in \mP\B\Env(\mM)_{\rS,\T}$. Then there is a morphism $\T \to \T'$ in $\mP\B\Env(\mM)_{\rS,\T}$, where $\T' \in \mM$, such that for every $\Z \in \mM$, $\V_1,...,\V_\n \in \mV, \W_1,...,\W_\m \in \mW$ for $\n,\m \geq 0$ the map (\ref{eqaso}) is an equivalence. Thus for every $\V \in \rS^{-1}\mP\Env(\mV), \W \in \T^{-1}\mP\Env(\mW)$ the following map is an equivalence since $\rS^{-1}\mP\Env(\mV)$ is generated under small colimits by tensor products of objects of $\mV:$
\begin{equation}\label{eqaso2}
\mP\B\Env(\mM)_{\rS,\T}(\V \ot \T' \ot \W, \Z) \to \mP\B\Env(\mM)_{\rS,\T}(\V \ot \T \ot \W, \Z).
\end{equation}	
	
\end{proof}

\begin{remark}\label{cohfun}Let $(\mV^\ot \to \Ass, \rS), (\mW^\ot \to \Ass, \T)$ be small localization pairs and $\mM^\circledast \to \mV^\ot \times \mW^\ot$ a small $\rS,\T$-bienriched $\infty$-category.
An enriched functor $\F: \mJ^\circledast \to \mM^\circledast$ admits a $\rH$-weighted colimit for a weight $\rH \in \mP\B\Env(\mJ)$ if and only if
the functor $ \mP\B\Env(\mJ) \xrightarrow{\F_!} \mP\B\Env(\mM) \xrightarrow{\L} \mP\B\Env(\mM)_{\rS,\T}$ lands in $\mM_{\rS,\T}.$
In this case there is a canonical equivalence $\colim^\rH(\F)\simeq \phi(\L(\F_!(\rH))).$

%$\mJ^\circledast \to \mV^\ot \times \mW^\ot, \mM^\circledast \to \mV^\ot \times \mW^\ot$ small $\rS,\T$-bienriched $\infty$-categories.
%Moreover, if $\Theta^\circledast \subset \Enr\Fun(\mJ,\mM) \times \mP\B\Env(\mJ)^\circledast$ is the full weakly bienriched subcategory spanned by those pairs $(\F, \rH)$ such that $\mM$ %^\circledast \to \mV^\ot \times \mW^\ot$ admits the $\rH$-weighted colimit of $\F$, the enriched functor$$ \Theta^\circledast \subset \Enr\Fun(\mJ,\mM) \times \mP\B\Env(\mJ)^\circledast \to \mP\B\Env(\mM)^\circledast, (\F,\rH) \mapsto \F_!(\rH)$$lands in $\mM'^\circledast$ and the enriched functor 
%$ \Theta^\circledast \subset \Enr\Fun(\mJ,\mM) \times \mP\B\Env(\mJ)^\circledast \to \mM'^\circledast \xrightarrow{\L}\bar{\mM}^\circledast$ sends $(\F,\rH) $ to $\colim^\rH(\F).$

\end{remark}

%The definition of weighted colimits simplifies for enriched %and pseud-enriched $\infty$-categories, where we use Notation \ref{enrprrr2}:

%The latter remark specializes to enriched $\infty$-categories: %left enriched, right enriched and bienriched $\infty$-categories, where we explain the case of left enrichment: %An important case is the following one:

%By the next remark the theory of weighted colimits of enriched functors between $\infty$-categories left enriched in a presentably monoidal $\infty$-category is a special case of the theory of weighted colimits of enriched functors between small bienriched $\infty$-categories:

%If we are in the situation of Definition \ref{weight} and $\mV^\ot=\mV'^\ot, \mW^\ot= \mW'^\ot$ and $\mM$ admits the $\rH$-weighted colimit of every $\mV,\mW$-enriched functor $\mJ^\circledast \to \mM^\circledast$, we say that $\mM$ admits $\rH$-weighted colimits. For the general situation we introduce the following terminology:

\begin{proposition}\label{weifun} Let $(\mV^\ot \to \Ass, \rS), (\mW^\ot \to \Ass, \T)$ be small localization pairs and $\mJ^\circledast \to \mV^\ot \times \mW^\ot, \mM^\circledast \to \mV^\ot \times \mW^\ot$ small $\rS,\T$-bienriched $\infty$-categories.
%Let $\mJ^\circledast \to \mV^\ot \times \mW^\ot, \mM^\circledast \to \mV^\ot \times \mW^\ot$ be absolute small weakly  bienriched $\infty$-categories.

\begin{enumerate}
\item Let $\F: \mJ^\circledast \to \mM^\circledast$ be a $\mV,\mW$-enriched functor such that
$\mM$ admits the $\rH$-weighted colimit of $\F$ for every $\rH \in \mP\B\Env(\mJ)_{\rS,\T}.$
The enriched functor $\F^*_{\mid \bar{\mM}^\circledast}: \bar{\mM}^\circledast \to \mP\B\Env(\mJ)_{\rS,\T}^\circledast$ admits a $\rS^{-1}\mP\Env(\mV), \T^{-1}\mP\Env(\mW)$-enriched left adjoint that sends $\rH$ to $\colim^\rH(\F)$ and extends $\F.$

\item Assume that $\mM$ admits the $\rH$-weighted colimit of every $\mV,\mW$-enriched functor $\F: \mJ^\circledast \to \mM^\circledast$ for every $\rH \in \mP\B\Env(\mJ)_{\rS,\T}.$
The induced functor $$\Enr\Fun_{\rS^{-1}\mP\Env(\mV),\T^{-1}\mP\Env(\mW)}(\mP\B\Env(\mJ)_{\rS,\T},\bar{\mM}) \to \Enr\Fun_{\mV,\mW}(\mJ,\mM)$$ admits a fully faithful left adjoint that lands in $\LinFun^\L_{\rS^{-1}\mP\Env(\mV),\T^{-1}\mP\Env(\mW)}(\mP\B\Env(\mJ)_{\rS,\T},\bar{\mM}).$
In particular, the following induced functor is an equivalence: $$\LinFun^\L_{\rS^{-1}\mP\Env(\mV),\T^{-1}\mP\Env(\mW)}(\mP\B\Env(\mJ)_{\rS,\T},\bar{\mM}) \to \Enr\Fun_{\mV,\mW}(\mJ,\mM)$$

\end{enumerate}	

\end{proposition}

\begin{proof}
(1): The $\rS^{-1}\mP\Env(\mV),\T^{-1}\mP\Env(\mW)$-enriched adjunction $\F_!: \mP\B\Env(\mJ)_{\rS,\T}^\circledast \rightleftarrows \mP\B\Env(\mM)_{\rS,\T}^\circledast : \F^*$ restricts to a $\rS^{-1}\mP\Env(\mV),\T^{-1}\mP\Env(\mW)$-enriched adjunction $\F_!: \mP\B\Env(\mJ)_{\rS,\T}^\circledast \rightleftarrows \mM_{\rS,\T}^\circledast : \F^*_{\mid \mM_{\rS,\T}^\circledast}.$
Composing the latter with the $\rS^{-1}\mP\Env(\mV),\T^{-1}\mP\Env(\mW)$-enriched localization $\L: \mM_{\rS,\T}^\circledast \rightleftarrows \mM^\circledast$ we obtain a $\rS^{-1}\mP\Env(\mV),\T^{-1}\mP\Env(\mW)$-enriched adjunction $\L \circ \F_!: \mP\B\Env(\mJ)_{\rS,\T}^\circledast \rightleftarrows \mM^\circledast : \F^*_{\mid \mM^\circledast}.$
The restriction of the enriched functor $\L \circ \F_!: \mP\B\Env(\mJ)_{\rS,\T}^\circledast \to \mM_{\rS,\T}^\circledast \to \mM^\circledast$ to $\mJ^\circledast$ factors as $\mJ^\circledast \to \mM^\circledast \subset \mM_{\rS,\T}^\circledast \to \mM^\circledast$ and so is $\F.$

(2): By Proposition \ref{pseuso} %Corollary \ref{envvcor}
the induced functor $$\Enr\Fun_{\rS^{-1}\mP\Env(\mV),\T^{-1}\mP\Env(\mW)}(\mP\B\Env(\mJ)_{\rS,\T},\mP\B\Env(\mM)_{\rS,\T}) \to \Enr\Fun_{\mV,\mW}(\mJ,\mP\B\Env(\mM)_{\rS,\T})$$ admits a fully faithful left adjoint, which lands in the full subcategory $$\LinFun^\L_{\rS^{-1}\mP\Env(\mV),\T^{-1}\mP\Env(\mW)}(\mP\B\Env(\mJ)_{\rS,\T},\mP\B\Env(\mM)_{\rS,\T}).$$
Let $$\Enr\Fun'_{\rS^{-1}\mP\Env(\mV),\T^{-1}\mP\Env(\mW)}(\mP\B\Env(\mJ)_{\rS,\T},\mM_{\rS,\T}) \subset$$$$ \Enr\Fun_{\rS^{-1}\mP\Env(\mV),\T^{-1}\mP\Env(\mW)}(\mP\B\Env(\mJ)_{\rS,\T},\mM_{\rS,\T})$$ be the full subcategory of enriched functors sending $\mJ$ to $\mM.$ In particular, the induced functor $$\Enr\Fun'_{\rS^{-1}\mP\Env(\mV),\T^{-1}\mP\Env(\mW)}(\mP\B\Env(\mJ)_{\rS,\T},\mM_{\rS,\T}) \to \Enr\Fun_{\mV,\mW}(\mJ,\mM)$$ admits a fully faithful left adjoint that lands in the full subcategory $$\LinFun^\L_{\rS^{-1}\mP\Env(\mV),\T^{-1}\mP\Env(\mW)}(\mP\B\Env(\mJ)_{\rS,\T},\mM_{\rS,\T}).$$
So the result follows from the fact that the equivalence $$\Enr\Fun_{\rS^{-1}\mP\Env(\mV),\T^{-1}\mP\Env(\mW)}(\mP\B\Env(\mJ)_{\rS,\T},\mM) \simeq \Enr\Fun'_{\rS^{-1}\mP\Env(\mV),\T^{-1}\mP\Env(\mW)}(\mP\B\Env(\mJ)_{\rS,\T},\mM_{\rS,\T})$$ is inverse to the following functor post-composing with $\L:$ $$ \Enr\Fun'_{\rS^{-1}\mP\Env(\mV),\T^{-1}\mP\Env(\mW)}(\mP\B\Env(\mJ)_{\rS,\T},\mM_{\rS,\T}) \to \Enr\Fun_{\rS^{-1}\mP\Env(\mV),\T^{-1}\mP\Env(\mW)}(\mP\B\Env(\mJ)_{\rS,\T},\mM)$$

\end{proof}

\begin{corollary}Let $(\mV^\ot \to \Ass, \rS), (\mW^\ot \to \Ass, \T)$ be small localization pairs and $\mJ^\circledast \to \mV^\ot \times \mW^\ot$ a small $\rS,\T$-bienriched $\infty$-category, $\rH \in \mP\B\Env(\mJ)_{\rS,\T}$ and $\iota: \mJ^\circledast \subset  \mP\B\Env(\mJ)^\circledast_{\rS,\T}$ the canonical enriched embedding.
There is a canonical equivalence $ \colim^\rH \iota \simeq \rH.$

\end{corollary}

\subsection{Weighted colimits for enriched $\infty$-categories}Next we consider weighted colimits for enriched $\infty$-categories, which is of special relevance for us.
In the following we use Notation \ref{enrprrr}.
Let $\mM^\circledast \to \mV^\ot \times \mW^\ot$ be an $\infty$-category left enriched in a presentably monoidal $\infty$-category and $\L : \widehat{\mP}\B\Env(\mM)^\circledast \rightleftarrows \widehat{\mP}\B\Env(\mM)_{\L\Enr}^\circledast$ the enriched localization
of Lemma \ref{Lauf}. For every $\mV,\mW$-enriched functor $\F: \mJ^\circledast \to \mM^\circledast$ the $\mV,\mW$-enriched functor $$\widehat{\mP}\B\Env(\mJ)^\circledast_{\L\Enr} \subset \widehat{\mP}\B\Env(\mJ)^\circledast \xrightarrow{\F_!} \widehat{\mP}\B\Env(\mM)^\circledast \xrightarrow{\L} \widehat{\mP}\B\Env(\mM)_{\L\Enr}^\circledast $$ restricts to a $\mV,\mW$-enriched functor $\mP\B\Env(\mJ)_{\L\Enr}^\circledast \to \mP\B\Env(\mM)_{\L\Enr}^\circledast$ that admits a $\mV,\mW$-enriched right adjoint $\F^*.$
Specializing to left enriched $\infty$-categories Lemma \ref{Enristo} gives the following corollary:

\begin{corollary}\label{obsto}
Let $\mM^\circledast \to \mV^\ot \times \mW^\ot$ be an $\infty$-category left enriched in a presentably monoidal $\infty$-category,
%$\alpha: \mV'^\ot \to \mV^\ot, \beta: \mW^\ot \to \mW'^\ot$ be monoidal functors preserving small colimits between monoidal $\infty$-categories compatible with small colimits.
$\F: \mJ^\circledast \to \mM^\circledast$ a $\mV,\mW$-enriched functor, $\rH \in\mP\B\Env(\mJ)_{\L\Enr}, \Y \in \mM$ and $\F_!(\rH) \to \Y $ a morphism in $\widehat{\mP}\B\Env(\mM)$. The following conditions are equivalent:
%The enriched functor $\F_!: \mP\B\Env(\mJ)^\circledast \to \mP\B\Env(\mM)^\circledast$ descends to an enriched functor $\widetilde{\F_!}: \mP\B\Env(\mJ)_{\rS',\T'}^\circledast \to \mP\B\Env(\mM)_{\rS,\T}^\circledast$ (Remark \ref{Functol}).

\begin{enumerate}
\item The morphism $\F_!(\rH) \to \Y $ in $\widehat{\mP}\B\Env(\mM)$
%corresponds to a morphism $\psi': \widetilde{\F_!}(\rH) \to \Y $ in $\mP\B\Env(\mM)_{\rS,\T}$. The morphism $\psi$ 
exhibits $\Y$ as the $\rH$-weighted colimit of $\F$.
	
\item The induced morphism $\L(\F_!(\rH)) \to \Y$ in $\mP\B\Env(\mM)_{\L\Enr}$ induces for every $\Z \in \mM$ an equivalence:
$$\hspace{10mm}\L\Mor_\mM(\Y, \Z) \to \L\Mor_{\mP\B\Env(\mM)_{\L\Enr}}(\L(\F_!(\rH)), \Z).$$

\item The corresponding morphism $\psi': \rH \to \F^*(\Y) $ in $\mP\B\Env(\mJ)_{\L\Enr}$ induces for every $\Z \in \mM$ an equivalence:
$$\hspace{11mm}\L\Mor_{\mM}(\Y, \Z) \to \L\Mor_{\mP\B\Env(\mJ)_{\L\Enr}}(\rH,\F^*(\Z)).$$
\end{enumerate}

%between left enriched $\infty$-categories lying over $\alpha, \beta.$ %from the left over $\alpha$ and from the right over a map of $\infty$-operads $ \beta: \mW'^\ot \to \mW^\ot $.
% and lying from the left over a monoidal functor $\alpha: \mV'^\ot \to \mV^\ot$ preserving small colimits between monoidal $\infty$-categories compatible with small colimits.
%$ \beta: \mW'^\ot \to \mW^\ot $ a map of $\infty$-operads and $\F: \mJ^\circledast \to \mM^\circledast$ an enriched functor between left enriched $\infty$-categories lying over $\alpha, \beta$ and $\Y\in\mM$ and 
	
%If $\mV$ is presentable and $\mW^\ot \to \Ass$ and $ \mM$ are small, then %ly monoidal $\infty$-category $\mP\B\Env(\mM)_{\L\Enr}^\circledast \to \mV^\ot \times \mW^\ot$ is a left enriched $\infty$-category by Proposition \ref{Corfols}.Hence $\psi$ exhibits $\Y$ as the $\rH$-weighted colimit of $\F$ if and only if for every $\Z \in \mM$ the morphism $\psi'$ induces an equivalence:
%\begin{equation*}\label{qwerd}\L\Mor_\mM(\Y, \Z) \to \L\Mor_{\mP\B\Env(\mM)_{\L\Enr}}(\L(\F_!(\rH)), \Z). \end{equation*}
\end{corollary}

\begin{remark}\label{sosos}
	
%Let $\mJ^\circledast \to \mV^\ot \times \mW^\ot$ be an absolute small weakly  bienriched $\infty$-category, $ \mM^\circledast \to \mV^\ot \times \mW^\ot$ a presentably  bi-tensored $\infty$-category, $\rH$ a weight over $\mV,\mW$ and $\F: \mJ^\circledast \to \mM^\circledast$ an enriched functor, which by Corollary \ref{envvcor} (2) uniquely extends to a left adjoint linear functor$\bar{\F}: \mP\B\Env(\mJ)^\circledast \to \mP\B\Env(\mM)^\circledast$.By Corollary \ref{cohfun} the functor $\bar{\F}$ sends $\rH$ to the $\rH$-weighted colimit of $\F.$
Hinich \cite[§ 6]{hinich2021colimits} defines a notion of weighted colimit for his model of enriched $\infty$-categories \cite{HINICH2020107129}, which is equivalent to our model by 
\cite{MR4185309}, \cite[Theorem 6.7.]{HEINE2023108941} and \cite{heine2024equivalence}. The description of the weighted colimit given by Corollary \ref{obsto} implies that our notion of weighted colimit for left enriched $\infty$-categories agrees with Hinich's definition.
This definition of weighted colimits was also introduced and used in \cite[Definition A.34.]{heine2021real}.

\end{remark}

One proves the following proposition like Proposition \ref{weifun}:
\begin{proposition}\label{weifunt} Let $\mV^\ot \to \Ass$ be a presentably monoidal $\infty$-category, $\mW^\ot \to \Ass$ a small $\infty$-operad and $\mJ^\circledast \to \mV^\ot \times \mW^\ot, \mM^\circledast \to \mV^\ot \times \mW^\ot$ small left enriched $\infty$-categories.
\begin{enumerate}
\item Let $\F: \mJ^\circledast \to \mM^\circledast$ be a $\mV,\mW$-enriched functor such that
$\mM$ admits the $\rH$-weighted colimit of $\F$ for every $\rH \in \mP\B\Env(\mJ)_{\L\Enr}.$
The $\mV,\mW$-enriched functor $\F^*_{\mid \mM^\circledast}: \mM^\circledast \to \mP\B\Env(\mJ)_{\L\Enr}^\circledast$ admits a $\mV,\mW$-enriched left adjoint that sends $\rH$ to $\colim^\rH(\F)$ and extends $\F.$
		
\item Assume that $\mM$ admits the $\rH$-weighted colimit of every $\mV,\mW$-enriched functor $\F: \mJ^\circledast \to \mM^\circledast$ for every $\rH \in \mP\B\Env(\mJ)_{\L\Enr}.$
The induced functor $$\Enr\Fun_{\mV,\mW}(\mP\B\Env(\mJ)_{\L\Enr},\mM) \to \Enr\Fun_{\mV,\mW}(\mJ,\mM)$$ admits a fully faithful left adjoint that lands in $\LinFun^\L_{\mV,\mW}(\mP\B\Env(\mJ)_{\L\Enr},\mM).$
In particular, the following induced functor is an equivalence: $$\LinFun^\L_{\mV,\mW}(\mP\B\Env(\mJ)_{\L\Enr}, \mM) \to \Enr\Fun_{\mV,\mW}(\mJ,\mM)$$
		
\end{enumerate}	
	
\end{proposition}

We use the following enriched variants of weights, which put us in the situation of Corollary \ref{obsto} via Lemma \ref{sasa} and Remark \ref{rest}:

\begin{definition}Let $\mV^\ot \to \Ass, \mW^\ot \to \Ass$ be $\infty$-operads.
	
\begin{enumerate}
\item A small left enriched weight over $\mV,\mW$ is a weight $(\alpha: \mV'^\ot \to \mV^\ot,
\beta: \mW'^\ot \to \mW^\ot, \mJ^\circledast \to \mV'^\ot \times \mW'^\ot, \rH)$ over $\mV,\mW$ such that $\mV^\ot \to \Ass, \mV'^\ot \to \Ass$ are presentably monoidal $\infty$-categories, $\mW^\ot \to \Ass, \mW'^\ot \to \Ass$ are small $\infty$-operads,
$\alpha$ is a left adjoint monoidal functor, $ \mJ^\circledast \to \mV'^\ot \times \mW'^\ot$ is a small left enriched $\infty$-category and $\rH \in \mP\B\Env(\mM)_{\L\Enr}.$

\item A small right enriched weight over $\mV,\mW$ is a weight $(\alpha: \mV'^\ot \to \mV^\ot,
\beta: \mW'^\ot \to \mW^\ot, \mJ^\circledast \to \mV'^\ot \times \mW'^\ot, \rH)$ over $\mV,\mW$ such that $\mW^\ot \to \Ass, \mW'^\ot \to \Ass$ are presentably monoidal $\infty$-categories, $\mV^\ot \to \Ass, \mV'^\ot \to \Ass$ are small $\infty$-operads,
$\beta$ is a left adjoint monoidal functor, $ \mJ^\circledast \to \mV'^\ot \times \mW'^\ot$ is a small right enriched $\infty$-category and $\rH \in \mP\B\Env(\mM)_{\R\Enr}.$

\item A small enriched weight over $\mV,\mW$ is a weight $(\alpha: \mV'^\ot \to \mV^\ot, \beta: \mW'^\ot \to \mW^\ot, \mJ^\circledast \to \mV'^\ot \times \mW'^\ot, \rH)$ over $\mV,\mW$ such that $\mV^\ot \to \Ass, \mV'^\ot \to \Ass, \mW^\ot \to \Ass, \mW'^\ot \to \Ass$ are presentably monoidal $\infty$-categories, $\alpha, \beta$ are left adjoint monoidal functors, $ \mJ^\circledast \to \mV'^\ot \times \mW'^\ot$ is a small bienriched $\infty$-category and $\rH \in \mP\B\Env(\mM)_{\B\Enr}.$

\end{enumerate}
\end{definition}

%\begin{definition}\label{Dewe}Let $\mM^\circledast \to \mV^\ot \times \mW^\ot$ be small bienriched $\infty$-category and $\T=(\alpha,\beta,\mJ^\circledast \to \mV'^\ot \times \mW'^\ot, \rH)$ a small enriched weight over $\mV,\mW$.
%\begin{enumerate}
%\item We say that $\mM$ admits small left enriched weighted colimits if it admits $\T$-weighted colimits for every small left enriched weight $\T$ over $\mV.$

%\item We say that $\mM$ admits small right enriched weighted colimits if it admits $\T$-weighted colimits for every small right enriched weight $\T$ over $\mW.$

%\item Let $\mH$ be a collection of small left weights over $\mV$.We say that $\mM$ admits $\mH$-weighted colimits if $\mM$ admits $\T$-weighted colimits for every left weight $\T \in \mH.\item Let $\mH$ be a collection of small right weights over $\mW$.We say that $\mM$ admits $\mH$-weighted colimits if $\mM$ admits $\T$-weighted colimits for every right weight $\T \in \mH.$\end{enumerate}\end{definition}

%We apply Definition \ref{Dewe} in particular to the case that $\mH$ is a collection of left enriched weights over $\mV$ or a collection of right enriched weights over $\mW$ viewed as enriched weights over $\mV,\mW$ in the canonical way (Remark \ref{labo}).

\subsection{Conical colimits and tensors}

Next we consider our first example of weighted colimits, which are conical colimits and left and right tensors.

%We consider the following types of examples of weighted colimits that give rise to dual examples of weighted limits:

%\begin{itemize}
%\item conical colimits
%\item left and right tensors
%\item enriched coends
%\item lax and oplax colimits\end{itemize}
%We define conical colimits:

\begin{definition}Let $\mM^\circledast \to \mV^\ot \times \mW^\ot$ be a weakly bienriched $\infty$-category.

\begin{enumerate}
\item A functor $\F: \K^\triangleright \to \mM$ is a conical colimit diagram %or that the induced morphism
if for every $\V_1,...,\V_\n \in \mV, \W_1,...,\W_\m \in \mW$ for $\n,\m \geq 0$
and $\Y \in \mM$ the presheaf $\Mul_\mM(\V_1,...,\V_\n,-, \W_1,...,\W_\m;\Y)$ on $\mM$ sends $\F^\op: (\K^\op)^\triangleleft \to \mM^\op$ to a limit diagram.
An $\infty$-category $\mM$ admits $\K$-indexed conical colimits if $\mM$ admits the conical colimit of every functor starting at $\K.$

\item A functor $\F: \K^\triangleleft \to \mM$ is a conical limit diagram %or that the induced morphism
if for every $\V_1,...,\V_\n \in \mV, \W_1,...,\W_\m \in \mW$ for $\n,\m \geq 0$
and $\X \in \mM$ the presheaf $\Mul_\mM(\V_1,...,\V_\n,\X, \W_1,...,\W_\m;-)$
on $\mM$ sends $\F$ to a limit diagram.
An $\infty$-category $\mM$ admits $\K$-indexed conical limits if $\mM$ admits the conical limit of every functor starting at $\K.$

\end{enumerate}
%The conical colimit of $\F$ is the colimit of $\F$ weighted at the final presheaf $*$on $\mJ$ seen as an object of $\mP(\mJ) \subset \mP\B\Env(\mJ).$
\end{definition}

\begin{remark}
Stefanich \cite[Definition 4.3.2.]{stefanich2020presentable} introduces conical colimits for Gepner-Haugseng's model of enriched $\infty$-categories.	
	
\end{remark}

\begin{remark}\label{laulo}
Let $\mM^\circledast \to \mV^\ot \times \mW^\ot$ be a weakly bienriched $\infty$-category. 

\begin{enumerate}
\item Every conical colimit diagram is a colimit diagram.

\item If $\mM^\circledast \to \mV^\ot \times \mW^\ot$ admits left and right cotensors,
every colimit diagram is conical. % colimit diagram.

\item Assume that $\mM^\circledast \to \mV^\ot \times \mW^\ot$ admits left tensors
and let $\G: \K \to \mM$ be a functor.
The colimit of $\G$ is conical if and only if forming left tensors preserves the colimit of $\G$ and for every $\W_1,...,\W_\m \in \mW$ for $\m \geq 0$
and $\Y \in \mM$ the presheaf $\Mul_\mM(-, \W_1,...,\W_\m;\Y)$ on $\mM$ preserves the limit of $\G^\op.$

\item Assume that $\mM^\circledast \to \mV^\ot \times \mW^\ot$ admits left and right tensors and let $\G: \K \to \mM$ be a functor.
The colimit of $\G$ is conical if and only if forming left and right tensors preserves the colimit of $\G$.
\end{enumerate}

\end{remark}

%\begin{example}Let $\mK \subset \Cat_\infty$ be a full subcategory.A left tensored, right tensored, bitensored $\infty$-category is compatible with $\mK$-indexed colimits if and only if it admits all $\mK$-indexed conical colimits.\end{example}

%\begin{example}A monoidal $\infty$-category is compatible with $\mK$-indexed colimits if and only if it is compatible with $\mK$-indexed colimits as an $\infty$-operad.\end{example}

\begin{definition}Let $\mK \subset \Cat_\infty$ be a full subcategory.
% and $\phi: \mM^\circledast \to \mV^\ot \times \mW^\ot$ a weakly bienriched $\infty$-category.
%We say that $\phi$ is compatible with $\mK$-indexed colimits if the $\infty$-operad $\mV^\ot \to \Ass$ is compatible with $\mK$-indexed colimits, $\mM$ admits $\mK$-indexed colimits and for every $\X, \Y \in \mM$, $\V_1,...,\V_\n \in \mV, \W_1,...,\W_\m \in \mW$ for $\n,\m \geq0$ and $0 \leq \bi \leq \n$ the functors $$\Mul_\mM(\V_1,...,\V_\n,-,\W_1,...,\W_\m;\Y): \mM^\op \to \mS,$$$$ \Mul_\mM(\V_1,..,\V_\bi,-, \V_{\bi+1},...,\V_\n,\X,\W_1,...,\W_\m;\Y): \mV^\op \to \mS $$ preserve small limits.
%Similarly, we define when a weakly bienriched $\infty$-category $\phi: \mM^\circledast \to \mV^\ot \times \mW^\ot$ exhibits $\mM$ as weakly right $\mW$-enriched compatible with small colimits.
A weakly bienriched $\infty$-category $\phi: \mM^\circledast \to \mV^\ot \times \mW^\ot$ is compatible with $\mK$-indexed colimits if $\phi$ admits $\mK$-indexed conical colimits and for every $\X, \Y \in \mM$, $\V_1,...,\V_\n \in \mV, \W_1,...,\W_\m \in \mW$ for $\n,\m \geq0$ and $0 \leq \bi \leq \n, 0 \leq \bj \leq \m$ the presheaves %$$\Mul_\mM(\V_1,...,\V_\n,-,\W_1,...,\W_\m;\Y): \mM^\op \to \mS,$$
$$  \Mul_\mM(\V_1,..,\V_\bi,-, \V_{\bi+1},...,\V_\n,\X,\W_1,...,\W_\m;\Y)$$$$ \Mul_\mM(\V_1,...,\V_\n,\X,\W_1,..,\W_\bj,-, \W_{\bj+1},...,\W_\m;\Y)$$ on $\mV,\mW$, respectively, preserve $\mK$-indexed limits.

%A weakly bienriched $\infty$-category $\phi: \mM^\circledast \to \mV^\ot \times \mW^\ot$ is compatible with small colimits if the $\infty$-operads $\mV^\ot \to \Ass, \mW^\ot \to \Ass$ are compatible with small colimits, $\mM$ admits small colimitsand for every $\X, \Y \in \mM$, $\V_1,...,\V_\n \in \mV, \W_1,...,\W_\m \in \mW$ for $\n,\m \geq0$ and $0 \leq \bi \leq \n, 0 \leq \bj \leq \m$ the functors $$\Mul_\mM(\V_1,...,\V_\n,-,\W_1,...,\W_\m;\Y): \mM^\op \to \mS,$$$$  \Mul_\mM(\V_1,..,\V_\bi,-, \V_{\bi+1},...,\V_\n,\X,\W_1,...,\W_\m;\Y): \mV^\op \to \mS,$$$$ \Mul_\mM(\V_1,...,\V_\n,\X,\W_1,..,\W_\bj,-, \W_{\bj+1},...,\W_\m;\Y): \mW^\op \to \mS$$ preserve small limits.
\end{definition}

\begin{notation}
Let $\K$ be a small $\infty$-category and $\mV^\ot \to \Ass, \mW^\ot \to \Ass$ small $\infty$-operads.
By Corollary \ref{leiik} the equivalence $\K_{\mV,\mW}^\circledast \simeq *_\mV^\circledast \times\K\times *_\mW^\circledast$ extends to an equivalence \begin{equation}\label{ttt}
\mP\B\Env(\K_{\mV,\mW}) \simeq \mP\Env(\mV) \otimes \mP(\K) \otimes \mP\Env(\mW).\end{equation}
For every $\V \in \mP\Env(\mV),\W \in \mP\Env(\mW)$ we define the following weight $$\langle\V, \K, \W\rangle \in \mP\B\Env(\K_{\mV,\mW}) $$ corresponding to the image of $(\V,*,\W) \in \mP\Env(\mV) \times \mP(\K) \times \mP\Env(\mW) $ in $ \mP\Env(\mV) \otimes \mP(\K) \otimes \mP\Env(\mW).$

\end{notation}

\begin{remark}
Equivalence (\ref{ttt}) restricts to equivalences
$$\mP\L\Env(\K_{\mV,\mW}) \simeq \mP\Env(\mV) \otimes \mP(\K) \ot \mP(*_\mW),$$$$
\mP\R\Env(\K_{\mV,\mW}) \simeq  \mP(*_\mV) \ot \mP(\K) \ot \mP\Env(\mW),$$$$
\mP(\K_{\mV,\mW}) \simeq \mP(*_\mV) \ot \mP(\K) \ot \mP(*_\mW).$$	

Thus $\langle\V,*,\tu_{\mP\Env(\mW)}\rangle$ is a left weight and $\langle \tu_{\mP\Env(\mV)},*, \W\rangle$ is a right weight.
	
\end{remark}

\begin{definition}
Let $\K$ be a small $\infty$-category and $\mV^\ot \to \Ass, \mW^\ot \to \Ass$ small $\infty$-operads.
The trivial weight on $\K$ is the weight $\langle \tu_{\mP\Env(\mV)},\K, \tu_{\mP\Env(\mW)}\rangle.$
	
\end{definition}

\begin{lemma}\label{exorr}
Let $\mM^\circledast \to \mV^\ot \times \mW^\ot$ be an absolute small weakly  bienriched $\infty$-category, $\K$ an $\infty$-category, $\V \in \mV,\W \in \mW$ and $\F: \K_{\mV, \mW}^\circledast \to \mM^\circledast$ a $\mV,\mW$-enriched functor and $\X \in \mM$ corresponding to a $\mV,\mW$-enriched functor $ *_{\mV, \mW}^\circledast \to \mM^\circledast$.
% sending the unique object to an object $\X \in \mM.$
\begin{enumerate}
\item The colimit of $\X$ weighted at the weight $\langle\V,*,\W\rangle$ is the bitensor $\V\ot\X \ot \W$. 	
\item The colimit of $\F$ weighted at the left weight $\langle\V,*,\tu_{\mP\Env(\mW)}\rangle$ is the left tensor $\V\ot \X$. 
\item The colimit of $\F$ weighted at the right weight $\langle \tu_{\mP\Env(\mV)},*, \W\rangle$
is the right tensor $\X\ot \W$. 
\item The colimit of $\F$ weighted at the trivial weight on $\K$ is the conical colimit of $\F: \K \to \mM$.

%\item The limit of $\F$ weighted at the weight $\langle\V,\K,\W\rangle$ is the bicotensor ${^\V\colim}(\F)^\W$. \item The limit of $\F$ weighted at the weight $\langle\V,\K,\tu_{\mP\Env(\mW)}\rangle$ is the left cotensor ${^{\V}\colim(\F)}.$\item The limit of $\F$ weighted at the weight $\langle \tu_{\mP\Env(\mV)},\K,\W\rangle$ is the right cotensor $\colim(\F)^{\W}.$\item The limit of $\F$ weighted at $\langle\tu_{\mP\Env(\mV)},\K, \tu_{\mP\Env(\mW)}\rangle$ is the conical limit of $\F: \K \to \mM$.

\end{enumerate}
\end{lemma}

\begin{proof}We prove (1). The proofs of (2), (3), (4) are similar.
%(5)-(8) are dual. 
The induced left adjoint $\mP\Env(\mV), \mP\Env(\mW)$-linear functor $\F_!: (\mP\Env(\mV) \otimes \mP(\K) \ot \mP\Env(\mW))^\circledast \simeq \mP\B\Env(\K_{\mV, \mW})^\circledast \to \mP\B\Env(\mM)^\circledast $
sends $\langle\V,\K,\W\rangle$ to $\V \ot \F_!(*) \ot \W$. % using linearity.
The final presheaf $*$ on $\K$ is the colimit of the Yoneda-embedding of $\K$
so that $\F_!(*)$ is the colimit of $\F.$
%the tensor unit $\tu_{\Env(\mV)}, \tu_{\Env(\mW)}$ to $\X$ and so sends $\V,\W$ to $ \V \ot \X \ot \W.$ 
%Let $\rH$ be the weight for the left $\V$-tensor. 
So for every $\Z \in \mM, \V_1,...,\V_\n \in \mV,\W_1,...,\W_\m$ for $\n,\m \geq 0$ there is an equivalence
$$\Mul_{\mP\B\Env(\mM)}(\V_1,...,\V_\n, \F_!(\langle\V,\K,\W\rangle), \W_1,...,\W_\m; \Z) \simeq$$$$ \Mul_{\mP\B\Env(\mM)}(\V_1,...,\V_\n, \V, \X,\W, \W_1,...,\W_\m; \Z)\simeq \Mul_{\mM}(\V_1,...,\V_\n,\V,\X,\W,\W_1,...,\W_\m; \Z).$$

\end{proof}

%\begin{remark}Let $\kappa$ be a small regular cardinal, $ \mV^\ot \to \Ass, \mW^\ot \to \Ass$ small$\infty$-operads and $\V \in \mV, \W \in \mW.$The weights for the left $\V$-tensor and right $\W$-tensor are $\kappa$-smallsince the equivalence$ \mP\B\Env(*_{\mV,\mW}) \simeq \mP(\Env(\mV)\times \Env(\mW)) $restricts to an equivalence $ \B\Env(*_{\mV,\mW}) \simeq \Env(\mV)\times \Env(\mW)$so that objects of $\Env(\mV)\times \Env(\mW)$ correspond to representable presheaves on $\B\Env(*_{\mV,\mW}).$	\end{remark}

\begin{lemma}
Let $\kappa$ be a small regular cardinal, $\K$ a $\kappa$-small $\infty$-category, $\mV^\ot \to \Ass, \mW^\ot \to \Ass$ small $\infty$-operads, $\V $ in the full subcategory of $\mP\Env(\mV) $ generated by $\mV$ under $\kappa$-small colimits, $\W $ in the full subcategory of $\mP\Env(\mW) $ generated by $\mW$ under $\kappa$-small colimits. The weight $\langle\V, \K, \W\rangle$ is $\kappa$-small.
\end{lemma}

\begin{proof}
	
Every presheaf $\X$ on $\K$ is the colimit of
the functor $\K \times_{\mP(\K)} \mP(\K)_{/\X} \to \K \subset \mP(\K).$ 
Thus the final presheaf on $\K$ is the colimit of the Yoneda-embedding $\K \subset \mP(\K)$ and so a $\kappa$-small colimit of representables if $\K$ is $\kappa$-small.
So the result follows since the functor 
$\mP\Env(\mV) \times \mP(\K) \times \mP\Env(\mW) \to \mP\Env(\mV) \otimes \mP(\K) \otimes \mP\Env(\mW)$ preserves small colimits component-wise and the product of two $\kappa$-small $\infty$-categories is again $\kappa$-small.
	
\end{proof}

\begin{corollary}\label{cuik}
Let $\kappa$ be a small regular cardinal. If a weakly bienriched $\infty$-category admits $\kappa$-small weighted colimits, it admits $\kappa$-small conical colimits
and left and right tensors.
%In particular, it admits $\kappa$-small
% for every $\V \in \mV, \W \in \mW$ the functor$\V\ot(-) \ot \W:\mM \to \mM$ preserves $\kappa$-small colimits.
	
\end{corollary}

\subsection{Weighted colimits via enriched coends}

%In this final subsection we describe morphism objects in the enriched $\infty$-category of enriched functors as an enriched end (Theorem \ref{end}).In the following we will use symmetric monoidal $\infty$-categories \cite[Definition 2.0.0.7.]{lurie.higheralgebra}, which are cocartesian fibrations$\mV^\boxtimes \to \Comm$ to the category $\Comm$ of finite pointed sets.There is a canonical functor $\theta: \Ass \to \Comm, [\n]\mapsto \{1,...,\n,*\} $\cite[Construction 4.1.2.9.]{lurie.higheralgebra} and we write $\mV^\ot \to \Ass$ for the pullback of a symmetric monoidal $\infty$-category $\mV^\boxtimes \to \Comm$ along $\theta,$ which is a monoidal $\infty$-category.
Next we discuss a further example of weighted colimits, the enriched coends \cite[Definition 5.78.]{heine2024bienriched} with respect to enrichment in any symmetric monoidal $\infty$-category.
Moreover we prove that general weighted colimits can be built from tensors and coends if such exist (Proposition \ref{heuu}).

In the following we will use symmetric monoidal $\infty$-categories. If $\mV^\ot \to \Ass$ is the pullback of a symmetric monoidal $\infty$-category $\mV^\boxtimes \to \Comm$, there is a canonical monoidal equivalence
$(\mV^\rev)^\ot \simeq \mV^\ot$ since $\theta$ factors as $ \Ass \xrightarrow{(-)^\op} \Ass \xrightarrow{\theta}\Comm$.
The latter induces an equivalence between left and right (weakly) enriched $\infty$-categories.
%Consequently, there is no need to distinguish between (weakly) left and right $\mV$-(pseudo)-enriched $\infty$-categories, which we therefore call (weakly) $\mV$-(pseudo)-enriched $\infty$-categories. 
%Similarly, there is no need to distinguish between left and right $\mV$-pseudo-enriched $\infty$-categories, which we therefore call $\mV$-pseudo-enriched $\infty$-categories. 
In particular, there is no need to distinguish between left and right morphism objects, which we call morphism objects.

\begin{notation}Let $\mV^\boxtimes \to \Comm$ be a symmetric monoidal $\infty$-category and $\mM^\circledast \to \mV^\ot $ a left $\mV$-enriched $\infty$-category. By \cite[Notation 5.36.]{heine2024bienriched} there is a left $\mV$-enriched functor $$\Mor_\mM: (\mM^\op  \otimes \mM)^\circledast \to \mV^\circledast$$ assigning the morphism object.
\end{notation}

%\begin{notation}\label{nnoo}Let $\mM^\circledast \to \mV^\ot $ be a left enriched $\infty$-category. The $\mV,\mV$-enriched left morphism object functor $$\L\Mor_\mM: \mM^\circledast \times (\mM^\op)^\circledast \to \mV^\circledast$$ is the restricted $\mV,\mV$-enriched evaluation functor $$\mM^\circledast \times (\mM^\op)^\circledast \subset \mM^\circledast \times \Enr\Fun_{\mV, \emptyset}(\mM,\mV)^\circledast \to \mV^\circledast.$$\end{notation}

%\begin{definition}Let $\mM^\circledast \to \mV^\ot \times \mV^\ot$ be a weakly bienriched $\infty$-category, $\mJ^\circledast \to \mV^\ot $ a right enriched $\infty$-category and $\F: (\mJ^\op)^\circledast \times \mJ^\circledast \to \mM^\circledast$ a $\mV,\mV$-enriched functor.\begin{itemize}\item The $\mV$-enriched coend of $\F$ is the $\L\Mor_{\mJ^\op}$-weighted colimit of $\F$.\item The $\mV$-enriched end of $\F$ is the $\R\Mor_\mJ$-weighted limit of $\F$.\end{itemize}\end{definition}

%\label{Aho}Let $\mM^\circledast \to \mV^\ot \times \mV^\ot$ be a weakly bienriched $\infty$-category, $\mJ^\circledast \to \mV^\ot $ a right enriched $\infty$-category and $\F: (\mJ^\op)^\circledast \times \mJ^\circledast \to \mM^\circledast$ a $\mV,\mV$-enriched functor.
\begin{definition}\label{Aho} Let $\mV^\ot \to \Comm$ be a symmetric monoidal $\infty$-category, $\mM^\circledast \to \mV^\ot, \mJ^\circledast \to \mV^\ot$ left pseudo-enriched $\infty$-categories and $\F: (\mJ^\op\ot\mJ)^\circledast \to \mM^\circledast$ a left $\mV$-enriched functor.
	
\begin{itemize}
%object representing the $\mV$-enriched presheaf
%$$\mV^\op \to \mV, \V \mapsto \Mor_{\Enr\Fun_{\mV}(\mJ^\op \otimes \mJ,\mV)}(\Mor_\mJ, \Mor_{\mM}(\V,-) \circ \F).$$ %\simeq \mV(\V,\Mor_{\Enr\Fun_{\mV,\mV}(\mJ^\op \times \mJ,\mV)}(\Mor_\mJ, \F))$$representing an equivalence$$ \int^\mJ \F \simeq \Mor_{\Enr\Fun_{\mV,\mV}(\mJ^\op \times \mJ,\mV)}(\Mor_\mJ, \F).$$
		
\item The $\mV$-enriched coend of $\F$, denoted by $\int^\mJ \F$, is the $\Mor_{\mJ^\op}$-weighted colimit of $\F.$
%$$\mV \to \mV, \V \mapsto \Mor_{\Enr\Fun_{\mV}(\mJ \otimes \mJ^\op,\mV)}(\Mor_{\mJ^\op}, \Mor_{\mM}(-,\V) \circ \F^\op).$$ 	

\item The $\mV$-enriched end of $\F$, denoted by $\int_\mJ \F$, is the $\Mor_\mJ$-weighted limit of $\F.$
\end{itemize}	
	
\end{definition}

\begin{remark}

The $\mV$-enriched end of $\F$ is the $\Mor_\mJ$-weighted limit of $\F,$
which is the $\Mor_\mJ$-weighted colimit of $\F^\op: (\mJ\ot\mJ^\op)^\circledast \to (\mM^\op)^\circledast $ and so agrees with the $\mV$-enriched coend of $\F^\op.$

\end{remark}

\begin{construction}
Let $\mV^\boxtimes \to \Comm$ be a presentably symmetric monoidal $\infty$-category
and $\mM^\circledast \to \mV^\ot $ a left $\mV$-enriched $\infty$-category that admits left tensors.
%By ... there is a canonical $\mV$-enriched functor $\Mor_\mM: (\mM^\op \otimes \mM)^\circledast \to \mV^\circledast$.
The composition of left $\mV$-enriched functors $$(\mV^\op \ot \mM^\op \ot \mM)^\circledast \xrightarrow{\id \otimes \Mor_\mM} (\mV^\op \ot \mV)^\circledast \xrightarrow{\Mor_\mV} \mV^\circledast$$ corresponds to a left $\mV$-enriched functor 
$$(\mV \ot \mM)^\circledast \to (\Enr\Fun_\mV(\mM, \mV)^\op)^\circledast $$
that induces a functor $\ot : (\mV \ot \mM)^\circledast \to \mM^\circledast$ by the existence of tensors.
We obtain an equivalence
\begin{equation}\label{gggz}
\Mor_\mM((-)\ot(-),-) \simeq \Mor_\mV(-,\Mor_\mM(-,-))
\end{equation}
of left $\mV$-enriched functors $ (\mM^\op \otimes \mM^\op \ot \mM)^\circledast \to \mV^\circledast.$ 
	
\end{construction}

%\begin{theorem}\label{end}Let $\mV^\boxtimes \to \Comm$ be a presentably symmetric monoidal $\infty$-category, $\mJ^\circledast \to \mV^\ot $ a left pseudo-enriched $\infty$-category and
%$\mM^\circledast \to \mV^\ot $ a left enriched $\infty$-category and %$\F: \ot_!(\mJ^\op \times_{\Ass} \mJ)^\circledast=(\mJ^\op\ot^\mV\mJ)^\circledast \to \mM^\circledast$ a left $\mV$-enriched functor.
%$\F,\G : (\mJ^\op)^\circledast \to \mV^\circledast $ left $\mV$-enriched functors.	
%There is an equivalence$$ \Mor_{\mP_\mV(\mJ)}(\F,\G) \simeq \int_\mJ \Mor_\mV \circ (\F^\op \otimes \G).$$\end{theorem}

\begin{theorem}\label{heuu}Let $\mV^\boxtimes \to \Comm$ be a presentably symmetric monoidal $\infty$-category, $\mJ^\circledast \to \mV^\ot $ a left $\mV$-enriched $\infty$-category,
$\mM^\circledast \to \mV^\ot $ a left $\mV$-enriched $\infty$-category that admits left tensors, $\F: \mJ^\circledast \to \mM^\circledast$ a left $\mV$-enriched functor and
$ \rH \in \mP_\mV(\mJ)$. %  and $\rH':=\Mor_{\mP_\mV(\mJ)}((-)_{\mid\mJ},\rH): (\mJ^\op)^\circledast \to \mV^\circledast.$ %: (\mJ^\op)^\circledast \to \mV^\circledast$ be a $\mV$-enriched functor. If $\mM^\circledast \to \mV^\ot \times \mV^\ot$ admits small colimits and left tensors, 
There is a canonical equivalence
$$ \colim^\rH(\F) \simeq \int^\mJ\ot \circ (\mid \rH\mid \otimes \F).$$
		
%\item Let $\rH: \mJ^\circledast \to \mV^\circledast$ be a  right $\mV$-enriched functor. If $\mM^\circledast \to \mV^\ot \times \mV^\ot$ admits small limits and left cotensors, there is a canonical equivalence $$ \lim^\rH(\F) \simeq \int^\mJ(-)^{(-)} \circ (\rH^\op \times \F).$$
%The $\mH$-weighted colimit of$ \F$ is the $\mV$-enriched coend of the $\mV$-enriched functor$$ (\mJ^\op)^\circledast \ot^\mV \mJ^\circledast \xrightarrow{\F^\op \ot^\mV \G}	\end{enumerate}
	
\end{theorem}

\begin{proof}Equivalence (\ref{gggz}) gives for every $\X \in \mM$ a canonical equivalence 
\begin{equation}\label{ggg}
\Mor_\mM(\ot \circ (\mid \rH\mid \otimes \F),\X) \simeq \Mor_\mV \circ(\mid \rH\mid ^\op \otimes (\Mor_\mM(-,\X)\circ \F^\op))
\end{equation}
of $\mV$-enriched functors $ (\mJ^\op \otimes \mJ)^\circledast \to \mV^\circledast.$
We obtain the following chain of natural equivalences:
$$\Mor_\mM(\int^\mJ\ot \circ (\mid \rH\mid  \otimes \F),\X) \simeq \int_\mJ \Mor_\mM(\ot \circ (\mid \rH\mid  \otimes \F),\X) \simeq %\Mor_{\mP_\mV(\mJ^\op \otimes^\mV \mJ)}(\Mor_{\mJ^\op}, \Mor_\mV \circ (\rH^\op \ot^\mV (\Mor_\mM(-,\X) \circ \F))) $$
\int_\mJ \Mor_\mV \circ (\mid \rH\mid ^\op \otimes (\Mor_\mM(-,\X)\circ \F^\op)) $$
$$\simeq \Mor_{\Enr\Fun_{\mV, \emptyset}(\mJ^\op,\mV)}(\mid \rH\mid ,\Mor_\mM(-,\X)\circ \F^\op)
\simeq \Mor_{\mP_\mV(\mJ)}(\rH,\F^*(\X)), $$ %\simeq \Mor_{\mP_\mV(\mM)}(\F_!(\rH),\X).$$
where the first equivalence is by Corollary \ref{transco}, %and Corollary \ref{unitol}
the second equivalence is by (\ref{ggg}), the third equivalence is by \cite[Theorem 5.84.]{heine2024bienriched} and the last equivalence is by Theorem \ref{unitol}. 
\end{proof}	

%\subsection{Functoriality of weighted colimits}

%In the following we study functoriality of forming weighted colimits.We start with the following remark:Next we consider coherent functoriality. 
%For that we recallthat by Corollary \ref{envvcor} for every absolute small weakly  bienriched $\infty$-categories $\mJ^\circledast \to \mV'^\ot \times \mW'^\ot, \mM^\circledast \to \mV^\ot \times \mW^\ot$ there is an embedding 
%\begin{equation}\label{eiii}
%\Enr\Fun(\mJ,\mM) \subset \Enr\Fun(\mJ,\mP\B\Env(\mM)) \simeq \LinFun^\L(\mP\B\Env(\mJ),\mP\B\Env(\mM)), \F \mapsto \F_!. \end{equation}
%whose essential image are the linear functors lying over left adjoint monoidal functors.Let $\rho: \Enr\Fun(\mJ,\mM) \times \mP\B\Env(\mJ)^\circledast \to \mP\B\Env(\mM)^\circledast$ be the corresponding enriched functor.

\subsection{A Bousfield Kan formula for weighted colimits}

In the following we prove a Bousfield-Kan formula for weighted colimits (Theorem \ref{BK}).

\begin{notation}
Let $\mJ^\circledast \to \mV^\ot$ be a left enriched $\infty$-category, $\mD^\circledast \to \mV^\ot \times \mW^\ot$ a weakly bienriched $\infty$-category that admits right cotensors and $\Y \in \mJ.$
By \cite[Corollary 5.23. (1)]{heine2024bienriched} the right $\mW$-enriched functor $\Enr\Fun_{\mV,\emptyset}(\mJ,\mD)^\circledast \to \mD^\circledast$ evaluating at $\Y$ admits a $\mW$-enriched right adjoint $\mathrm{Ran}_\Y$ that sends 
$\Z \in \mD $ to $ ^{{\L\Mor_\mJ(-,\Y)}} \Z.$
	
\end{notation}

In the following we use the notation of Theorem \ref{unitol}.

\begin{proposition}
	
Let $\mD^\circledast \to \mV^\ot \times \mW^\ot$ be a weakly bienriched $\infty$-category that admits right cotensors, $\F: \mJ^\circledast \to \mD^\circledast$ a left $\mV$-enriched functor and $\rH \in \mP\L\Env(\mJ)$ a left weight on $\mJ.$
For every $\X \in \mD, \W_1,...,\W_\m \in \mW, \m \geq 0$ there is a canonical equivalence
$$\Mul_{\mP\L\Env(\mD)}(\F_!(\rH), \W_1,...,\W_\m;\X) \simeq \Mul_{\Enr\Fun_{\mV,\emptyset}(\mJ,\mD)}(\F, \W_1,...,\W_\m; {^{(-)}\X}\circ {\mid \rH\mid}).$$
	
\end{proposition}

\begin{proof}
	
There is a canonical equivalence
$$\Mul_{\mP\L\Env(\mD)}(\F_!(\rH), \W_1,...,\W_\m;\X) \simeq$$
$$
\Mul_{\Enr\Fun_{\mP\Env(\mV),\emptyset}(\mP\L\Env(\mJ),\mP\L\Env(\mD))}(\F_!, \W_1,...,\W_\m; {^{(-)}\X}\circ \mP\L\Env(\mJ)(-,\rH)) \simeq$$
$$ \Mul_{\Enr\Fun_{\mV,\emptyset}(\mJ,\mD)}(\F, \W_1,...,\W_\m; {^{(-)}\X}\circ {\mid \rH\mid}).$$
	
\end{proof}

\begin{corollary}Let $\mJ^\circledast \to \mV^\ot $ be a weakly left enriched $\infty$-category, $\rH$ a left weight on $\mJ$ and $\mD^\circledast \to \mV^\ot \times \mW^\ot$ a weakly bienriched $\infty$-category that admits right cotensors and $\F: \mJ^\circledast \to \mD^\circledast$ a left $\mV$-enriched functor that admits a $\rH$-weighted colimit. There is a canonical equivalence
$$\Mul_{\mD}(\colim^\rH(\F), \W_1,...,\W_\m;\X) \simeq \Mul_{\Enr\Fun_{\mV,\emptyset}(\mJ,\mD)}(\F, \W_1,...,\W_\m; {^{(-)}\X}\circ {\mid \rH\mid}).$$

\end{corollary}

\begin{proof}
For every $\X \in \mD, \V_1,...,\V_\n \in \mV, \W_1,...,\W_\m \in \mW, \n,\m \geq 0$ there is a canonical equivalence 
$$\Mul_{\mD}(\colim^\rH(\F), \W_1,...,\W_\m;\X) \simeq \Mul_{\mP\L\Env(\mD)}(\F_!(\rH), \W_1,...,\W_\m;\X) \simeq$$
$$ \Mul_{\Enr\Fun_{\mV,\emptyset}(\mJ,\mD)}(\F, \W_1,...,\W_\m; {^{(-)}\X}\circ {\mid \rH\mid}).$$

\end{proof}

\begin{corollary}\label{Funf} Let $\mJ^\circledast \to \mV^\ot $ be a weakly left enriched $\infty$-category, $\rH$ a left weight on $\mJ$ and $\mD^\circledast \to \mV^\ot \times \mW^\ot$ a weakly bienriched $\infty$-category that admits right cotensors and $\rH$-weighted colimits. The right $\mW$-enriched functor
$$\mD^\circledast \xrightarrow{\mathrm{Ran}_\rH} \Enr\Fun_{\mP\Env(\mV),\emptyset}(\mP\L\Env(\mJ),\bar{\mD})^\circledast \xrightarrow{(-)_{\mid\mJ}} \Enr\Fun_{\mV,\emptyset}(\mJ,\mD)^\circledast, \X \mapsto {^{(-)}\X}\circ {\mid \rH\mid} $$
admits a $\mW$-enriched left adjoint that sends $\F$ to $\colim^\rH(\F).$

\end{corollary}

%\begin{corollary}\label{soslo} Let $(\mV^\ot \to \Ass, \rS)$ be a small localization pair and $\mJ^\circledast \to \mV^\ot $ a small left $\rS$-enriched $\infty$-category, $\rH \in \mP\L\Env(\mJ)_{\rS}$.For every $\Y \in \mJ$ there is a canonical equivalence in $\rS^{-1}\mP\Env(\mV) :$ $$ \colim^\rH (\Mor_\mJ(\Y,-)) \simeq \mid\rH\mid(\Y).$$\end{corollary}

\begin{theorem}\label{BK}
Let $\mJ^\circledast \to \mV^\ot \times \mW^\ot$ be a small weakly bienriched $\infty$-category, $\mD^\circledast \to \mV^\ot \times \mW^\ot$ a weakly bienriched $\infty$-category that admits left and right tensors and small conical colimits,
$\F : \mJ^\circledast \to \mD^\circledast$ a $\mV,\mW$-enriched functor and $\rH \in \mP\B\Env(\mJ).$
Then $\colim^\rH(\F)$ is the colimit of a canonical simplicial object $\mY$ in $ \mD$ such that for every $\n \geq 0$ there is a canonical equivalence in $\mN:$ $$\mY_\n \simeq \colim_{\Z_1, ...., \Z_\n \in \mJ^\simeq}(\mid\rH\mid(\Z_{\n})) \ot \L\Mor_{\bar{\mJ}}(\Z_{\n-1},\Z_\n)\ot ... \ot \L\Mor_{\bar{\mJ}}(\Z_1,\Z_2) \ot \F(\Z_1).$$	
		
\end{theorem}

\begin{proof}

By \cite[Proposition 5.26.]{heine2024bienriched} every $\mV,\mW$-enriched functor $\F : \mJ^\circledast \to \mD^\circledast$ is the colimit of a canonical simplicial object $\mX$ in $ \Enr\Fun_{\mV, \mW}(\mJ,\mD)$ such that for every $\n \geq 0$ there is a canonical equivalence in $\Enr\Fun_{\mV, \mW}(\mJ,\mD):$ $$\mX_\n \simeq \colim_{\Z_1, ...., \Z_\n \in \mM^\simeq}\L\Mor_{\bar{\mJ}}(\Z_{\n},-)\ot \L\Mor_{\bar{\mJ}}(\Z_{\n-1},\Z_\n)\ot ... \ot \L\Mor_{\bar{\mJ}}(\Z_2,\Z_3) \ot \L\Mor_{\bar{\mJ}}(\Z_1,\Z_2) \ot \F(\Z_1).$$	
By Proposition \ref{weifun} there is a left adjoint functor 
$\colim^\rH: \Enr\Fun_{\mV, \mW}(\mJ,\mD) \to \mD, \F \mapsto \colim^\rH(\F).$ Let $\mY:= \colim^\rH \circ \mX.$
Consequently, there are canonical equivalences: $\colim^\rH(\F) \simeq
\colim_{[\n]\in \Delta^\op} \mY_\n $
and $$\mY_\n \simeq \colim_{\Z_1, ...., \Z_\n \in \mJ^\simeq}\colim^\rH(\L\Mor_{\bar{\mJ}}(\Z_{\n},-)) \ot \L\Mor_{\bar{\mJ}}(\Z_{\n-1},\Z_\n)\ot ... \ot \L\Mor_{\bar{\mJ}}(\Z_1,\Z_2) \ot \F(\Z_1)$$$$\simeq \colim_{\Z_1, ...., \Z_\n \in \mJ^\simeq}(\mid\rH\mid(\Z_{\n})) \ot \L\Mor_{\bar{\mJ}}(\Z_{\n-1},\Z_\n)\ot ... \ot \L\Mor_{\bar{\mJ}}(\Z_1,\Z_2) \ot \F(\Z_1),$$
where the last equivalence is by Corollary \ref{Funf}.
\end{proof}

%\begin{proposition}
%There is a canonical equivalence
%$$\Mul_\mD(\V_1,...,\V_\n,\colim^\rH(\F), \W_1,...,\W_\m;\X) \simeq$$$$ \Mul_{\Enr\Fun_{\mV,\emptyset}(\mJ,\mD)}(\V_1,...,\V_\n,\F, \W_1,...,\W_\m; \X^{(-)}\circ \mP\L\Env(\mJ)(-,\rH)_{\mid \mJ}).$$
%\end{proposition}

%\begin{proof}There is a canonical equivalence$$\Mul_\mD(\V_1,...,\V_\n,\colim^\rH(\F), \W_1,...,\W_\m;\X) \simeq\Mul_\mD(\V_1,...,\V_\n,\F'(\rH), \W_1,...,\W_\m;\X) \simeq $$$$\Mul_{\Enr\Fun_{\mP\Env(\mV),\emptyset}(\mP\L\Env(\mJ),\mD)}(\V_1,...,\V_\n,\F', \W_1,...,\W_\m; \X^{(-)}\circ \mP\L\Env(\mJ)(-,\rH)) \simeq$$$$ \Mul_{\Enr\Fun_{\mV,\emptyset}(\mJ,\mD)}(\V_1,...,\V_\n,\F, \W_1,...,\W_\m; \X^{(-)}\circ \rH^\op).$$

%$$\L\Mor_\mD(\colim^\rH(\F),\X) \simeq \L\Mor_\mD(\F'(\rH),\X) \simeq\L\Mor_{\Fun_{\mV,\emptyset}(\mP_\mV(\mJ),\mD)}(\F', \X^{(-)}\circ \L\Mor_{\mP_\mV(\mJ)}(-,\rH)  )\simeq \L\Mor_{\Fun_{\mV,\emptyset}(\mJ,\mD)}(\F, \X^{(-)}\circ \rH^\op).$$\end{proof}

\subsection{Preservation of weighted colimits}

Next we prove that enriched left adjoints preserve weighted colimits and dually that enriched right adjoints preserve weighted limits.
We first define what it means that an enriched functor preserves weighted colimits.

\begin{definition}

%Let $\mM^\circledast \to \mV^\ot \times \mW^\ot$ be an absolute small weakly  bienriched $\infty$-category. 
A small diagram is a triple $(\rH, \F,\psi)$ consisting of 
\begin{itemize}
\item an enriched functor $\F:\mJ^\circledast \to \mM^\circledast$ of absolute small weakly  bienriched $\infty$-categories,
\item an object $\rH \in  \mP\B\Env(\mJ),$
\item a morphism $\psi: \F_!(\rH) \to \Y$ in $\mP\B\Env(\mM)$, where $\Y \in \mM.$ 

\end{itemize}
More precisely, we call $(\rH, \F,\psi)$ a small $\rH$-weighted diagram on $\mM$.
If $\alpha,\beta$ are the maps of $\infty$-operads underlying $\F:\mJ^\circledast \to \mM^\circledast$, we call $(\alpha, \beta,\mJ^\circledast, \rH)$ the weight underlying the diagram $(\rH, \F,\psi)$.
If $\psi$ exhibits $\Y$ as the $\rH$-weighted colimit, we call $(\rH, \F,\psi)$ a small $\rH$-weighted colimit diagram on $\mM.$
A small left (right) diagram is a small diagram lying over a left (right) weight.
\end{definition}

\begin{definition}A diagram is a left (right) diagram if the underlying weight is a left (right) weight.
\end{definition}

\begin{definition}
Let $\kappa, \tau$ be small regular cardinals.	
A diagram is a left enriched, right enriched, bienriched, left $\kappa$-enriched, right $\tau$-enriched, $\kappa,\tau$-bienriched diagram, respectively, if
the underlying weight is left enriched, right enriched, bienriched, left $\kappa$-enriched, right $\tau$-enriched, $\kappa,\tau$-bienriched.
\end{definition}
\begin{notation}\label{trans}
	
An enriched functor $\theta: \mM^\circledast \to \mN^\circledast$ lying over maps of $\infty$-operads $\alpha',\beta'$ sends a diagram $(\rH, \F,\psi)$ on $\mM$ lying over the weight $(\alpha, \beta,\mJ^\circledast, \rH)$ to the diagram
$$(\rH, \theta \circ \F,\psi': \theta_!(\F_!(\rH)) \xrightarrow{\theta_!(\psi)} \theta(\Y)) $$ on $\mN$ lying over the weight $(\alpha' \circ \alpha, \beta' \circ \beta,\mJ^\circledast, \rH)$. We call the latter diagram the image of $(\rH, \F,\psi)$ under $\theta.$
We call the weight $(\alpha' \circ \alpha, \beta' \circ \beta,\mJ^\circledast, \rH)$ the image of $(\alpha, \beta,\mJ^\circledast, \rH)$ under $\alpha', \beta'.$
\end{notation}

\begin{definition}
Let $\theta: \mM^\circledast \to \mN^\circledast$ be an enriched functor of absolute small weakly  bienriched $\infty$-categories and $(\rH, \F: \mJ^\circledast \to \mM^\circledast,\psi)$ a $\rH$-weighted colimit diagram on $\mM$.
We say that $\theta$ preserves the $\rH$-weighted colimit of $\F$ if $\theta$ sends the $\rH$-weighted colimit diagram $(\rH, \F: \mJ^\circledast \to \mM^\circledast,\psi)$ to a $\rH$-weighted colimit diagram on $\mN$.

\end{definition}

\begin{definition}Let $\mM^\circledast \to \mV^\ot \times \mW^\ot, \mN^\circledast \to \mV''^\ot \times \mW''^\ot $ be absolute small weakly  bienriched $\infty$-categories and $\theta: \mM^\circledast \to \mN^\circledast$ an enriched functor.
		
% of absolute small weakly  bienriched $\infty$-categories
% lying over maps of $\infty$-operads $\alpha: \mV^\ot \to \mV'^\ot, \beta: \mW^\ot \to \mW'^\ot$ 
\begin{itemize}
%\item If $\psi: \F_!(\rH) \to \Y$ exhibits $\Y$ as the $\rH$-weighted colimit of $\F$ and $\psi'$ exhibits $\theta(\Y)$ as the $\rH$-weighted colimit of $\theta \circ \F$, we say that $\theta$ preserves the $\rH$-weighted colimit of $\F.$

\item For every absolute small weight $(\alpha, \beta, \mJ^\circledast,\rH)$ over $\mV,\mW$	
we say that $\theta$ preserves $(\alpha, \beta, \mJ^\circledast, \rH)$-weighted colimits if $\theta$ preserves the $\rH$-weighted colimit of any enriched functor $ \mJ^\circledast \to \mM^\circledast$ lying over $\alpha,\beta.$

\item For every collection $\mH$ of weights over $\mV,\mW$ we say that $\theta$ preserves $\mH$-weighted colimits if $\theta$ preserves the $(\alpha, \beta, \mJ^\circledast,\rH)$-weighted colimit for every $(\alpha, \beta, \mJ^\circledast,\rH) \in \mH.$ 

\item We say that $\theta$ preserves $\kappa$-small weighted colimits if 
$\theta$ preserves the $(\alpha, \beta, \mJ^\circledast,\rH)$-weighted colimit for every $\kappa$-small weight $(\alpha, \beta, \mJ^\circledast,\rH)$ over $\mV,\mW.$
\end{itemize}

\end{definition}
\begin{notation}\label{notabene}
For every weakly bienriched $\infty$-categories $\mM^\circledast \to \mV^\ot \times \mW^\ot,\mN^\circledast \to \mV^\ot \times \mW^\ot$ and collections $\Lambda, \Lambda'$ of diagrams in $\mM, \mN$, respectively, let
$$\Enr\Fun^{\Lambda, \Lambda'}_{\mV, \mW}(\mM, \mN) \subset \Enr\Fun_{\mV,\mW}(\mM,\mN) $$ be the full subcategory of $\mV, \mW$-enriched functors sending diagrams of $\Lambda$ to diagrams of $\Lambda'.$
Let $$\Enr\Fun^{\Lambda}_{\mV,\mW}(\mM,\mN) \subset \Enr\Fun_{\mV,\mW}(\mM,\mN) $$ be the full subcategory of $\mV, \mW$-enriched functors sending diagrams of $\Lambda$ to weighted colimit diagrams.
% (or equivalently send $\mH$-weighted colimit diagrams to weighted colimit diagrams and lie over a pair of maps of $\infty$-operads that send $\mH$ to $\mH'$).
	
\end{notation}
\begin{notation}Let $\mV^\otimes \to \Ass, \mW^\ot \to \Ass$ be $\infty$-operads, $\mN^\circledast \to \mV^\ot \times \mW^\ot$ a weakly bienriched $\infty$-category, $\mM^\circledast \to \mV^\ot$ a weakly left enriched $\infty$-category, $\mO^\circledast \to \mW^\ot$ a weakly right enriched $\infty$-category, $\Lambda$ a collection of diagrams in $\mM$ and $\Lambda'$ a collection of diagrams in $\mO.$ 
Let
$$\Enr\Fun_{\mV, \emptyset}^{\Lambda}(\mM, \mN)^\circledast \subset \Enr\Fun_{\mV, \emptyset}(\mM, \mN)^\circledast$$ be the full weakly right enriched subcategory of left $\mV$-enriched functors $\mM \to \mN$ sending diagrams of $\Lambda$ to weighted colimit diagrams.
Let
$$\Enr\Fun_{\emptyset, \mW}^{\Lambda'}(\mO, \mN)^\circledast \subset \Enr\Fun_{\emptyset, \mW}(\mO, \mN)^\circledast$$ be the full weakly left enriched subcategory of right $\mW$-enriched functors $\mO \to \mN$ sending diagrams of $\Lambda'$ to weighted colimit diagrams.
\end{notation}

\begin{definition}
Let $\mM^\circledast \to \mV^\ot \times \mW^\ot$ be a weakly bienriched $\infty$-category and $\mH$ a collection of weights over $\mV, \mW$.
A collection $\Lambda$ of diagrams in $\mM$ is $\mH$-weighted if every diagram
of $\Lambda$ is $\rH$-weighted for some $\rH \in \mH$.
	
\end{definition}

\begin{notation}\label{Lamb}
Let $\mM^\circledast \to \mV^\ot \times \mW^\ot$ be a weakly bienriched $\infty$-category. For every collection $\mH$ of weights over $\mV, \mW$ let $\Lambda^{(\mM,\mH)}$ be the collection of $\rH$-weighted colimit diagrams in $\mM$ for some $\rH \in \mH$.
	
\end{notation}

\begin{remark}\label{remros}
Let $\mM^\circledast \to \mV^\ot \times \mW^\ot,\mN^\circledast \to \mV^\ot \times \mW^\ot$ be weakly bienriched $\infty$-categories, $\mH$ a collection of weights over $\mV, \mW$ and $\Lambda$ a collection of $\mH$-weighted diagrams in $\mM$.
Then $\Enr\Fun^{\Lambda}_{\mV,\mW}(\mM,\mN)= \Enr\Fun^{\Lambda,  \Lambda^{(\mN,\mH)}}_{\mV,\mW}(\mM,\mN).$
\end{remark}

\begin{proposition}\label{remqalo}

Let $\mM^\circledast \to \mV^\ot \times \mW^\ot, \mN^\circledast \to \mV^\ot \times \mW^\ot $ be weakly bienriched $\infty$-categories and $\phi: \mM^\circledast \to \mN^\circledast $ a $\mV, \mW$-enriched functor.
\begin{enumerate}
\item If $\phi$ admits a $\mV, \mW$-enriched right adjoint, $\phi$ preserves all weighted colimits.

\item If $\phi$ admits a $\mV, \mW$-enriched left adjoint, $\phi$ preserves all weighted limits.
\end{enumerate}
\end{proposition}

\begin{proof}

We prove (1). (2) is similar. A $\mV,\mW$-enriched adjunction $\phi: \mM^\circledast \rightleftarrows \mN^\circledast: \gamma$ gives rise to a $\mP\Env(\mV),\mP\Env(\mW)$-enriched adjunction
$\phi_!: \mP\Env(\mM)^\circledast \rightleftarrows \mP\Env(\mN)^\circledast: \gamma_!$. 
Let $\F: \mJ^\circledast \to \mM^\circledast$ be an enriched functor, $\Y\in\mM, \rH\in\mP\B\Env(\mJ)$ and $\psi: \F_!(\rH) \to \Y $ a map in $\mP\B\Env(\mM)$ that exhibits $\Y$ as the $\rH$-weighted colimit of $\F$.
We like to see that the map $\phi_!(\psi): \phi_!(\F_!(\rH)) \to \phi(\Y) $ in $\mP\B\Env(\mN)$ exhibits $\phi(\Y)$ as the $\rH$-weighted colimit of $\phi \circ \F$, i.e. that for any $\Z \in \mN$, $\V_1,...,\V_\n \in \mV, \W_1,...,\W_\m \in \mW$ for some $\n,\m \geq 0$ the map $$\Mul_\mN(\V_1,...,\V_\n, \phi(\Y), \W_1,...,\W_\m; \Z) \to \Mul_{\mP\B\Env(\mN)}(\V_1,...,\V_\n, \phi_!(\F_!(\rH)), \W_1,...,\W_\m; \Z)$$ is an equivalence.
The latter map identifies with the equivalence
$$\Mul_\mM(\V_1,...,\V_\n, \Y, \W_1,...,\W_\m; \gamma(\Z)) \to \Mul_{\mP\B\Env(\mM)}(\V_1,...,\V_\n, \F_!(\rH), \W_1,...,\W_\m; \gamma(\Z)).$$

\end{proof}

\begin{definition}
Let $\mM^\circledast \to \mV^\ot \times \mW^\ot$ be a weakly bienriched $\infty$-category, $\Lambda$ a collection of diagrams in $\mM$ and $\mH$ be a collection of weights over $\mV,\mW.$

\begin{enumerate}
\item A full weakly bienriched subcategory $\mN^\circledast \subset \mM^\circledast$
is closed under diagrams of $\Lambda$ if for every diagram $(\rH,\tau, \psi: \tau_!(\rH) \to \Y)$ in $\mM$ the object $\Y\in \mM $ belongs to $\mN$ if
$\tau: \mJ^\circledast \to \mM^\circledast $ factors through $\mN^\circledast.$

\item A full weakly bienriched subcategory $\mN^\circledast \subset \mM^\circledast$
is closed under $\mH$-weighted colimits if it is closed under the collection of $\mH$-weighted colimit diagrams in $\mM.$ 

\end{enumerate}

\end{definition}

\begin{example}\label{Closure}
Let $\alpha:\F\to \G: \mM^\circledast \to \mN^\circledast$ be a
morphism of $\Enr(\mM,\mN)$ and $\T=(\rH,\tau: \mJ^\circledast \to \mM^\circledast, \psi: \tau_!(\rH) \to \Y)$ a diagram in $\mM$. %, where $\Y \in \mM$, $\tau: \mJ^\circledast \to \mM^\circledast$ is an enriched functor and $\psi$is a morphism in $\mP\Env(\mM).$
If $\F,\G$ send $\T$ to a weighted colimit diagram and $\alpha_{\tau(\Z)}$ is an equivalence for every $\Z \in \mJ$, then using Remark \ref{zujj} the morphism $\alpha_\Y$ is an equivalence.
In particular, the full subcategory of $\mM$ spanned by all $\X \in \mM$
such that $\alpha_\X$ is an equivalence, is closed under all weighted colimits
that are preserved by $\F,\G.$

\end{example}

%\subsection{Features of weighted colimits}
	
%In the following we prove the following results about weighted colimits:
%which we apply in the next section:
%\begin{enumerate}
%\item We prove an existence result for weighted colimits that splits the existence of weighted colimits in the existence of tensors and conical colimits
%criterion for the existence of weighted colimits % and dually weighted limits 
%(Proposition \ref{weeei}).

%\item We %prove that enriched left adjoints preserve weighted colimits (Proposition \ref{remqalo}) and 
%prove an enriched version of the adjoint functor theorem (Proposition \ref{remwq}).

%\item We study the behaviour of weighted colimits under pulling back the enrichment along a map of $\infty$-operads (Corollary \ref{eccco}).

%\item We prove that weighted colimits in enriched functor $\infty$-categories are formed object-wise (Proposition \ref{Enric}).\end{enumerate}

%To study existence of weighted colimits we start with functorialty of weighted colimits:
%There is the following functoriality of forming weighted colimits and dually weighted limits:

\subsection{Stability of weighted colimits under restriction of enrichment}

Next we prove that the existence of weighted colimits is stable under pulling back the enrichment (Corollary \ref{eccco}).
%We start with the following lemma:

\begin{lemma}\label{lemlems}
Let $\alpha: \mV'^\ot \to \mV^\ot, \beta: \mW'^\ot \to \mW^\ot$ be maps of small $\infty$-operads, $\mM^\circledast \to \mV^\ot \times \mW^\ot$ an absolute small weakly  bienriched $\infty$-category, $\F: \mJ^\circledast \to (\alpha, \beta)^*(\mM)^\circledast:=\mV'^\ot \times_{\mV^\ot}\mM^\circledast \times_{\mW^\ot} \mW'^\ot $ an enriched functor, $ \rH \in \mP\B\Env(\mJ)$ and $\Y \in \mM, \V'_1,... \V'_\n \in \mV', \W'_1,...,\W'_\m \in \mW'$ for some $\n,\m \geq 0$.
The projection $\rho: (\alpha, \beta)^*(\mM)^\circledast \to \mM^\circledast $ induces an equivalence $$\Mul_{\mP\L\Env((\alpha, \beta)^*(\mM))}(\V'_1,... \V'_\n, \F_!(\rH), \W'_1,...,\W'_\m; \Y) \to $$$$ \Mul_{\mP\L\Env(\mM)}(\alpha(\V'_1),...,\alpha(\V'_\n), \rho_! (\F_!(\rH)), \beta(\W'_1),...,\beta(\W'_\m); \rho(\Y)).$$

\end{lemma}

\begin{proof}
It is enough to prove that the unit $\Y \to \rho^\ast(\rho_!(\Y)) \simeq \rho^\ast(\rho(\Y)) $
in $ \mP\B\Env((\alpha, \beta)^*(\mM))$ is an equivalence.
The unit induces at $\V_1 \ot ... \ot \V_\n \ot \Z \ot \W_1 \ot ... \ot \W_\m \in \B\Env((\alpha, \beta)^*(\mM))$ the equivalence
$$ \Mul_{(\alpha, \beta)^*(\mM)}(\V_1,...,\V_\n, \Z,\W_1,...,\W_\m; \Y) \to  \Mul_{\mM}(\alpha(\V_1),...,\alpha(\V_\n), \rho(\Z), \beta(\W_1),...,\beta(\W_\m); \rho(\Y)).$$
\end{proof}

\begin{proposition}\label{pullba}

Let $\alpha: \mV'^\ot \to \mV^\ot, \beta: \mW'^\ot \to \mW^\ot$ be maps of small $\infty$-operads, $\mJ^\circledast \to \mQ^\ot \times \mU^\ot, \mM^\circledast \to \mV^\ot \times \mW^\ot$ absolute small weakly  bienriched $\infty$-categories, $ \rH \in \mP\B\Env(\mJ)$, $\F: \mJ^\circledast \to (\alpha, \beta)^*(\mM)^\circledast:=\mV'^\ot \times_{\mV^\ot}\mM^\circledast \times_{\mW^\ot} \mW'^\ot $ an enriched functor, $\Y \in \mM$ and $\rho: (\alpha, \beta)^*(\mM)^\circledast \to \mM^\circledast $ the projection.

\begin{enumerate}
\item A morphism $$\psi: \F_!(\rH) \to \Y $$ in $\mP\B\Env((\alpha, \beta)^*(\mM))$
exhibits $\Y$ as the $\rH$-weighted colimit of $\F$ if 
$$\rho_!(\psi): \rho_!(\F_!(\rH)) \to \rho(\Y) $$ in $\mP\B\Env(\mM)$
exhibits $\rho(\Y)$ as the $\rH$-weighted colimit of $\rho \circ \F$.

\vspace{1mm}
\item The converse holds by the uniqueness of weighted colimits if the $ \rH$-weighted colimit of $\F$ exists.

\end{enumerate}

\end{proposition}

\begin{proof}
If $\rho_!(\psi): \rho_! (\F_!(\rH)) \to \rho(\Y) $
exhibits $\rho(\Y)$ as the $\rH$-weighted colimit of $\rho \circ \F$, for every $\Z \in (\alpha, \beta)^*(\mM)$ and $\V'_1,...,\V'_\n \in \mV', \W'_1,..., \mW'_\m \in \mW'$ for some $\n, \m \geq 0$ the canonical map
$$\Mul_{\mM}(\alpha(\V'_1),...,\alpha(\V'_\n), \rho(\Y), \beta(\W'_1),..., \beta(\W'_\m); \rho(\Z)) \to $$$$ \Mul_{\mP\L\Env(\mM)}(\alpha(\V'_1),...,\alpha(\V'_\n), \rho_! (\F_!(\rH)), \beta(\W'_1),...,\beta(\W'_\m); \rho(\Z))$$
is an equivalence. But this map canonically identifies with the map
$$\Mul_{(\alpha, \beta)^*(\mM)}(\V'_1,... \V'_\n, \Y, \W'_1,...,\W'_\m; \Z) \to \Mul_{\mP\L\Env((\alpha, \beta)^*(\mM))}(\V'_1,... \V'_\n, \F_!(\rH), \W'_1,...,\W'_\m; \Z) \xrightarrow{\sigma}$$$$ \Mul_{\mP\L\Env(\mM)}(\alpha(\V'_1),...,\alpha(\V'_\n), \rho_! (\F_!(\rH)), \beta(\W'_1),...,\beta(\W'_\m); \rho(\Z)),$$
where $\sigma$ is induced by $\rho_!: \mP\L\Env((\alpha, \beta)^*(\mM))^\circledast \to \mP\L\Env(\mM)^\circledast $. By Lemma \ref{lemlems} $\rho$ is an equivalence.

(2) immediately follows from 1. and Lemma \ref{lemlems}.

\end{proof}

\begin{corollary}\label{eccco}
Let $\tau: \mV'^\ot \to \mV^\ot, \rho: \mW'^\ot \to \mW^\ot$ be maps of small $\infty$-operads, $\mJ^\circledast \to \mQ^\ot \times \mU^\ot$ an absolute small weakly  bienriched $\infty$-category and $(\alpha, \beta, \mJ^\circledast, \rH)$ a weight over $\mV',\mW'.$
For every absolute small weakly  bienriched $\infty$-category $\mM^\circledast \to \mV^\ot \times \mW^\ot$ that admits $(\tau \circ \alpha, \rho \circ \beta, \mJ^\circledast, \rH)$-weighted colimits the pullback $\mV'^\ot \times_{\mV^\ot}\mM^\circledast \times_{\mW^\ot} \mW'^\ot$ admits $(\alpha, \beta, \mJ^\circledast, \rH)$-weighted colimits and the projection $\mV'^\ot \times_{\mV^\ot}\mM^\circledast \times_{\mW^\ot} \mW'^\ot \to \mM^\circledast$ sends $(\alpha, \beta, \mJ^\circledast, \rH)$-weighted colimits to $(\tau \circ \alpha, \rho \circ \beta, \mJ^\circledast, \rH)$-weighted colimits.

\end{corollary}
\subsection{Existence of weighted colimits}

In the following we prove an existence result for weighted colimits that splits the existence of weighted colimits in the existence of tensors and conical colimits (Proposition \ref{weeei}).

\begin{proposition}\label{ukwa}
	
Let $ \mJ^\circledast \to \mQ^\ot \times \mU^\ot, \mM^\circledast \to \mV^\ot \times \mW^\ot$ be absolute small weakly  bienriched $\infty$-categories, $ \F: \mJ^\circledast \to \mM^\circledast$ an enriched functor lying over maps of $\infty$-operads $\alpha:\mQ^\ot \to \mV^\ot, \beta: \mU^\ot \to \mW^\ot$ and $\mK \subset \Cat_\infty, \mT \subset \mV, \mT' \subset \mW$ full subcategories.
Let $\mH \subset  \mP\B\Env(\mJ)$ be the smallest full subcategory containing $\mJ$, closed under $\K$-indexed colimits for any $\K \in \mK$ and such that for every $\Y \in \alpha^{-1}(\mT) \subset \mQ, \Z \in \beta^{-1}(\mT') \subset \mU $ the functors $\Y \ot (-), (-) \ot \Z: \mP\B\Env(\mJ) \to \mP\B\Env(\mJ)$ preserves $\mH$.
Assume that $\mM$ admits left tensors of objects $\Y \in \mT$ and right tensors of objects $ \Z \in \mT'$ and $\K$-indexed conical colimits for any $\K \in \mK.$

\begin{enumerate}
\item Then $\mM $ admits the $\rH$-weighted colimit of $\F$ for any $\rH \in \mH$.

\item An enriched functor $ \phi: \mM^\circledast \to \mN^\circledast$ preserves the $\rH$-weighted colimit of $\F$ for any $\rH \in \mH$ if $\phi$ preserves left tensors of objects $\Y \in \mT$ and right tensors of objects $ \Z \in \mT'$ and $\K$-indexed conical colimits.

\end{enumerate}

\end{proposition}
\begin{proof}
(1): Let $\mH' \subset  \mP\B\Env(\mJ)$ be the full subcategory spanned by those $\rH \in \mP\B\Env(\mJ)$ such that $\mM$ admits the $\rH$-weighted colimit of $\F$.
Then $\mH'$ contains $\mJ$, is closed under $\K$-indexed conical colimits for any $\K \in \mK$ and for every $\Y \in \alpha^{-1}(\mT), \Z \in \beta^{-1}(\mT') $ the functor $\Y \ot (-), (-) \ot \Z: \mP\B\Env(\mJ) \to  \mP\B\Env(\mJ)$ preserves $\mH'$, where we use Remark \ref{cohfun}. Thus $\mH \subset \mH'.$

(2): Let $\mH' \subset \mP\B\Env(\mJ)$ be the full subcategory spanned by those $\rH \in \mP\B\Env(\mJ)$ such that $\mM$ admits the $\rH$-weighted colimit of $\F$ and $\phi$ preserves the $\rH$-weighted colimit of $\F.$
Then $\mH'$ contains $\mJ$, is closed under $\K$-indexed conical colimits for any $\K \in \mK$, and for every $\Y \in \alpha^{-1}(\mT), \Z \in \beta^{-1}(\mT') $ the functor $\Y \ot (-), (-) \ot \Z: \mP\B\Env(\mJ) \to  \mP\B\Env(\mJ)$ preserves $\mH'$. Thus $\mH \subset \mH'.$

\end{proof}

Proposition \ref{ukwa} implies the following corollary:

\begin{corollary}\label{Cov}
Let $\mJ^\circledast \to \mQ^\ot \times \mU^\ot, \mM^\circledast \to \mV^\ot \times \mW^\ot$ be absolute small weakly  bienriched $\infty$-categories,
$ \F: \mJ^\circledast \to \mM^\circledast$ an enriched functor and $\mK \subset \Cat_\infty$ a full subcategory.
Let $\mH \subset \mP\B\Env(\mJ)$ be the smallest full subcategory containing $\mJ$, closed under $\K$-indexed colimits for any $\K \in \mK$ and such that for any $\Y \in \mQ, \Z \in \mU$ the functors $\Y \ot (-), (-) \ot \Z:\mP\B\Env(\mJ) \to \mP\B\Env(\mJ)$ preserve $\mH$.
Assume that $\mM$ admits left and right tensors and $\K$-indexed colimits for any $\K \in \mK$ and that $\K$-indexed colimits are preserved by the functors $\V \ot (-), (-) \ot \W:\mM \to \mM$ for every $\V \in \mV, \W \in \mW$.

\begin{enumerate}
\item Then $\mM $ admits the $\rH$-weighted colimit of $\F$ for any $\rH \in \mH$.

\item An enriched functor $ \phi: \mM^\circledast \to \mN^\circledast$ preserves the $\rH$-weighted colimit of $\F$ for any $\rH \in \mH$ if $\phi$ preserves left and right tensors and $\K$-indexed colimits.

\end{enumerate}

\end{corollary}

We obtain the following proposition:

\begin{proposition}\label{weeei}Let $\kappa$ be a small regular cardinal and $ \mM^\circledast \to \mV^\ot \times \mW^\ot$ an absolute small weakly  bienriched $\infty$-category.
	
\begin{enumerate}
\item Then $\mM $ admits $\kappa$-small weighted colimits if and only if $\mM$ admits left and right tensors and $\kappa$-small colimits and forming left and right tensors commutes with $\kappa$-small colimits.

\item Then $\mM $ admits $\kappa$-small left (right) weighted colimits if and only if $\mM$ admits left (right) tensors and $\kappa$-small colimits and forming left
(right) tensors commutes with $\kappa$-small colimits.

\item If $\mM$ admits $\kappa$-small weighted colimits, an enriched functor $ \mM^\circledast \to \mN^\circledast$
preserves $\kappa$-small weighted colimits if and only if it preserves $\kappa$-small colimits and left and right tensors.

\item If $\mM$ admits $\kappa$-small left (right) weighted colimits, an enriched functor $ \mM^\circledast \to \mN^\circledast$
preserves $\kappa$-small left (right) weighted colimits if and only if it preserves $\kappa$-small colimits and left (right) tensors.
\end{enumerate}

\end{proposition}

\begin{proof}We prove (1) and (3). The proofs of (2) and (4) are similar.
(1): If $\mM$ admits left and right tensors and $\kappa$-small colimits and forming left and right tensors commutes with $\kappa$-small colimits, then $\mM$ admits $\kappa$-small weighted colimits by Corollary \ref{Cov}.
 
%Conversely, if $\mM$ admits $\kappa$-small weighted colimits, then $\mM$ admitsbitensors by Lemma \ref{exorr}. % because $\id: \Ass \times \Ass \to \Ass \times \Ass$ is a $\kappa$-absolute small weakly  bienriched $\infty$-category.Finally, if $\mM$ admits $\kappa$-small weighted colimits, then $\mM$ admits$\kappa$-small colimits by Remark \ref{rehmm} and Example \ref{coni}.
%since for every $\kappa$-small $\infty$-category $\K$ the weakly bienriched $\infty$-category $\Ass \times \K \times \Ass \to \Ass \times \Ass$ is $\kappa$-small.Moreover if $\mM$ admits $\kappa$-small weighted colimits,
The converse is Corollary \ref{cuik} and Remark \ref{laulo}.
(3) follows like (1) from Corollary \ref{Cov}.
\end{proof}

\begin{corollary}\label{weeeisol}Let $\kappa$ be a small regular cardinal
and $ \mM^\circledast \to \mV^\ot \times \mW^\ot$ an absolute small weakly  bienriched $\infty$-category.
Then $\mM $ admits $\kappa$-small weighted colimits if and only if $\mM$ admits
$\kappa$-small weighted colimits of $\mV,\mW$-enriched functors $\K_{\V,\W}^\circledast \to \mM^\circledast$ for some $\infty$-category $\K.$
 
\end{corollary}

\begin{corollary}\label{transco}
	
Let $\mM^\circledast \to \mV^\ot \times \mW^\ot$ be an absolute small weakly  bienriched $\infty$-category.
The embedding $\mM^\circledast \subset \mP\B\Env(\mM)^\circledast$
preserves weighted limits.	
	
\end{corollary}

\begin{proof}By Proposition \ref{weeei} it is enough to show that the embedding $\mM^\circledast \subset \mP\B\Env(\mM)^\circledast$ preserves left and right cotensors and all conical limits.
Let $\W \in \mW, \Y \in \mM$ such that for every $\X \in \mM, \V_1,...,\V_\n \in \mV, \W_1,...,\W_\m \in \mW$ the induced map
$$ \Mul_\mM(\V_1,...,\V_\n,\X,\W_1,...,\W_\m;\Y^\W) \to  \Mul_\mM(\V_1,...,\V_\n,\X,\W_1,...,\W_\m, \W;\Y)$$ is an equivalence.
The latter map identifies with the map
$$ \mP\B\Env(\mM)(\V_1 \ot ... \ot \V_\n \ot \X \ot \W_1 \ot ... \ot \W_\m, \Y^\W) \to \Mul_{\mP\B\Env(\mM)}(\V_1 \ot ...\ot \V_\n \ot \X \ot \W_1 \ot ... \ot \W_\m,\W;\Y).$$
Hence the map
$ \mP\B\Env(\mM)(\T, \Y^\W) \to \Mul_{\mP\B\Env(\mM)}(\T,\W;\Y)$
is an equivalence for every $\T \in \B\Env(\mM).$
Since the functors $\mP\B\Env(\mM)(-, \Y^\W), \Mul_{\mP\B\Env(\mM)}(-,\W;\Y): \mP\B\Env(\mM)^\op \to \mS$ preserve small limits, the latter map is also an equivalence for every $\T \in \mP\B\Env(\mM)$. So the right cotensor is preserved by the embedding $\mM \subset \mP\B\Env(\mM)$. The case of the left cotensor is similar.
We prove the case of conical limits.

Let $\F: \K \to \mM$ be a functor such that for every $\X \in \mM, \V_1,...,\V_\n \in \mV, \W_1,...,\W_\m \in \mW$ the map
$$ \Mul_\mM(\V_1,...,\V_\n,\X,\W_1,...,\W_\m;\lim(\F)) \to \lim  \Mul_\mM(\V_1,...,\V_\n,\X,\W_1,...,\W_\m;\F(-))$$ is an equivalence.
The latter map identifies with the map
$$ \mP\B\Env(\mM)(\V_1 \ot ... \ot \V_\n \ot \X \ot \W_1 \ot ... \ot \W_\m,\lim(\F)) \to$$$$ \lim \mP\B\Env(\mM)(\V_1 \ot ...\ot \V_\n \ot \X \ot \W_1 \ot ... \ot \W_\m,\F(-)).$$
Hence the map
$ \mP\B\Env(\mM)(\Y,\lim(\F)) \to \lim \mP\B\Env(\mM)(\Y,\F(-))$
is an equivalence for every $\Y \in \B\Env(\mM).$
Since for every $\Z \in \mP\B\Env(\mM)$ the functor $\mP\B\Env(\mM)(-,\Z): \mP\B\Env(\mM)^\op \to \mS$ preserves small limits, the latter map is also an equivalence for every $\Y \in \mP\B\Env(\mM)$. So the limit of $\F$ is preserved by the embedding $\mM \subset \mP\B\Env(\mM)$.

% i.e. that for every $ \Y \in \mP\B\Env(\mM)$ the induced map$$ \mP\B\Env(\mM)(\Y,\lim(\F)) \to \lim \mP\B\Env(\mM)(\Y,\F(-))$$ is an equivalence.Since for every $\Z \in \mP\B\Env(\mM)$ the functor $\mP\B\Env(\mM)(-,\Z): \mP\B\Env(\mM)^\op \to \mS$ preserves small limits, we can assume that $\Y \in \B\Env(\mM)$.
	
\end{proof}

\begin{corollary}\label{yow}
	
Let $(\mV^\ot \to \Ass, \rS), (\mW^\ot \to \Ass, \T)$ be small localization pairs
and $\mM^\circledast \to \mV^\ot \times \mW^\ot$ a small $\rS,\T$-bienriched $\infty$-category.
The embedding $\mM^\circledast \subset \mP\B\Env(\mM)_{\rS, \T}^\circledast$
preserves weighted limits.
	
\end{corollary}

\begin{corollary}\label{Yowei}
	
%Let $\kappa$ be a small regular cardinal. 

\begin{enumerate}
\item For every left quasi-enriched $\infty$-category $\mM^\circledast \to \mV^\ot \times \mW^\ot$ the embedding $\mM^\circledast \subset \mP\B\Env(\mM)_{\L\Enr}^\circledast$
preserves left weighted limits and exhibits $\mM$ as free cocompletion under small left weighted colimits.
\item For every right quasi-enriched $\infty$-category $\mM^\circledast \to \mV^\ot \times \mW^\ot$ the embedding $\mM^\circledast \subset \mP\B\Env(\mM)_{\R\Enr}^\circledast$
preserves right weighted limits and exhibits $\mM$ as free cocompletion under small right weighted colimits.
\item For every bi-quasi-enriched $\infty$-category $\mM^\circledast \to \mV^\ot \times \mW^\ot$ the embedding $\mM^\circledast \subset \mP\B\Env(\mM)_{\B\Enr}^\circledast$ preserves weighted limits and exhibits $\mM$ as free cocompletion under small weighted colimits.
\end{enumerate}	
\end{corollary}

\subsection{An enriched adjoint functor theorem}

In Proposition \ref{remqalo} we proved that enriched left adjoints preserve weighted colimits and dually that enriched right adjoints preserve weighted limits. In the following we prove a converse, an enriched version of the adjoint functor theorem: % (Theorem \ref{remwq}). 

\begin{lemma}\label{pressel}
Let $ \mM^\circledast \to \mV^\ot \times \mW^\ot$ be a weakly bienriched $\infty$-category.
The following are equivalent:

\begin{enumerate}
\item $\mM$ is presentable and admits small weighted colimits.

\item $\mM$ is presentable and admits left and right tensors and
for every $\V \in \mV$ the left tensor $\V \ot (-): \mM \to \mM$ preserves small colimits and for every $\W \in \mW$ the right tensor $(-) \ot \W : \mM \to \mM$ preserves small colimits.

\item $\mM$ is presentable and admits small weighted limits.

\item $\mM$ is presentable and admits left and right cotensors and
for every $\V \in \mV$ the left cotensor ${^\V(-)}: \mM \to \mM$ preserves small limits and for every $\W \in \mW$ the right cotensor $(-)^\W : \mM \to \mM$ preserves small limits.

\end{enumerate}

\end{lemma}

\begin{proof}
The equivalence between (1) and (2) is Proposition \ref{weeei} (1) and the equivalence between (3) and (4) is the opposite of Proposition \ref{weeei} (1).
To see that (2) and (4) are equivalent, by \cite[Corollary 4.22. (7)]{heine2024bienriched} we can assume that $ \mM^\circledast \to \mV^\ot \times \mW^\ot$ is a bipseudo-enriched $\infty$-category. In this case the equivalence between (2) and (4) follows from the adjoint functor theorem \cite[Corollary 5.5.2.9]{lurie.HTT}.
	
\end{proof}

\begin{definition}
\begin{itemize}
\item A weakly bienriched $\infty$-category $ \mM^\circledast \to \mV^\ot \times \mW^\ot$ is presentable if it admits one of the equivalent conditions of Lemma \ref{pressel}.
	
\item A bipseudo-enriched $\infty$-category $ \mM^\circledast \to \mV^\ot \times \mW^\ot$ is presentable if it admits one of the equivalent conditions of Lemma \ref{pressel}.

\item A bienriched $\infty$-category $ \mM^\circledast \to \mV^\ot \times \mW^\ot$ is presentable if it admits one of the equivalent conditions of Lemma \ref{pressel}.
	
\end{itemize}
\end{definition}	
\begin{example}

A weakly bienriched $\infty$-category $  \mM^\circledast \to \mV^\ot \times \mW^\ot$ is a presentable bipseudo-enriched $\infty$-category if and only if it is a bitensored $\infty$-category, $\mM$ is presentable and for every $\V \in \mV$ the left tensor $\V \ot (-): \mM \to \mM$ preserves small colimits and for every $\W \in \mW$ the right tensor $(-) \ot \W : \mM \to \mM$ preserves small colimits.

\end{example}
\begin{example}
A weakly bienriched $\infty$-category $  \mM^\circledast \to \mV^\ot \times \mW^\ot$ is a presentably bitensored $\infty$-category if and only if it is a presentable bienriched $\infty$-category and $\mV^\ot \to \Ass, \mW^\ot \to \Ass$ are presentably monoidal $\infty$-categories.

\end{example}

\begin{theorem} (Enriched adjoint functor theorem) 
\label{remwq}
Let $\mM^\circledast \to \mV^\ot \times \mW^\ot, \mN^\circledast \to \mV^\ot \times \mW^\ot $ be a presentable weakly bienriched $\infty$-categories and $\phi: \mM^\circledast \to \mN^\circledast $ a $\mV, \mW$-enriched functor.

\begin{enumerate}
\item %If $\mM$ admits small weighted colimits, 
Then $\phi: \mM^\circledast \to \mN^\circledast$ admits a $\mV, \mW$-enriched right adjoint if and only if $\phi$ preserves small weighted colimits.

\vspace{1mm}

\item %If $\mM$ admits small weighted limits, 
Then $\phi: \mM^\circledast \to \mN^\circledast$ admits a $\mV, \mW$-enriched left adjoint if and only if $\phi$ preserves small weighted limits and the underlying functor $\mM \to \mN$ of $\phi$ is accessible.\end{enumerate}
\end{theorem}

\begin{proof}
(1): If $\phi$ admits a $\mV, \mW$-enriched right adjoint, it  preserves all weighted colimits
by Proposition \ref{remqalo} (1).
If $\phi$ preserves small weighted colimits, it preserves left and right tensors and small conical colimits. Since $\mM$ admits small weighted colimits, every small
%for every $\V \in \mV$ the left tensor $\V \ot (-): \mM \to \mM$ preserves small colimits and for every $\W \in \mW$ the right tensor $(-) \ot \W : \mM \to \mM$ preserves small colimits, every
colimit in $\mM$ is conical. Hence the functor underlying $\phi$ preserves small colimits and so by the adjoint functor theorem \cite[Corollary 5.5.2.9]{lurie.HTT} admits a right adjoint. So the result follows from Lemma \ref{Adj}, where we use that $\mM$ admits left and right tensors and $\phi$ preserves left and right tensors.

(2): If $\phi$ admits a $\mV, \mW$-enriched left adjoint, it  preserves all weighted limits
by Proposition \ref{remqalo} (2).
If $\phi$ admits a $\mV, \mW$-enriched left adjoint, the functor underlying $\phi$ admits a left adjoint and so is accessible.
 
If $\phi$ preserves small weighted limits, it preserves left and right cotensors and small conical limits. Since $\mM$ admits small weighted limits, every small limit in $\mM$ is conical.
%Since for every $\V \in \mV$ the left cotensor ${^\V(-)}: \mM \to \mM$ preserves small limits and for every $\W \in \mW$ the right cotensor $(-)^\W : \mM \to \mM$ preserves small limits, every limit is conical. 
Hence the functor underlying $\phi$ preserves small limits and is accessible and so by the adjoint functor theorem \cite[Corollary 5.5.2.9]{lurie.HTT} admits a left adjoint. So the result follows from the opposite of Lemma \ref{Adj}, where we use that $\mM$ admits left and right cotensors and
$\phi$ preserves left and right cotensors.

\end{proof}

\subsection{Adjoining weighted colimits}

In this section we prove that weighted colimits can be universally adjoined
(Theorem \ref{wooond}).
We start with the following lemmas:

\begin{lemma}\label{small}

Let $\mM^\circledast \to \mV^\ot \times \mW^\ot$ be a locally small weakly bienriched $\infty$-category that admits $\mH$-weighted colimits for some set $\mH$ of small weights over $\mV,\mW$ and $\mB \subset \mM$ a small full subcategory.
The full subcategory of $\mM$ generated by $\mB$ under $\mH$-weighted colimits is small.
% if $\mV^\ot \to \Ass, \mW^\ot \to \Ass$ are locally small.
% andthe multi-morphism spaces of $\mV^\ot \to \Ass, \mW^\ot \to \Ass, \mM^\circledast \to \mV^\ot \times \mW^\ot $ are small.
%and $\mH$ consists of small weights.

%one of the following conditions is satisfied:\begin{enumerate}\item The $\infty$-operads $\mV^\ot \to \Ass, \mW^\ot \to \Ass$ are small,the multi-morphism spaces of $\mM^\circledast \to \mV^\ot \times \mW^\ot $ are smalland $\mH$ consists of small weights.\item The $\infty$-operad $\mV^\ot \to \Ass$ is a presentably monoidal $\infty$-category, $\mM^\circledast \to \mV^\ot \times \mW^\ot$ is a left enriched $\infty$-category and $\mH$ consists of small left enriched weights over $\mV, \mW$.\item The $\infty$-operad $\mW^\ot \to \Ass$ is a presentably monoidal $\infty$-category, $\mM^\circledast \to \mV^\ot \times \mW^\ot$ is a right enriched $\infty$-category and $\mH$ consists of small right enriched weights over $\mV, \mW$.\item The $\infty$-operads $\mV^\ot \to \Ass, \mW^\ot \to \Ass$ are presentably monoidal $\infty$-categories, $\mM^\circledast \to \mV^\ot \times \mW^\ot$ is a bienriched $\infty$-category and $\mH$ consists of small enriched weights over $\mV, \mW$.\end{enumerate}

\end{lemma}

\begin{proof}
We need to see that the collection of equivalence classes of the full subcategory of $\mM$ generated by $\mB$ under $\mH$-weighted colimits is a set.
Since $\mH$ is a set and for every weight $(\alpha, \beta, \mJ^\circledast \to \mV'^\ot \times \mW'^\ot, \rH \in \mP\B\Env(\mJ)) \in \mH$
the $\infty$-category $\mJ$ is small, there is a small regular cardinal $\kappa$ such that for every weight $(\alpha, \beta, \mJ^\circledast \to \mV'^\ot \times \mW'^\ot, \rH \in \mP\B\Env(\mJ)) \in \mH$ the collection of equivalence classes of objects of $\mJ$ is $\kappa$-small. 
We will first prove that the full subcategory of $\mM$ generated by $\mB$ under $\mH$-weighted colimits is the $\kappa$-filtered colimit of the 
sequence $(\mB_\lambda)_{\lambda < \kappa}$ of full subcategories of $\mM$ inductively defined by:
\begin{enumerate}
\item $\mB_0 := \mB,$
\item $\mB_{\lambda +1} $ is the union of $\mB_\lambda$ and the full subcategory of $\mM$ of $\rH$-weighted colimits for some weight $(\alpha, \beta, \mJ^\circledast \to \mV'^\ot \times \mW'^\ot, \rH \in \mP\B\Env(\mJ)) \in \mH$
of all enriched functors $\F: \mJ^\circledast \to \mM^\circledast$, where the collection of equivalence classes in $\mJ$ is $\lambda$-small and $\F$ lies over $\alpha, \beta$ and carries $\mJ$ to $\mB_\lambda$.

\item If $\lambda < \kappa$ is a limit ordinal, $\mB_\lambda := \cup_{\lambda' < \lambda} \mB_{\lambda'}.$ 
\end{enumerate}

Indeed, for any enriched functor $\F: \mJ^\circledast \to \mM^\circledast$ lying over $\alpha, \beta$ that carries $\mJ$ to $\colim_{\lambda < \kappa} \mB_\lambda$ for some weight $(\alpha, \beta, \mJ^\circledast \to \mV'^\ot \times \mW'^\ot, \rH \in \mP\B\Env(\mJ)) \in \mH$ the collection of equivalence classes of objects of $\mJ$ is $\kappa$-small and so a $\kappa$-compact object in the category of small sets. 
Since taking equivalence classes preserves filtered colimits 
as a functor from $\infty$-categories to sets, there is a $\lambda < \kappa$ such that $\F$ carries $\mJ$ to $\mB_\lambda$.
So the $\rH$-weighted colimit of $\F$ belongs to $\mB_{\lambda +1}.$
This shows that $\colim_{\lambda < \kappa} \mB_\lambda \subset \mM$ is the smallest full subcategory of $\mM$ containing $\mB$ and closed under $\mH$-weighted colimits.

Therefore we have to see that the colimit $\colim_{\lambda < \kappa} \mB_\lambda $ is small. We use induction:
by assumption $\mB_0=\mB$ is small. The case of a limit ordinal is clear using
that $\kappa$ is small.
For the case of a successor ordinal assume that $\lambda < \kappa$ and $\mB_\lambda$ is small. We want to see that $\mB_{\lambda+1} $ is small.
We fix the following notation: for any weight $(\alpha, \beta, \mJ^\circledast \to \mV'^\ot \times \mW'^\ot, \rH \in \mP\B\Env(\mJ)) \in \mH$ let $\mB_{\lambda+1}^\rH$ be the full subcategory of $\mM$ spanned by 
the objects of $\mB_\lambda$ and the $\rH$-weighted colimits of all enriched functors $\F: \mJ^\circledast \to \mM^\circledast$ that lie over
$\alpha, \beta$ and carry $\mJ$ to $\mB_\lambda.$
Then $\mB_{\lambda+1} \subset \mM $ is the union $\cup_{\rH \in \mH} \mB_{\lambda+1}^\rH$. Since $\mH$ is a set, it is enough to see that
$\mB_{\lambda+1}^\rH$ is small.
So it is enough to see that the full subcategory of $\mM$ spanned by the objects $\colim^\rH(\F)$ for $\F \in \Enr\Fun_{\mV', \mW'}(\mJ, (\alpha,\beta)^*(\mB_{\lambda})) $ is small, where $ \mB_\lambda^\circledast \subset \mM^\circledast$ is the full weakly bienriched subcategory. 
This follows from the fact that the $\infty$-category %$\Enr\Fun_{\mV, \mW}((\alpha,\beta)_!(\mJ), \mB_{\lambda}) $ 
$\Enr\Fun_{\mV', \mW'}(\mJ, (\alpha,\beta)^*(\mB_{\lambda})) $ is small because $(\alpha,\beta)^*(\mB_{\lambda})^\circledast \to \mV'^\ot \times \mW'^\ot $ and $\mJ^\circledast \to \mV'^\ot \times \mW'^\ot$ are small. 
The first is small since the multi-morphism spaces of $\mM^\circledast \to \mV^\ot \times \mW^\ot $ and so $\mB_\Lambda^\circledast \to \mV^\ot \times \mW^\ot  $ are small so that also the multi-morphism spaces of $(\alpha,\beta)^*(\mB_{\lambda})^\circledast \to \mV'^\ot \times \mW'^\ot $ are small,
and $\mV',\mW ', \mB_\Lambda$ are small.

%by Corollary \ref{presy}. % under one of the assumptions.

\end{proof}

\begin{lemma}\label{leuz}
Let $\mM^\circledast \to \mV^\ot \times \mW^\ot$ be a small weakly bienriched $\infty$-category, $\mH$ a set of absolute small weights over $\mV, \mW$ and $\Lambda$ a collection of $\mH$-weighted diagrams in $\mM$. Then $\Lambda$ is a set.
%\begin{enumerate}
%\item $\mM^\circledast \to \mV^\ot \times \mW^\ot$ is an absolute small weakly  bienriched $\infty$-category and $\mH$ consists of small weights. %for every weight $(\alpha, \beta, \mJ^\circledast \to \mQ^\ot \times \mU^\ot, \rH) \in \mH$ the weakly bienriched $\infty$-category$\mJ^\circledast \to \mQ^\ot \times \mU^\ot$ is small.\vspace{1mm}

%\item $\mM^\circledast \to \mV^\ot \times \mW^\ot$ is a small $\infty$-category  bienriched in presentably monoidal $\infty$-categories and $\mH$ consists of small enriched weights. 
%for every weight $(\alpha, \beta, \mJ^\circledast \to \mQ^\ot \times \mU^\ot, \rH) \in \mH$ the weakly bienriched $\infty$-category$\mJ^\circledast \to \mQ^\ot \times \mU^\ot$ is a small $\infty$-category bienriched in presentably monoidal $\infty$-categories and $\rH \in \mP\Env(\mJ)$ belongs to$\mP\mB\Env(\mJ)_{\B\mathrm{enr}} \subset \mP\mB\Env(\mJ).$\end{enumerate}	

\end{lemma}

\begin{proof}

Let $(\alpha, \beta, \mJ^\circledast \to \mQ^\ot \times \mU^\ot, \rH) \in \mH$ be a weight.
By assumption $\mJ^\circledast\to \mQ^\ot \times \mU^\ot$ is an absolute small weakly  bienriched $\infty$-category so that the $\infty$-category
$\Enr\Fun_{\mQ, \mU}(\mJ, (\alpha, \beta)^*(\mM))$ is small.
Since $\mM$ is small and for every enriched functor $\F: \mJ^\circledast \to \mM^\circledast$ lying over $\alpha, \beta$ and $\Y \in \mM$ the space $\mP\mB\Env(\mJ)(\rH, \F^*(\Y))$ is small, $\Lambda$ is small.

\end{proof}

%One proves the next lemma the similar way via  and Proposition \ref{rembrako}:

\begin{lemma}\label{leuzos}
\begin{enumerate}
\item Let $\mV^\ot \to \Ass$ be a presentably monoidal
$\infty$-category, $\mW^\ot \to \Ass$ a small $\infty$-operad, $\mM^\circledast \to \mV^\ot \times \mW^\ot$ a small left enriched $\infty$-category, $\mH$ a set of small left enriched weights over $\mV, \mW$ and $\Lambda$ a collection of $\mH$-weighted left enriched diagrams in $\mM$. Then $\Lambda$ is a set.

\item Let $\mV^\ot \to \Ass, \mW^\ot \to \Ass$ be presentably monoidal
$\infty$-categories, $\mM^\circledast \to \mV^\ot \times \mW^\ot$ a small bienriched $\infty$-category, $\mH$ a set of small enriched weights over $\mV, \mW$ and $\Lambda$ a collection of $\mH$-weighted enriched diagrams in $\mM$. Then $\Lambda$ is a set.
\end{enumerate}	
	
\end{lemma}

\begin{proof} Let $(\alpha, \beta, \mJ^\circledast \to \mQ^\ot \times \mU^\ot, \rH) \in \mH$. Then $\Enr\Fun_{\mQ, \mU}(\mJ, (\alpha, \beta)^*(\mM))$ is small by \cite[Corollary 4.25.]{heine2024bienriched}.
Since $\mM$ is small and for every enriched functor $\F: \mJ^\circledast \to \mM^\circledast$ lying over $\alpha, \beta$ and $\Y \in \mM$ the spaces $\mP\mB\Env(\mM)_{\L\Enr}(\F_!(\rH), \Y), \mP\mB\Env(\mM)_{\B\Enr}(\F_!(\rH), \Y)$ are small under the assumptions of (1) and (2), respectively, Proposition \ref{rembrako}, also $\Lambda$ is small.
	
\end{proof}

\begin{notation}\label{Notali}
Let $(\mV^\ot \to \Ass, \rS)$, $(\mW^\ot \to \Ass, \T)$	be small localization pairs,  $\mM^\circledast \to \mV^\ot \times \mW^\ot$ a small $\rS, \T$-bienriched  $\infty$-category and $\Lambda$ a set of diagrams in $\mM$.
\begin{enumerate}
\item Let $\mQ_\Lambda$ be the set of morphisms $\theta: \colim^{\rH}(\iota \circ \F) \to \iota(\Y)$ in $\mP\B\Env(\mM)_{\rS,\T}$ adjoint to the morphism $\rH \xrightarrow{\lambda} \F^\ast(\Y) \to (\iota\circ \F)^\ast(\iota(\Y))$
in $\mP\B\Env(\mJ)$ for some diagram $(\F:\mJ^\circledast \to \mM^\circledast, \Y, \lambda:  \rH \to \F^*(\Y)) \in \Lambda$, where $\iota: \mM^\circledast \to \mP\B\Env(\mM)_{\rS,\T}^\circledast$ is the canonical embedding.

\item Let $\mQ'_\Lambda$ be the set of morphisms in $\mP\B\Env(\mM)_{\rS,\T}^\circledast$ of the form $\V_1 \ot ... \ot \V_\n \ot \f \ot \W_1 \ot ... \ot \W_\m$ for $\f \in \mQ_\Lambda$ and $\V_1 ,..., \V_\n \in \mV, \mW_1, ..., \W_\m \in \mW$ and $\n, \m \geq 0.$
\end{enumerate}
\end{notation}

\begin{notation}
Let $(\mV^\ot \to \Ass, \rS)$, $(\mW^\ot \to \Ass, \T)$	be small localization pairs,  $\mM^\circledast \to \mV^\ot \times \mW^\ot$ a small $\rS, \T$-bienriched  $\infty$-category and $\Lambda$ a set of diagrams in $\mM$.

\begin{enumerate}
\item Let $$\mP\B\Env_\Lambda(\mM)_{\rS,\T}^\circledast \subset \mP\B\Env(\mM)_{\rS,\T}^\circledast $$
be the full weakly bienriched subcategory spanned the $\mQ'_\Lambda$-local objects.

%by all $\X$ such that for every morphism $\colim^{\rH}(\iota \circ \F) \to \iota(\Y)$ in $\mP\B\Env(\mM)_{\rS,\T}$ adjoint to a morphism $$\rH \xrightarrow{\lambda} \F^\ast(\Y) \to (\iota\circ \F)^\ast(\iota(\Y))$$in $\mP\B\Env(\mJ)$ for some diagram $(\F:\mJ^\circledast \to \mM^\circledast, \Y, \lambda:  \rH \to \F^*(\Y)) \in \Lambda$the induced morphism $$\Mor_{\mP\B\Env(\mM)_{\rS,\T}}(\iota(\Y), \X)  \to \Mor_{\mP\B\Env(\mM)_{\rS,\T}}(\colim^{\rH}(\iota \circ \F), \X)$$ is an equivalence.

\item Let $$\mP\widetilde{\B\Env}_\Lambda(\mM)_{\rS,\T}^\circledast := \mV^\ot
\times_{\rS^{-1}\mP\Env(\mV)^\ot} \mP\B\Env_\Lambda(\mM)_{\rS,\T}^\circledast \times_{\T^{-1}\mP\Env(\mW)^\ot} \mW^\ot.$$
\end{enumerate}
\end{notation}

\begin{remark}\label{umbes}
Note that $\mP\B\Env_\Lambda(\mM)_{\rS,\T}^\circledast \subset \mP\B\Env(\mM)_{\rS,\T}^\circledast $ is the full weakly bienriched subcategory spanned by all $\X$ such that the $\mW, \mV$-enriched functor
$\Mor_{\mP\B\Env(\mM)_{\rS,\T}}((-)_{\mid\mM},\X): (\mM^\op)^\circledast \to (\rS^{-1}\mP\Env(\mV) \otimes \T^{-1}\mP\Env(\mW))^\circledast$
sends every diagram of $ \Lambda$ to a weighted limit diagram.	
	
\end{remark}

\begin{lemma}\label{lolet}
Let $(\mV^\ot \to \Ass, \rS)$, $(\mW^\ot \to \Ass, \T)$	be small localization pairs,  $\mM^\circledast \to \mV^\ot \times \mW^\ot$ a small $\rS, \T$-bienriched  $\infty$-category and $\Lambda$ a set of diagrams in $\mM$.

\begin{enumerate}
\item The full weakly bienriched subcategory $\mP\B\Env_\Lambda(\mM)_{\rS,\T}^\circledast \subset \mP\B\Env(\mM)_{\rS,\T}^\circledast $ is an accessible $\rS^{-1}\mP\Env(\mV), \T^{-1}\mP\Env(\mW) $-enriched localization.

\item The restricted localization functor $\mM^\circledast \subset \mP\B\Env(\mM)_{\rS,\T}^\circledast \xrightarrow{\L} \mP\B\Env_\Lambda(\mM)_{\rS,\T}^\circledast $ %relative to $\rS^{-1}\mP\Env(\mV)^\ot \times \T^{-1}\mP\Env(\mW)^\ot $
sends diagrams of $\Lambda $ to weighted colimit diagrams.	
\end{enumerate}
\end{lemma}

\begin{proof}
(1): Let $\iota: \mM^\circledast \to \mP\B\Env(\mM)_{\rS,\T}^\circledast$ be the canonical embedding. %Let $\mQ_\Lambda$ be the set of morphisms $\theta: \colim^{\rH}(\iota \circ \F) \to \iota(\Y)$ in $\mP\B\Env(\mM)_{\rS,\T}$ adjoint to the morphism $$\rH \xrightarrow{\lambda} \F^\ast(\Y) \to (\iota\circ \F)^\ast(\iota(\Y))$$in $\mP\B\Env(\mJ)$ for some diagram $(\F:\mJ^\circledast \to \mM^\circledast, \Y, \lambda:  \rH \to \F^*(\Y)) \in \Lambda$.
%Let $\mQ'$ be the set of morphisms in $\mP\B\Env(\mM)_{\rS,\T}^\circledast$ of the form $\V_1 \ot ... \ot \V_\n \ot \f \ot \W_1 \ot ... \ot \W_\m$ for $\f \in \mQ$ and $\V_1 ,..., \V_\n \in \mV, \mW_1, ..., \W_\m \in \mW$ and $\n, \m \geq 0.$
The saturated class generated by $\mQ'_\Lambda$ is closed under small colimits and so closed under the $\rS^{-1}\mP\Env(\mV), \T^{-1}\mP\Env(\mW)$-biaction. 
%which are generated under small colimits by tensor products of $\mV, \mW$, respectively.
Consequently, because $\mP\B\Env(\mM)_{\rS,\T}^\circledast \to \rS^{-1}\mP\Env(\mV)^\ot \times \T^{-1}\mP\Env(\mW)^\ot $ is a presentably bitensored $\infty$-category, the full weakly bienriched subcategory $\mP\B\Env_\Lambda(\mM)_{\rS,\T}^\circledast \subset \mP\B\Env(\mM)^\circledast_{\rS,\T}$ is an accessible $\rS^{-1}\mP\Env(\mV), \T^{-1}\mP\Env(\mW)$-enriched localization. % of $ \mP\B\Env(\mM)^\circledast_{\rS,\T}\to \rS^{-1}\mP\Env(\mV)^\ot \times \T^{-1}\mP\Env(\mW)^\ot $.
% with respect to $\mQ'_\Lambda$.
% which agrees with $\mP\B\Env_\Lambda(\mM)_{\rS,\T}^\circledast\to \rS^{-1}\mP\Env(\mV)^\ot \times \T^{-1}\mP\Env(\mW)^\ot.$

(2): The induced morphism
$\rH \xrightarrow{\lambda} \F^\ast(\Y) \to (\L \circ \iota \circ \F)^\ast(\L(\iota(\Y))) $ is adjoint to the equivalence
$$\colim^{\rH}(\L \circ \iota \circ \F) \simeq \L(\colim^{\rH}(\iota \circ \F)) \xrightarrow{\L(\theta)} \L(\iota(\Y)),$$
where $\theta \in \mQ$, and so exhibits $\L(\iota((\Y))$ as the $\rH$-weighted colimit of $\L \circ \iota \circ \F.$
%In other words $\psi$ sends diagrams of $\Lambda $ to weighted colimit diagrams.	

\end{proof}

\begin{notation}
Let $(\mV^\ot \to \Ass, \rS)$, $(\mW^\ot \to \Ass, \T)$	be small localization pairs,  $\mM^\circledast \to \mV^\ot \times \mW^\ot$ a small $\rS, \T$-bienriched $\infty$-category, $\mH$ a set of absolute small weights over $\mV, \mW$ and $\Lambda$ a collection of $\mH$-weighted diagrams in $\mM$.
By Lemma \ref{leuz} (1) the collection $\Lambda$ is a set.
Let $$\mP\B\Env^\mH_\Lambda(\mM)_{\rS,\T}^\circledast \subset \mP\widetilde{\B\Env}_\Lambda(\mM)_{\rS,\T}^\circledast $$ be the smallest full weakly bienriched subcategory containing the essential image of the restricted $\mV, \mW$-enriched localization functor
$ \mM^\circledast \to \mP\widetilde{\B\Env}_\Lambda(\mM)_{\rS,\T}^\circledast$
and closed under $\mH$-weighted colimits.

If $\rS=\T=\emptyset$, we drop $\rS,\T$ from the notation.
If $\Lambda=\emptyset$, we drop $\Lambda$ from the notation.

\end{notation}

\begin{remark}
We obtain a $\mV, \mW$-enriched functor $\mM^\circledast \to \mP\B\Env^\mH_\Lambda(\mM)_{\rS,\T}^\circledast. $
	
\end{remark}

%\begin{notation}
%Let $(\mV^\ot \to \Ass, \rS)$, $(\mW^\ot \to \Ass, \T)$	be small localization pairs,  $\mM^\circledast \to \mV^\ot \times \mW^\ot$ a small $\rS, \T$-bienriched $\infty$-category, $\kappa$ a regular cardinal and $\Lambda$ a collection of diagrams in $\mM$ whose underlying weights are $\kappa$-small.Let $$\mP\L\Env^\kappa_\Lambda(\mM)_{\rS,\T}^\circledast := \mP\B\Env^{\mH}_\Lambda(\mM)_{\rS,\T}^\circledast,$$$$\mP\R\Env^\kappa_\Lambda(\mM)_{\rS,\T}^\circledast := \mP\B\Env^{\mH'}_\Lambda(\mM)_{\rS,\T}^\circledast$$for $\mH$ the set of $\kappa$-small left weights over $\mV$and $\mH'$ the set of $\kappa$-small right weights over $\mW.$
%\end{notation}

%\begin{remark}By Proposition \ref{weeei} the full weakly bienriched subcategories $$\mP\L\Env^\kappa_\Lambda(\mM)_{\rS,\T}^\circledast, \mP\R\Env^\kappa_\Lambda(\mM)_{\rS,\T}^\circledast \subset  \mP\B\Env_\Lambda(\mM)_{\rS,\T}^\circledast$$are generated under $\kappa$-small colimits and left tensors, right tensors, respectively, by objects of $\mM.$\end{remark}

\begin{remark}\label{bipre}
Lemma \ref{lolet} implies that $\mP\B\Env_\Lambda(\mM)_{\rS,\T}^\circledast \to \rS^{-1}\mP\Env(\mV)^\ot \times \T^{-1}\mP\Env(\mW)^\ot$ is a presentably bitensored $\infty$-category. Hence $\mP\widetilde{\B\Env}_\Lambda(\mM)_{\rS,\T}^\circledast  \to \mV^\ot \times \mW^\ot$ is a $\rS, \T$-bienriched $\infty$-category by Proposition \ref{eqq} that admits small weighted colimits. Thus $\mP\B\Env^\mH_\Lambda(\mM)_{\rS,\T}^\circledast \to \mV^\ot \times \mW^\ot$ is a small $\rS, \T$-bienriched $\infty$-category that admits $\mH$-weighted colimits by Lemma \ref{small}.
	
\end{remark}

\begin{remark}\label{embesto}By Lemma \ref{lolet} (2) the restricted localization functor $\mM^\circledast \to \mP\B\Env_\Lambda(\mM)_{\rS,\T}^\circledast $ and so also $\mM^\circledast \to \mP\B\Env^\mH_\Lambda(\mM)_{\rS,\T}^\circledast $ send diagrams of $\Lambda $ to $\mH$-weighted colimit diagrams.	
%Hence $\psi$ sends diagrams of $\Lambda $ to $\mH$-weighted colimit diagrams.	

\end{remark}

\begin{remark}\label{embes}

If $\Lambda$ consists of weighted colimit diagrams, every object of $\mM$ belongs to $\mP\B\Env_\Lambda(\mM)_{\rS,\T} $.
So the restricted localization functor
$\mM^\circledast \to \mP\B\Env_\Lambda(\mM)_{\rS,\T}^\circledast$
and so $\mM^\circledast \to \mP\B\Env^\mH_\Lambda(\mM)_{\rS,\T}^\circledast$ are embeddings.
	
\end{remark}

For the next notation we use the following terminology:
\begin{definition}
A weakly bienriched $\infty$-category $\mM^\circledast \to \mV^\ot \times \mW^\ot$ exhibits $\mM$ as 
\begin{itemize}
\item locally left pseudo-enriched in $\mV$ if $\mV^\ot \to \Ass$ admits a tensor unit and the pullback $\Ass \times_{\mV^\ot} \mM^\circledast \times_{\mW^\ot} \emptyset^\ot \to \Ass$ along the unique map $\emptyset^\ot \to \mW^\ot$ and the unique map $ \Ass \to \mV^\ot$ preserving the tensor unit exhibits $\mM$ as left pseudo-enriched.
		
\item locally right pseudo-enriched in $\mW$ if $\mW^\ot \to \Ass$ admits a tensor unit and the pullback $\emptyset^\ot \times_{\mV^\ot} \mM^\circledast \times_{\mW^\ot} \Ass \to \Ass$ along the unique map $\emptyset^\ot \to \mV^\ot$ and the unique map $ \Ass \to \mW^\ot$ preserving the tensor unit exhibits $\mM$ as right pseudo-enriched.
		
\item locally bipseudo-enriched in $\mV, \mW$ if $\mV^\ot \to \Ass, \mW^\ot \to \Ass$ admit a tensor unit and the pullback $\Ass \times_{\mV^\ot} \mM^\circledast\times_{\mW^\ot} \Ass \to \Ass \times \Ass$ along the unique maps $\Ass \to \mV^\ot, \Ass \to \mW^\ot$ preserving the tensor unit exhibits $\mM$ as bipseudo-enriched.
\end{itemize}
	
\end{definition}

\begin{proposition}\label{Corall}

Let $\n \geq 1$ and $(\mV_\bi^\ot \to \Ass, \rS_\bi)$, $(\mW_\bi^\ot \to \Ass, \T_\bi)$ small localization pairs for $1 \leq \bi \leq \n$.
Let $\mM_\bi^\circledast \to \mV_\bi^\ot \times \mW_\bi^\ot$ be small $\rS_\bi, \T_\bi$-bienriched $\infty$-categories such that for every $1 \leq \bi \leq \n$ the weakly bienriched $\infty$-category
$\mM_\bi^\circledast \to \mV_\bi^\ot \times \mW_\bi^\ot$ is locally bipseudo-enriched if there is some $1 \leq \bj \leq \n$ with $\bj \neq \bi$ such that $\rS_\bj$ or $ \T_\bj$ do not only contain equivalences.
Let $\mN^\circledast \to \mV^\ot \times \mW^\ot$ be a weakly bienriched $\infty$-category,
where $ \mV^\ot:= \mV_1^\ot \times_\Ass ... \times_\Ass \mV^\ot_\n, \mW^\ot:= \mW_1^\ot \times_\Ass ... \times_\Ass \mW^\ot_\n.$
Assume that for every $1 \leq \bi \leq \n$ the pullback
$\mN_\bi^\circledast := \mV_\bi^\ot \times_{\mV^\ot} \mN^\circledast \times_{\mW^\ot} \mW_\bi^\ot$ along $\tau_\bi: \mV_\bi^\ot \to \mV^\ot, \tau_\bi: \mW_\bi^\ot \to \mW^\ot$ is a $\rS_\bi, \T_\bi$-bienriched $\infty$-category that admits left and right tensors and small colimits preserved by forming left and right tensors.

\begin{enumerate}
\item The induced functor $$ \Enr\Fun_{\mV,\mW}(\prod_{1 \leq \bi \leq \n} \mP\widetilde{\B\Env}(\mM_\bi)_{\rS_\bi,\T_\bi}, \mN) \to \Enr\Fun_{\mV,\mW}(\prod_{1 \leq \bi \leq \n}\mM_\bi, \mN)$$	
admits a fully faithful left adjoint that lands in the full subcategory
$$ \LinFun^{\L,\L}_{\mV,\mW}(\prod_{1 \leq \bi \leq \n} \mP\widetilde{\B\Env}(\mM_\bi)_{\rS_\bi,\T_\bi}, \mN)$$ of $\mV, \mW$-linear functors that admit a right adjoint component-wise.

\item The following induced functor is an equivalence: $$\LinFun^{\L,\L}_{\mV,\mW}(\prod_{1 \leq \bi \leq \n} \mP\widetilde{\B\Env}(\mM_\bi)_{\rS_\bi,\T_\bi}, \mN) \to \Enr\Fun_{\mV,\mW}(\prod_{1 \leq \bi \leq \n}\mM_\bi, \mN).$$
\end{enumerate}

\end{proposition}

\begin{proof}
(2) follows immediately from (1) because the functor of (2) is conservative.	
(1): The existence of a fully faithful left adjoint in (1) follows from \cite[Proposition 2.65.]{heine2024bienriched}.
Let $1 \leq \bi \leq \n$ and $\X:=(\X_\bj \in \mP\widetilde{\B\Env}(\mM_\bj)_{\rS_\bj,\T_\bj})_{1 \leq \bj \leq \n, \bj \neq \bi}$
a family of objects.
We like to see that the composition
$$\Enr\Fun_{\mV,\mW}(\prod_{1 \leq \bi \leq \n}\mM_\bi, \mN) \to 
\Enr\Fun_{\mV,\mW}(\prod_{1 \leq \bi \leq \n} \mP\widetilde{\B\Env}(\mM_\bi)_{\rS_\bi,\T_\bi}, \mN) $$$$\xrightarrow{\ev_\X} \Enr\Fun_{\mV_\bi,\mW_\bi}(\mP\widetilde{\B\Env}(\mM_\bi)_{\rS_\bi,\T_\bi}, \mN_\bi)$$
lands in $\LinFun^{\L}_{\mV_\bi,\mW_\bi}(\mP\widetilde{\B\Env}(\mM_\bi)_{\rS_\bi,\T_\bi}, \mN_\bi).$
By \cite[Proposition 8.29. (5), Proposition 3.42.]{HEINE2023108941} the latter functor factors as
$$\Enr\Fun_{\mV,\mW}(\prod_{1 \leq \bi \leq \n}\mM_\bi, \mN) \to \Enr\Fun_{\mV,\mW}(\prod_{1 \leq \bj < \bi}\mP\widetilde{\B\Env}(\mM_\bj)_{\rS_\bj,\T_\bj} \times \mM_\bi \times \prod_{\bi < \bj \leq \n}\mP\widetilde{\B\Env}(\mM_\bj)_{\rS_\bj,\T_\bj}, \mN) \xrightarrow{\ev_\X} $$$$
\Enr\Fun_{\mV_\bi, \mW_\bi}(\mM_\bi, \mN_\bi) \to \Enr\Fun_{\mV_\bi,\mW_\bi}(\mP\widetilde{\B\Env}(\mM_\bi)_{\rS_\bi,\T_\bi}, \mN_\bi).$$
So we can reduce to the case $\n=1,$ where we use that $\mN_\bi^\circledast \to \mV_\bi^\ot \times \mW_\bi^\ot$ is a $\rS_\bi, \T_\bi$-bienriched $\infty$-category that admits left and right tensors and small colimits preserved by forming left and right tensors.

So let $\n=1.$
By Corollary \ref{eqq} and \cite[Lemma 4.31.]{heine2024bienriched} the functor $\mN^\circledast \to \mV^\ot \times \mW^\ot$ is the pullback of a unique bitensored $\infty$-category $ \bar{\mN}^\circledast \to \rS^{-1}\mP\Env(\mV)^\ot \times \T^{-1}\mP\Env(\mW)^\ot$ compatible with small colimits.
By Corollary \ref{eqq} the functor $$\beta: \Enr\Fun_{\rS^{-1}\mP\Env(\mV), \T^{-1}\mP\Env(\mW)}(\mP\B\Env(\mM)_{\rS,\T},\bar{\mN}) \to \Enr\Fun_{\mV,\mW}(\mP\widetilde{\B\Env}(\mM)_{\rS,\T},\mN)$$ is an equivalence.
Since $\mP\Env(\mV)$ is generated by $\mV$ under small colimits and the tensor product, also the localization $\rS^{-1}\mP\Env(\mV)$ is generated by $\mV$ under small colimits and the tensor product. This implies that the functor
$$\alpha: \LinFun^\L_{\rS^{-1}\mP\Env(\mV), \T^{-1}\mP\Env(\mW)}(\mP\B\Env(\mM)_{\rS,\T},\bar{\mN}) \to \LinFun^\L_{\mV,\mW}(\mP\widetilde{\B\Env}(\mM)_{\rS,\T},\mN) $$
is the pullback of the functor $\beta$ and so an equivalence, too. 
By Proposition \ref{pseuso} the composition $$\Enr\Fun_{\rS^{-1}\mP\Env(\mV), \T^{-1}\mP\Env(\mW)}(\mP\B\Env(\mM)_{\rS,\T},\bar{\mN}) \xrightarrow{\beta} \Enr\Fun_{\mV,\mW}(\mP\widetilde{\B\Env}(\mM)_{\rS,\T},\mN) $$$$ \to \Enr\Fun_{\mV,\mW}(\mM,\mN)$$
admits a fully faithful left adjoint that lands in the full subcategory
$$\LinFun^\L_{\rS^{-1}\mP\Env(\mV), \T^{-1}\mP\Env(\mW)}(\mP\B\Env(\mM)_{\rS,\T},\bar{\mN}).$$
%So the result follows.	

\end{proof}

\begin{notation}\label{zzzzo}

For every $\n \geq 1$ and $1 \leq \bi \leq \n$ let $\mV_\bi^\ot \to  \Ass,
\mV_\bi^\ot \to  \Ass$ be small $\infty$-operads and $\mH_\bi$ a set of weights over $\mV_\bi, \mW_\bi$.
Assume that for every $1 \leq \bi \leq \n$ the $\infty$-operad $\mV_\bi^\ot \to  \Ass$ admits a tensor unit if for every $1 \leq \bj \leq \n, \bj \neq \bi$ the set $\mH_\bj$ does not only consist of right weights.
Assume that for every $1 \leq \bi \leq \n$ the $\infty$-operad $\mW_\bi^\ot \to  \Ass$ admits a tensor unit if for every $1 \leq \bj \leq \n, \bj \neq \bi$ the set $\mH_\bj$ does not only consist of left weights.
We define a set $\underline{\mH}_\bi$ of weights over $\mV_1 \times... \times \mV_\n, \mW_1 \times... \times \mW_\n$ the following way:
if $\mH_\bi$ does not only consist of left weights and does not only consist of right weights, $\underline{\mH}_\bi$ is the image of $\mH_\bi$ under
the maps of $\infty$-operads (using Lemma \ref{preten}):
\begin{equation}\label{qqo}
\{\tu\} \times... \times \{\tu\} \times \mV_\bi \times \{\tu\} \times ... \times \{\tu\} \to \mV_1 \times... \times \mV_\n, \end{equation}
\begin{equation}\label{qqoo}
\{\tu\} \times... \times \{\tu\} \times \mW_\bi \times \{\tu\} \times ... \times \{\tu\} \to \mW_1 \times... \times \mW_\n.\end{equation}
If $\mH_\bi$ only consists of left weights, $\underline{\mH}_\bi$ is the image of $\mH_\bi$ under the map of $\infty$-operads (\ref{qqo}).
If $\mH_\bi$ only consists of right weights, $\underline{\mH}_\bi$ is the image of $\mH_\bi$ under the map of $\infty$-operads (\ref{qqoo}).
If $\mH_\bi$ is empty, $\underline{\mH}_\bi$ is empty.
Let $$\mH_1 \boxtimes...\boxtimes\mH_\n:= \bigcup_{\bi=1}^\n \underline{\mH}_\bi.$$
If $\n=0$, we agree that $\mH_1 \boxtimes...\boxtimes\mH_\n $ is empty. 

%We make a dual definition for right weights.
%For every sets $\mH_\bi$ of weights over $\mV_\bi, \mW_\bi$ for $1 \leq \bi \leq \n$ we set $$\mH_1 \boxtimes...\boxtimes\mH_\n:= \bigcup_{\bi=1}^\n \omega(\theta^\bi, \theta^\bi)(\mH_\bi) \in \omega(\mV_1 \times... \times \mV_\n, \mW_1 \times... \times \mW_\n). $$If $\n=0$, we agree that $\mH_1 \boxtimes...\boxtimes\mH_\n \in \omega(\ast,\ast) $ is empty.For every sets $\mH_\bi$ of left weights over $\mV_\bi$ for$1 \leq \bi \leq \n$ we define a set $$\mH_1 \boxtimes...\boxtimes\mH_\n:= \bigcup_{\bi=1}^\n \omega(\theta^\bi, \emptyset)(\mH_\bi) \in \omega(\mV_1 \times... \times \mV_\n, \emptyset).$$If $\n=0$, we agree that $\mH_1 \boxtimes...\boxtimes\mH_\n \in \omega(\ast,\emptyset) $ is empty. We make a dual definition for right weights.
\end{notation}

\begin{notation}
For every $\n \geq 1$ and $1 \leq \bi \leq \n$ let $\mM_\bi^\circledast \to \mV_\bi^\ot \times \mW_\bi^\ot$ be an absolute small weakly  bienriched $\infty$-category that admits $\mH_\bi$-weighted colimits for some set $\mH_\bi$ of weights over $\mV_\bi, \mW_\bi$ and $\Lambda_\bi$ a collection of $\mH_\bi$-weighted diagrams in $\mM_\bi$.
Assume that for every $1 \leq \bi \leq \n$ the weakly bienriched $\infty$-category $\mM_\bi^\circledast \to \mV_\bi^\ot \times \mW_\bi^\ot$ is a weakly locally bipseudo-enriched $\infty$-category if for every $1 \leq \bj \leq \n, \bj \neq \bi$ the set $\mH_\bj$ does not only consist of left weights and does not only consist of right weights.
Assume that for every $1 \leq \bi \leq \n$ the weakly bienriched $\infty$-category $\mM_\bi^\circledast \to \mV_\bi^\ot \times \mW_\bi^\ot$ is a locally left (right) pseudo-enriched $\infty$-category if for every $1 \leq \bj \leq \n, \bj \neq \bi$ the set $\mH_\bj$ only consists of left (right) weights.
For every $1 \leq \bi \leq \n$ we define a set $\underline{\Lambda}_\bi$ of $\rH$-weighted diagrams in $\mM_1 \times ... \times \mM_\n $ for some $\rH \in \underline{\mH}_\bi$ the following way:
%if $\Lambda_\bi$ does not only consist of left weights and does not only consist of right weights, 
$\underline{\Lambda}_\bi$ is the image of $\Lambda_\bi$ under
all enriched functors of the form
\begin{equation}\label{qqol}
\{\X_1\} \times... \times \{\X_{\bi-1}\} \times \mM_\bi \times \{\X_{\bi+1}\} \times ... \times \{\X_\n\} \to \mM_1 \times... \times \mM_\n \end{equation}
for some objects $\X_\bj \in \mM_\bj$ for $1 \leq \bj \leq \n, \bj \neq \bi.$
%If $\Lambda_\bi$ only consists of left diagrams, $\underline{\Lambda}_\bi$ is the image of $\Lambda_\bi$ under all enriched functor (\ref{qqol}).If $\Lambda_\bi$ only consists of right diagrams, $\underline{\Lambda}_\bi$ is the image of $\Lambda_\bi$ under all enriched functors (\ref{qqol}).
If $\Lambda_\bi$ is empty, $\underline{\Lambda}_\bi$ is empty.
Let $$\Lambda_1 \boxtimes ... \boxtimes \Lambda_\n:= \bigcup_{\bi=1}^\n \underline{\Lambda}_\bi.$$
If $\n=0$, we agree that $\Lambda_1 \boxtimes ... \boxtimes \Lambda_\n$ is empty. 
\end{notation}

\begin{notation}
Let $\n \geq 1$ and for every $1 \leq \bi \leq \n$ let $\mV_\bi^\ot \to \Ass $ be a small $\infty$-operad that admits a tensor unit if there is some $1 \leq \bj \leq \n$ with $\bj \neq \bi$ such that $\rS_\bj$ does not contain only equivalences.
Set $ \mV^\ot:= \mV_1^\ot \times_\Ass ... \times_\Ass \mV^\ot_\n$.
For every $1 \leq \bi \leq \n$ let $\rS_\bi$ be a collection of morphisms of $\mP\Env(\mV_\bi)$.
Let $$\rS_1 \boxtimes ... \boxtimes \rS_\n$$ be the collection of morphisms $\g$ of $\mP\Env(\mV)$ such that there is an $1 \leq \bi \leq \n$ and a morphism $\f\in \rS_\bi$ that is sent to $\g$ by the left adjoint monoidal functor $(\tau_\bi)_!:\mP\Env(\mV_\bi)^\ot \to \mP\Env(\mV)^\ot$ induced by the map of $\infty$-operads $\tau_\bi: \mV_\bi^\ot \to \mV^\ot$ that is the identity on the $\bi$-th factor and the constant map preserving the tensor unit on every other factor. 

%Then $(\mV^\ot \to \Ass, \rS_1 \boxtimes ... \boxtimes \rS_\n)$ is a localization pair.A weakly bienriched $\infty$-category $\mM^\circledast \to \mV^\ot \times \mW^\ot$ is left $\rS$-enriched if and only if for every $1 \leq \bi \leq \n$ the pullback along the map of $\infty$-operads $\tau_\bi: \mV_\bi^\ot \to \mV^\ot$ is a left $\rS_\bi$-enriched $\infty$-category.	
\end{notation}

\begin{notation}
Let $\mM^\circledast \to \mV^\ot \times \mW^\ot, \mN^\circledast \to \mV'^\ot \times \mW'^\ot$ be weakly bienriched $\infty$-categories,
%Let $\mM^\circledast \to \mV^\ot \times \mW^\ot,\mN^\circledast \to \mV'^\ot \times \mW'^\ot$ be weakly bienriched $\infty$-categories, 
$\mH$ a collection of weights over $\mV,\mW, $ $\mH'$ a collection of weights over $\mV',\mW'$ and $\Lambda$ a collection of $\mH$-weighted diagrams in $\mM$. Let $\rS$ be a collection of morphisms of $\mP\Env(\mV)$ and $\T$ a collection of morphisms of $\mP\Env(\mW)$.
%\begin{enumerate}
%\item Let $$\Enr\Fun^{\Lambda}_{\mH, \mH'}(\mM, \mN)\subset \Enr\Fun(\mM,\mN) $$ be the full subcategory of enriched functors sending diagrams of $\Lambda$ to weighted colimit diagrams and lying over maps of $\infty$-operads that send weights of $\mH$ to weights of $\mH'$.
%\item Let $$\Enr\Fun_{\rS, \rS',\T, \T'}(\mM, \mN)\subset \Enr\Fun(\mM, \mN) $$ be the full subcategory of enriched functors lying over maps of $\infty$-operads whose induced left adjoint monoidal functors on closed monoidal envelopes send $\rS$ to $\rS'$ and $\T$ to $\T'$, respectively.
Let $$\Enr\Fun^{\Lambda}_{\rS, \rS',\T, \T', \mH, \mH'}(\mM, \mN)\subset \Enr\Fun_{\rS, \rS',\T, \T'}(\mM, \mN) $$ 
be the full subcategory of enriched functors sending diagrams of $\Lambda$ to weighted colimit diagrams and lying over maps of $\infty$-operads that send weights of $\mH$ to weights of $\mH'$.

%be the intersection ofthe full subcategories $\Enr\Fun^{\Lambda}_{\mH, \mH'}(\mM, \mN)$ and$\Enr\Fun_{\rS, \rS',\T, \T'}(\mM, \mN).$\end{enumerate}
\end{notation}

\begin{theorem}\label{wooond}
Let $\n \geq 1$. For $1 \leq \bi \leq \n$ let $(\mV_\bi^\ot \to \Ass, \rS_\bi)$, $(\mW_\bi^\ot \to \Ass, \T_\bi), (\mV'^\ot \to \Ass, \rS'), (\mW'^\ot \to \Ass, \T')$ be small localization pairs, $\mH_\bi$ a set of absolute small weights over $\mV_\bi, \mW_\bi$ and $\mH'$ a collection of weights over $\mV',\mW'$.
Let $\mM_\bi^\circledast \to \mV_\bi^\ot \times \mW_\bi^\ot$ be small $\rS_\bi, \T_\bi$-bienriched $\infty$-categories, $\Lambda_\bi$ a collection of $\mH_\bi$-weighted diagrams in $\mM_\bi$ and $\mN^\circledast \to \mV'^\ot\times \mW'^\ot $ a $\rS',\T'$-bienriched $\infty$-category that admits $\mH'$-weighted colimits. Let $\mH:= \mH_1 \boxtimes... \boxtimes \mH_\n, \rS:= \rS_1 \boxtimes... \boxtimes \rS_\n, \T:= \T_1 \boxtimes... \boxtimes \T_\n.$
Assume that the following conditions hold for every $1 \leq \bi \leq \n$:
\begin{enumerate}
\item %The weakly bienriched $\infty$-category
$\mM_\bi^\circledast \to \mV_\bi^\ot \times \mW_\bi^\ot$ is a locally left pseudo-enriched $\infty$-category if there is some $1 \leq \bj \leq \n$ with $\bj \neq \bi$ such that $\rS_\bj$ does not contain only equivalences or $\Lambda_\bj$ does not only consist of right diagrams.
\item %The weakly bienriched $\infty$-category 
$\mM_\bi^\circledast \to \mV_\bi^\ot \times \mW_\bi^\ot$ is a locally right pseudo-enriched $\infty$-category if there is some $1 \leq \bj \leq \n$ with $\bj \neq \bi$ such that $\T_\bj$ does not contain only equivalences or $\Lambda_\bj$ does not only consist of left diagrams.
\end{enumerate}
%contain weights over $\mV,\mW$.is non-empty or $\rS_\bj$ or $ \T_\bj$ do not contain only equivalences.
% such that for every $1 \leq \bi \leq \n$ the pullback $\mN_\bi^\circledast := \mV_\bi^\ot \times_{\mV^\ot} \mN^\circledast \times_{\mW^\ot} \mW_\bi^\ot$ along $\tau_\bi: \mV_\bi^\ot \to \mV^\ot, \tau_\bi: \mW_\bi^\ot \to \mW^\ot$ is a bi-$\rS_\bi, \T_\bi$-enriched $\infty$-category.

The induced functor $$\gamma_\mN: \Enr\Fun_{\rS,\rS', \T,\T', \mH,\mH'}^{\Lambda^{\mH_1} \boxtimes ... \boxtimes \Lambda^{\mH_\n}}(\prod_{1 \leq \bi \leq \n} \mP\B\Env^{\mH_\bi}_{\Lambda_\bi}(\mM_\bi)_{\rS_\bi,\T_\bi}, \mN) \to  \Enr\Fun_{\rS,\rS', \T,\T', \mH,\mH'}^{\Lambda_1 \boxtimes ... \boxtimes \Lambda_\n}(\prod_{1 \leq \bi \leq \n}\mM_\bi, \mN)$$ is an equivalence.
%If for every $1 \leq \bi \leq \n$ the collection $\Lambda_\bi$ is empty,
The induced functor $$\delta_\mN: \Enr\Fun_{\rS,\rS', \T,\T', \mH,\mH'}(\prod_{1 \leq \bi \leq \n} \mP\B\Env^{\mH_\bi}(\mM_\bi)_{\rS_\bi,\T_\bi}, \mN) \to  \Enr\Fun_{\rS,\rS', \T,\T', \mH,\mH'}(\prod_{1 \leq \bi \leq \n}\mM_\bi, \mN)$$ admits a fully faithful left adjoint that lands in $\Enr\Fun_{\rS,\rS', \T,\T', \mH,\mH'}^{\Lambda^{\mH_1} \boxtimes ... \boxtimes \Lambda^{\mH_\n}}(\prod_{1 \leq \bi \leq \n} \mP\B\Env^{\mH_\bi}_{\Lambda_\bi}(\mM_\bi)_{\rS_\bi,\T_\bi}, \mN).$

\end{theorem}

\begin{proof}
%The latter induced functor is conservative as for every $1 \leq \bi \leq \n$the $\infty$-category $\mP\B\Env^{\mH_\bi}_{\Lambda_\bi}(\mM_\bi)_{\rS_\bi,\T_\bi}$is generated by $\L_\bi(\mM_\bi)$ under $\mH_\bi$-weighted colimits. So the first statement implies the second one.

We first reduce to the case that $\mN^\circledast \to \mV'^\ot\times \mW'^\ot $ admits large weighted colimits.
%Let $$\rS:=\rS_1 \boxtimes ... \boxtimes \rS_\n, \T:=\T_1 \boxtimes ... \boxtimes \T_\n, \mH:= \mH_1 \boxtimes ... \boxtimes \mH_\n$$so that an $\infty$-category weakly bienriched in $\mV,\mW$ is bi-$\rS,\T$-enriched if and only if for for every $1 \leq \bi, \bj \leq \n$ the pullback along the canonical maps of $\infty$-operads $\mV_\bi^\ot \to \mV^\ot, \mW_\bj^\ot \to \mW^\ot$ is bi-$\rS_\bi, \T_\bj$-enriched.
By Remarks \ref{embesto}, \ref{embes} there is a $\mV', \mW'$-enriched embedding $\mN^\circledast \to \widehat{\mP}\widetilde{\B\Env}_{\Lambda^{\mH'}}(\mN)_{\rS', \T'}^\circledast$ that preserves $\mH'$-weighted colimits (Lemma \ref{lolet} (2)) into a bi-$\rS',\T'$-enriched $\infty$-category that admits large weighted colimits (Remark \ref{bipre}).
% and so induces a $\mV, \mW$-enriched embedding $\mN^\circledast \to \mO^\circledast$ that preserves $\mH$-weighted colimits(Corollary \ref{eccco}).
%The canonical embedding $\mN^\circledast \subset \mV^\ot \times_{\rS^{-1}\mP\Env(\mV)^\ot} \mP\B\Env_{\Lambda^{\mH}}(\mN)_{\rS, \T}^\circledast \times_{\T^{-1}\mP\Env(\mW)} \mW^\ot$ induces a $\mV, \mW$-enriched embedding $\mN^\circledast \subset \mO^\circledast$ that preserves $\mH$-weighted colimits by Corollary \ref{eccco}.
The functor $\gamma_\mN$ is the pullback of $\gamma_{\widehat{\mP}\widetilde{\B\Env}_{\Lambda^{\mH'}}(\mN)_{\rS', \T'}}$
because $\mP\B\Env^{\mH_\bi}_{\Lambda_\bi}(\mM_\bi)_{\rS_\bi,\T_\bi}$ is generated by $\L_\bi(\mM_\bi)$ under $\mH_\bi$-weighted colimits. 
Moreover if $\delta_{\mP\widetilde{\B\Env}_{\Lambda^{\mH'}}(\mN)_{\rS', \T'}}$admits a fully faithful left adjoint that lands in $ \Enr\Fun_{\rS,\rS', \T,\T', \mH,\mH'}^{\Lambda^{\mH_1} \boxtimes ... \boxtimes \Lambda^{\mH_\n}}(\prod_{1 \leq \bi \leq \n} \mP\B\Env^{\mH_\bi}(\mM_\bi)_{\rS_\bi,\T_\bi}, \widehat{\mP}\widetilde{\B\Env}_{\Lambda^{\mH'}}(\mN)_{\rS', \T'}),$ the left adjoint sends $\Enr\Fun_{\rS,\rS', \T,\T', \mH,\mH'}(\prod_{1 \leq \bi \leq \n}\mM_\bi, \mN)$ to $\Enr\Fun_{\rS,\rS', \T,\T', \mH,\mH'}^{\Lambda^{\mH_1} \boxtimes ... \boxtimes \Lambda^{\mH_\n}}(\prod_{1 \leq \bi \leq \n} \mP\B\Env^{\mH_\bi}(\mM_\bi)_{\rS_\bi,\T_\bi}, \mN)$ because $\mP\B\Env^{\mH_\bi}(\mM_\bi)_{\rS_\bi,\T_\bi}$ is generated by $\L_\bi(\mM_\bi)$ under $\mH_\bi$-weighted colimits. So in this case the fully faithful left adjoint of $\delta_{\mP\widetilde{\B\Env}_{\Lambda^{\mH'}}(\mN)_{\rS', \T'}}$ restricts to a fully faithful left adjoint of $\delta_\mN$ that lands in $ \Enr\Fun_{\rS,\rS', \T,\T', \mH,\mH'}^{\Lambda^{\mH_1} \boxtimes ... \boxtimes \Lambda^{\mH_\n}}(\prod_{1 \leq \bi \leq \n} \mP\B\Env^{\mH_\bi}(\mM_\bi)_{\rS_\bi,\T_\bi}, \mN).$
This way we can assume that $\mN^\circledast \to \mV'^\ot\times \mW'^\ot $ admits large weighted colimits.

Let $\overline{\mP\B\Env}^{\mH_\bi}_{\Lambda_\bi}(\mM_\bi)_{\rS_\bi,\T_\bi}^\circledast \subset \mP\B\Env^{\mH_\bi}_{\Lambda_\bi}(\mM_\bi)_{\rS_\bi,\T_\bi}^\circledast$ be the preimage of $ \mP\B\Env^{\mH_\bi}_{\Lambda_\bi}(\mM_\bi)_{\rS_\bi,\T_\bi}^\circledast$ under the left adjoint $\L_\bi: \mP\widetilde{\B\Env}(\mM_\bi)_{\rS_\bi,\T_\bi}^\circledast \to \mP\widetilde{\B\Env}_{\Lambda_\bi}(\mM_\bi)_{\rS_\bi,\T_\bi}^\circledast.$
The full subcategory $\overline{\mP\B\Env}^{\mH_\bi}_{\Lambda_\bi}(\mM_\bi)_{\rS_\bi,\T_\bi}^\circledast$ is closed in $ \mP\widetilde{\B\Env}(\mM_\bi)^\circledast_{\rS_\bi,\T_\bi}$ under $\mH_\bi$-weighted colimits because $\mP\B\Env^{\mH_\bi}_{\Lambda_\bi}(\mM_\bi)_{\rS_\bi,\T_\bi}^\circledast$ is closed in $ \mP\widetilde{\B\Env}_{\Lambda_\bi}(\mM_\bi)^\circledast_{\rS_\bi,\T_\bi}$ under $\mH_\bi$-weighted colimits and the left adjoint $\L_\bi$ preserves small weighted colimits.
The localization $\L_\bi: \mP\widetilde{\B\Env}(\mM_\bi)_{\rS_\bi,\T_\bi}^\circledast \rightleftarrows \mP\widetilde{\B\Env}_{\Lambda_\bi}(\mM_\bi)_{\rS_\bi,\T_\bi}^\circledast$ relative to $\mV_\bi^\ot \times \mW_\bi^\ot$ %$\rS_\bi^{-1}\mP\Env(\mV_\bi)^\ot \times \T_\bi^{-1}\mP\Env(\mW_\bi)^\ot$ 
restricts to a $\mV_\bi, \mW_\bi$-localization $\L'_\bi:  \overline{\mP\B\Env}^{\mH_\bi}_{\Lambda_\bi}(\mM_\bi)_{\rS_\bi,\T_\bi}^\circledast \rightleftarrows \mP\B\Env^{\mH_\bi}_{\Lambda_\bi}(\mM_\bi)_{\rS_\bi,\T_\bi}^\circledast$. Precomposition with $\L'_\bi$ induces an embedding 
$$\Enr\Fun_{\rS,\rS', \T,\T', \mH,\mH'}^{\Lambda^{\mH_1} \boxtimes ... \boxtimes \Lambda^{\mH_\n}}(\prod_{1 \leq \bi \leq \n} \mP\B\Env^{\mH_\bi}_{\Lambda_\bi}(\mM_\bi)_{\rS_\bi,\T_\bi}, \mN) \hookrightarrow \Enr\Fun_{\rS,\rS', \T,\T', \mH,\mH'}^{\Lambda^{\mH_1} \boxtimes ... \boxtimes \Lambda^{\mH_\n}}(\prod_{1 \leq \bi \leq \n} \overline{\mP\B\Env}^{\mH_\bi}_{\Lambda_\bi}(\mM_\bi)_{\rS_\bi,\T_\bi}, \mN)$$
whose essential image is the full subcategory $\Enr\Fun_{\rS,\rS', \T,\T', \mH,\mH'}^{\Lambda^{\mH_1} \boxtimes ... \boxtimes \Lambda^{\mH_\n}}(\prod_{1 \leq \bi \leq \n} \overline{\mP\B\Env}^{\mH_\bi}_{\Lambda_\bi}(\mM_\bi)_{\rS_\bi,\T_\bi}, \mN)'$ of enriched functors
%$\widetilde{\mP\B\Env}^\mH_\Lambda(\mM)_{\rS,\T}^\circledast \times_{(\Ass \times \Ass)} \mM'^\circledast \to \mN^\circledast$ such that for every $\X \in \mM'$ the induced $\mV, \mW$-enriched functor $ \widetilde{\mP\B\Env}^\mH_\Lambda(\mM)_{\rS,\T}^\circledast\to \mN^\circledast$ 
inverting all morphisms inverted by $\L:=\prod_{1 \leq \bi \leq \n} \L_\bi.$
The functor $\gamma_{\mN}$ factors as 
$$ \Enr\Fun_{\rS,\rS', \T,\T', \mH,\mH'}^{\Lambda^{\mH_1} \boxtimes ... \boxtimes \Lambda^{\mH_\n}}(\prod_{1 \leq \bi \leq \n} \mP\B\Env^{\mH_\bi}_{\Lambda_\bi}(\mM_\bi)_{\rS_\bi,\T_\bi}, \mN) \simeq \Enr\Fun_{\rS,\rS', \T,\T', \mH,\mH'}^{\Lambda^{\mH_1} \boxtimes ... \boxtimes \Lambda^{\mH_\n}}(\prod_{1 \leq \bi \leq \n} \overline{\mP\B\Env}^{\mH_\bi}_{\Lambda_\bi}(\mM_\bi)_{\rS_\bi,\T_\bi}, \mN)'$$$$ \xrightarrow{\beta} \Enr\Fun_{\rS,\rS', \T,\T', \mH,\mH'}^{\Lambda_1 \boxtimes ... \boxtimes \Lambda_\n}(\prod_{1 \leq \bi \leq \n}\mM_\bi, \mN).$$
The functors $\gamma_{\mN}$ and so $\beta$ are conservative because $\mP\B\Env^{\mH_\bi}_{\Lambda_\bi}(\mM_\bi)_{\rS_\bi,\T_\bi}$ is generated by $\L_\bi(\mM_\bi)$ under $\mH_\bi$-weighted colimits. By \cite[Proposition 5.16. (1)]{heine2024bienriched} the induced functors $$\Enr\Fun_{\rS,\rS', \T,\T'}(\prod_{1 \leq \bi \leq \n} \overline{\mP\B\Env}^{\mH_\bi}_{\Lambda_\bi}(\mM_\bi)_{\rS_\bi,\T_\bi}, \mN) \to \Enr\Fun_{\rS,\rS', \T,\T'}(\prod_{1 \leq \bi \leq \n}\mM_\bi, \mN),$$$$
\Enr\Fun_{\rS,\rS', \T,\T'}(\prod_{1 \leq \bi \leq \n} \mP\widetilde{\B\Env}(\mM_\bi)_{\rS_\bi,\T_\bi}, \mN) \to \Enr\Fun_{\rS,\rS', \T,\T'}(\prod_{1 \leq \bi \leq \n} \overline{\mP\B\Env}^{\mH_\bi}_{\Lambda_\bi}(\mM_\bi)_{\rS_\bi,\T_\bi}, \mN) $$ admit fully faithful left adjoints $\phi, \psi$, respectively, since $\mN^\circledast \to \mV'^\ot \times \mW'^\ot$ admits large weighted colimits.
Consequently, it will be enough to verify that the functor $\phi$ takes  $\Enr\Fun_{\rS,\rS', \T,\T', \mH,\mH'}^{\Lambda_1 \boxtimes ... \boxtimes \Lambda_\n}(\prod_{1 \leq \bi \leq \n}\mM_\bi, \mN)$
to $\Enr\Fun_{\rS,\rS', \T,\T', \mH,\mH'}^{\Lambda^{\mH_1} \boxtimes ... \boxtimes \Lambda^{\mH_\n}}(\prod_{1 \leq \bi \leq \n} \overline{\mP\B\Env}^{\mH_\bi}_{\Lambda_\bi}(\mM_\bi)_{\rS_\bi,\T_\bi}, \mN)'.$ 
Because $\mN^\circledast \to \mV'^\ot \times \mW'^\ot$ is bi-$\rS', \T'$-enriched, by Propositions \ref{Corall} and \ref{ukwa} the induced functor 
$$\Enr\Fun_{\rS,\rS', \T,\T'}(\prod_{1 \leq \bi \leq \n} \mP\widetilde{\B\Env}(\mM_\bi)_{\rS_\bi,\T_\bi}, \mN) \to \Enr\Fun_{\rS,\rS', \T,\T'}(\prod_{1 \leq \bi \leq \n}\mM_\bi, \mN)$$
admits a fully faithful left adjoint $\rho$ that lands in the full subcategory of enriched functors preserving small weighted colimits component-wise, which send diagrams of $\Lambda^{\mH_1} \boxtimes ... \boxtimes \Lambda^{\mH_\n}$ to weighted colimit diagrams.
By adjointness the functor $\rho$ factors as fully faithful left adjoints  $$\Enr\Fun_{\rS,\rS', \T,\T'}(\prod_{1 \leq \bi \leq \n}\mM_\bi, \mN)  \xrightarrow{\phi} \Enr\Fun_{\rS,\rS', \T,\T'}(\prod_{1 \leq \bi \leq \n} \overline{\mP\B\Env}^{\mH_\bi}_{\Lambda_\bi}(\mM_\bi)_{\rS_\bi,\T_\bi}, \mN) $$$$\xrightarrow{\psi} \Enr\Fun_{\rS,\rS', \T,\T'}(\prod_{1 \leq \bi \leq \n} \mP\widetilde{\B\Env}(\mM_\bi)_{\rS_\bi,\T_\bi}, \mN)$$
so that $\phi$ factors as $\rho$ followed by the induced functor $$\Enr\Fun_{\rS,\rS', \T,\T'}(\prod_{1 \leq \bi \leq \n} \mP\widetilde{\B\Env}(\mM_\bi)_{\rS_\bi,\T_\bi}, \mN)  \to \Enr\Fun_{\rS,\rS', \T,\T'}(\prod_{1 \leq \bi \leq \n} \overline{\mP\B\Env}^{\mH_\bi}_{\Lambda_\bi}(\mM_\bi)_{\rS_\bi,\T_\bi}, \mN).$$
Since the embedding $\overline{\mP\B\Env}^{\mH_\bi}_{\Lambda_\bi}(\mM_\bi)_{\rS_\bi,\T_\bi}^\circledast \subset \mP\widetilde{\B\Env}(\mM_\bi)_{\rS_\bi,\T_\bi}$ preserves $\mH_\bi$-weighted colimits, the embedding $\prod_{1 \leq \bi \leq \n} \overline{\mP\B\Env}^{\mH_\bi}_{\Lambda_\bi}(\mM_\bi)_{\rS_\bi,\T_\bi}^\circledast \subset \prod_{1 \leq \bi \leq \n} \mP\widetilde{\B\Env}(\mM_\bi)_{\rS_\bi,\T_\bi}$
preserves diagrams of $\Lambda^{\mH_1} \boxtimes ... \boxtimes \Lambda^{\mH_\n}.$
Thus $\phi$ lands in enriched functors sending diagrams of $\Lambda^{\mH_1} \boxtimes ... \boxtimes \Lambda^{\mH_\n}$ to weighted colimit diagrams.
So we have to prove that an enriched functor $$\alpha: \overline{\mP\B\Env}^{\mH_1}_{\Lambda_1}(\mM_1)_{\rS_1,\T_1}^\circledast \times_{(\Ass \times \Ass)} ... \times_{(\Ass \times \Ass)} \overline{\mP\B\Env}^{\mH_\n}_{\Lambda_\n}(\mM_\n)_{\rS_\n,\T_\n}^\circledast\to \mN^\circledast $$ sending diagrams of $\Lambda^{\mH_1} \boxtimes ... \boxtimes \Lambda^{\mH_\n}$ to weighted colimit diagrams inverts all morphisms inverted by $\L$ if its restriction to $\mM_1^\circledast \times_{(\Ass \times \Ass)} ... \times_{(\Ass \times \Ass)} \mM_\n^\circledast$ sends diagrams of $\Lambda_1 \boxtimes ... \boxtimes \Lambda_\n$ to weighted colimit diagrams. 
%We prove next that the following conditions are equivalent for any $\mV, \mW$-enriched functor $$\alpha: \overline{\mP\B\Env}^{\mH_1}_{\Lambda_1}(\mM_1)_{\rS_1,\T_1}^\circledast \times_{(\Ass \times \Ass)} ... \times_{(\Ass \times \Ass)} \overline{\mP\B\Env}^{\mH_\n}_{\Lambda_\n}(\mM_\n)_{\rS_\n,\T_\n}^\circledast\to \mN^\circledast $$ sending diagrams of $\Lambda^{\mH_1} \boxtimes ... \boxtimes \Lambda^{\mH_\n}$ to weighted colimit diagrams and diagrams of $\Lambda$ to weighted colimit diagrams:\begin{enumerate}\item $\alpha$ inverts morphisms inverted by $\L$.\item $\alpha$ inverts any morphism $(\f_1,...,\f_\n)$ of $\prod_{1 \leq \bi \leq \n} \mP\widetilde{\B\Env}(\mM_\bi)_{\rS_\bi,\T_\bi}$, where $\f_\bi$ is a generating local equivalence of the localization $\mP\B\Env_{\Lambda_\bi}(\mM_\bi)_{\rS_\bi,\T_\bi} \subset \mP\B\Env(\mM_\bi)_{\rS_\bi,\T_\bi}$.\item \end{enumerate}
Let $ \mQ_\bi$ be the set of morphisms $\colim^{\rH}(\iota \circ \F) \to \iota(\Y)$ in $\mP\B\Env(\mM_\bi)_{\rS_\bi,\T_\bi}$ adjoint to a morphism $$\rH \xrightarrow{\lambda} \F^\ast(\Y) \to (\iota\circ \F)^\ast(\iota(\Y))$$
in $\mP\B\Env(\mJ)$ for some diagram $(\F:\mJ^\circledast \to \mM_\bi^\circledast, \Y, \lambda:  \rH \to \F^*(\Y)) \in \Lambda_\bi$, where $\iota: \mM_\bi^\circledast \to \mP\B\Env(\mM_\bi)_{\rS_\bi,\T_\bi}^\circledast$ is the canonical embedding.
By construction of the localization the set $\mQ_\bi$ is the set of generating local 
equivalences of the localization $\mP\B\Env_{\Lambda_\bi}(\mM_\bi)_{\rS_\bi,\T_\bi} \subset \mP\B\Env(\mM_\bi)_{\rS_\bi,\T_\bi}$, which are morphisms of the full subcategory
$\overline{\mP\B\Env}^{\mH_\bi}_{\Lambda_\bi}(\mM_\bi)_{\rS_\bi,\T_\bi} \subset \mP\B\Env(\mM_\bi)_{\rS_\bi,\T_\bi}$ closed under $\mH_\bi$-weighted colimits.
Therefore the enriched functor $\alpha$ inverts morphisms inverted by $\L$ if and only if it inverts any $(\f_1,...,\f_\n) \in \prod_{1 \leq \bi \leq \n} \mQ_\bi.$ 
%of $\prod_{1 \leq \bi \leq \n} \overline{\mP\B\Env}^{\mH_\bi}_{\Lambda_\bi}(\mM_\bi)_{\rS_\bi,\T_\bi}$, where $\f_\bi \in \mQ_\bi$ for $1 \leq \bi \leq \n.$
%is a generating local equivalence of the localization $\mP\B\Env_{\Lambda_\bi}(\mM_\bi)_{\rS_\bi,\T_\bi} \subset \mP\B\Env(\mM_\bi)_{\rS_\bi,\T_\bi}$.Since the generating local equivalences of the localization $\mP\B\Env_{\Lambda_\bi}(\mM_\bi)_{\rS_\bi,\T_\bi} \subset \mP\B\Env(\mM_\bi)_{\rS_\bi,\T_\bi}$ belong to the full subcategory $\overline{\mP\B\Env}^{\mH_\bi}_{\Lambda_\bi}(\mM_\bi)_{\rS_\bi,\T_\bi} \subset \mP\B\Env(\mM_\bi)_{\rS_\bi,\T_\bi}$,
For $1 \leq \bi \leq \n$ let $ \bar{\mM}^\circledast_\bi \subset \mP\B\Env_{\Lambda_\bi}(\mM_\bi)_{\rS_\bi,\T_\bi}^\circledast$ be the full weakly bienriched subcategory generated by $\mM_\bi$ under $\mH_\bi$-weighted colimits.
Since morphisms of $\mP\B\Env(\mM_\bi)_{\rS_\bi,\T_\bi} $ that belong to $\mQ_\bi$ are morphisms of $\bar{\mM}_\bi$, the enriched functor $\alpha$ inverts every  $(\f_1,...,\f_\n) \in \prod_{1 \leq \bi \leq \n} \mQ_\bi$ if for every family $(\X_\bj \in \bar{\mM}_\bj)_{1 \leq \bj \leq \n, \bj \neq \bi}$ 
the induced enriched functor $\alpha(\X_1,...,\X_{\bi-1},-, \X_{\bi+1},...,\X_\n): \overline{\mP\B\Env}^{\mH_\bi}_{\Lambda_\bi}(\mM_\bi)_{\rS_\bi,\T_\bi}^\circledast \to \mO^\circledast$ inverts morphisms of $\mQ_\bi.$
To see the latter, by Remark \ref{Closure} and the fact that $\alpha$ sends diagrams of $\Lambda^{\mH_1} \boxtimes ... \boxtimes \Lambda^{\mH_\n}$ to weighted colimit diagrams, we can assume that $\X_\bj \in \mM_\bj$ for $1 \leq \bj \leq \n, \bj \neq \bi$.
In this case the induced enriched functor $\alpha(\X_1,...,\X_{\bi-1},-, \X_{\bi+1},...,\X_\n): \overline{\mP\B\Env}^{\mH_\bi}_{\Lambda_\bi}(\mM_\bi)_{\rS_\bi,\T_\bi}^\circledast \to \mO^\circledast$
inverts morphisms of $\mQ_\bi$ because it sends diagrams of $\Lambda_\bi$ and $\mH_\bi$-weighted colimit diagrams to weighted colimit diagrams.
\end{proof}

Specializing Theorem \ref{wooond} gives the following corollaries:

\begin{corollary}\label{wooonda}
Let $\n \geq 1$. For $1 \leq \bi \leq \n$ let $(\mV_\bi^\ot \to \Ass, \rS_\bi), (\mV'^\ot \to \Ass, \rS')$ be small localization pairs, $\mH_\bi$ a set of small left weights over $\mV_\bi$ and $\mH'$ a collection of left weights over $\mV'$.
Let $\mM_\bi^\circledast \to \mV_\bi^\ot $ be small left $\rS_\bi$-enriched $\infty$-categories, $\Lambda_\bi$ a collection of $\mH_\bi$-weighted diagrams in $\mM_\bi$ and $\mN^\circledast \to \mV'^\ot\times \mW'^\ot $ a left $\rS'$-enriched $\infty$-category that admits $\mH'$-weighted colimits. Let $\mH:= \mH_1 \boxtimes... \boxtimes \mH_\n, \rS:= \rS_1 \boxtimes... \boxtimes \rS_\n, \T:= \T_1 \boxtimes... \boxtimes \T_\n.$
Assume that for every $1 \leq \bi \leq \n$ the weakly left enriched $\infty$-category
$\mM_\bi^\circledast \to \mV_\bi^\ot $ is locally left pseudo-enriched if there is some $1 \leq \bj \leq \n$ with $\bj \neq \bi$ such that $\rS_\bj$ does not contain only equivalences or $\Lambda_\bj$ is not empty.
The induced functor $$\gamma_\mN: \Enr\Fun_{\rS,\rS', \T,\T', \mH,\mH'}^{\Lambda^{\mH_1} \boxtimes ... \boxtimes \Lambda^{\mH_\n}}(\prod_{1 \leq \bi \leq \n} \mP\B\Env^{\mH_\bi}_{\Lambda_\bi}(\mM_\bi)_{\rS_\bi,\T_\bi}, \mN) \to  \Enr\Fun_{\rS,\rS', \T,\T', \mH,\mH'}^{\Lambda_1 \boxtimes ... \boxtimes \Lambda_\n}(\prod_{1 \leq \bi \leq \n}\mM_\bi, \mN)$$ is an equivalence.
%If for every $1 \leq \bi \leq \n$ the collection $\Lambda_\bi$ is empty,
The induced functor $$\delta_\mN: \Enr\Fun_{\rS,\rS', \T,\T', \mH,\mH'}(\prod_{1 \leq \bi \leq \n} \mP\B\Env^{\mH_\bi}(\mM_\bi)_{\rS_\bi,\T_\bi}, \mN) \to  \Enr\Fun_{\rS,\rS', \T,\T', \mH,\mH'}(\prod_{1 \leq \bi \leq \n}\mM_\bi, \mN)$$ admits a fully faithful left adjoint that lands in $\Enr\Fun_{\rS,\rS', \T,\T', \mH,\mH'}^{\Lambda^{\mH_1} \boxtimes ... \boxtimes \Lambda^{\mH_\n}}(\prod_{1 \leq \bi \leq \n} \mP\B\Env^{\mH_\bi}_{\Lambda_\bi}(\mM_\bi)_{\rS_\bi,\T_\bi}, \mN).$
	
\end{corollary}

\begin{corollary}\label{wooond0}
	
Let $(\mV^\ot \to \Ass, \rS), (\mW^\ot \to \Ass,\T)$ be small localization pairs, 
$\mM^\circledast \to \mV^\ot \times \mW^\ot$ an absolute small $\rS,\T$-bienriched $\infty$-category, $\mH$ a set of absolute small weights over $\mV, \mW$, $\Lambda$ a collection of $\mH$-weighted diagrams in $\mM$ and $\mN^\circledast \to \mV^\ot\times \mW^\ot $ a $\rS,\T$-bienriched $\infty$-category that admits $\mH$-weighted colimits.
The following functor is an equivalence: $$\Enr\Fun_{\mV, \mW}^{\Lambda^{\mH}}( \mP\B\Env^{\mH}_{\Lambda}(\mM)_{\rS,\T}, \mN) \to \Enr\Fun_{\mV, \mW}^{\Lambda}(\mM, \mN).$$	
	
\end{corollary}

\begin{corollary}\label{wooond00}Let $(\mV^\ot \to \Ass, \rS)$ be a small localization pair, $\mM^\circledast \to \mV^\ot$ an absolute small left $\rS$-enriched $\infty$-category, $\mH$ a set of small left weights over $\mV$, $\Lambda$ a collection of $\mH$-weighted left diagrams in $\mM$ and $\mN^\circledast \to \mV^\ot\times \mW^\ot $ a left $\rS$-enriched $\infty$-category that admits $\mH$-weighted colimits.
The following right $\mW$-enriched functor is an equivalence: $$\Enr\Fun_{\mV, \emptyset}^{\Lambda^{\mH}}( \mP\B\Env^{\mH}_{\Lambda}(\mM)_{\rS,\T}, \mN)^\circledast \to \Enr\Fun_{\mV, \emptyset}^{\Lambda}(\mM, \mN)^\circledast.$$	
	
\end{corollary}

\begin{notation}Let $\kappa, \tau$ be regular cardinals, $\n \geq 1$ and for every $1 \leq \bi \leq \n$ let $\mM_\bi^\circledast \to \mV_\bi^\ot \times \mW_\bi^\ot, \mN^\circledast \to \mV'^\ot \times \mW'^\ot$ be $\infty$-categories weakly bienriched in monoidal $\infty$-categories compatible with $\kappa$-small colimits, $\tau$-small colimits, respectively. Let $\mH$ be a collection of weights over $\prod_{\bi=1}^\n\mV_\bi,\prod_{\bi=1}^\n\mW_\bi $ and $\mH'$ a collection of weights over $\mV',\mW'$ and $\Lambda$ a collection of $\mH$-weighted diagrams in $\prod_{\bi=1}^\n\mM_\bi$.
Set $\mV^\ot:= \mV_1^\ot \times_\Ass ... \times_\Ass \mV_\n^\ot, \mW^\ot:= \mW_1^\ot \times_\Ass ... \times_\Ass \mW_\n^\ot.$
%\begin{enumerate}
%\item Let $$ _\tu\Enr\Fun^{\Lambda}_{\mH, \mH'}(\prod_{1 \leq \bi \leq \n}\mM_\bi, \mN),\ \Enr\Fun^{\Lambda}_{\tu, \mH, \mH'}(\prod_{1 \leq \bi \leq \n}\mM_\bi, \mN),\ _\tu\Enr\Fun^{\Lambda}_{\tu, \mH, \mH'}(\prod_{1 \leq \bi \leq \n}\mM_\bi, \mN)$$$$ \subset \Enr\Fun^{\Lambda}_{\mH, \mH'}(\prod_{1 \leq \bi \leq \n}\mM_\bi, \mN) $$ be the full subcategories of enriched functors lying from the left, from the right, from the left and right, respectively, over a map of $\infty$-operads preserving the tensor unit.
%\item If for every $1 \leq \bi \leq \n$ the $\infty$-operads $\mV_\bi^\ot \to \Ass, \mW_\bi^\ot \to \Ass$ are monoidal $\infty$-categories, let $$ \Enr\Fun^{\Lambda}_{\ot, \mH, \mH'}(\prod_{1 \leq \bi \leq \n}\mM_\bi, \mN) \subset \Enr\Fun^{\Lambda}_{\mH, \mH'}(\prod_{1 \leq \bi \leq \n}\mM_\bi, \mN)$$ be the full subcategory of enriched functors 
%sending diagrams of $\Lambda$ to $\mH$-weighted colimit diagrams and 
%lying over maps of $\infty$-operads $\mV^\ot \to \mV'^\ot, \mW^\ot \to \mW'^\ot $such that for every $1 \leq \bi \leq \n$ the compositions $\mV_\bi^\ot \to \mV^\ot \to \mV'^\ot, \mW_\bi^\ot \to \mW^\ot \to \mW'^\ot$ are monoidal functors. % that admits a right adjoint. % relative to $\Ass.$
Let $$\hspace{12mm} _{\ot,\kappa}\Enr\Fun^{\Lambda}_{\mH, \mH'}(\prod_{1 \leq \bi \leq \n}\mM_\bi, \mN), \ \Enr\Fun^{\Lambda}_{\ot, \tau, \mH, \mH'}(\prod_{1 \leq \bi \leq \n}\mM_\bi, \mN), \ _{\ot,\kappa}\Enr\Fun^{\Lambda}_{\ot, \tau, \mH, \mH'}(\prod_{1 \leq \bi \leq \n}\mM_\bi, \mN)$$$$ \subset \Enr\Fun^{\Lambda}_{\mH, \mH'}(\prod_{1 \leq \bi \leq \n}\mM_\bi, \mN)$$ be the full subcategories of enriched functors 
%sending diagrams of $\Lambda$ to $\mH$-weighted colimit diagrams and 
lying from the left over a map of $\infty$-operads $\mV^\ot \to \mV'^\ot$ such that for every $1 \leq \bi \leq \n$ the composition $\mV_\bi^\ot \to \mV^\ot \to \mV'^\ot$ is a monoidal functor preserving $\kappa$-small colimits,
lying from the right over a map of $\infty$-operads $\mW^\ot \to \mW'^\ot$ such that for every $1 \leq \bi \leq \n$ the composition $\mW_\bi^\ot \to \mW^\ot \to \mW'^\ot$ is a monoidal functor preserving $\tau$-small colimits,
lying from the left and right over a map of $\infty$-operads with these properties, respectively. If $\kappa=\emptyset$, we drop $\kappa.$
If $\kappa$ is the large strongly inaccessible cardinal corresponding to the small universe, we replace $\kappa$ by $\rc\rc.$
%be the full subcategory of enriched functors 
%sending diagrams of $\Lambda$ to $\mH$-weighted colimit diagrams and lying over monoidal functors $\mV_1^\ot \times_\Ass ... \times_\Ass \mV_\n^\ot \to \mV'^\ot, \mW_1^\ot \times_\Ass ... \times_\Ass \mW_\n^\ot \to \mW'^\ot $such that for every $1 \leq \bi \leq \n$ the compositions $$\mV_\bi^\ot \to \mV_1^\ot \times_\Ass ... \times_\Ass \mV_\n^\ot \to \mV'^\ot, \mW_\bi^\ot \to \mW_1^\ot \times_\Ass ... \times_\Ass \mW_\n^\ot \to \mW'^\ot$$ admit a right adjointrelative to $\Ass.$	\end{enumerate}

\end{notation}

\begin{corollary}\label{wooond1}
Let $\kappa, \tau$ be small regular cardinals, $\n \geq 1$ and for every  $1 \leq \bi \leq \n$ let %$\mV_\bi^\ot \to \Ass, \mW_\bi^\ot \to \Ass$ small $\infty$-operads with tensor unit.
$\mM_\bi^\circledast \to \mV_\bi^\ot \times \mW_\bi^\ot$ be an absolute small weakly  bienriched $\infty$-category, $\mH_\bi$ a set of absolute small weights over $\mV_\bi, \mW_\bi$, $\Lambda_\bi$ a collection of $\mH_\bi$-weighted diagrams in $\mM_\bi$ and $\mN^\circledast \to \mV'^\ot\times \mW'^\ot $ a weakly bienriched $\infty$-category that admits $\mH'$-weighted colimits for a collection $\mH'$ of weights over $\mV', \mW'.$
%Set $\mH:=\mH_1 \boxtimes ... \boxtimes \mH_\n.$
%where $ \mV^\ot:= \mV_1^\ot \times_\Ass ... \times_\Ass \mV^\ot_\n, \mW^\ot:= \mW_1^\ot \times_\Ass ... \times_\Ass \mW^\ot_\n.$

\begin{enumerate}

\item If for every  $1 \leq \bi \leq \n$ the weakly bienriched $\infty$-categories $\mM_\bi^\circledast \to \mV_\bi^\ot \times \mW_\bi^\ot$ and $\mN^\circledast \to \mV'^\ot\times \mW'^\ot $ are left $\kappa$-enriched $\infty$-categories and $\mH_\bi, \mH'$ consist of left weights, the following functor is an equivalence: 
$$\hspace{6mm}_{\ot, \kappa}\Enr\Fun_{ \mH, \mH'}^{\Lambda^{\mH_1} \boxtimes ... \boxtimes \Lambda^{\mH_\n}}(\prod_{1 \leq \bi \leq \n} \mP\B\Env^{\mH_\bi}_{\Lambda_\bi}(\mM_\bi)_{\L\Enr_{\kappa}}, \mN) \to {_{\ot, \kappa}\Enr\Fun}_{\mH, \mH'}^{\Lambda_1 \boxtimes ... \boxtimes \Lambda_\n}(\prod_{1 \leq \bi \leq \n}\mM_\bi, \mN).$$

\item If for every  $1 \leq \bi \leq \n$ the weakly bienriched $\infty$-categories $\mM_\bi^\circledast \to \mV_\bi^\ot \times \mW_\bi^\ot$ and $\mN^\circledast \to \mV'^\ot\times \mW'^\ot $ are right $\tau$-enriched $\infty$-categories and $\mH_\bi, \mH'$ consist of right weights, the following functor is an equivalence: 
$${\Enr\Fun}_{\ot, \tau, \mH, \mH'}^{\Lambda^{\mH_1} \boxtimes ... \boxtimes \Lambda^{\mH_\n}}(\prod_{1 \leq \bi \leq \n} \mP\B\Env^{\mH_\bi}_{\Lambda_\bi}(\mM_\bi)_{\R\Enr_{\tau}}, \mN) \to {\Enr\Fun}_{\ot, \tau, \mH, \mH'}^{\Lambda_1 \boxtimes ... \boxtimes \Lambda_\n}(\prod_{1 \leq \bi \leq \n}\mM_\bi, \mN).$$

\item If for every  $1 \leq \bi \leq \n$ the weakly bienriched $\infty$-categories $\mM_\bi^\circledast \to \mV_\bi^\ot \times \mW_\bi^\ot$ and $\mN^\circledast \to \mV'^\ot\times \mW'^\ot $ are $\kappa, \tau$-bienriched $\infty$-categories, 
the following functor is an equivalence: 
$$\hspace{12mm}_{\ot, \kappa}\Enr\Fun_{\ot, \tau, \mH, \mH'}^{\Lambda^{\mH_1} \boxtimes ... \boxtimes \Lambda^{\mH_\n}}(\prod_{1 \leq \bi \leq \n} \mP\B\Env^{\mH_\bi}_{\Lambda_\bi}(\mM_\bi)_{\B\Enr_{\kappa, \tau}}, \mN) \to {_{\ot,\kappa}\Enr\Fun}_{\ot, \tau, \mH, \mH'}^{\Lambda_1 \boxtimes ... \boxtimes \Lambda_\n}(\prod_{1 \leq \bi \leq \n}\mM_\bi, \mN).$$
\end{enumerate}
	
\end{corollary}

For the next lemma we use the following terminology:

\begin{definition}Let $\mV^\ot \to \Ass, \mW^\ot \to \Ass$ be small $\infty$-operads
and $\kappa, \tau$ small regular cardinals.
	
\begin{enumerate}
\item An absolute small left $\kappa$-enriched weight over $\mV,\mW$ is a weight $(\alpha: \mV'^\ot \to \mV^\ot, \beta: \mW'^\ot \to \mW^\ot, \mJ^\circledast \to \mV'^\ot \times \mW'^\ot, \rH)$ over $\mV,\mW$ such that $\mV^\ot \to \Ass, \mV'^\ot \to \Ass$ are small monoidal $\infty$-categories compatible with $\kappa$-small colimits, $\mW^\ot \to \Ass, \mW'^\ot \to \Ass$ are small $\infty$-operads,
$\alpha$ is a monoidal functor preserving $\kappa$-small colimits, $ \mJ^\circledast \to \mV'^\ot \times \mW'^\ot$ is a small left $\kappa$-enriched $\infty$-category and $\rH \in \mP\B\Env(\mM)_{\L\Enr_{\kappa}}.$
Dually, we define absolute small right $\tau$-enriched weights over $\mV,\mW$.
%is a weight $(\alpha: \mV'^\ot \to \mV^\ot, \beta: \mW'^\ot \to \mW^\ot, \mJ^\circledast \to \mV'^\ot \times \mW'^\ot, \rH)$ over $\mV,\mW$ such that $\mW^\ot \to \Ass, \mW'^\ot \to \Ass$ are small monoidal $\infty$-categories compatible with $\tau$-small colimits, $\mV^\ot \to \Ass, \mV'^\ot \to \Ass$ are small $\infty$-operads,$\beta$ is a monoidal functor preserving $\tau$-small colimits, $ \mJ^\circledast \to \mV'^\ot \times \mW'^\ot$ is a small right $\tau$-enriched $\infty$-category and $\rH \in \mP\B\Env(\mM)_{\R\Enr_{\tau}}.$
		
\item An absolute small $\kappa, \tau$-enriched weight over $\mV,\mW$ is a weight $(\alpha: \mV'^\ot \to \mV^\ot, \beta: \mW'^\ot \to \mW^\ot, \mJ^\circledast \to \mV'^\ot \times \mW'^\ot, \rH)$ over $\mV,\mW$ such that $\mV^\ot \to \Ass, \mV'^\ot \to \Ass$
are small monoidal $\infty$-categories compatible with $\kappa$-small colimits, $ \mW^\ot \to \Ass, \mW'^\ot \to \Ass$ are monoidal $\infty$-categories compatible with $\tau$-small colimits, $\alpha$ is a monoidal functor preserving $\kappa$-small colimits, $\beta$ is a monoidal functor preserving $\tau$-small colimits, $ \mJ^\circledast \to \mV'^\ot \times \mW'^\ot$ is a small $\kappa, \tau$-bienriched $\infty$-category and $\rH \in \mP\B\Env(\mM)_{\B\Enr_{\kappa,\tau}}.$
		
\end{enumerate}
\end{definition}

\begin{lemma}\label{mini} \begin{enumerate}\item Let $\alpha: \mV'^\ot \to \mV^\ot$ be a left adjoint monoidal functor between presentably monoidal $\infty$-categories and $ \beta: \mW'^\ot \to \mW^\ot $ a map of $\infty$-operads. By \cite[Proposition 7.15.]{Rune} there are small regular cardinals $\kappa, \kappa'$ such that $\mV^\ot \to \Ass$ is $\kappa$-compactly generated, $\mV'^\ot \to \Ass$ is $\kappa'$-compactly generated and $\alpha$ sends $\kappa'$-compact objects to $\kappa$-compact objects. 		
For any left enriched weight $\rH$ on a small left enriched $\infty$-category $\mJ^\circledast \to \mV'^\ot \times \mW'^\ot$ there is a canonical left $\kappa'$-enriched weight $\rH' \in \mP\B\Env(\mJ_{\kappa'})_{\L\Enr_{\kappa'}}$ on $\mJ_{\kappa'}^\circledast \to (\mV'^{\kappa'})^\ot \times \mW'^\ot$ such that for any enriched functor $\F: \mJ^\circledast \to \mM^\circledast$ lying over $\alpha, \beta$ there is a canonical equivalence $$\colim^\rH(\F) \simeq \colim^{\rH'}(\F'),$$ where $\F': \mJ_{\kappa'}^\circledast \to \mM_\kappa^\circledast$ is the restriction of $\F.$

\item Let $\alpha: \mV'^\ot \to \mV^\ot, \beta: \mW'^\ot \to \mW^\ot$ be left adjoint monoidal functors between presentably monoidal $\infty$-categories. By \cite[Proposition 7.15.]{Rune} there are small regular cardinals $\kappa, \kappa', \tau, \tau'$ such that $\mV^\ot \to \Ass$ is $\kappa$-compactly generated, $\mV'^\ot \to \Ass$ is $\kappa'$-compactly generated, $\mW^\ot \to \Ass$ is $\tau$-compactly generated, $\mW'^\ot \to \Ass$ is $\tau'$-compactly generated, $\alpha$ sends $\kappa'$-compact objects to $\kappa$-compact objects and $\beta$ sends $\tau'$-compact objects to $\tau'$-compact objects. For any enriched weight $\rH$ on $\mJ^\circledast \to \mV'^\ot \times \mW'^\ot$ there is a canonical $\kappa',\tau'$-enriched weight $\rH' \in \mP\B\Env(\mJ_{\kappa',\tau'})_{\B\Enr_{\kappa', \tau'}}$
on $\mJ_{\kappa', \tau'}^\circledast \to (\mV'^{\kappa'})^\ot \times (\mW'^{\tau'})^\ot$
such that for any enriched functor $\F: \mJ^\circledast \to \mM^\circledast$ lying over $\alpha, \beta$ there is a canonical equivalence $$\colim^\rH(\F) \simeq \colim^{\rH'}(\F'),$$ where $\F': \mJ_{\kappa',\tau'}^\circledast \to \mM_{\kappa,\tau}^\circledast$ is the restriction of $\F.$
\end{enumerate}\end{lemma}

\begin{proof}We prove (1). The proof of (2) is similar. By Corollary \ref{cosqa} the projections $\bj: \mJ_{\kappa'}^\circledast \to \mJ^\circledast, \iota: \mM_\kappa^\circledast \to \mM^\circledast$ induce
$\mV,\mW$-enriched equivalences $$\hspace{6mm} \L \circ \bj_!: \mP\B\Env(\mJ_{\kappa'})^\circledast_{\L\Enr_{\kappa'}} \simeq \mP\B\Env(\mJ)_{\L\Enr}^\circledast, \L \circ \iota_!: \mP\B\Env(\mM_{\kappa})^\circledast_{\L\Enr_{\kappa}} \simeq \mP\B\Env(\mM)_{\L\Enr}^\circledast.$$
Setting $\rH':= (\L \circ \tau_!)^{-1}(\rH)$ there is an equivalence
$\L(\F_!(\rH)) \simeq \L \circ \F_! \circ\L \circ \tau_!(\rH') \simeq \L \circ \iota_! \circ\L \circ \F'_!(\rH')$. So $\psi: \L(\F_!(\rH)) \to \Y $
is the image under $\L \circ \iota_!$ of a morphism
$\psi': \L \circ \F'_!(\rH') \to \Y $ in $\mP\B\Env(\mM_{\kappa})_{\L\Enr_{\kappa}}.$ The morphism $\psi$ exhibits $\Y$ as the $\rH$-weighted colimit of $\F $ if and only if $\psi'$ exhibits $\Y$ as the $\rH'$-weighted colimit of $\F'$ since 
for every $\Z \in \mM$, $\mV \in \mV^\kappa, \W_1, ...,\W_\m \in \mW$ for $\m \geq0$ there is a commutative square: % where $\L$ is the left adjoint of the embedding $\mP\B\Env(\mM)_{\B\Enr} \subset \mP\B\Env(\mM):$ 
$$\begin{xy}
\xymatrix{\Mul_{\mM_\kappa}(\V,\Y,\W_1,...,\W_\m; \Z)  \ar[d]\ar[rr]
&& \Mul_{\mP\B\Env(\mM_\kappa)_{\L\Enr_\kappa}}(\V,\L(\F'_!(\rH')),\W_1,...,\W_\m; \Z) \ar[d]^{} 
\\
\Mul_\mM(\V,\Y,\W_1,...,\W_\m; \Z) \ar[rr]^{} && \Mul_{\mP\B\Env(\mM)_{\L\Enr}}(\V,\L(\F_!(\rH)),\W_1,...,\W_\m; \Z).}
\end{xy}$$

\end{proof}

To apply Corollary \ref{wooond1} we use the following proposition:

\begin{proposition}\label{quirk}
	
\begin{enumerate}
\item Let $\mM^\circledast \to \mV^\ot \times \mW^\ot$ be a small $\infty$-category left enriched in a presentably monoidal $\infty$-category, $\mH$ a set of small left enriched weights over $\mV,\mW$ and $\Lambda$ a collection of $\mH$-weighted diagrams in $\mM$. 
Then $\widehat{\mP}\B\Env^\mH_{\Lambda}(\mM)^\circledast_{\L\Enr} \to \mV^\ot \times \mW^\ot$ is large but locally small.

\item Let $\mM^\circledast \to \mV^\ot \times \mW^\ot$ be a small $\infty$-category right enriched in a presentably monoidal $\infty$-category, $\mH$ a set of small right enriched weights over $\mV,\mW$ and $\Lambda$ a collection of $\mH$-weighted diagrams in $\mM$. 
Then $\widehat{\mP}\B\Env^\mH_{\Lambda}(\mM)^\circledast_{\R\Enr} \to \mV^\ot \times \mW^\ot$ is large but locally small.

\item Let $\mM^\circledast \to \mV^\ot \times \mW^\ot$ be a small $\infty$-category bienriched in presentably monoidal $\infty$-categories, $\mH$ a set of small enriched weights over $\mV,\mW$ and $\Lambda$ a collection of $\mH$-weighted diagrams in $\mM$. 
Then $\widehat{\mP}\B\Env^\mH_{\Lambda}(\mM)^\circledast_{\B\Enr} \to \mV^\ot \times \mW^\ot$ is large but locally small.

\end{enumerate}

\end{proposition}

\begin{proof}We prove (3). The proofs of (1) and (2) are similar.
By Notation \ref{enrprrrt} there is an enriched embedding
$\bj: \mP\B\Env(\mM)^\circledast_{\B\Enr} \subset \widehat{\mP}\B\Env(\mM)^\circledast_{\B\Enr}$
that preserves small weighted colimits. Let $\iota: \mM^\circledast \subset \mP\B\Env(\mM)^\circledast_{\B\Enr}$ be the canonical enriched embedding.
By Proposition \ref{rembrako} the weakly bienriched $\infty$-category $ \mP\B\Env(\mM)_{\B\Enr}^\circledast \to \mV^\ot \times \mW^\ot $ is a presentably bitensored $\infty$-category.
Let $\Sigma$ be the collection of the following absolute small diagrams:
let $(\F: \mJ^\circledast \to \mM^\circledast, \rH \in \mP\B\Env(\mJ)_{\B\Enr},\Y\in \mM) \in \Lambda$.
By \cite[Proposition 7.15.]{Rune} there are small regular cardinals $\kappa, \kappa', \tau, \tau'$ such that $\mV^\ot \to \Ass$ is $\kappa$-compactly generated, $\mV'^\ot \to \Ass$ is $\kappa'$-compactly generated, $\mW^\ot \to \Ass$ is $\tau$-compactly generated, $\mW'^\ot \to \Ass$ is $\tau'$-compactly generated, $\alpha$ sends $\kappa'$-compact objects to $\kappa$-compact objects and $\beta$ sends $\tau'$-compact objects to $\tau'$-compact objects.
By Lemma (2) \ref{mini} there is a canonical $\kappa',\tau'$-enriched weight $\rH' \in \mP\B\Env(\mJ_{\kappa',\tau'})_{\B\Enr_{\kappa', \tau'}}$
on $\mJ_{\kappa', \tau'}^\circledast \to (\mV'^{\kappa'})^\ot \times (\mW'^{\tau'})^\ot$
such that for any enriched functor $\F: \mJ^\circledast \to \mM^\circledast$ lying over $\alpha, \beta$ there is a canonical equivalence \begin{equation}\label{ppss}
\colim^\rH(\iota \circ \F) \simeq \colim^{\rH'}(\iota \circ \F'),\end{equation} where $\F': \mJ_{\kappa',\tau'}^\circledast \to \mM_{\kappa,\tau}^\circledast$ is the restriction of $\F.$
Then $(\F': \mJ_{\kappa', \tau'}^\circledast \to \mM_{\kappa, \tau}^\circledast, \rH' \in \mP\B\Env(\mJ_{\kappa',\tau'})_{\B\Enr_{\kappa', \tau'}}, \Y\in \mM)$.
Then $\Sigma$ is a collection of absolute small weights over $\mV,\mW$, which is a set by Lemma \ref{leuz} and by equivalence \ref{ppss} there is a canonical equivalence $\widehat{\mP}\B\Env_\Lambda(\mM)^\circledast_{\B\Enr} \simeq \widehat{\mP}\B\Env_\Sigma(\mM)^\circledast_{\B\Enr}$
by definition of the latter.

Let $\mQ'_\Sigma$ be the set of Notation \ref{Notali} and $\mP\B\Env_\Sigma(\mM)^\circledast_{\B\Enr} \subset \mP\B\Env(\mM)^\circledast_{\B\Enr}$ the full weakly bienriched subcategory spanned by the $\mQ'_\Sigma$-local objects.
Since $\mQ'_\Sigma$ is preserved by the $\mV, \mW$-biaction, the full weakly bienriched subcategory $\mP\B\Env_\Sigma(\mM)^\circledast_{\B\Enr} \subset \mP\B\Env(\mM)^\circledast_{\B\Enr}$ is an accessible $\mV,\mW$-enriched localization and so a presentably bitensored $\infty$-category.
Let $\mP\B\Env^\mH_\Sigma(\mM)^\circledast_{\B\Enr} \subset \mP\B\Env_\Sigma(\mM)^\circledast_{\B\Enr}$ be the full weakly bienriched subcategory
generated by $\mM$ under $\mH$-weighted colimits.
By definition $\widehat{\mP}\B\Env_\Sigma(\mM)^\circledast_{\B\Enr}$
is the localization with respect to the set $\bj(\mQ'_\Sigma)$, the image of $\mQ'_\Sigma$ under $\bj$.
Hence $\bj$ preserves and reflects local objects (being an embedding) for the localizations with respect to the sets $\mQ'_\Sigma, \bj(\mQ'_\Sigma)$. So $\bj$ restricts to an embedding $ \mP\B\Env_\Sigma(\mM)^\circledast_{\B\Enr} \subset \widehat{\mP}\B\Env_\Sigma(\mM)^\circledast_{\B\Enr} $ preserving small weighted colimits.
Consequently, the latter embedding restricts to an equivalence
$\mP\B\Env^\mH_\Sigma(\mM)^\circledast_{\B\Enr} \simeq \widehat{\mP}\B\Env^\mH_\Sigma(\mM)^\circledast_{\B\Enr}. $ 

\end{proof}

\begin{corollary}\label{wooond0000}Let $\mV^\ot \to \Ass$ be a presentably monoidal $\infty$-category, $\mM^\circledast \to \mV^\ot$ a small left $\mV$-enriched $\infty$-category, $\mH$ a set of small left enriched weights over $\mV$, $\Lambda$ a collection of $\mH$-weighted diagrams in $\mM$ and $\mN^\circledast \to \mV^\ot\times \mW^\ot $ a left $\mV$-enriched $\infty$-category that admits $\mH$-weighted colimits.
The following right $\mW$-enriched functor is an equivalence: $$\Enr\Fun_{\mV, \emptyset}^{\Lambda^{\mH}}( \mP\B\Env^{\mH}_{\Lambda}(\mM)_{\L\Enr}, \mN)^\circledast \to \Enr\Fun_{\mV, \emptyset}^{\Lambda}(\mM, \mN)^\circledast.$$	
	
\end{corollary}

\subsection{Weighted colimits of diagrams of enriched functors}

In the following we prove that weighted colimits in enriched $\infty$-categories of enriched functors are formed object-wise.

\begin{theorem}\label{Enric}
Let $\mV^\otimes \to \Ass, \mW^\ot \to \Ass$ be $\infty$-operads, $\mM^\circledast \to \mV^\ot$ a weakly left enriched $\infty$-category, %, $\mO^\circledast \to \mW^\ot$ a weakly right enriched $\infty$-category,
$\mN^\circledast \to \mV^\ot \times \mW^\ot$ a weakly bienriched $\infty$-category and %$\mH$ a set of small left weights over $\mV$ and 
$\mH$ a set of small right weights over $\mW$. % (seen as weights over $\mV,\mW$ in the canonical way).
If $\mN^\circledast \to \mV^\otimes \times \mW^\ot$ admits $\sigma_!(\mH)$-weighted colimits,
where $\sigma: \emptyset^\ot \to \mV^\ot$ is the unique map of $\infty$-operads, the weakly right enriched $\infty$-category $\Enr\Fun_{\mV, \emptyset}(\mM,\mN)^\circledast \to \mW^\ot $ admits $\mH$-weighted colimits and for every $\X \in \mM$ the right $\mW$-enriched functor $\Enr\Fun_{\mV, \emptyset}(\mM,\mN)^\circledast \to \mN^\circledast$ evaluating at $\X$ preserves $\mH$-weighted colimits.
Moreover for every set of diagrams $\Lambda$ in $\mM$ the full subcategory $\Enr\Fun_{\mV, \emptyset}^{\Lambda}(\mM, {\mN})^\circledast \subset \Enr\Fun_{\mV, \emptyset}(\mM,\mN)^\circledast $ is closed under $\mH$-weighted colimits. 

%\item If the underlying weakly left enriched $\infty$-category of $\mN^\circledast \to \mV^\otimes \times \mW^\ot$ admits $\mH$-weighted colimits, the weakly left enriched $\infty$-category $\Enr\Fun_{\emptyset, \mW}(\mM,\mN)^\circledast \to \mV^\ot $ admits $\mH$-weighted colimits and for every $\X \in \mO$ the left $\mV$-enriched functor $\Enr\Fun_{\emptyset, \mW}(\mO,\mN)^\circledast \to \mN^\circledast$ preserves $\mH$-weighted colimits.Moreover for every set of diagrams $\Lambda$ in $\mO$ the full subcategory $\Enr\Fun_{\emptyset, \mW}^{\Lambda}(\mO, {\mN})^\circledast \subset \Enr\Fun_{\emptyset, \mW}(\mO,\mN)^\circledast $ is closed under $\mH$-weighted colimits. \end{enumerate}
\end{theorem}

\begin{proof} %We prove (1). (2) is dual.
%Let $\iota: \mN \subset \mP\B\Env(\mN)$ be the canonical embedding. Let $\mN'^\circledast \subset \mP\B\Env(\mN)^\circledast$ be the localization with respect to the set of morphisms $ \colim^{\rH}(\iota \circ \F) \to \iota(\Y)$ in $ \mP\B\Env(\mN)$ associated to the morphism $\F_!(\rH) \xrightarrow{\lambda} \iota(\Y) $ in $\mP\B\Env(\mM)$ for some $\mH$-weighted colimit diagram $(\F:\mJ^\circledast \to \mN^\circledast \times_{\mW^\ot} \emptyset^\ot, \Y, \lambda:  \F_!(\rH) \to \Y)$ in $\mN^\circledast \times_{\mW^\ot} \emptyset^\ot \to \mV^\ot.$ %and the set of morphisms $ \colim^{\rH}(\iota \circ \F) \to \iota(\Y)$ in $\mP\B\Env(\mN)$ corresponding to the morphism $\rH \xrightarrow{\lambda} \F^\ast(\Y) \to (\iota\circ \F)^\ast(\iota(\Y))$in $\mP\R\Env(\mJ)$ for some diagram $(\F:\mJ^\circledast \to \emptyset^\ot \times_{\mV^\ot} \mN^\circledast, \Y, \lambda:  \rH \to \F^*(\Y))$ of $\Lambda$.
Let $\mN'^\circledast:= \mP\widetilde{\B\Env}_{\Lambda^{\sigma_!(\mH)}}(\mN)^\circledast \to \mV^\ot \times \mW^\ot.$
%By ... the embedding $\mN^\circledast \subset \mP\B\Env(\mN)^\circledast$ induces an embedding $\mN^\circledast \subset \mN'^\circledast$ %whose underlying embedding on weakly left enriched $\infty$-categories that preserves $\mH$-weighted colimits. %and whose underlying embedding on weakly right enriched $\infty$-categories sends diagrams of $\Lambda$ to weighted colimit diagrams.
By Remark \ref{embes} there is a $\mV,\mW$-enriched embedding $\mN^\circledast \subset \mN'^\circledast$
that preserves $\sigma_!(\mH)$-weighted colimits. The latter induces a right $\mW$-enriched embedding: % of weakly right enriched $\infty$-categories:
$$\Enr\Fun_{\mV, \emptyset}(\mM,\mN)^\circledast \subset \Enr\Fun_{\mV, \emptyset}(\mM,\mN')^\circledast.$$
%\ \Enr\Fun^{\Lambda}_{\mV, \emptyset}(\mM,\mN)^\circledast \subset \Enr\Fun^{\Lambda}_{\mV, \emptyset}(\mM,\mN')^\circledast.$$

By Remark \ref{bipre} the weakly bienriched $\infty$-category $\mN'^\circledast= \mP\widetilde{\B\Env}_{\Lambda^{\sigma_!(\mH)}}(\mN)^\circledast \to \mV^\ot \times \mW^\ot$
admits right tensors and small conical colimits. 
By \cite[Proposition 5.15.]{heine2024bienriched} the weakly right enriched $\infty$-category $\Enr\Fun_{\mV, \emptyset}(\mM,\mN')^\circledast \to \mW^\ot $ admits right tensors and small conical colimits and the forgetful functor $\Enr\Fun_{\mV, \emptyset}(\mM,\mN')^\circledast \to (\mN'^\mM)^\circledast$ 
preserves right tensors and small conical colimits.
Corollary \ref{weeei} guarantees that $\Enr\Fun_{\mV, \emptyset}(\mM,\mN')^\circledast \to \mW^\ot $
admits small weighted colimits and the forgetful functor $\Enr\Fun_{\mV, \emptyset}(\mM,\mN')^\circledast \to (\mN'^\mM)^\circledast$ 
preserves small weighted colimits.
%that are preserved by evaluation at any object of $\mM$. 
Since the embedding $ \mN^\circledast \subset \mN'^\circledast$ %on underlying weakly right enriched $\infty$-categories 
preserves $\sigma_!(\mH)$-weighted colimits, by Corollary \ref{eccco} the full weakly right enriched subcategory $\Enr\Fun_{\mV, \emptyset}(\mM, \mN)^\circledast \subset \Enr\Fun_{\mV, \emptyset}(\mM,\mN')^\circledast$ is closed under $\mH$-weighted colimits.
Consequently, it is enough to check that $\Enr\Fun^{\Lambda}_{\mV, \emptyset}(\mM, \mN') $ is closed in $\Enr\Fun_{\mV, \emptyset}(\mM,\mN') $ under small weighted colimits, or equivalently by Corollary \ref{weeei} under small colimits and right tensors.
For every left $\mV$-enriched functor $\F: \mM^\circledast \to \mN'^\circledast$ and $\W \in \mW$ the right tensor
$\F \ot \W$ factors as $ \mM^\circledast \xrightarrow{\F} \mN'^\circledast \xrightarrow{(-)\ot\W}\mN'^\circledast$, where $ \mN'^\circledast \xrightarrow{(-)\ot\W}\mN'^\circledast$ admits a left $\mV$-enriched right adjoint and so preserves weighted colimits by Proposition \ref{remqalo}.
The colimit of any functor $\G: \K \to \Enr\Fun_{\mV, \emptyset}(\mM,\mN') $ factors as the left $\mV$-enriched functor $\mM^\circledast \xrightarrow{\G'} (\mN'^\K)^\circledast \xrightarrow{\colim} \mN'^\circledast$, where $\G'$ corresponds to $\G$ under the equivalence
$ \Fun(\K, \Enr\Fun_{\mV, \emptyset}(\mM,\mN')) \simeq \Enr\Fun_{\mV, \emptyset}(\mM, \mN'^\K)$ and $\colim: (\mN'^\K)^\circledast \to \mN'^\circledast$ is the left $\mV$-enriched left adjoint of the left $\mV$-linear diagonal functor $\mN'^\circledast \to  (\mN'^\K)^\circledast$. % that preserves weighted colimits by Proposition \ref{remqalo}.
Therefore it is enough to observe that the left $\mV$-enriched functor $\mM^\circledast \xrightarrow{\G'} (\mN'^\K)^\circledast$ sends diagrams of $\Lambda$ to weighted colimit diagrams if the functor $\G: \K \to \Enr\Fun_{\mV, \emptyset}(\mM,\mN') $ lands in $\Enr\Fun^\Lambda_{\mV, \emptyset}(\mM,\mN') $: this follows immediately from the fact that $(\mN'^\K)^\circledast \to \mV^\ot$ is a left tensored $\infty$-category compatible with small colimits and so by Corollary \ref{weeei} admits small weighted colimits, and for every $\Z \in  \K$ the left $\mV$-linear functor $(\mN'^\K)^\circledast \to \mN'^\circledast$ evaluating at $\Z$
preserves small colimits 
%admits a left $\mV$-enriched right adjoint (Corollary \ref{remqa}) 
and so preserves small weighted colimits by Proposition \ref{weeei}.

\end{proof}

\section{A monoidal structure for enriched $\infty$-categories with weighted colimits}\label{HW}

\subsection{Enriched $\infty$-categories with weighted colimits}

In this section we organize weakly bienriched $\infty$-categories equipped
with a diagram to a symmetric monoidal $\infty$-category. 
%For the next two notations we use Notation \ref{pose}:
\begin{notation}\label{pose}Let $\mV^\ot \to \Ass, \mW^\ot \to \Ass$ be small $\infty$-operads and $\mM^\circledast \to \mV^\ot \times \mW^\ot$ a  weakly bienriched $\infty$-category.
	
\begin{enumerate}
\item Let $\B\D(\mM)$ be the poset of sets of absolute small diagrams on $\mM$ ordered by inclusion. 

\item Let $\omega(\mV, \mW)$ be the poset of sets of absolute small weights over $\mV, \mW$ ordered by inclusion.

\end{enumerate}	
Since we can transport diagrams and weights (Notation \ref{trans}), we obtain functors to the category of large posets $\widehat{\Poset}:$
$$\B\D: \Ho(\omega\B\Enr) \to \widehat{\Poset},\ \omega: \Ho(\Op_\infty) \times \Ho(\Op_\infty) \to \widehat{\Poset}.$$

\end{notation}

\begin{notation}
Let $$ {\omega\B\Enr}_* \to \omega\B\Enr$$ be the cocartesian fibrations classifying
the composition $$\omega\B\Enr \to \Ho(\omega\B\Enr) \xrightarrow{\B\D} \widehat{\Poset} \subset \widehat{\Cat}_\infty.$$
\end{notation}
We call $\omega\B\Enr_* $ the $\infty$-category of absolute small weakly bienriched $\infty$-categories with diagrams.

\begin{notation}
	
Let %$$(\Op_\infty)_* \to \Op_\infty,$$$$_*(\Op_\infty) \to \Op_\infty,$$
$$(\Op_\infty \times \Op_\infty)_* \to \Op_\infty \times \Op_\infty$$
be the cocartesian fibrations classifying the composition %$$\Op_\infty \to \Ho(\Op_\infty) \xrightarrow{\L\omega} \widehat{\Poset} \subset \widehat{\Cat}_\infty,$$$$\Op_\infty \to \Ho(\Op_\infty) \xrightarrow{\R\omega} \widehat{\Poset} \subset \widehat{\Cat}_\infty,$$
$$\Op_\infty \times \Op_\infty \to \Ho(\Op_\infty \times \Op_\infty) \xrightarrow{\omega} \widehat{\Poset} \subset \widehat{\Cat}_\infty.$$
\end{notation}

We call $ {(\Op_\infty \times \Op_\infty)_*}$ the $\infty$-category of pairs of small $\infty$-operads with weights.

\begin{remark}
Since $\omega\B\Enr, \Op_\infty$ are locally small and posets are locally small, the $\infty$-categories $ \omega\B\Enr_\ast, {(\Op_\infty \times \Op_\infty)_*}$ are large but locally small:
for $(\mV^\ot, \mW^\ot, \mH), (\mV'^\ot, \mW'^\ot,\mH') \in  {(\Op_\infty \times \Op_\infty)_*}$
%$\infty$-operads $\mV^\ot \to \Ass, \mW^\ot \to \Ass, \mV'^\ot \to \Ass, \mW'^\ot \to \Ass$
the induced map $$(\Op_\infty \times \Op_\infty)_*((\mV^\ot, \mW^\ot, \mH), (\mV'^\ot, \mW'^\ot,\mH')) \to \Op_\infty(\mV^\ot, \mV'^\ot) \times
\Op_\infty(\mW^\ot, \mW'^\ot)$$ has empty or contractible fibers
and so is fully faithful. It identifies the source with the full subspace of $\Op_\infty(\mV^\ot, \mV'^\ot) \times
\Op_\infty(\mW^\ot, \mW'^\ot)$ spanned by the pairs of maps of $\infty$-operads that send weights in $\mH$ to weights in $\mH'.$
Similarly, for every $(\mM^\circledast \to \mV^\ot \times \mW^\ot,\Lambda), (\mM'^\circledast \to \mV'^\ot \times \mW'^\ot,\Lambda') \in {\omega\B\Enr_\ast}$
% weakly bienriched $\infty$-categories $\mM^\circledast \to \mV^\ot \times \mW^\ot, \mM'^\circledast \to \mV'^\ot \times \mW'^\ot$
the space $\omega\B\Enr_*((\mM^\circledast,\Lambda), (\mM'^\circledast, \Lambda'))$ is the full subspace of $\omega\B\Enr(\mM^\circledast,\mM'^\circledast)$ spanned by the enriched functors sending diagrams in $\Lambda$ to diagrams in $\Lambda'.$

\end{remark}

\begin{remark}\label{colass}The functor $\omega\B\Enr_* \to \omega\B\Enr$ admits a left and right adjoint. The left adjoint equips an absolute small weakly  bienriched $\infty$-category with the empty set of diagrams and the right adjoint equips it with the set of all diagrams.
\end{remark}

\begin{construction}Let $\mM^\circledast \to \mV^\ot \times \mW^\ot$ be a weakly bienriched $\infty$-category.
%Sending a weighted diagram $(\rH, \F:\mJ^\circledast \to \mM^\circledast, \lambda: \F_!(\rH) \to \Y)$ on $\mM$ to the underlying weight $(\alpha: \mV'^\ot \to \mV^\ot, \beta: \mW'^\ot \to \mW^\ot, \mJ^\circledast \to \mV'^\ot \times \mW'^\ot, \rH)$ over $\mV,\mW$, where $\F$ lies over $\alpha, \beta,$ 
Sending a diagram on $\mM$ to the underlying weight defines an order preserving map $$ \B\D(\mM) \to \omega(\mV, \mW)$$ that determines a transformation from the functor
%weighted diagram on $\mM$ to the underlying weight over $\mV, \mW$ defines an order preserving map $\B\D(\mM) \to \omega(\mV, \mW)$ that determines a natural transformation from the functor
$\B\D: \Ho(\omega\B\Enr) \xrightarrow{} \widehat{\Poset}$ to the functor
$\Ho(\omega\B\Enr) \to \Ho(\Op_\infty \times \Op_\infty) \xrightarrow{\omega} \widehat{\Poset}$.
The latter classifies a map of cocartesian fibrations over $\omega\B\Enr:$ 
\begin{equation}\label{kupre2}
{\omega\B\Enr}_\ast \to \omega\B\Enr \times_{(\Op_\infty \times \Op_\infty)} {(\Op_\infty \times \Op_\infty)_*}. \end{equation}
\end{construction}

\begin{lemma}\label{lllee}\label{llleee}

The functor $ \omega\B\Enr_\ast \to {(\Op_\infty \times \Op_\infty)_*}$ is a cocartesian and cartesian fibration and the functor $\omega\B\Enr_\ast \to \omega\B\Enr$ sends (co)cartesian lifts of morphisms in $ (\Op_\infty \times \Op_\infty)_*$ to (co)cartesian lifts in $ \Op_\infty \times \Op_\infty$.

\end{lemma}	

\begin{proof}
%We prove (3). The proofs of (1) and (2) are similar.
We start with proving that the functor is a cocartesian fibration.
Let $(\mM^\circledast \to \mV^\ot \times \mW^\ot, \Lambda) \in {\omega\B\Enr}_\ast$
and $(\alpha, \beta): (\mV^\ot, \mW^\ot, \mH) \to (\mV'^\ot, \mW'^\ot,\mH')$
a morphism in $(\Op_\infty \times \Op_\infty)_*$, where $\mH$ is the set of underlying weights of $\Lambda.$
The cocartesian lift $\mM^\circledast \to \mM'^\circledast:=(\alpha, \beta)_!(\mM)^\circledast $ of $\alpha, \beta$ sends the set $\Lambda$ of $\mH$-weighted diagrams
to a set $\Lambda'$ of $\mH'$-weighted diagrams. 
By the choice of $\Lambda'$ the morphism 
$(\mM^\circledast \to \mV^\ot \times \mW^\ot, \Lambda) \to (\mM'^\circledast \to \mV'^\ot \times \mW'^\ot, \Lambda') $ in $\omega\B\Enr_\ast$
lies over $(\alpha, \beta): (\mV^\ot, \mW^\ot, \mH) \to (\mV'^\ot, \mW'^\ot,\mH')$ in $(\Op_\infty \times \Op_\infty)_*$ and induces for every $(\mN^\circledast \to \mV''^\ot \times \mW''^\ot, \Lambda'') \in {\omega\B\Enr_\ast}$
lying over $(\mV''^\ot \to \Ass, \mW''^\ot \to \Ass, \mH'') \in (\Op_\infty \times \Op_\infty)_*$ a pullback square
\begin{equation*}
\begin{xy}
\xymatrix{
\omega\B\Enr_\ast((\mM', \Lambda'),(\mN, \Lambda'')) \ar[d] \ar[r]
&\omega\B\Enr_\ast((\mM, \Lambda),(\mN, \Lambda'')) \ar[d]
\\ 
(\Op_\infty \times \Op_\infty)_*((\mV',\mW', \mH'),(\mV'', \mW'',\mH'')) \ar[r] &(\Op_\infty \times \Op_\infty)_*((\mV, \mW, \mH),(\mV'',\mW'',\mH'')). 
}
\end{xy} 
\end{equation*} 

%\end{proof}
%\begin{lemma}\label{llleee}The functor $ \omega\B\Enr_\ast \to (\Op_\infty \times \Op_\infty)_*$ is a cartesian fibration and the functor $\omega\B\Enr_\ast \to \omega\B\Enr$ sends cartesian lifts of morphisms in $ (\Op_\infty \times \Op_\infty)_*$ to cartesian lifts of morphisms in $ \Op_\infty \times \Op_\infty$.\end{lemma}	\begin{proof}
	
%Let $(\mM^\circledast \to \mV^\ot \times \mW^\ot, \Lambda) \in \omega\B\Enr_\ast$and $(\alpha, \beta): (\mV'^\ot, \mW'^\ot, \mH') \to (\mV^\ot, \mW^\ot,\mH)$a morphism in $(\Op_\infty \times \Op_\infty)_*$, where $\mH$ is the set of underlying weights of $\Lambda.$
We continue with proving that the functor is a cartesian fibration.
Let $\Lambda'$ be the set of $\mH'$-weighted diagrams on $(\alpha, \beta)^\ast(\mM)^\circledast \to \mV'^\ot \times \mW'^\ot $ whose image under the projection $(\alpha, \beta)^\ast(\mM)^\circledast \to \mM^\circledast$ belongs to $\Lambda.$
By the choice of $\Lambda'$ the morphism 
$((\alpha, \beta)^\ast(\mM)^\circledast \to \mV'^\ot \times \mW'^\ot, \Lambda') \to (\mM^\circledast \to \mV^\ot \times \mW^\ot, \Lambda) $ in $\omega\B\Enr_\ast$
lies over the morphism $(\alpha, \beta): (\mV'^\ot, \mW'^\ot, \mH') \to (\mV^\ot, \mW^\ot,\mH)$ of $(\Op_\infty \times \Op_\infty)_*$ and induces for any $(\mN^\circledast \to \mV''^\ot \times \mW''^\ot, \Lambda'') \in {\omega\B\Enr_\ast}$
lying over $(\mV''^\ot \to \Ass, \mW''^\ot \to \Ass, \mH'') \in (\Op_\infty \times \Op_\infty)_*$ a pullback square
\begin{equation*}
\begin{xy}
\xymatrix{
\omega\B\Enr_\ast((\mN, \Lambda''), ((\alpha, \beta)^\ast(\mM), \Lambda')) \ar[d] \ar[r]
& \omega\B\Enr_\ast((\mN, \Lambda''), (\mM, \Lambda)) \ar[d]
\\ 
(\Op_\infty \times \Op_\infty)_*((\mV'',\mW'', \mH''),(\mV', \mW',\mH')) \ar[r] &(\Op_\infty \times \Op_\infty)_*((\mV'', \mW'', \mH''),(\mV,\mW,\mH)). 
}
\end{xy} 
\end{equation*} 
	
\end{proof}
\begin{notation}Let $\mV^\ot \to \Ass, \mW^\ot \to \Ass$ be small $\infty$-operads.
For every set $\mH$ of absolute small weights over $\mV,\mW$ 
let ${_\mV\B\Enr}_{\mW}(\mH)_*$ be the fiber of the functor
$ \omega\B\Enr_\ast \to {(\Op_\infty \times \Op_\infty)_*}$ over $\mH.$

%For small monoidal $\infty$-categories $\mV^\ot \to \Ass, \mW^\ot \to \Ass$let $$ {_\mV\B\P\Enr}_{\mW}(\mH)_* \subset  {_\mV\omega\B\Enr}_{\mW}(\mH)_*$$ be the full subcategory of bipseudo-enriched $\infty$-categories with $\mH$-weighted diagrams. \item For monoidal $\infty$-categories $\mV^\ot \to \Ass, \mW^\ot \to \Ass$ compatible with small colimits let $$ {_\mV\B\Enr}_{\mW}(\mH)_* \subset {_\mV\B\P\widehat{\Enr}}_{\mW}(\mH)_*$$ be the full subcategory of small bienriched $\infty$-categories with $\mH$-weighted diagrams. 

\end{notation}

\begin{notation}
Let $$\omega\B\Enr_{\mathrm{w}} \subset {(\Op_\infty \times \Op_\infty)_*} \times_{(\Op_\infty \times \Op_\infty)} \omega\B\Enr$$ be the subcategory whose objects are pairs $(\mM^\circledast \to \mV^\ot \times \mW^\ot,\mH)$ such that $\mM^\circledast \to \mV^\ot \times \mW^\ot$ admits $\mH$-weighted colimits and whose morphisms $$(\mM^\circledast \to \mV^\ot \times \mW^\ot,\mH) \to (\mN^\circledast \to \mV'^\ot \times \mW'^\ot,\mH')$$ correspond to
an enriched functor $ \mM^\circledast \to \mN^\circledast$ that sends $\mH$-weighted colimits to $\mH'$-weighted colimits.

\end{notation}

\begin{lemma}\label{eqas} The inclusion $\omega\B\Enr_{\mathrm{w}} \subset {(\Op_\infty \times \Op_\infty)_* \times_{(\Op_\infty \times \Op_\infty)} \omega\B\Enr}$
lifts to an embedding $$\omega\B\Enr_{\mathrm{w}} \to {\omega\B\Enr_\ast}.$$ 

\end{lemma}

\begin{proof}We prove (1). The proofs of (2) is similar. Let $\mQ \subset {\omega\B\Enr_\ast}$ be the full subcategory spanned by the pairs $(\mM^\circledast \to \mV^\ot \times \mW^\ot, \Lambda)$ with the following property:
if $\mH$ is the set of weights underlying $\Lambda$,
then $\mM$ admits $\mH$-weighted colimits and $\Lambda$ is the set of all $\mH$-weighted colimit diagrams on $\mM.$
The functor $ \gamma: {\omega\B\Enr_\ast} \to {(\Op_\infty \times \Op_\infty)_* \times_{(\Op_\infty \times \Op_\infty)} \omega\B\Enr} $
restricts to a functor $\gamma': \mQ \to \omega\B\Enr_{\mathrm{w}}$. The functor $\gamma'$ is essentially surjective since for any $(\mM^\circledast \to \mV^\ot \times \mW^\ot ,\mH)$ we can equip $\mM^\circledast \to \mV^\ot \times \mW^\ot $ with the set of all $\mH$-weighted colimit diagrams on $\mM.$
The functor $\gamma $ induces embeddings on mapping spaces so that the restriction $\gamma'$ also does.
But by definition of $\mQ$ the functor $\gamma'$ induces equivalences on mapping spaces and so is an equivalence. This proves the claim.
\end{proof}

\begin{notation}

Let $\mV^\ot \to \Ass, \mW^\ot \to \Ass$ be small $\infty$-operads and $\mH$ a set of absolute small weights over $\mV,\mW.$ 
We write $$ {_\mV\omega\B\Enr}_{\mW}(\mH)$$ for the fiber of the functor
$ \omega\B\Enr_\wc \to {(\Op_\infty \times \Op_\infty)_*}$ over $\mH.$

\end{notation}

\begin{notation}
Let $(\mV^\ot \to \Ass, \rS)$, $(\mW^\ot \to \Ass, \T)$	be small localization pairs and $\mH$ a set of absolute small weights over $\mV, \mW.$

\begin{itemize}
\item Let $^\rS_\mV\B\Enr^\T_{\mW}(\mH)$ be the pullback $^\rS_\mV\B\Enr_\mW^\T \times_{{_\mV\omega\B\Enr}_{\mW}} {_\mV\omega\B\Enr}_{\mW}(\mH).$

\item Let ${^\rS_\mV\B\Enr^\T_{\mW}(\mH)_*}$ be the pullback $^\rS_\mV\B\Enr^\T_\mW \times_{{_\mV\omega\B\Enr}_{\mW}}{{_\mV\omega\B\Enr}_{\mW}(\mH)_*}.$

\end{itemize}

\end{notation}
Let $(\mV^\ot \to \Ass, \rS)$, $(\mW^\ot \to \Ass, \T)$	be small localization pairs and $\mH$ a set of absolute small weights over $\mV, \mW.$
%Let $\mV^\ot \to \Ass, \mW^\ot \to \Ass$ be small $\infty$-operads and $\mH$ a set of small weights over $\mV, \mW.$
The embedding $\omega\B\Enr_{\mathrm{w}} \to {\omega\B\Enr_\ast}$ 
over ${(\Op_\infty \times \Op_\infty)_*}$ of Lemma \ref{eqas} (1) induces on the fiber over $(\mV, \mW, \mH)$ an embedding ${_\mV\omega\B\Enr}_{\mW}(\mH) \subset {{_\mV\omega\B\Enr}_{\mW}(\mH)_*}$
that restricts to an embedding ${_\mV^\rS\B\Enr_\mW^\T}(\mH) \subset {{^\rS_\mV\B\Enr^\T_\mW}(\mH)_*}.$

\begin{notation}\label{presen}
\begin{enumerate}
		
\item Let $\mV^\ot \to \Ass$ be a presentably monoidal $\infty$-category, $\mW^\ot \to \Ass$ a small $\infty$-operad and $\mH$ a set of small left enriched diagrams over $\mV, \mW$.
Let $$ {{_{\mV}}\L\Enr_\mW}(\mH)_* \subset {_\mV\omega\widehat{\B\Enr}}_{\mW}(\mH)_*$$ be the full subcategory of small left enriched $\infty$-categories equipped with a set of small left enriched diagrams.
		
\item Let $\mW^\ot \to \Ass$ be a presentably monoidal $\infty$-category, $\mV^\ot \to \Ass$ a small $\infty$-operad and $\mH$ a set of small right enriched diagrams over $\mV, \mW$.
Let $$ {{_{\mV}}\R\Enr_\mW}(\mH)_*\subset {_\mV\omega\widehat{\B\Enr}}_{\mW}(\mH)_*$$ be the full subcategory of small right enriched $\infty$-categories equipped with a set of small right enriched diagrams.
		
\item Let $\mV^\ot \to \Ass, \mW^\ot \to \Ass$ be presentably monoidal $\infty$-categories and $\mH$ a set of small bienriched diagrams over $\mV, \mW$.
Let $$ {{_{\mV}}\B\Enr_\mW}(\mH)_*\subset {_\mV\omega\widehat{\B\Enr}}_{\mW}(\mH)_*$$ be the full subcategory of small bienriched $\infty$-categories equipped with a set of small bienriched diagrams.
\end{enumerate}	
	
%\vspace{1mm}\item Let $\mV^\ot \to \Ass, \mW^\ot \to \Ass$ be small $\infty$-operads with tensor unit.Then ${_\mV\ell\B\P\Enr}_{\mW}(\mH)_*$ is presentable.\vspace{1mm}\item Let $\mV^\ot \to \Ass, \mW^\ot \to \Ass$ be small monoidal $\infty$-categories.The $\infty$-category ${_\mV\B\P\Enr}_{\mW}(\mH)_*$ is presentable.\item Let $\mV^\ot \to \Ass, \mW^\ot \to \Ass$ be presentably monoidal $\infty$-categories. The $\infty$-category ${_\mV\B\Enr}_{\mW}(\mH)_*$ is presentable.
	
\end{notation}

\begin{lemma}\label{gene}

Let $\mC$ be a presentable $\infty$-category and $\gamma: \mC \to \Set$ an accessible functor.
Let $\mD \to \mC$ be the cocartesian fibration classifying the functor
$\mC \xrightarrow{\gamma} \Set \xrightarrow{\mP} \widehat{\Poset}.$
The $\infty$-category $\mD$ is presentable.

\end{lemma}

\begin{proof}
For every set $\X$ the power set $\mP(\X)$ seen as poset is presentable as $\infty$-category.
The functor $$\mC \xrightarrow{\gamma} \Set \xrightarrow{\mP} \widehat{\Poset} \subset \widehat{\Cat}_\infty$$
is accessible because the power set functor is.
Thus the cocartesian fibration $\mD \to \mC$ classifies an accessible functor
and is a cartesian fibration, whose fibers are presentable.
By \cite[Theorem 10.3.]{articles} this implies the claim.

\end{proof}

%\begin{notation}Let $(\mV^\ot \to \Ass, \rS), (\mW^\ot \to \Ass, \T)$ be small localization pairs and $\mH$ a set of small weights over $\mV,\mW.$ Let $$ {{^\rS_{\mV}}\B\Enr_\mW^\T}(\mH)_* \subset {_{\mV}\B\Enr}_{\mW}(\mH)_*$$ be the full subcategory of $\rS,\T$-bienriched $\infty$-categories with $\mH$-weighted diagrams. 	\end{notation}

\begin{proposition}\label{presen}
\begin{enumerate}
\item Let $(\mV^\ot \to \Ass, \rS), (\mW^\ot \to \Ass, \T)$ be small localization pairs and $\mH$ a set of absolute small weights over $\mV,\mW.$ 
The $\infty$-category $ {{^\rS_{\mV}}\B\Enr_\mW^\T}(\mH)_*$ is presentable.

\item Let $\mV^\ot \to \Ass$ be a presentably monoidal $\infty$-category, $\mW^\ot \to \Ass$ a small $\infty$-operad and $\mH$ a set of small left enriched diagrams over $\mV, \mW$.
The $\infty$-category $ {{_{\mV}}\L\Enr_\mW}(\mH)_*$ is presentable.

\item Let $\mW^\ot \to \Ass$ be a presentably monoidal $\infty$-category, $\mV^\ot \to \Ass$ a small $\infty$-operad and $\mH$ a set of small right enriched diagrams over $\mV, \mW$.
The $\infty$-category $ {{_{\mV}}\R\Enr_\mW}(\mH)_*$ is presentable.

\item Let $\mV^\ot \to \Ass, \mW^\ot \to \Ass$ be presentably monoidal $\infty$-categories and $\mH$ a set of small bienriched diagrams over $\mV, \mW$.
The $\infty$-category $ {{_{\mV}}\B\Enr_\mW}(\mH)_*$ is presentable.

\end{enumerate}	
	
%\vspace{1mm}\item Let $\mV^\ot \to \Ass, \mW^\ot \to \Ass$ be small $\infty$-operads with tensor unit.Then ${_\mV\ell\B\P\Enr}_{\mW}(\mH)_*$ is presentable.\vspace{1mm}\item Let $\mV^\ot \to \Ass, \mW^\ot \to \Ass$ be small monoidal $\infty$-categories.The $\infty$-category ${_\mV\B\P\Enr}_{\mW}(\mH)_*$ is presentable.\item Let $\mV^\ot \to \Ass, \mW^\ot \to \Ass$ be presentably monoidal $\infty$-categories. The $\infty$-category ${_\mV\B\Enr}_{\mW}(\mH)_*$ is presentable.
	
\end{proposition}

\begin{proof}
We apply Lemma \ref{gene}. (1): We apply Lemma \ref{gene}. By \cite[Proposition 2.60.]{heine2024bienriched} the $\infty$-category ${^\rS_\mV\omega\B\Enr_\mW^\T}$ is compactly generated.
%$ {_\mV\omega\B\Enr}_{\mW}, {_\mV\B\P\Enr}_{\mW}, {_\mV\B\Enr}_{\mW} $ are presentable.
Let $\gamma$ be the functor $ {^\rS_\mV\omega\B\Enr^\T_{\mW}} \to \Set$ that sends a $\rS,\T$-bienriched $\infty$-category $\mM^\circledast \to \mV^\ot \times \mW^\ot$ to the set of equivalence classes of $\rH$-weighted diagrams on $\mM$ for some $\rH \in \mH.$
Since $\mH$ is a set of absolute small weights, the functor $\gamma$ is accessible.
%The full subcategories $ _\mV\B\P\Enr_{\mW}, \ {_\mV\B\Enr_{\mW}} \subset {_\mV\omega\B\Enr}_{\mW} $ are closed under filtered colimits. So the restrictions of $\tau$ to $ {_\mV\B\P\Enr}_{\mW},  {_\mV \B\Enr}_{\mW} $ are accessible, too. 
So we can apply Lemma \ref{gene} to deduce the claim. 

(2): Let $\kappa$ be a small regular cardinal such that $\mV^\ot \to \Ass$ is a $\kappa$-compactly generated monoidal $\infty$-category.
By Corollary \ref{cosqa} (1) %and Lemma \ref{mini} (1) 
taking pullback along the monoidal embedding $(\mV^\kappa)^\ot \subset \mV^\ot$ induces an equivalence
$ {{_{\mV}}\L\Enr_\mW}\simeq {_{\mV^\kappa}^\kappa\L\Enr_{\mW}}.$
Thus $ {{_{\mV}}\L\Enr_\mW}$ is presentable.
Let $\gamma$ be the functor $ {_\mV\L\Enr_{\mW}} \to \Set$ that sends a left enriched $\infty$-category $\mM^\circledast \to \mV^\ot \times \mW^\ot$ to the set of equivalence classes of $\rH$-weighted diagrams on $\mM$ for some $\rH \in \mH.$
Since $\mH$ is a set of small left enriched weights, the functor $\gamma$ is accessible.
%The full subcategories $ _\mV\B\P\Enr_{\mW}, \ {_\mV\B\Enr_{\mW}} \subset {_\mV\omega\B\Enr}_{\mW} $ are closed under filtered colimits. So the restrictions of $\tau$ to $ {_\mV\B\P\Enr}_{\mW},  {_\mV \B\Enr}_{\mW} $ are accessible, too. 
So we can apply Lemma \ref{gene} to deduce the claim. 
%where $\mH:= \{\rH' \mid \rH \in \mH\}$ is a set of small left $\kappa$-enriched weights and $\rH'$ is defined like in Lemma \ref{mini} (1). So the claim follows from Corollary \ref{presenta}.
(3) is dual to (2).	
(4): Let $\kappa, \tau$ be small regular cardinals such that $\mV^\ot \to \Ass$ is a $\kappa$-compactly generated monoidal $\infty$-category and $\mW^\ot \to \Ass$ is a $\tau$-compactly generated monoidal $\infty$-category.
By Corollary \ref{cosqa} (3) taking pullback along the monoidal embeddings $(\mV^\kappa)^\ot \subset \mV^\ot, (\mW^\tau)^\ot \subset \mW^\ot$ induces an equivalence
$ {{_{\mV}}\B\Enr_\mW} \simeq {_{\mV^\kappa}^\kappa\B\Enr^\tau_{\mW^\tau}}.$
Thus $ {{_{\mV}}\B\Enr_\mW}$ is presentable.
Let $\gamma$ be the functor $ {_\mV\B\Enr_{\mW}} \to \Set$ that sends a bienriched $\infty$-category $\mM^\circledast \to \mV^\ot \times \mW^\ot$ to the set of equivalence classes of $\rH$-weighted diagrams on $\mM$ for some $\rH \in \mH.$
Since $\mH$ is a set of small enriched weights, the functor $\gamma$ is accessible.
%The full subcategories $ _\mV\B\P\Enr_{\mW}, \ {_\mV\B\Enr_{\mW}} \subset {_\mV\omega\B\Enr}_{\mW} $ are closed under filtered colimits. So the restrictions of $\tau$ to $ {_\mV\B\P\Enr}_{\mW},  {_\mV \B\Enr}_{\mW} $ are accessible, too. 
So we apply Lemma \ref{gene}.
\end{proof}

%Let $(\mV^\ot \to \Ass, \rS)$, $(\mW^\ot \to \Ass, \T)$	be localization pairs and $\mH$ a set of small weights over $\mV, \mW.$The latter embedding restricts to an embedding $^\rS_\mV\B\P\Enr^\T_{\mW}(\mH) \subset ^\rS_\mV\B\P\Enr^\T_{\mW}(\mH)_*$.

\begin{remark}\label{cola} The $\infty$-category $\omega\B\Enr_* $ admits small colimits. The colimit of a functor $\K \to \omega\B\Enr_*$ is the pair consisting of the colimit of the functor $\F: \K \to \omega\B\Enr_* \to \omega\B\Enr$ and the set of diagrams
that precisely consists of the diagrams that are the image of a diagram on $\F(\bk)$ under the enriched functor $\F(\bk)^\circledast \to \colim(\F)^\circledast$ for some $\bk \in \K.$

Let $(\mV^\ot \to \Ass, \rS)$, $(\mW^\ot \to \Ass, \T)$	be small localization pairs and $\mH$ a set of absolute small weights over $\mV, \mW.$
%The functor ${_\mV\omega\B\Enr}_\mW(\mH)_* \to {_\mV\omega\B\Enr}_\mW$ admits a left and right adjoint, where the left adjoint equips a small $\infty$-category weakly bienriched in $\mV, \mW$ with the empty set of diagrams and the right adjoint equips a small $\infty$-category weakly bienriched in $\mV, \mW$ with the set of all $\mH$-weighted diagrams.
Similarly, the $\infty$-category ${^\rS_\mV\omega\B\Enr^\T_{\mW}}(\mH)_* $ admits small colimits. The colimit of a functor $\K \to {^\rS_\mV\omega\B\Enr^\T_{\mW}}(\mH)_*$ is the pair consisting of the colimit of the functor $\F: \K \to {^\rS_\mV\omega\B\Enr^\T_{\mW}}(\mH)_* \to {^\rS_\mV\omega\B\Enr^\T_{\mW}}$ and the set of diagrams
that precisely consists of the $\mH$-weighted diagrams that are the image of a $\mH$-weighted diagram on $\F(\bk)$ under the $\mV, \mW$-enriched functor $\F(\bk)^\circledast \to \colim(\F)^\circledast$ for some $\bk \in \K.$

\end{remark}

\begin{proposition}\label{gcp}\label{gcpp}
\begin{enumerate}
\item Let $(\mV^\ot \to \Ass, \rS)$, $(\mW^\ot \to \Ass, \T)$ be small localization pairs and $\mH$ a set of absolute small weights over $\mV, \mW.$
The embedding $${_\mV^\rS\B\Enr_\mW^\T}(\mH) \subset {{^\rS_\mV\B\Enr^\T_\mW}(\mH)_*}$$ is accessible and admits a left adjoint.
\item Let $\mV^\ot \to \Ass$ be a presentably monoidal $\infty$-category, $\mW^\ot \to \Ass$ a small $\infty$-operad and $\mH$ a set of small left enriched diagrams over $\mV, \mW$.
The embedding $${_\mV\L\Enr_\mW}(\mH) \subset {{_\mV\L\Enr_\mW}(\mH)_*}$$ is accessible and admits a left adjoint.

\item Let $\mW^\ot \to \Ass$ be a presentably monoidal $\infty$-category, $\mV^\ot \to \Ass$ a small $\infty$-operad and $\mH$ a set of small right enriched diagrams over $\mV, \mW$.
The embedding $${_\mV\R\Enr_\mW}(\mH) \subset {{_\mV\R\Enr_\mW}(\mH)_*}$$ is accessible and admits a left adjoint.

\item Let $\mV^\ot \to \Ass, \mW^\ot \to \Ass$ be presentably monoidal $\infty$-categories and $\mH$ a set of small bienriched diagrams over $\mV, \mW$.
The embedding $${_\mV\B\Enr_\mW}(\mH) \subset {{_\mV\B\Enr_\mW}(\mH)_*}$$ is accessible and admits a left adjoint.

\end{enumerate}	
	
\end{proposition}

\begin{proof}
By Corollary \ref{wooond0} the embeddings of the statement admit a left adjoint.	
So it remains to see that the embeddings of the statement are accessible. %In view of Remark \ref{cola} we can reduce to the case that $\rS=\T=\emptyset.$	
We prove this for (1). The proofs of (2)-(4) are similar.
Since $\mH$ is small and ${^\rS_\mV\B\Enr^\T_{\mW}}$ is compactly generated by \cite[Proposition 2.60.]{heine2024bienriched}, there is a small regular cardinal $\kappa$ such that for any weight $(\rH, \mJ^\circledast \to \mV^\ot \times \mW^\ot) \in \mH$ the $\rS,\T$-bienriched $\infty$-category $\mJ^\circledast \to \mV^\ot \times \mW^\ot$ is a $\kappa$-compact object of ${_\mV^\rS\B\Enr^\T_{\mW}}.$
We will prove that the embedding ${_\mV^\rS\B\Enr^\T_{\mW}}(\mH) \subset {{_\mV^\rS\B\Enr^\T_{\mW}}(\mH)_*}$ preserves small $\kappa$-filtered colimits.

Let $\K$ be a small $\kappa$-filtered $\infty$-category, $\Psi : \K \to {^\rS_\mV\omega\B\Enr^\T_{\mW}}(\mH) $ a functor and $(\mM^\circledast \to \mV^\ot\times \mW^\ot,\Lambda)$ the colimit of $\Psi$ taken in ${{^\rS_\mV\omega\B\Enr^\T_{\mW}}(\mH)_*} .$
We will show that $\mM^\circledast \to \mV^\ot \times \mW^\ot$ admits $\mH$-weighted colimits and $\Lambda$ is precisely the set of $\rH$-weighted colimit diagrams on $\mM$ for some $\rH \in \mH.$
Observe that for every $\bk \in \K$ the canonical map $\psi: \Psi(\bk)^\circledast \to \mM^\circledast $ preserves $\mH$-weighted colimits because for every $\V_1,...,\V_\n \in \mV, \W_1,..., \W_\m \in \mW$ for $\n, \m \geq 0$ and $\X,\Y \in \Psi(\bk)$
the canonical map 
\begin{equation*}\label{abc}
\colim_{\alpha:\bk \to \ell}\Mul_{\Psi(\ell)}(\V_1, ..., \V_\n, \Psi(\alpha)(\X), \W_1, ..., \W_\m;\Psi(\alpha)(\Y)) \to$$$$ \Mul_\mM(\V_1, ..., \V_\n,\psi(\X), \W_1, ..., \W_\m;\psi(\Y)) 
\end{equation*}
is an equivalence, and similarly for $\X,\Y \in \mP\B\Env(\Psi(\bk))$ the canonical map 
\begin{equation}\label{aabccoo}
\colim_{\alpha:\bk \to \ell}\Mul_{\mP\B\Env(\Psi(\ell))}(\V_1, ..., \V_\n, \Psi(\alpha)_!(\X), \W_1, ..., \W_\m;\Psi(\alpha)_!(\Y)) \to$$$$ \Mul_{\mP\B\Env(\mM)}(\V_1, ..., \V_\n,\psi_!(\X), \W_1, ..., \W_\m;\psi_!(\Y)) 
\end{equation} is an equivalence.
This implies that $\Lambda$ consists of $\rH$-weighted colimit diagrams on $\mM$
for some $\rH \in \mH.$
Let $(\rH, \mJ^\circledast \to \mV^\ot\times \mW^\ot) \in \mH$ be a weight.
Then there is a $\bk \in \K$ so that for any morphism $\bk \to \ell$ in $\K$ the $\rS,\T$-bienriched $\infty$-category $\Psi(\ell)^\circledast \to \mV^\ot \times \mW^\ot$ admits $\rH$-weighted colimits.
Let $\F: \mJ^\circledast \to \mM^\circledast$ be an enriched functor.
By assumption $\mJ^\circledast \to \mV^\ot \times \mW^\ot$ is $\kappa$-compact so that
$\F$ factors as $\mJ^\circledast \xrightarrow{\G} \Psi(\br)^\circledast \to \mM^\circledast$ for some $\br \in \K$, which we can assume to receive a morphism $\bk \to \br$.
The map $\G: \mJ^\circledast \to \Psi(\br)^\circledast$
admits a $\rH$-weighted colimit that is sent to the $\rH$-weighted colimit of $\F.$
So $\mM^\circledast \to \mV^\ot \times \mW^\ot$ admits $\mH$-weighted colimits.

It remains to see that every $\rH$-weighted colimit diagram $(\rH, \F:\mJ^\circledast \to \mM^\circledast, \lambda: \F_!(\rH)\to \Y)$ on $\mM$
for some $\rH \in \mH$ belongs to $\Lambda$.
Since $\F: \mJ^\circledast \to \mV^\ot \times \mW^\ot$ is $\kappa$-compact,
$\F$ factors as $\mJ^\circledast \xrightarrow{\G} \Psi(\bk)^\circledast \to \mM^\circledast$ for some $\bk \in \K$.
By equivalence (\ref{aabccoo}) there is a morphism $\bk \to \ell $ in $\K$ and a diagram $\nu$ on $\Psi(\ell)$ that is transported by the enriched functor $\Psi(\ell)^\circledast \to \mM^\circledast$ to $(\rH, \F:\mJ^\circledast \to \mM^\circledast, \lambda: \F_!(\rH)\to \Y)$.
As $\nu$ is transported by $\Psi(\ell)^\circledast \to \mM^\circledast$ to a $\mH$-weighted colimit diagram and $\Psi$ sends morphisms in $\K$ to enriched functors preserving $\mH$-weighted colimits and for any $\bk \in \K$ the $\infty$-category $\Psi(\bk)$ admits $\mH$-weighted colimits, there is a morphism $\ell \to \br$ such that $\nu$ is transported by the enriched functor $\Psi(\ell) \to \Psi(\br)$ to a $\mH$-weighted colimit diagram $\nu'.$
Since $\Psi(\br) \in {^\rS_\mV\omega\B\Enr^\T_{\mW}}(\mH) $, we find that $\nu'$ is a diagram $\Psi(\br)$ is equipped with.
Hence $(\rH, \F:\mJ^\circledast \to \mM^\circledast, \lambda: \F_!(\rH)\to \Y)$ belongs to $\Lambda.$

\end{proof}

\begin{corollary}\label{presenta} 
\begin{enumerate}
\item Let $(\mV^\ot \to \Ass, \rS), (\mW^\ot \to \Ass, \T)$ be small localization pairs and $\mH$ a set of absolute small weights over $\mV,\mW.$ 
The $\infty$-category $ {{^\rS_{\mV}}\B\Enr_\mW^\T}(\mH)$ is presentable.	
	
\item Let $\mV^\ot \to \Ass$ be a presentably monoidal $\infty$-category, $\mW^\ot \to \Ass$ a small $\infty$-operad and $\mH$ a set of small left enriched diagrams over $\mV, \mW$.
The $\infty$-category $ {{_{\mV}}\L\Enr_\mW}(\mH)$ is presentable.
		
\item Let $\mW^\ot \to \Ass$ be a presentably monoidal $\infty$-category, $\mV^\ot \to \Ass$ a small $\infty$-operad and $\mH$ a set of small right enriched diagrams over $\mV, \mW$.
The $\infty$-category $ {{_{\mV}}\R\Enr_\mW}(\mH)$ is presentable.
		
\item Let $\mV^\ot \to \Ass, \mW^\ot \to \Ass$ be presentably monoidal $\infty$-categories and $\mH$ a set of small bienriched diagrams over $\mV, \mW$.
The $\infty$-category $ {{_{\mV}}\B\Enr_\mW}(\mH)$ is presentable.
		
\end{enumerate}	

\end{corollary}

%\begin{proof}(3): Let $\kappa$ be a small regular cardinal such that $\mV^\ot \to \Ass$ is a $\kappa$-compactly generated monoidal $\infty$-category.By Corollary \ref{cosqa} (1) %and Lemma \ref{mini} (1) taking pullback along the monoidal embedding $(\mV^\kappa)^\ot \subset \mV^\ot$ induces an equivalence$ {{_{\mV}}\L\Enr_\mW}\simeq {_{\mV^\kappa}^\kappa\L\Enr_{\mW}}.$
%where $\mH:= \{\rH' \mid \rH \in \mH\}$ is a set of small left $\kappa$-enriched weights and $\rH'$ is defined like in Lemma \ref{mini} (1). So the claim follows from Corollary \ref{presenta}.
%(3): Let $\kappa, \tau$ be small regular cardinals such that $\mV^\ot \to \Ass$ is a $\kappa$-compactly generated monoidal $\infty$-category and $\mW^\ot \to \Ass$ is a $\tau$-compactly generated monoidal $\infty$-category.By Corollary \ref{cosqa} (3) %and Lemma \ref{mini} (3) taking pullback along the monoidal embeddings $(\mV^\kappa)^\ot \subset \mV^\ot, (\mW^\tau)^\ot \subset \mW^\ot$ induces an equivalence$ {{_{\mV}}\B\Enr_\mW} \simeq {_{\mV^\kappa}^\kappa\B\Enr^\tau_{\mW^\tau}}.$By ... 
%where $\mH:= \{\rH' \mid \rH \in \mH\}$ is a set of small $\kappa,\tau$-enriched weights and $\rH'$ is defined like in Lemma \ref{mini} (3). So the claim follows from Corollary \ref{presenta}.\end{proof}	

\begin{corollary}\label{ujpp}
	
The functor $ \omega\B\Enr_{\mathrm{w}} \to {(\Op_\infty \times \Op_\infty)_*}$ is a cocartesian and cartesian fibration and the functor $\omega\B\Enr_{\mathrm{w}} \to \omega\B\Enr$ sends cartesian lifts of morphisms in $ (\Op_\infty \times \Op_\infty)_*$ to cartesian lifts in $ \Op_\infty \times \Op_\infty$.	
	
\end{corollary}

\begin{proof}
By Lemma \ref{lllee} the functor $ \omega\B\Enr_* \to {(\Op_\infty \times \Op_\infty)_*}$ is a cartesian fibration and the functor $\omega\B\Enr_* \to \omega\B\Enr$ sends cartesian lifts of morphisms in $ (\Op_\infty \times \Op_\infty)_*$ to cartesian lifts in $ \Op_\infty \times \Op_\infty$.	
By Corollary \ref{eccco} the latter cartesian fibration restricts to a cartesian fibration 
$ \omega\B\Enr_{\mathrm{w}} \to {(\Op_\infty \times \Op_\infty)_*}$ with the same cartesian morphisms.
This cartesian fibration is also a cocartesian fibration since the fiber transports admit left adjoints by Proposition \ref{gcpp}.

\end{proof}

\subsection{The tensor product of enriched $\infty$-categories with weighted colimits}

In this subsection we construct a tensor product for left enriched $\infty$-categories with weighted diagrams and apply the theory of adjoining weighted colimits (Theorem \ref{wooond}) to construct a tensor product for left enriched $\infty$-categories with weighted colimits (Theorem \ref{thqaz} and Theorem \ref{Thor}) from the first one.

%Next we construct symmetric monoidal structures on the $\infty$-categoriesof small bienriched $\infty$-categories with diagrams.
%, that restrict to pseudo-enriched and enriched $\infty$-categories with diagrams.
%For the next remark we use Notation \ref{noturn}:

\begin{notation}\label{pose}Let $\mV^\ot \to \Ass$ be a small $\infty$-operad and $\mM^\circledast \to \mV^\ot $ a  weakly left enriched $\infty$-category.
	
\begin{enumerate}
\item Let $\L\D(\mM)$ be the poset of sets of absolute small left diagrams on $\mM$ ordered by inclusion. 
	
\item Let $\omega(\mV)$ be the poset of sets of absolute small left weights over $\mV$ ordered by inclusion.
\end{enumerate}	
%Since we can transport diagrams and weights (Notation \ref{trans}), we obtain functors to the category of large posets $\widehat{\Poset}:$$$\B\D: \Ho(\omega\B\Enr) \to \widehat{\Poset},\ \omega: \Ho(\Op_\infty) \times \Ho(\Op_\infty) \to \widehat{\Poset}.$$
	
\end{notation}

\begin{remark}\label{opm}

\begin{enumerate}
\item The functor $$\omega: \Ho(\Mon) \to \widehat{\Poset}, $$ 
is lax symmetric monoidal with respect to cartesian symmetric monoidal structures. 

For every $\mV_1^\ot,...,\mV_\n^\ot \in \Mon$ for $\n \geq 0$ the structure maps $$ \prod_{\bi=1}^\n \omega(\mV_\bi) \to \omega(\prod_{\bi=1}^\n \mV_\bi)$$ send $ (\mH_1,...,\mH_\n) $ to $ \mH_1 \boxtimes...\boxtimes\mH_\n.$

\item The functor $$\L\D: \Ho(\L\P\Enr)\to \widehat{\Poset} $$
is lax symmetric monoidal with respect to cartesian symmetric monoidal structures. 

For every $\mM_1^\circledast \to \mV_1^\ot,...,\mM_\n^\circledast \to \mV_\n^\ot $ for $\n \geq 0$ the structure maps $$ \prod_{\bi=1}^\n \L\D(\mM_\bi) \to \L\D(\prod_{\bi=1}^\n \mM_\bi) $$ send $ (\Lambda_1,...,\Lambda_\n) $ to $ \Lambda_1 \boxtimes...\boxtimes\Lambda_\n.$

%\item The functor $$\R\D: \Ho(\ell\R\P\Enr)\to \widehat{\Poset} $$ is lax symmetric monoidal with respect to cartesian symmetric monoidal structures. For every $\mM_1^\circledast \to \mV_1^\ot \times \mW_1^\ot,...,\mM_\n^\circledast \to \mV_\n^\ot \times \mW_\n^\ot \in \omega\B\Enr$ for $\n \geq 0$ the structure map $$ \prod_{\bi=1}^\n \R\D(\mM_\bi) \to \R\D(\prod_{\bi=1}^\n \mM_\bi) $$ sends $ (\Lambda_1,...,\Lambda_\n) $ to $ \Lambda_1 \boxtimes...\boxtimes\Lambda_\n.$

%\item The functor $$\B\D: \Ho(\ell\B\P\Enr)\to \widehat{\Poset} $$ is lax symmetric monoidal with respect to cartesian symmetric monoidal structures. For every $\mM_1^\circledast \to \mV_1^\ot \times \mW_1^\ot,...,\mM_\n^\circledast \to \mV_\n^\ot \times \mW_\n^\ot \in \omega\B\Enr$ for $\n \geq 0$ the structure map $$ \prod_{\bi=1}^\n \B\D(\mM_\bi) \to \B\D(\prod_{\bi=1}^\n \mM_\bi) $$ sends $ (\Lambda_1,...,\Lambda_\n) $ to $ \Lambda_1 \boxtimes...\boxtimes\Lambda_\n.$

\vspace{1mm}
\item Let $\nu:\Ho(\L\P\Enr) \to \Ho(\Mon) $ be the forgetful functor. 
The natural transformation $ \L\D \to \omega \circ \nu$ 
is symmetric monoidal.
\end{enumerate}
\end{remark}

\begin{notation}
Let $\kappa$ be a small regular cardinal.
Let $\Mon_\kappa^\ot \subset \Mon^\times$ be the symmetric suboperad whose colors are the monoidal $\infty$-categories compatible with $\kappa$-small colimits and whose multimorphisms $\mV_1^\ot, ..., \mV_\n^\ot \to \mW^\ot$ correspond to monoidal functors $\mV_1^\ot \times_\Ass...\times_\Ass \mV_\n^\ot \to \mW^\ot$ preserving $\kappa$-small colimits component-wise.
\end{notation}

\begin{notation}
For $\sigma$ the strongly inaccessible cardinal corresponding to the small universe we set $$ \cc\cc\Mon^\ot:=\widehat{\Mon}_\sigma^\otimes.$$
\end{notation}
\begin{notation}
Let $$\Pr\Mon^\ot \subset \cc\cc\Mon^\otimes$$ be the full symmetric suboperad 
spanned by the presentably monoidal $\infty$-categories.
% and whose multimorphisms $\mV_1^\ot, ..., \mV_\n^\ot \to \mW^\ot$ correspond to the monoidal functors $\mV_1^\ot \times_\Ass...\times_\Ass \mV_\n^\ot \to \mW^\ot$ that preserve small colimits component-wise.
	
\end{notation}

By \cite[Proposition 4.8.1.15.]{lurie.higheralgebra} $\Mon_\kappa^\ot \to \Comm, \Pr\Mon^\ot \to \Comm$ are symmetric monoidal $\infty$-categories.

\begin{notation}
	
Let $\kappa$ be a small regular cardinal. Let %$$%{^\kappa\L\Enr}^\otimes \subset \Mon_\kappa^\ot \times_{\Mon^\times}\L\P\Enr^\times, \ \R\Enr^{\tau \otimes} \subset  \R\P\Enr^\times\times_{\Mon^\times}\Mon_\tau^\ot,$$
$$ {^\kappa\L\Enr^{\otimes}} \subset \Mon_\kappa^\ot \times_{\Mon^\times}(\L\P\Enr_\emptyset)^\times$$
be the full symmetric suboperad spanned by the left $\kappa$-enriched $\infty$-categories.
\end{notation}

\begin{proposition}\label{eras}
	
Let $\kappa$ be a small regular cardinal. The map of symmetric $\infty$-operads $${^\kappa\L\Enr^{\otimes}} \to \Mon_\kappa^\ot $$
is a cocartesian fibration.
	
\end{proposition}

\begin{proof}
We first prove that the symmetric monoidal functor 
$\phi: (\L\P\Enr_\emptyset)^\times \to \Mon^\times$ is a cocartesian fibration.
By Proposition \ref{bica} the underlying functor is a cocartesian fibration. So by \cite[Lemma 2.44.]{heine2023monadicity} it is enough to check that %taking product with any bipseudo-enriched $\infty$-category $\mM''^\circledast \to \mV''^\ot \times \mW''^\ot$ preserves $\phi$-cocartesian morphisms, i.e. that for any enriched functor $\mM^\circledast\to \mM'^\circledast$ between bipseudo-enriched $\infty$-categories the induced enriched functor $$\mM^\circledast \times_{(\Ass \times \Ass)} \mM''^\circledast \to \mM'^\circledast \times_{(\Ass \times \Ass)} \mM''^\circledast$$ is $\phi$-cocartesian.
the collection of $\phi$-cocartesian morphisms is stable under product. This follows from the description of $\phi$-cocartesian morphisms of Proposition \ref{bica}.
By \cite[Corollary 4.61.]{heine2024bienriched} for every small monoidal $\infty$-category $\mV^\ot \to \Ass$ compatible with $\kappa$-small colimits the full subcategory ${_\mV^\kappa\L\Enr} \subset {_\mV\L\P\Enr}$ is a localization.
Moreover for every monoidal functor $\mV^\ot \to \mV'^\ot$ preserving $\kappa$-small colimits the induced functor ${_\mV\L\P\Enr} \to {_{\mV'}\L\P\Enr}$ preserves local equivalences. 
Thus by \cite[Lemma 2.2.4.11.]{lurie.higheralgebra} the statement follows if we have shown that for every small monoidal $\infty$-categories $\mV^\ot \to \Ass, \mV'^\ot \to \Ass, \mV''^\ot \to \Ass$ and monoidal functors $\alpha: \mV^\ot \times_{\Ass} \mV'^\ot \to \mV''^\ot$ preserving $\kappa$-small colimits component-wise the induced functor
${_\mV\L\P\Enr} \times {_{\mV'}\L\P\Enr} \to {_{\mV''}\L\P\Enr} $ preserves local equivalences.
In view of Proposition \ref{bica} and the description of local equivalences of \cite[Corollary 4.61.]{heine2024bienriched} this follows from the fact that the functor
$\mP(\mV) \times \mP(\mV') \to \mP(\mV\times\mV') \to \mP(\mV'')$ sends local equivalences for the localization $\Ind_\kappa(\mV) \times \Ind_\kappa(\mV')$ to local equivalences for the localization $\Ind_\kappa(\mV'')$.
This follows from the fact that $\alpha$ preserves $\kappa$-small colimits component-wise.
\end{proof}

\begin{notation}\label{aloiso} 
	
Let $\kappa$ be a small regular cardinal. Let $$ {^\kappa\L\Enr_*^{\otimes}} \to {^\kappa\L\Enr^{\otimes}}$$
be the cocartesian fibration of symmetric monoidal $\infty$-categories classifying the lax symmetric monoidal functor $$^\kappa\L\Enr \subset \L\P\Enr \to \Ho(\L\P\Enr) \xrightarrow{\L\D} \widehat{\Poset} \subset \widehat{\Cat}_\infty.$$
Let \begin{equation}\label{cho} (\Mon_\kappa)_\ast^\ot  \to \Mon_\kappa^\ot \end{equation}
be the cocartesian fibration of symmetric monoidal $\infty$-categories classifying the lax symmetric monoidal functor
$$\Mon_\kappa \subset \Mon \to \Ho(\Mon) \xrightarrow{\omega} \widehat{\Poset} \subset \widehat{\Cat}_\infty.$$
\end{notation}

\begin{notation}
For $\sigma$ the strongly inaccessible cardinal corresponding to the small universe
we set 
$$\rc\rc\Mon_\ast^\ot:= (\Mon_\sigma)_\ast^\ot, \ {\L\Q\Enr_\ast^\ot} :={^\sigma\widehat{\L\Enr}^{\ot}_*}.$$
%\end{enumerate}
\end{notation}

\begin{remark}
By Remark \ref{opm} (3) the transformation $ \L\D \to \omega \circ \nu $ is symmetric monoidal and so classifies a map
\begin{equation}\label{chor}{^\kappa\L\Enr_*^{\otimes}} \to {(\Mon_\kappa)_\ast^\ot  \times_{ \Mon_\kappa^\ot} {^\kappa\L\Enr^{\otimes}}} \end{equation}
of cocartesian fibrations over ${^\kappa\L\Enr^{\otimes}}$
and so a symmetric monoidal functor.

\end{remark}
 
\begin{remark}\label{ohuz}
Since the fibers of the cocartesian fibration %$\gamma:{_*(\Mon_\kappa \times \Mon_\tau)_\ast^\ot} \to \Mon_\kappa^\ot \times_{\Comm} \Mon_\tau^\otimes$ 
(\ref{cho}) are posets, for any symmetric $\infty$-operad $\mO^\ot \to \Comm$ the canonical functor
$$\Alg_\mO((\Mon_\kappa)_\ast) \to \Alg_\mO(\Mon_\kappa) \times_{(\prod_{\X \in \mO} \Mon_\kappa)}\prod_{\X \in \mO} {(\Mon_\kappa)_\ast} $$ is fully faithful and the essential image are 
the triples $(\mV, \mH)$, where $\mV \in  \Alg_\mO(\Mon_\kappa)$ and $\mH:= (\mH_\X)_{\X \in \mO}$ is a family such that $\mH_\X$ is a set of absolute small weights over $ \mV(\X)$ for every $\X \in \mO$ subject to the condition that for every multi-morphism $\alpha: \X_1,...,\Y_\n \to \Y$ in $\mO$ we have $\alpha_!(\mH_{\X_1} \boxtimes... \boxtimes \mH_{\X_\n}) \subset \mH_{\Y}.$ 
%$\mO$-algebras $(\mV, \mW)$ in $\Mon_\kappa \times \Mon_\tau$ equipped with sets $\mH_\X$ of small weights over $\mV(\X), \mW(\X)$ for every $\X \in \mO$ such that for every multi-morphism $\alpha: \X_1,...,\Y_\n \to \Y$ in $\mO$ we have $\alpha_!(\mH_{\X_1} \boxtimes... \boxtimes \mH_{\X_\n}) \subset \mH_{\Y}.$ 
Thus we will denote an $\mO$-algebra in $(\Mon_\kappa)_*$ as a pair $(\mV,\mH)$.
%, where $\mV \in  \Alg_\mO(\Mon_\kappa), \mW \in  \Alg_\mO(\Mon_\tau)$ and $\mH:= (\mH_\X)_{\X \in \mO}$ is a family such that $\mH_\X$ is a set of small weights over $ \mV(\X), \mW(\X)$ for every $\X \in \mO.$And similar for the cocartesian fibrations (\ref{aho}), (\ref{bho}).

%Similarly, the canonical functor$$\Alg_\mO(_*{(\Op^\tu_\infty)}) \to 
%\Alg_\mO(\Op^\tu_\infty) \times_{(\prod_{\X \in \mO} \Op^\tu_\infty) } \prod_{\X \in \mO} {_*{(\Op^\tu_\infty)} }$$ is fully faithful and the essential image are the $\mO$-algebras $\mV$ in $\Mon $ equipped with sets $\mH_\X$ of small left weights over $\mV(\X)$ for every $\X \in \mO$ preserved by the tensor product.
%such that for every multi-morphism $\alpha: \X_1,...,\Y_\n \to \Y$ in $\mO$ we have $\alpha_!(\mH_{\X_1} \boxtimes... \boxtimes \mH_{\X_\n}) \subset \mH_{\Y}.$ 
%And similar for right weights.
\end{remark}

\begin{example}
Let $\mO^\ot\to \Comm$ be an unital monochromatic symmetric $\infty$-operad,
i.e. %$\mO^\simeq$ is connected and the space of unary operations is contractible
%(which is equivalent to ask that 
$\mO$ is contractible, for example the $\infty$-operads $\bE_\n$ for $1 \leq \n \leq \infty.$
Then an $\mO$-algebra in ${(\Mon_\kappa)_\ast^\ot}$
is a pair $(\mV, \mH)$, where $\mV \in \Alg_\mO(\Mon_\kappa)$
and $\mH$ is a set of absolute small left weights over $ \mV$.
The condition that for every multi-morphism $\alpha $ in $\mO$ corresponding to an $\n$-ary operation we have $\alpha_!(\mH \boxtimes... \boxtimes \mH) \subset \mH$ 
is automatic since the composition $\{\tu\} \times ... \times \{\tu\} \times \mC \times \{\tu\} \times ... \times \{\tu\} \to \mC^{\times\n} \xrightarrow{\alpha_!} \mC$ is the identity.

\end{example}

\begin{example}\label{infi}
	
Let $\mO \to \Comm$ be a symmetric $\infty$-operad, $\kappa$ a small regular cardinal and $\mV: \mO^\ot \to \Mon^\times$ an $\mO$-algebra.
%\begin{itemize}
%\item Let $\L\emptyset, \R\emptyset $ be the families that assign to every $\X \in \mO$ the empty set of left weights over $\mV(\X),$ the empty set of right weights over $\mW(\X),$ respectively.\item Let $\B\emptyset $ be the family that assigns to every $\X \in \mO$ the empty set of weights over $\mV(\X), \mW(\X)$.

There is a family $\L\kappa $ that assigns to every $\X \in \mO$ the set of $\kappa$-small left weights over $\mV(\X)$.
%There is a family $\L_{\mathrm{con}}\kappa $ that assigns to every $\X \in \mO$ the set of trivial left $\mV(\X)$-weights on $\kappa$-small $\infty$-categories.

%Let $\L\rc\rc $ be the family that assigns to every $\X \in \mO$ the set of absolute small left weights over $\mV(\X)$.
%\item Let $\L\mathrm{Ten}, \R\Ten, \B\Ten $ be the families that assign to $\X \in \mO$ the set of $\mV(\X)$-left weights corresponding to objects of $\mP\L\Env(*_{\mV(\X)}) \simeq \mP\Env(\mV(\X))$, the set of $\mW(\X)$-right weights corresponding to objects of $\mP\R\Env(*_{\mW(\X)}) \simeq \mP\Env(\mW(\X)),$the set of $\mV(\X), \mW(\X)$-weights corresponding to objects of $\mP\B\Env(*_{\mV(\X), \mW(\X)}) \simeq \mP\Env(\mV(\X)) \otimes \mP\Env(\mW(\X)),$respectively.\end{itemize}

\end{example}

\begin{example}\label{leff}
Let $\LM^\ot \to \Comm$ be the symmetric $\infty$-operad governing left modules
\cite[Notation 4.2.1.6.]{lurie.higheralgebra} and $\kappa$ a small regular cardinal.
A $\LM$-algebra in $\Mon_\kappa $ corresponds to a left module in $\Mon_\kappa $ \cite[Proposition 4.2.2.9.]{lurie.higheralgebra}.
So a $\LM$-algebra in $\Mon_\kappa$ encodes a left action of an associative algebra in $\Mon_\kappa$ - corresponding to two braided monoidal $\infty$-categories $\mV^\boxtimes \to \bE_2$ compatible with $\kappa$-small colimits, - on an object $\mW^\ot \to \Ass $ of $\Mon_\kappa$.
%, acting from the left in $\Mon_\kappa $ on a monoidal $\infty$-category $\mW^\ot \to \Ass$ compatible with $\kappa$-small colimits corresponding to a $\LM$-algebra in $\Mon_\kappa$.
Let $\mu: \mV^\ot \times_{\Comm} \mW^\ot \to \mW^\ot$ be the underlying monoidal left action functor and $\mH$ a set of absolute small left weights over $\mV$.
By the axioms of a left action the functor $\mu(\tu_\mV,-):\mW^\ot \to \mW^\ot $ is the identity and so sends weights of $\mH'$ to weights of $\mH'$.
If also $\mu(-,\tu_\mW):\mV^\ot \to \mW^\ot$ sends weights of $\mH$ to weights of $\mH'$, then by Remark \ref{ohuz} the given $\LM$-algebra uniquely refines to a $\LM$-algebra in $(\Mon_\kappa)_*$.
	
\end{example}

\begin{notation}
Let $$ \Pr\Mon_*^\ot \subset \Pr\Mon^\ot \times_{\rc\rc\Mon^\ot}\rc\rc\Mon _\ast^\ot$$ 	
be the full symmetric monoidal subcategory spanned by the presentably monoidal $\infty$-categories equipped with a set of small left enriched weights.
	
\end{notation}

\begin{notation}%Let $\sigma$ be the strongly inaccessible cardinal corresponding to the small universe.
Let $$\L\Enr_\ast^\ot \subset {\Pr\Mon_*^\ot} \times_{{\cc\cc\Mon_*^\ot}} {\L\Q\Enr^{\ot}_*} $$ be the full symmetric monoidal subcategory spanned by the small left enriched $\infty$-categories equipped with a set of small left enriched diagrams.

%\Pr\Mon_*^\otimes\times_{(\widehat{\Mon} \times \widehat{\Mon})_*^\otimes}\P\B\widehat{\Enr}_\ast^\ot $$ be the full symmetric monoidal subcategories of bienriched $\infty$-categories with diagrams.
	
\end{notation}

The functor (\ref{chor}) for $\kappa$ the strongly inaccessible cardinal corresponding to the small universe restricts to a map
\begin{equation}\label{chora}{\L\Enr_*^{\otimes}} \to {\Pr\Mon_\ast^\ot  \times_{\Pr\Mon^\ot} {\L\Enr^{\otimes}}} \end{equation}
of cocartesian fibrations over ${\L\Enr^{\otimes}}$ and so a symmetric monoidal functor.

\begin{proposition}\label{prpp}
\begin{enumerate}
\item Let $\kappa$ be a small regular cardinal.
The symmetric monoidal functor 
$${^\kappa\L\Enr_*^{\otimes}} \to (\Mon_\kappa)_\ast^\ot$$
is a cocartesian fibration and the symmetric monoidal functor (\ref{chor}) is a map of cocartesian fibrations over $(\Mon_\kappa)_\ast^\ot.$

\item The symmetric monoidal functor 
$${\L\Enr_*^{\otimes}} \to \Pr\Mon_\ast^\ot $$
is a cocartesian fibration and the symmetric monoidal functor (\ref{chora}) is a map of cocartesian fibrations over $\Pr\Mon_\ast^\ot .$

\end{enumerate}	
	
\end{proposition}

\begin{proof}We prove (1). The proof of (2) is similar.
By Proposition \ref{eras} the map of symmetric $\infty$-operads $\psi: {^\kappa\L\Enr^{\otimes}} \to \Mon_\kappa^\ot$ is a cocartesian fibration.
Let $\psi'$ be the underlying functor.		
We start with proving that the functor ${^\kappa\L\Enr_*} \to (\Mon_\kappa)_\ast$ is a cocartesian fibration.
Let $(\mM^\circledast \to \mV^\ot, \Lambda) \in {^\kappa\L\Enr}_\ast$
and $\alpha: (\mV^\ot, \mH) \to (\mV'^\ot,\mH')$
a morphism in $(\Mon_\kappa)_*$, where $\mH$ is the set of underlying weights of $\Lambda.$
The $\psi' $-cocartesian lift $\mM^\circledast \to \mM'^\circledast:=\alpha_!(\mM)^\circledast $ of $\alpha$ sends the set $\Lambda$ of $\mH$-weighted diagrams to a set $\Lambda'$ of $\mH'$-weighted diagrams. 
By the choice of $\Lambda'$ the morphism 
$(\mM^\circledast \to \mV^\ot, \Lambda) \to (\mM'^\circledast \to \mV'^\ot, \Lambda') $ in ${^\kappa\L\Enr_*} $
lies over $\alpha: (\mV^\ot, \mH) \to (\mV'^\ot,\mH')$ in $(\Mon_\kappa)_\ast$ and induces for every $(\mN^\circledast \to \mV''^\ot, \Lambda'') \in {^\kappa\L\Enr_*}$
lying over $(\mV''^\ot \to \Ass, \mH'') \in (\Mon_\kappa)_\ast$ a pullback square
\begin{equation*}
\begin{xy}
\xymatrix{
{^\kappa\L\Enr_*}((\mM', \Lambda'),(\mN, \Lambda'')) \ar[d] \ar[r]
&{^\kappa\L\Enr_*}((\mM, \Lambda),(\mN, \Lambda'')) \ar[d]
\\ 
(\Mon_\kappa)_\ast((\mV', \mH'),(\mV'',\mH'')) \ar[r] &(\Mon_\kappa)_\ast((\mV, \mH),(\mV'',\mH'')). 
}
\end{xy} 
\end{equation*} 	
	
Hence the functor $\phi: {^\kappa\L\Enr_*} \to (\Mon_\kappa)_\ast$ is a cocartesian fibration. To prove the result, by \cite[Lemma 2.44.]{heine2023monadicity} it is enough to see that the collection of $\phi$-cocartesian morphisms is stable under the tensor product on ${^\kappa\L\Enr_*}, $ i.e. that for any $(\mM''^\circledast \to \mV''^\ot, \Lambda'') \in {^\kappa\L\Enr_*}$ and any morphism 
$(\mM^\circledast \to \mV^\ot, \Lambda) \to (\mM'^\circledast \to \mV'^\ot, \Lambda')$ in $^\kappa\L\Enr_*$ the tensor product is $\phi$-cocartesian: $$((\mM \otimes \mM'')^\circledast, \Lambda \boxtimes \Lambda'') \to ((\mM' \otimes \mM'')^\circledast, \Lambda' \boxtimes \Lambda'').$$
The image of the latter morphism in $^\kappa\L\Enr $ is $\psi$-cocartesian since $\psi$ is a cocartesian fibration.
Consequently, by the description of $\phi$-cocartesian morphisms, it is enough to show that every diagram of $\Lambda'\boxtimes \Lambda''$ is transported from $\Lambda\boxtimes \Lambda''$. This follows immediately from the definition
of $\Lambda'\boxtimes \Lambda''$ and that $\Lambda'$ is transported from  $\Lambda$ by the $\psi'$-cocartesian morphism $ \mM^\circledast \to \mM'^\circledast$ that induces an essentially surjective functor $\mM \to \mM'$ by Proposition \ref{bica}.

\end{proof}

Proposition \ref{prpp} implies the following corollary:
\begin{corollary}\label{notay}
Let $\mO^\ot \to \Comm$ be a symmetric $\infty$-operad and $\kappa$ a small regular cardinal.

\begin{enumerate}
%\item For every $\mO$-algebras $(\mV, \mH): \mO^\ot \to {_*(\Mon)^\ot}$ and $\mW: \mO^\ot \to \Op_\infty^\times $ the $\mO$-operad $$_*{_\mV\ell\L\P\Enr}_{\mW}(\mH)^\ot:= \mO^\ot \times_{_*(\Mon)^\ot \times_{\Comm} \Op_{\infty}^\times} {_*\ell\L\P\Enr^\ot} \to \mO^\ot$$is corepresentable.
%\item For every $\mO$-algebras $\mV: \mO^\ot \to \Op_\infty^\times$ and $ (\mW,\mH): \mO^\ot \to (\Mon)_*^\ot$ the $\mO$-operad $${_\mV\ell\R\P\Enr}_{\mW}(\mH)_*^\ot := \mO^\ot \times_{\Op_\infty^\times \times_{\Comm} (\Mon)_*^\otimes} \ell\R\P\Enr_\ast^\ot \to \mO^\ot$$is corepresentable.
%\item For every $\mO$-algebra $(\mV, \mW, \mH): \mO^\ot \to (\Mon \times \Mon)_*^\ot$ the $\mO$-operad $$_*{_\mV\ell\B\P\Enr}_{\mW}(\mH)_*^\ot:= \mO^\ot \times_{_*(\Mon \times \Mon)_*^\otimes} {_*\ell\B\P\Enr_\ast^\ot} \to \mO^\ot$$is corepresentable.
		
%\item For any $\mO$-algebra $(\mV, \mW, \mH): \mO^\ot \to (\Mon \times \Mon)_*^\otimes$ the $\mO$-operad $${_\mV\B\P\Enr}_{\mW}(\mH)_*^\ot \to \mO^\ot$$is an $\mO$-monoidal $\infty$-category.

\item For every $\mO$-algebra $(\mV, \mH): \mO^\ot \to (\Mon_\kappa)_*^\otimes$ the $\mO$-operad $${^\kappa_\mV\L\Enr}(\mH)_*^\ot := \mO^\ot \times_{(\Mon_\kappa)_*^\otimes} {^\kappa\L\Enr_\ast^\ot}\to \mO^\ot$$ is an $\mO$-monoidal $\infty$-category.

\item For every $\mO$-algebra $(\mV, \mH): \mO^\ot \to \Pr\Mon_*^\otimes$ the $\mO$-operad $${_\mV\L\Enr}(\mH)_*^\ot:= \mO^\ot \times_{{\Pr\Mon_*^\otimes}} {\L\Enr_\ast^\ot} \to \mO^\ot$$
is an $\mO$-monoidal $\infty$-category.

\end{enumerate}	

\end{corollary}

\begin{notation}Let $\kappa$ be a small regular cardinal.
Let $${^\kappa\L\Enr_\wc^{\otimes}} \subset {(\Mon_\kappa)_*^\otimes} \times_{ {\Mon_\kappa^\ot}} {^\kappa\L\Enr^{\otimes}}$$
be the symmetric suboperad whose colors are pairs $(\mM^\circledast \to \mV^\ot,\mH)$ such that $\mM^\circledast \to \mV^\ot$ admits $\mH$-weighted colimits and whose multi-morphisms $$(\mM_1^\circledast \to \mV_1^\ot,\mH_1),...,(\mM_\n^\circledast \to \mV_\n^\ot,\mH_\n) \to (\mN^\circledast \to \mV^\ot,\mH)$$ correspond to
an enriched functor $ \mM_1^\circledast \times_{\Ass} ... \times_{\Ass} \mM_\n^\circledast \to \mN^\circledast$ that sends $\mH_1 \boxtimes... \boxtimes \mH_\n$-weighted colimits
to $\mH$-weighted colimits.

\end{notation}

\begin{notation}Let $\sigma$ be the strongly inaccessible cardinal corresponding to the small universe. 
Let $$ \L\Q\Enr_{\mathrm{w}}^\ot:= {^\sigma\widehat{\L\Enr}_\wc^{\ot}} \to {(\widehat{\Mon}_\sigma)_*^\ot}.$$
\end{notation}

\begin{notation}Let $\sigma$ be the strongly inaccessible cardinal corresponding to the small universe. 
Let $$\L\Enr_{\mathrm{w}}^\ot \subset \L\Q\Enr_{\mathrm{w}}^\ot$$
be the full symmetric suboperad whose colors are the pairs $(\mM^\circledast \to \mV^\ot,\mH)$ such that $\mV^\ot \to \Ass$ is a presentably monoidal $\infty$-category, $\mH$ is a set of small left enriched weights over $\mV$ and $\mM^\circledast \to \mV^\ot $ is a small left enriched $\infty$-category that admits $\mH$-weighted colimits.

\end{notation}

Lemma \ref{eqas} implies the following:

\begin{corollary}
	
\begin{enumerate}
\item Let $\kappa $ be a small regular cardinal.
The inclusion $${^\kappa\L\Enr_\wc^{\otimes}} \subset {(\Mon_\kappa)_*^\otimes} \times_{ {\Mon_\kappa^\ot}} {^\kappa\L\Enr^{\otimes}}$$
lifts to an embedding ${^\kappa\L\Enr_\wc^{\otimes}} \to {^\kappa\L\Enr_*^{\otimes}}.$
\item The inclusion $${\L\Enr_\wc^{ \otimes}} \subset {\Pr\Mon_*^\otimes}\times_{ {\Pr\Mon^\ot}} {\L\Enr^{\otimes}}$$
lifts to an embedding ${\L\Enr_\wc^{\otimes}} \to {\L\Enr_*^{\otimes}}.$
\end{enumerate}	

\end{corollary}

%By Proposition \ref{bica} the functor \begin{equation}\label{cart0}(\Op_\infty \times \Op_\infty)_* \times_{(\Op_\infty\times \Op_\infty)} \omega\B\Enr \to {(\Op_\infty \times \Op_\infty)_*}\end{equation} is a cartesian fibration that restricts to a cartesian fibration
%\begin{equation}\label{cart1}
%_*(\Mon) \times_{\Mon}\ell\L\P\Enr \to {_*(\Mon) \times \Op_\infty},\end{equation}
%\begin{equation}\label{cart2} \ell\R\P\Enr \times_{\Mon}(\Mon)_* \to \Op_\infty \times (\Mon)_*,\end{equation}
%\begin{equation}\label{cart3}_*(\Mon \times \Mon)_* \times_{(\Mon \times \Mon)}\ell\B\P\Enr \to {_*(\Mon \times \Mon)_*},\end{equation}
%\begin{equation}\label{cart4}(\Mon_\kappa)_* \times_{\Mon_\kappa} {^\kappa\L\Enr} \to (\Mon_\kappa)_* \times \Op_\infty,\end{equation}
%\begin{equation}\label{cart5} \R\Enr^{\tau}\times_{\Mon_\tau}(\Mon_\tau)_*   \to \Op_\infty \times (\Mon_\tau)_*,\end{equation}
%\begin{equation}\label{cart6} (\Mon_\kappa \times \Mon_\tau)_* \times_{(\Mon_\kappa \times \Mon_\tau)} {^\kappa\B\Enr}^{\tau} \to {(\Mon_\kappa \times \Mon_\tau)_*}.\end{equation} 

%Corollary \ref{eccco} implies the following:

Corollary \ref{ujpp} and Corollary \ref{eccco} imply the following:

\begin{corollary}\label{coarst} Let $\kappa$ be a small regular cardinal.
The cartesian fibration of Corollary \ref{ujpp} restricts to the following cartesian fibrations with the same cartesian morphisms:
$$^\kappa\L\Enr_{\wc} \to {(\Mon_\kappa)_*}, \
\L\Enr_{\wc} \to {\Pr\Mon_*}.$$
\end{corollary}

\begin{theorem}\label{thqaz} \label{thqazz}\label{Thor} Let $\mO^\ot \to \Comm$ be a symmetric $\infty$-operad and $\kappa$ a small regular cardinal.
	
\begin{enumerate}
\item %\begin{enumerate}
%\item Let $(\mV,\mH): \mO^\ot \to {_*(\Op_\infty^\tu)^\ot}, \mW: \mO^\ot \to \Op_\infty^\times $ be $\mO$-algebras. The full suboperad $${_\mV\ell\L\P\Enr}_{\mW}(\mH)^\ot:= \mO^\ot \times_{{_*(\Op_\infty^\tu)^\ot} \times_{\Comm} \Op_\infty^\times} \ell\L\P\Enr_\wc^\ot \subset {_{*\mV}\ell\L\P\Enr}_{\mW}(\mH)^\ot$$is an accessible localization relative to $\mO^\ot$and so a corepresentable $\mO$-operad.

%\item Let $\mV: \mO^\ot \to {\Op_\infty^\times}, (\mW, \mH): \mO^\ot \to (\Op_\infty^\tu)^\ot_* $ be $\mO$-algebras. The full suboperad $${_\mV\ell\R\P\Enr}_{\mW}(\mH)^\ot:= \mO^\ot \times_{\Op_\infty^\times \times_{\Comm} (\Op_\infty^\tu)_*^\otimes} \ell\R\P\Enr_\wc^\ot \subset {_{\mV}\ell\R\P\Enr}_{\mW}(\mH)_*^\ot$$is an accessible localization relative to $\mO^\ot$and so a corepresentable $\mO$-operad.

%\vspace{1mm}

%\item Let $(\mV,\mW, \mH): \mO^\ot \to {{_*(\Mon_\kappa \times \Mon_\tau)^\ot_*}^\otimes}$ be an $\mO$-algebra. The full suboperad $${_\mV\ell\B\P\Enr}_{\mW}(\mH)^\ot:= \mO^\ot \times_{{_*(\Mon_\kappa \times \Mon_\tau)^\ot_*}^\otimes}  {^\kappa\B\Enr_\wc^{\tau \ot }} \subset {_{*\mV}\ell\B\P\Enr}_{\mW}(\mH)^\ot_*$$is an accessible localization relative to $\mO^\ot$and so a corepresentable $\mO$-operad.\vspace{1mm}

Let $(\mV, \mH): \mO^\ot \to {(\Mon_\kappa)^\ot_*}$ be an $\mO$-algebra. The full suboperad $${_\mV^\kappa\L\Enr}(\mH)^\ot:= \mO^\ot \times_{(\Mon_\kappa)_*^\otimes} {^\kappa\L\Enr_\wc}^\ot \subset {^\kappa_\mV\L\Enr}(\mH)_*^\ot$$ is an accessible localization relative to $\mO^\ot$ and so a presentably $\mO$-monoidal $\infty$-category.

\item Let $(\mV, \mH): \mO^\ot \to {(\cc\cc\Mon)^\ot_*}$ be an $\mO$-algebra. The full suboperad $${_\mV\L\Q\Enr}(\mH)^\ot:= \mO^\ot \times_{\cc\cc\Mon_*^\otimes} \L\Q\Enr_\wc^\ot \subset {{_\mV\L\Q\Enr}(\mH)_*^\ot}$$ is an accessible localization relative to $\mO^\ot$ and so a $\mO$-monoidal $\infty$-category.

\item Let $(\mV,\mH): \mO^\ot \to {\Pr\Mon^\ot_*}$ be an $\mO$-algebra.
The full suboperad \begin{equation}\label{fff}
{_\mV\L\Enr}(\mH)^\ot:= \mO^\ot \times_{\Pr\Mon_*^\otimes} \L\Enr_\wc^\ot \subset {{_\mV\L\Enr}(\mH)_*^\ot}\end{equation} is an accessible localization relative to $\mO^\ot$
and so a presentably $\mO$-monoidal $\infty$-category. 

The embedding $${_\mV\L\Enr}(\mH)^\ot \subset {_\mV\L\Q\Enr}(\mH)^\ot$$ is an $\mO$-monoidal functor.
\end{enumerate}
\end{theorem}

\begin{proof}
(1): By Corollary \ref{gcpp} the embedding of (1) induces fiberwise accessible functors.
In view of Corollary \ref{presenta} (1) it remains to see that the embedding of (1) admits a left adjoint relative to $\mO^\ot.$
We prove that the full suboperad $ {^\kappa\L\Enr_\wc^{\ot }} \subset  {^\kappa\L\Enr_*^{\ot }} $ is a localization relative to ${(\Mon_\kappa)^\ot_*}.$
%the base \subset {_*^\kappa\B\P\Enr^{\tau}_*}^\ot, \B\Enr_\wc^\ot \subset {_*\B\Enr_*^\ot}{^\kappa\B\P\Enr^{\tau}_\mathrm{w}}^\ot} \
%${_*(\Mon_\kappa \times \Mon_\tau)^\ot_*}, {_*(\Mon_\kappa \times \Mon_\tau)_*^\otimes}, {_*\Pr\Mon_*^\otimes}$, respectively.
For this it is enough to see that the full suboperad $ {^\kappa\L\Enr_\wc^{\ot }} \subset  {^\kappa\L\Enr_*^{\ot }} $  % of (\ref{Fua})
%$$  {^\kappa\B\Enr_\wc^{\tau \ot }} \subset \ell\B\P\Enr_*^\ot, \B\P\Enr_\wc^\ot \subset \B\P\Enr_*^\ot, \B\P\Enr(\kappa)_\wc^\ot \subset \B\P\Enr(\kappa)_*^\ot, \B\Enr_\wc^\ot \subset \B\Enr_*^\ot $$ 
is a localization relative to $\Comm$ and that the local equivalences lie over equivalences in ${(\Mon_\kappa)^\ot_*}.$
% {_*(\Mon_\kappa \times \Mon_\tau)_*^\otimes}, {_*\Pr\Mon_*^\otimes}$, respectively.
By \cite[Lemma 2.2.4.11.]{lurie.higheralgebra} it suffices to prove that the full subcategory %$
$ {^\kappa\L\Enr_\wc} \subset {^\kappa\L\Enr_*}$
is a localization whose local equivalences lie over equivalences in ${(\Mon_\kappa)^\ot_*}$ %, {_*(\Mon_\kappa \times \Mon_\tau)_*}, {_*(\Pr\Mon)_*},$$ respectively, 
and that the collection of local equivalences is closed under the 
%for every$\Z \in \ell\B\P\Enr_*, \B\P\Enr_*,  \B\Enr_*$, respectively, the functor $\Z \ot (-)$ preserves local equivalences, where $\ot$ refers to the 
tensor product provided by the symmetric monoidal $\infty$-category ${^\kappa\L\Enr_*^{\ot}} \to \Comm. $ %, {{_*^\kappa\B\P\Enr^{\tau}_*}^\ot} \to \Comm, {_*\B\Enr_*^\ot} \to \Comm.$ %respectively.
Let $(\mM^\circledast \to \mV^\ot, \Lambda)\in {^\kappa\L\Enr_*}$
and $\mH$ the set of left weights underlying the set of left diagrams $\Lambda$.
By Corollary \ref{wooond1} (1) the morphism $$(\mM^\circledast \to \mV^\ot, \Lambda) \to (\mP\B\Env^\mH_\Lambda(\mM)^\circledast_{\L\Enr_{\kappa}} \to \mV^\ot, \mH)$$ in ${^\kappa\L\Enr_*}$ whose target belongs to ${^\kappa\L\Enr_\wc}$, is a local equivalence with respect to the full subcategory ${^\kappa\B\Enr_\wc^{\tau}} \subset {^\kappa\B\Enr_*^{\tau}}$ and lies over the identity in ${(\Mon_\kappa)^\ot_*}.$
Moreover by Corollary \ref{wooond1} (1) the local equivalences are closed under the 
tensor product provided by %the symmetric monoidal $\infty$-category 
$ {^\kappa\L\Enr_*^{\ot }} \to \Comm.$
%The proofs of (1)-(2) are the same using Corollary \ref{wooond1} (1)-(2), respectively.
% and Lemma \ref{quirk}.

(2) follows immediately from (1).
(3): By Corollary \ref{gcpp} the embedding (\ref{fff}) induces fiberwise accessible functors.
So it is enough to prove that the accessible localization $${_\mV\L\Q\Enr}(\mH)^\ot \subset {_\mV\L\Q\Enr}(\mH)_*^\ot$$ 
relative to $\mO^\ot$ of Corollary \ref{thqazz} restricts to a localization
${_\mV\L\Enr}(\mH)^\ot\subset {_\mV\L\Enr}(\mH)_*^\ot$ relative to $\mO^\ot$.
This follows from Lemma \ref{quirk} in view of the construction of the localization functor.
Presentability follows from Corollary \ref{presenta} (4).

\end{proof}

\begin{remark}\label{Exx}
Let $\kappa$ be a small regular cardinal,
$\mV^\boxtimes \to \bE_2$ a braided monoidal $\infty$-category compatible with $\kappa$-small colimits, seen as an associative algebra in the monoidal $\infty$-category 
$\Mon_\kappa$, acting from the left on a monoidal $\infty$-category $\mW^\ot \to \Ass$ compatible with $\kappa$-small colimits.
Let $\mu: \mV^\ot \times_\Ass \mW^\ot \to \mW^\ot$ be the monoidal left action functor. %, which preserves $\kappa$-small colimits component-wise.
Let $\mH$ be a set of small left weights over $\mV$ and $\mH'$ a set of small left weights over $\mW$ such that $\mu(-,\tu_\mW):\mV^\ot \to \mW^\ot$ sends weights of
$\mH$ to weights of $\mH'$.
By Example \ref{leff} the left action of $\mV^\ot \to \Ass$ on $\mW^\ot \to \Ass$ in $\Mon_\kappa$ corresponds to a $\LM$-algebra in $_*\Mon_\kappa$.

\begin{enumerate}
\item By Theorem \ref{thqaz} the $\infty$-category $^\kappa_\mV\L\Enr(\mH)$
carries a monoidal structure.
For two left $\kappa$-enriched $\infty$-categories $\mM^\circledast \to \mV^\ot, \mN^\circledast \to \mV^\ot$ that admit $\mH$-weighted colimits let $(\mM \otimes^{\L\Enr_\kappa}_{\mH} \mN)^\circledast $ the tensor product.
By construction there is a canonical equivalence $$(\mM \otimes^{\L\Enr_\kappa}_{\mH} \mN)^\circledast \simeq \mP\B\Env_{\ot_!(\Lambda^{(\mM,\mH)} \boxtimes \Lambda^{(\mN,\mH)})}^{\mH}(\ot_!(\mM \times \mN))^\circledast_{\L\Enr_\kappa}.$$ 
	
\item By Theorem \ref{thqaz} the $\infty$-category $^\kappa_\mW\L\Enr(\mH')$ carries a left action of $^\kappa_\mV\L\Enr(\mH)$.
For a left $\kappa$-enriched $\infty$-category $\mM^\circledast \to \mV^\ot$ that admits $\mH$-weighted colimits and a left $\kappa$-enriched $\infty$-category $\mN^\circledast \to \mW^\ot$ that admits $\mH'$-weighted colimits let $(\mM \otimes^{\L\Enr_\kappa}_{\mH,\mH'} \mN)^\circledast $ the left tensor.
By construction there is a canonical equivalence $$(\mM \otimes^{\L\Enr_\kappa}_{\mH,\mH'} \mN)^\circledast \simeq \mP\B\Env_{\mu_!(\Lambda^{(\mM,\mH)} \boxtimes \Lambda^{(\mN,\mH')})}^{\mH'}(\mu_!(\mM \times \mN))^\circledast_{\L\Enr_\kappa}.$$ 
	
%\item For a left $\kappa$-enriched $\infty$-category $\mO^\circledast \to \mV^\ot$ that admits $\mH$-weighted colimits let $(\mM \otimes^{\Enr_\kappa}_{\mH} \mO)^\circledast $ denote the image under the tensor product.By construction there is a canonical equivalence $$(\mM \otimes^{\Enr_\kappa}_{\mH} \mO)^\circledast \simeq \mP\B\Env_{\ot_!(\Lambda^{(\mM,\mH)} \boxtimes \Lambda^{(\mO,\mH)})}^{\mH}(\ot_!(\mM \times \mO))^\circledast_{\L\Enr_\kappa}.$$ 

%\item For $\kappa$ the strongly inaccessible cardinal corresponding to the small universe the monoidal structure on $^\kappa_\mV\widehat{\L\Enr}(\mH)_\emptyset$restricts to the full subcategory $_\mV\L\Enr(\mH)_\emptyset$,and the left action of $^\kappa_\mV\widehat{\L\Enr}(\mH)_\emptyset$ on $^\kappa_\mW\widehat{\L\Enr}(\mH')_\emptyset$restricts to a left action of $_\mV\L\Enr(\mH)_\emptyset$ on $_\mW\L\Enr(\mH')_\emptyset$ (Theorem \ref{Thor}).
\end{enumerate}
%the right $\mW$-enriched functor $\mu_!(\mM \times \mO)^\circledast \to (\mM \otimes^{\Enr_\kappa}_\mH \mO)^\circledast.$ 

%induces an equivalence$$\Enr\Fun^{\mH'}_{\mW, \emptyset}(\mM \otimes^{\Enr_\kappa}_\mH \mO, \mN) \simeq \Enr\Fun^{\mu_!(\Lambda^{(\mM,\mH)} \boxtimes \Lambda^{(\mO,\mH')})}_{\mW, \emptyset}(\mu_!(\mM \times \mO), \mN).$$	

\end{remark}

\begin{notation}
Let $\kappa, \tau$ be small regular cardinals. 

\begin{enumerate}
\item Let $(\kappa, \tau)\LMod^{\ot} \subset (\LMod_{\emptyset})^\times $ be the 
symmetric suboperad whose colors are the left tensored $\infty$-categories that admit
$\kappa$-small conical colimits, i.e. $\kappa$-small colimits preserved by the left action,
and that are left $\tau$-enriched $\infty$-categories, i.e. such that for every $\X \in \mM$ the functor $(-)\ot \X: \mV \to \mM$ preserves $\tau$-small colimits, and whose multi-morphisms $\mM_1^\circledast, ...,\mM_\n^\circledast \to \mN^\circledast$ for $\n \geq0$ correspond to a left linear functor $\mM_1^\circledast \times_{\Ass} ... \times_{\Ass} \mM_\n^\circledast \to \mN^\circledast$ 
preserving $\kappa$-small colimits component-wise.

%\item Let $ ^\kappa\LMod^{\ot} \subset \LMod^\times $ be the full symmetric suboperadwhose colors are the left tensored $\infty$-categories $\mM^\circledast \to \mV^\ot$ that are left $\kappa$-enriched $\infty$-categories, i.e. such that for every $\X \in \mM$ the functor$(-)\ot \X: \mV \to \mM$ preserves $\kappa$-small colimits.

\item Let $ \kappa\LMod^{\ot} := (\kappa,\kappa)\LMod^{\ot} $ be the symmetric suboperad
of left tensored $\infty$-categories compatible with $\kappa$-small colimits. For $\kappa$ the strongly inaccessible cardinal corresponding to the small universe we write
$\rc\rc\LMod^{\ot}$ for $ \kappa\LMod^{\ot}.$

\item Let $\mathrm{\Pr}\LMod^\ot \subset \rc\rc\LMod^{\ot} $ be the full symmetric suboperad whose colors are the presentably left tensored $\infty$-categories.

\end{enumerate}

\end{notation}

For the next example we use Example \ref{infi}.

\begin{example}\label{rfgp}Let $\mO \to \Comm$ be a symmetric $\infty$-operad and $\kappa$ a small regular cardinal.
	
\begin{enumerate}
\item Let $\mV: \mO^\ot \to \Mon_\kappa^\otimes$ be an $\mO$-algebra.	
There are $\mO$-monoidal identities $${_\mV^\kappa\L\Enr(\L\kappa)}^\ot= {_\mV\kappa\LMod}^{\ot}, $$
$${_\mV^\kappa\L\Enr(\L\emptyset)}^\ot= {_\mV(\emptyset,\kappa)\LMod}^{\ot}.$$

\item Let $\mV: \mO^\ot \to \rc\rc\Mon^\otimes$ be an $\mO$-algebra.	
There is an $\mO$-monoidal identity
$${_\mV\L\Q\Enr(\L\rc\rc)}^\ot= {_\mV\rc\rc\LMod}^{\ot}.$$

\item Let $\mV: \mO^\ot \to \Mon^\times$ be an $\mO$-algebra.	
There are $\mO$-monoidal identities
$${_\mV\L\P\Enr(\L\kappa)}^\ot={_\mV(\kappa,\emptyset)\LMod}^{\ot}.$$
\end{enumerate}
\end{example}

\begin{proposition}\label{preee}Let $\mO^\ot \to \Comm$ be a symmetric $\infty$-operad and $\kappa$ a small regular cardinal.
	
\begin{enumerate}
\item Let $(\mV, \mH) \to (\mV, \mH')$ be a map of $\mO$-algebras in $(\Mon_\kappa)_*^\ot$ inducing the identity on $\mV.$ The inclusion $${^\kappa_\mV\L\Enr}(\mH')^\ot \subset {^\kappa_\mV\L\Enr}(\mH)^\ot $$ admits a left adjoint relative to $\mO^\ot. $

\item Let $(\mV, \mH) \to (\mV, \mH')$ be a map of $\mO$-algebras in $(\rc\rc\Mon )_*^\ot$ inducing the identity on $\mV.$ The inclusion $${_\mV\L\Q\Enr}(\mH')^\ot \subset {_\mV\L\Q\Enr}(\mH)^\ot $$ admits a left adjoint relative to $\mO^\ot.$

\item Let $(\mV,\mH) \to (\mV, \mH')$ be a map of $\mO$-algebras in $\Pr\Mon_*^\ot$ inducing the identities on $\mV$.
The inclusion $$ {_\mV\L\Enr}(\mH')^\ot \subset {_\mV\L\Enr}(\mH)^\ot $$
admits a left adjoint relative to $\mO^\otimes. $
\end{enumerate}	

\end{proposition}

\begin{proof} %We prove (1). The proof of (2) is similar.
(2) follows from (1) by applying (1) in a larger universe and taking $\tau=\sigma$ the strongly inaccessible cardinals corresponding to the small universe.
(3): The adjunction relative to $\mO^\ot$ of (2) restricts to the adjunction of (3) by construction of the left adjoint and Lemma \ref{quirk}.

(1): The map $(\mV, \mH) \to (\mV, \mH')$ of $\mO$-algebras in ${(\Mon_\kappa)^\ot_*}$
corresponds to a functor $\mO^\otimes \times [1] \to {{(\Mon_\kappa)^\ot_*}} $ over $\Comm$ such that the composition $\mO^\otimes \times [1] \to {(\Mon_\kappa)^\ot_*} \to \Mon_\kappa^\ot $ factors as the projection $\mO^\otimes \times [1]\to \mO^\ot$
followed by $\mV.$
Let $\phi: {^\kappa\L\Enr_\ast^{\ot}} \to {(\Mon_\kappa)^\ot_*}$ be the symmetric monoidal forgetful functor that is a locally cocartesian fibration by Proposition \ref{prpp}.

We first prove that every locally $\phi$-cocartesian morphism $\alpha$ of ${{^\kappa\L\Enr_\ast^{\ot}}}$ lying over an equivalence in $\Mon_\kappa^\ot$ is $\phi$-cocartesian: since $\alpha$ lies over an equivalence in $\Mon_\kappa^\ot$, it belongs to some fiber $({^\kappa\L\Enr_\ast^{\ot}})_{\langle \n \rangle} \simeq ({^\kappa\L\Enr_\ast})^{\times \n}$ for $\n \geq 0$ and is locally $\phi_{\langle \n \rangle}$-cocartesian and thus $\phi_{\langle \n \rangle}$-cocartesian because $\phi_{\langle \n \rangle}$ is a cocartesian fibration by Lemma \ref{lllee}.
Therefore $\alpha$ is $\phi$-cocartesian if its image $\ot(\alpha)$ under the tensor product functor
$\ot: ({^\kappa\L\Enr_\ast})^{\times \n} \to {^\kappa\L\Enr_\ast}$ is $\phi_{\langle 1 \rangle}$-cocartesian.
Since $\alpha$ is $\phi_{\langle \n \rangle}$-cocartesian, by the construction of the tensor product on ${^\kappa\L\Enr}$ the morphism $\ot(\alpha)$ is $\phi_{\langle 1 \rangle}$-cocartesian if the image of $\ot(\alpha)$ under the forgetful functor $\nu:{^\kappa\L\Enr^{\ot}_*} \to {^\kappa\L\Enr}$ is an equivalence.
%cocartesian over its image in $\Op_\infty^\tu \times \Op_\infty^\tu.$
The functor $\nu$ is symmetric monoidal and preserves cocartesian morphisms by Lemma \ref{lllee}. Thus the image of $\alpha$ in ${^\kappa\L\Enr^{\ot}}$ is cocartesian over the equivalence $\phi(\alpha)$ and so an equivalence.
Hence the image of $\ot(\alpha)$ in ${^\kappa\L\Enr}$ is an equivalence.
%This implies that $\ot_!(\alpha)$ is $\phi_{\langle 1 \rangle}$-cocartesian.

As every locally $\phi$-cocartesian morphism of $(\Mon_\kappa)_*^\ot$ lying over an equivalence in $ \Mon_\kappa^\ot$ is $\phi$-cocartesian, the locally cocartesian fibration $ (\mO^\ot \times [1]) \times_{{(\Mon_\kappa)_*^\ot}} {{^\kappa\L\Enr_*}^\ot} \to \mO^\otimes \times [1]$ is a map of cocartesian fibrations over $[1]$.
By Theorem \ref{thqaz} the full subcategory $$\Xi:= (\mO^\ot \times [1]) \times_{{(\Mon_\kappa)_*^\ot}} { {^\kappa\L\Enr_\wc^{\ot}}} \subset (\mO^\ot \times [1]) \times_{{(\Mon_\kappa)_*^\ot}} { {^\kappa\L\Enr_*}^\ot}$$ is a localization relative to $\mO^\otimes \times [1]$.
Therefore also the functor $\Xi \to \mO^\otimes \times [1]$ is a locally cocartesian fibration and map of cocartesian fibrations over $[1]$.
Since $\Xi \to \mO^\otimes \times [1]$ is a locally cocartesian fibration and $[1]$
has no non-trivial endomorphisms, the functor $\Xi \to \mO^\otimes $ is a locally cocartesian fibration and the functor $\Xi\to \mO^\otimes \times [1]$ is a map of locally cocartesian fibrations over $\mO^\ot.$

By Corollary \ref{coarst} the functor $\Xi\to \mO^\otimes \times [1]$ induces on the fiber over every $\X \in \mO$ a cartesian fibration $ [1] \times_{{(\Mon_\kappa)^\ot_*}} {^\kappa\L\Enr_\wc} \to [1]$ and so $\Xi\to \mO^\otimes \times [1]$ is a map of cartesian fibrations over $[1]$ by \cite[Lemma 2.44. (2)]{heine2023monadicity}, where we use that $\Xi\to \mO^\otimes \times [1]$ is a map of locally cocartesian fibrations over $\mO^\ot.$
The map $ \Xi \to \mO^\otimes \times [1]$ of cocartesian and cartesian fibrations over $[1]$ classifies the adjunction of (3).

\end{proof}

\subsection{Closedness of the monoidal structure}\label{closed}

%\subsection{Construction of the internal hom}

In the following we prove that for every braided monoidal $\infty$-category $\mV^\boxtimes \to \bE_2$, set $\mH$ of small left weights over $\mV$ and small regular cardinal $\kappa$ the monoidal structures on $_\mV^\kappa \L\Enr_\emptyset(\mH), {_\mV \L\Enr}_\emptyset(\mH) $ of Theorem \ref{thqaz} are closed (Corollary \ref{Qa})
by giving an explicite construction of the internal hom (Theorem \ref{Thhs}).

%\subsubsection{Closedness of the outer tensor product}
 
\begin{notation}
Let $\mV^\ot \to \Ass, \mW^\ot \to \Ass$ be small $\infty$-operads, $\mH$ a set of left weights over $\mV$ and $\mH'$ a set of right weights over $\mW$.
Let $\mH\oplus\mH'$ be the set of weights over $\mV,\mW $ consisting of the images of elements of $\mH$ under the map of $\infty$-operads
$\mV^\ot \simeq \mV^\ot \times \emptyset^\ot \to \mV^\ot \times \mW^\ot$ and the
images of elements of $\mH'$ under the map of $\infty$-operads $\mW^\ot \simeq \emptyset^\ot \times \mW^\ot \to \mV^\ot \times \mW^\ot$.
\end{notation}

For the next notation we use that for every weakly left enriched $\infty$-category $\mM^\circledast \to \mV^\ot$ and $\X\in\mM$ there is a unique left enriched functor $\emptyset^\ot \to \mM^\circledast$ (starting at the final object of $ {_{\emptyset}\omega\BMod_\emptyset}\simeq \Cat_\infty$) sending the unique object to $\X.$
%correspond to left $\emptyset$-enriched functors $\emptyset^\ot \to \emptyset^\ot \times_{\mV^\ot} \mM^\circledast$, which correspond to functors$* \to \mM$, i.e. objects of $\mM.$
\begin{notation}\label{addi}
Let $\mM^\circledast \to \mV^\ot, \mO^\circledast \to \mW^\ot, \mN^\circledast \to \mV^\ot \times \mW^\ot$ be absolute small weakly  left enriched, weakly right enriched, weakly bienriched $\infty$-categories, respectively, $\Lambda$ a collection of left diagrams in $\mM$ and $\Lambda'$ a collection of right diagrams in $\mO$.
Let $$\Lambda\oplus\Lambda' $$ be the set of diagrams in $\mM \times \mO $
consisting of the images of elements of $\Lambda$ under
all enriched functors $\mM^\circledast\simeq \mM^\circledast \times \emptyset^\ot \xrightarrow{\mM^\circledast \times \Y} \mM^\circledast \times \mO^\circledast$ for $\Y \in \mO$
and the images of elements of $\Lambda'$ under
all enriched functors $\mO^\circledast \simeq \emptyset^\ot \times \mO^\circledast \xrightarrow{\X \times \mO^\circledast} \mM^\circledast \times \mO^\circledast$ for $\X \in \mM.$	
	
\end{notation}

\begin{remark}
Let $\mH$ be a set of left weights over $\mV$ and $\mH'$ a set of right weights over $\mW$ and $\Lambda$ a collection of $\mH$-weighted left diagrams in $\mM$ and $\Lambda'$ a collection of $\mH'$-weighted right diagrams in $\mO$.
Then $\Lambda\oplus\Lambda' $ is a collection of $\mH \oplus \mH'$-weighted digrams in $\mM \times \mO.$
\end{remark}

%\begin{remark}Let $\mH$ be a set of left weights over $\mV$ and $\mH'$ a set of right weights over $\mW$ and $\Lambda$ a collection of left diagrams in $\mM$, each of which is $\rH$-weighted for some $\rH \in \mH$, and $\Lambda'$ a collection of right diagrams in $\mO$, each of which is $\rH$-weighted for some $\rH \in \mH'$.Then $\Lambda\oplus\Lambda' $ is $\rH$-weighted for some $\rH \in \mH \oplus \mH'$.\end{remark}

%\begin{notation}\label{addi}Let $\mM^\circledast \to \mV^\ot$ be an absolute small weakly  left enriched $\infty$-category that admits $\mH$-weighted colimits for some set $\mH$ of left weights over $\mV$ and$\mO^\circledast \to \mW^\ot$ an absolute small weakly  right enriched $\infty$-category that admits $\mH'$-weighted colimits for some set $\mH'$ of right weights over $\mW$.Let $\Lambda$ be a collection of left diagrams in $\mM$, each of which is $\rH$-weighted for some $\rH \in \mH$ and $\Lambda'$ a collection of right diagrams in $\mO$, each of which is $\rH$-weighted for some $\rH \in \mH'$.\end{notation}

\begin{remark}\label{hvL}
Let $\mM^\circledast \to \mV^\ot, \mO^\circledast \to \mW^\ot, \mN^\circledast \to \mV^\ot \times \mW^\ot$ be absolute small weakly  left enriched, weakly right enriched, weakly bienriched $\infty$-categories, respectively, $\Lambda$ a collection of left diagrams in $\mM$ and $\Lambda'$ a collection of right diagrams in $\mO$.
The canonical equivalences $$\Enr\Fun_{\mV, \emptyset}(\mM, \Enr\Fun_{\emptyset,\mW}(\mO, \mN)) \simeq \Enr\Fun_{\mV, \mW}(\mM \times \mO, \mN) \simeq \Enr\Fun_{\emptyset,\mW}(\mO, \Enr\Fun_{\mV,\emptyset}(\mM, \mN)) $$ restrict to equivalences
$$\Enr\Fun_{\mV, \emptyset}^\Lambda(\mM, \Enr\Fun^{\Lambda'}_{\emptyset,\mW}(\mO, \mN)) \simeq \Enr\Fun^{\Lambda \oplus \Lambda'}_{\mV, \mW}(\mM \times \mO, \mN) \simeq \Enr\Fun_{\emptyset,\mW}^{\Lambda'}(\mO, \Enr\Fun^\Lambda_{\mV, \emptyset}(\mM, \mN)).$$
\end{remark}

%\begin{construction}Let $\mV^\ot \to \Ass, \mW^\ot \to \Ass$ be small $\infty$-operads, $\mH$ a set of small left weights over $\mV$ and $\mH'$ a set of small right weights over $\mW$.Let $\mM^\circledast \to \mV^\ot$ be an absolute small weakly  locally left pseudo-enriched $\infty$-category that admits $\mH$-weighted colimits and $\mO^\circledast \to \mW^\ot$ an absolute small weakly  locally right pseudo-enriched $\infty$-category that admits $\mH'$-weighted colimits.The embedding of $\infty$-operads $\emptyset \to \Ass$ sends the set $\mH$ of small left weights over $\mV$ (seen as weights over $\mV, \emptyset$ to a set of small weights over $\mV,*$ and sends the set $\mH'$ of small right weights over $\mW$ (seen as weights over $\emptyset,\mW$ to a set of small weights over $*,\mW$,which we denote by the same name.Then $\mM'^\circledast:= \mM^\circledast \times \Ass \to \mV^\ot \times \Ass, \Ass \times \mO'^\circledast:= \mO^\circledast \to \Ass \times \mW^\ot$ are weakly locally bipseudo-enriched $\infty$-categories and $\mM^\circledast \simeq \mM'^\circledast \times_{\Ass} \emptyset^\ot, \mO^\circledast \simeq \emptyset^\ot \times_\Ass \mO'^\circledast$. So there are enriched embeddings $\mM^\circledast \subset \mM'^\circledast, \mO^\circledast \subset \mO'^\circledast$that send the set $\mH$ of small left weights over $\mV$ (seen as weights over $\mV, \emptyset$ to set of small weights over $\mV,*.$\end{construction}

\begin{lemma}\label{locL} Let $(\mV^\ot \to \Ass, \rS)$, $(\mW^\ot \to \Ass, \T)$ be small localization pairs such that $\rS$ is a set of morphisms of $\Env(\mV)$ and $\T$ is a set of morphisms of $\Env(\mW).$ Let $\mM^\circledast \to \mV^\ot $ be a small left $\rS$-enriched $\infty$-category and $\mN^\circledast \to \mW^\ot $ a small right $\T$-enriched $\infty$-category. Then $\mM^\circledast \times \mN^\circledast \to \mV^\ot \times \mW^\ot$ is a $\rS,\T$-bienriched $\infty$-category.\end{lemma}
\begin{proof}
We prove that $\mM^\circledast \times \mN^\circledast \to \mV^\ot \times \mW^\ot$ is a left $\rS$-enriched $\infty$-category. The proof of the right $\T$-enrichment is dual.
We like to see that for every $\f \in \rS, \X, \Y \in \mM, \X',\Y' \in \mN, \W_1,...\W_\m \in \mW$ for $\m \geq 0$ the induced map $$\mP\B\Env(\mM \times \mN)(\f \ot (\X,\X') \ot \W_1 \ot ... \ot \W_\m, (\Y,\Y')) \simeq $$$$\mP\Env(\mV)(\f,  \L\Mul\Mor_{\mP\B\Env(\mM \times \mN)}((\X,\X'),\W_1,...,\W_\m; (\Y,\Y')))$$ 
is an equivalence. Let $\delta$ be the diagonal functor $\mS \to \mP\Env(\mV).$
There is a canonical equivalence $$\L\Mul\Mor_{\mP\B\Env(\mM \times \mN)}((\X,\X'),\W_1,...,\W_\m; (\Y,\Y')) \simeq $$$$\L\Mor_{\mM}(\X, \Y) \times \delta(\Mul_{\mN}(\X',\W_1,...,\W_\m; \Y')) $$in $\mP\Env(\mV).$
Since $\mM^\circledast \to \mV^\ot $ is a left $\rS$-enriched $\infty$-category,
the following map is an equivalence: $$\mP\Env(\mV)(\f, \L\Mor_{\mP\L\Env(\mM)}(\X, \Y)) \simeq \mP\L\Env(\mM)(\f\ot \X,\Y).$$
So it remains to see that the induced map $\mP\Env(\mV)(\f, \delta(\Mul_{\mN}(\X',\W_1,...,\W_\m; \Y'))$ is an equivalence.
Because $\f$ is representable by assumption, by the Yoneda-lemma the latter map is the equivalence $\delta(\Mul_{\mN}(\X',\W_1,...,\W_\m; \Y'))(\f)$.
\end{proof}

For the next Notation we use Notation \ref{Lamb} and Lemma \ref{locL}:
\begin{notation}
Let $(\mV^\ot \to \Ass, \rS)$, $(\mW^\ot \to \Ass, \T)$	be small localization pairs
such that $\rS$ is a set of morphisms of $\Env(\mV)$ and $\T$ is a set of morphisms of $\Env(\mW).$
Let $\mH$ be a set of small left weights over $\mV$ and $\mH'$ a set of small right weights over $\mW$.
Let $\mM^\circledast \to \mV^\ot$ be a small left $\rS$-enriched $\infty$-category that admits $\mH$-weighted colimits and $\mO^\circledast \to \mW^\ot$ a small right $\T$-enriched $\infty$-category that admits $\mH'$-weighted colimits.
Let $$(\mM \times^{\mH,\mH'}_{\rS,\T} \mO)^\circledast:=\mP\B\Env^{\Lambda^{(\mM \times \mO,\mH \oplus \mH')}}_{\Lambda^{(\mM,\mH)} \oplus \Lambda^{(\mO,\mH')}}(\mM \times \mO)^\circledast_{\rS,\T}.$$

%Let $\mV^\ot \to \Ass, \mW^\ot \to \Ass$ be small monoidal $\infty$-categories, $\mH$ a set of small left weights over $\mV$ and $\mH'$ a set of small right weights over $\mW$.
%\item Let $\mM^\circledast \to \mV^\ot$ be a small left pseudo-enriched $\infty$-category that admits $\mH$-weighted colimits and $\mO^\circledast \to \mW^\ot$ a small right pseudo-enriched $\infty$-category that admits $\mH'$-weighted colimits.Let $$(\mM \times^{\P\Enr}_{\mH\oplus\mH'} \mO)^\circledast:=\mP\B\Env^{\Lambda^{(\mM \times \mO,\mH \oplus \mH')}}_{\Lambda^{(\mM,\mH)} \oplus \Lambda^{(\mO,\mH')}}(\mM\times \mO)^\circledast_{\B\P\Enr}.$$\end{enumerate}
	
\end{notation}

\begin{theorem}\label{Thhsa}
Let $(\mV^\ot \to \Ass, \rS)$, $(\mW^\ot \to \Ass, \T)$	be small localization pairs
such that $\rS$ is a set of morphisms of $\Env(\mV)$ and $\T$ is a set of morphisms of $\Env(\mW).$
Let $\mH$ be a set of small left weights over $\mV$ and $\mH'$ a set of small right weights over $\mW$.
Let $\mM^\circledast \to \mV^\ot$ be a small left $\rS$-enriched $\infty$-category that admits $\mH$-weighted colimits and $\mO^\circledast \to \mW^\ot$ a small right $\T$-enriched $\infty$-category that admits $\mH'$-weighted colimits.
There are canonical equivalences
$$\Enr\Fun^{\mH, \mH'}_{\mV,\mW}(\mM \times^{\mH,\mH'}_{\rS,\T} \mO, \mN) \simeq\Enr\Fun^{\mH'}_{\emptyset,\mW}(\mO,\Enr\Fun^\mH_{\mV,\emptyset}(\mM, \mN))\simeq\Enr\Fun^{\mH}_{\mV,\emptyset}(\mM,\Enr\Fun^{\mH'}_{\emptyset,\mW}(\mO, \mN)).$$

\end{theorem}
\begin{proof}We prove (1). The proof of (2) is similar.
By Remark \ref{hvL} and Notation \ref{Lamb} the equivalences $$\Enr\Fun_{\mV, \emptyset}(\mM, \Enr\Fun_{\emptyset,\mW}(\mO, \mN)) \simeq \Enr\Fun_{\mV, \mW}(\mM \times \mO, \mN) \simeq \Enr\Fun_{\emptyset,\mW}(\mO, \Enr\Fun_{\mV,\emptyset}(\mM, \mN)) $$ restrict to equivalences
$$\Enr\Fun_{\mV, \emptyset}^\mH(\mM, \Enr\Fun^{\mH'}_{\emptyset,\mW}(\mO, \mN)) \simeq \Enr\Fun^{\Lambda^{(\mM,\mH)} \oplus \Lambda^{(\mO,\mH')}}_{\mV, \mW}(\mM \times \mO, \mN)$$$$ \simeq \Enr\Fun_{\emptyset,\mW}^{\mH'}(\mO, \Enr\Fun^\mH_{\mV, \emptyset}(\mM, \mN)).$$
%where the $\infty$-category in the middle is the full subcategory of functorspreserving $\mH$-weighted colimits in the first component and $\mH'$-weighted colimits in the second component (see Notation \ref{notabene}).
The $\mV,\mW$-enriched functor $\mM^\circledast \times \mO^\circledast \to (\mM \times^{\mH,\mH'}_{\rS,\T} \mO)^\circledast=\mP\B\Env^{\Lambda^{(\mM \times \mO,\mH \oplus \mH')}}_{\Lambda^{(\mM,\mH)} \oplus \Lambda^{(\mO,\mH')}}(\mM\times \mO)_{\rS,\T}^\circledast $ induces the following equivalence by Theorem \ref{wooond}:
$$\Enr\Fun^{\mH \oplus \mH'}_{\mV,\mW}(\mM \times^{\mH,\mH'}_{\rS,\T} \mO, \mN) \simeq \Enr\Fun^{\Lambda^{(\mM,\mH)} \oplus \Lambda^{(\mO,\mH')}}_{\mV, \mW}(\mM \times \mO, \mN).$$
	
\end{proof}

%\subsubsection{Closedness of the inner tensor product}

\begin{notation}Let $\mV^\ot \to \Ass, \mW^\ot \to \Ass$ be small $\infty$-operads.
Let $\tau$ be the functor $$\mV^\ot \times_\Ass (\mW^\rev)^\ot \simeq \mV^\ot \times_\Ass \mW^\ot \to \mV^\ot \times \mW^\ot$$
induced by projection and the equivalence $(\mW^\rev)^\ot \simeq \mW^\ot.$
Taking pullback along $\tau$ defines a functor $$\theta: {_\mV\omega\B\Enr_\mW} \to {_{\mV\times\mW^\rev}\omega\L\Enr_\emptyset}$$
that restricts to a functor 
\begin{equation}\label{ujjj}
_\mV\P\B\Enr_\mW\simeq {_{\mV \times \mW^\rev}\P\L\Enr_\emptyset}.
\end{equation}
\end{notation}

\begin{theorem}\label{biii}
Let $\mV^\ot \to \Ass, \mW^\ot\to \Ass$ be small monoidal $\infty$-categories.
The functor (\ref{ujjj}) is an equivalence.
% \begin{equation*}\label{ujjj}_\mV\P\B\Enr_\mW\simeq {_{\mV \times \mW^\rev}\P\L\Enr_\emptyset}.\end{equation*}

\end{theorem}

\begin{notation}
Let $\mV^\ot \to \Ass, \mW^\ot \to \Ass$ be monoidal $\infty$-categories, $\mu: \mV^\ot \times_\Ass \mW^\ot \to \mW^\ot$ a monoidal functor and $\mN^\circledast \to \mW^\ot$ a left pseudo-enriched $\infty$-category.
By Theorem \ref{biii} the pullback $\mu^\ast(\mN)^\circledast \to \mV^\ot \times_\Ass \mW^\ot $ along $\mu: \mV^\ot \times_\Ass \mW^\ot \to \mW^\ot$ is the pullback of a unique bipseudo-enriched $\infty$-category $\mN_\mu^\circledast \to \mV^\ot\times (\mW^\rev)^\ot$.
Let $\widetilde{\mN}_\mu^\circledast \to \mW^\ot \times (\mV^\rev)^\ot$ be the corresponding bipseudo-enriched $\infty$-category via Notation \ref{invo}.
\end{notation}

%\begin{remark}If for every $\V \in \mV$ the functor $\mu(\V,-): \mW \to \mW $ admits a right adjoint $\Gamma_\V$ and $\mN^\circledast \to \mW^\ot$ exhibits $\mN$ as right enriched in $\mW$, then $\widehat{\mN}^\circledast \to \mV^\ot\times (\mW^\rev)^\ot$ exhibits $\mN$ as right enriched in $\mW^\rev$, where the right multi-morphism object for $\V_1,...,\V_\n \in \mV, \X,\Y \in \mN$ is $\R\Mul\Mor_{\widehat{\mN}}(\V_1,...,\V_\n, \X,\Y) =\Gamma_{\V_1 \ot ... \ot \V_\n}( \R\Mor_\mN(\X,\Y)).$\end{remark}
\begin{notation}
Let $\mV^\ot \to \Ass, \mW^\ot \to \Ass$ be monoidal $\infty$-categories
and $\mu: \mV^\ot \times_\Ass \mW^\ot \to \mW^\ot$ a monoidal functor.
Let $\mN^\circledast \to \mW^\ot$ be a left pseudo-enriched $\infty$-category and $\mM^\circledast \to \mV^\ot, \mO^\circledast \to \mW^\ot$ weakly left enriched $\infty$-categories.

\begin{enumerate}
\item Let $$\Enr\Fun_\mV(\mM,\mN)^\circledast \to \mW^\ot $$ be the weakly left enriched $\infty$-category corresponding to the weakly right enriched $\infty$-category $$\Enr\Fun_{\mV, \emptyset}(\mM,\mN_\mu)^\circledast \to (\mW^\rev)^\ot.$$

\item Let $$\Enr\Fun_\mW(\mO,\mN)^\circledast \to \mV^\ot $$ be
the weakly left enriched $\infty$-category corresponding to the weakly right enriched $\infty$-category $$\Enr\Fun_{\mW, \emptyset}(\mO,\widetilde{\mN}_\mu)^\circledast \to (\mV^\rev)^\ot. $$
\end{enumerate}

\end{notation}

Proposition \ref{psinho} has the following consequences: 
\begin{corollary}\label{psinhosp}
Let $\kappa$ be a small regular cardinal, 
$\mV^\ot \to \Ass, \mW^\ot \to \Ass$ monoidal $\infty$-categories compatible with $\kappa$-small colimits and $\mu: \mV^\ot \times_\Ass \mW^\ot \to \mW^\ot$ a monoidal functor preserving $\kappa$-small colimits component-wise.
Let $\mN^\circledast \to \mW^\ot$ be a left $\kappa$-enriched $\infty$-category and $\mM^\circledast \to \mV^\ot, \mO^\circledast \to \mW^\ot$ weakly left enriched $\infty$-categories.
The following weakly left enriched $\infty$-categories are left $\kappa$-enriched:
$$\Enr\Fun_{\mV}(\mM, {\mN})^\circledast \to \mW^\ot, \ \Enr\Fun_\mW(\mO,\mN)^\circledast \to \mV^\ot.$$
	
\end{corollary}

\begin{corollary}\label{psinhosp2}
Let $\mV^\ot \to \Ass, \mW^\ot \to \Ass$ be monoidal $\infty$-categories, $\mu: \mV^\ot \times_\Ass \mW^\ot \to \mW^\ot$ a monoidal functor, $\mM^\circledast \to \mV^\ot, \mO^\circledast \to \mW^\ot$ small weakly left enriched $\infty$-categories and $\mN^\circledast \to \mW^\ot$ a left enriched $\infty$-category. 
	
\begin{enumerate}
\item If $\mW$ admits small limits, the monoidal structure on $\mW$ is closed and for every $\mV \in \mV$ the functor $\mu(\V,-): \mW \to \mW$ admits a right adjoint, the weakly left enriched $\infty$-category		
$$\Enr\Fun_{\mV}(\mM, {\mN})^\circledast \to \mW^\ot$$
is a left enriched $\infty$-category.
		
\item If $\mV$ admits small limits, the monoidal structure on $\mV$ is closed and for every $\mW \in \mW$ the functor $\mu(-,\W): \mV \to \mW$ admits a right adjoint, the weakly left enriched $\infty$-category		
$$\Enr\Fun_\mW(\mO,\mN)^\circledast \to \mV^\ot $$ is a left enriched $\infty$-category.	
\end{enumerate}
\end{corollary}

\begin{notation}Let $\mV^\ot \to \Ass, \mW^\ot \to \Ass$ be monoidal $\infty$-categories and $\mu: \mV^\ot \times_\Ass \mW^\ot \to \mW^\ot$ a monoidal functor.
Let $\mN^\circledast \to \mW^\ot$ be a left pseudo-enriched $\infty$-category and $\mM^\circledast \to \mV^\ot, \mO^\circledast \to \mW^\ot$ weakly left enriched $\infty$-categories, $\mH$ a collection of left weights over $\mV$ and $\mH'$ a collection of left weights over $\mW.$

\begin{enumerate}
\item Let
$$\Enr\Fun_\mV^\mH(\mM, \mN)^\circledast \subset \Enr\Fun_\mV(\mM, \mN)^\circledast$$ be the full weakly left enriched subcategory of left $\mV$-enriched functors $\mM \to \mN$ preserving $\mH$-weighted colimits.
\item Let
$$\Enr\Fun_\mW^{\mH'}(\mO, \mN)^\circledast \subset \Enr\Fun_\mW(\mO, \mN)^\circledast$$ be the full weakly left enriched subcategory of left $\mW$-enriched functors $\mO \to \mN$ preserving $\mH'$-weighted colimits.
\end{enumerate}	
\end{notation}

Proposition \ref{Enric} gives the following corollary:

\begin{corollary}\label{Enrichi}
	
Let $\mV^\ot \to \Ass, \mW^\ot \to \Ass$ be monoidal $\infty$-categories and $\mu: \mV^\ot \times_\Ass \mW^\ot \to \mW^\ot$ a monoidal functor.
Let $\mN^\circledast \to \mW^\ot$ be a left pseudo-enriched $\infty$-category and $\mM^\circledast \to \mV^\ot, \mO^\circledast \to \mW^\ot$ weakly left enriched $\infty$-categories, $\mH$ a collection of left weights over $\mV$ and $\mH'$ a collection of left weights over $\mW.$

\begin{enumerate}
\item If $\mN^\circledast \to \mW^\ot$ admits $\mH'$-weighted colimits and $\mu(\tu_\mV,-):\mW^\ot \to \mW^\ot$ sends weights of $\mH$ to weights of $\mH'$,
then $\Enr\Fun_{\mV, \emptyset}(\mM,\mN)^\circledast \to \mW^\ot $ admits $\mH'$-weighted colimits, $\Enr\Fun_\mV^{\mH}(\mM, {\mN})^\circledast \subset \Enr\Fun_{\mV}(\mM,\mN)^\circledast $ is closed under $\mH'$-weighted colimits and for every $\X \in \mM$ the left $\mW$-enriched functor $\Enr\Fun_{\mV}(\mM,\mN)^\circledast \to \mN^\circledast$ preserves $\mH'$-weighted colimits.

\item If $\mN^\circledast \to \mW^\ot$ admits $\mH$-weighted colimits and $\mu(-,\tu_\mW):\mV^\ot \to \mW^\ot$ sends weights of $\mH$ to weights of $\mH'$,  then $\Enr\Fun_\mW(\mO, {\mN})^\circledast \to \mV^\ot $ admits $\mH$-weighted colimits, $\Enr\Fun_\mW^{\mH'}(\mO, {\mN})^\circledast \subset \Enr\Fun_\mW(\mO, {\mN})^\circledast $ is closed under $\mH$-weighted colimits and for every $\X \in \mO$ the left $\mV$-enriched functor $ \Enr\Fun_\mW(\mO, {\mN})^\circledast \to \mN^\circledast$ preserves $\mH$-weighted colimits.
\end{enumerate}	
\end{corollary}

For the next theorem we use Example \ref{Exx}:

\begin{theorem}\label{Thhs} Let $\kappa$ be a small regular cardinal, 
$\mV^\ot \to \Ass, \mW^\ot \to \Ass$ monoidal $\infty$-categories compatible with $\kappa$-small colimits and $\mu: \mV^\ot \times_\Ass \mW^\ot \to \mW^\ot$ a monoidal functor preserving $\kappa$-small colimits component-wise.
Let $\mH, \mH'$ be sets of small left weights over $\mV$, over $\mW$, respectively, such that $\mu(-,\tu_\mW):\mV^\ot \to \mW^\ot$ sends left weights of $\mH$ to left weights of $\mH'$.
Let $\mM^\circledast \to \mV^\ot,\mO^\ot \to \mW^\ot, \mN^\circledast \to \mW^\ot$ be left $\kappa$-enriched $\infty$-categories that admit $\mH$-weighted colimits, $\mH'$-weighted colimits, respectively. There are canonical equivalences
$$\Enr\Fun^{\mH'}_\mW(\mM \otimes^{\Enr_\kappa}_{\mH,\mH'} \mO, \mN) \simeq\Enr\Fun^{\mH'}_\mW(\mO,\Enr\Fun^\mH_\mV(\mM, \mN))\simeq\Enr\Fun^{\mH}_\mV(\mM,\Enr\Fun^{\mH'}_\mW(\mO, \mN)).$$

%\item If $\mW$ admits small limits, the monoidal structure on $\mW$ is closed, for every $\V \in \mV$ the functor $\V \ot (-): \mW \to \mW $ admits a right adjointand $\mM^\circledast \to \mV^\ot,\mO^\ot \to \mW^\ot, \mN^\circledast \to \mW^\ot$ are left enriched $\infty$-categories, there are canonical equivalences$$\Enr\Fun^{\mH'}_\mW(\mM \otimes^{\Enr}_\mH \mO, \mN) \simeq\Enr\Fun^{\mH'}_\mW(\mO,\Enr\Fun^\mH_\mV(\mM, \mN))$$$$\simeq\Enr\Fun^{\mH}_\mV(\mM,\Enr\Fun^\mH_\mV(\mO, \mN)).$$\end{enumerate}

\end{theorem}

\begin{proof}By \cite[Proposition 5.59.]{heine2024bienriched} there are canonical equivalences $$\Enr\Fun_{\mV, \emptyset}(\mM, \Enr\Fun_\mW(\mO, \mN)) \simeq \Enr\Fun_{\mW, \emptyset}(\mu_!(\mM \times \mO), \mN) \simeq \Enr\Fun_{\mW, \emptyset}(\mO, \Enr\Fun_\mV(\mM, \mN)),$$ which restrict to equivalences
$$\Enr\Fun_{\mV, \emptyset}^\mH(\mM, \Enr\Fun^{\mH'}_\mW(\mO, \mN)) \simeq \Enr\Fun^{\mu!(\Lambda^{(\mM,\mH)} \boxtimes \Lambda^{(\mO,\mH')})}_{\mW, \emptyset}(\mu_!(\mM \times \mO), \mN)$$$$ \simeq \Enr\Fun_{\mW, \emptyset}^{\mH'}(\mO, \Enr\Fun^\mH_\mV(\mM, \mN)).$$
%where the $\infty$-category in the middle is the full subcategory of functorspreserving $\mH$-weighted colimits in the first component and  $\mH'$-weighted colimits in the second component (see Notation \ref{notabene}).

By construction of the left action as $(\mM \otimes^{\Enr_\kappa}_{\mH.\mH'} \mO)^\circledast \simeq \mP\B\Env_{\mu_!(\Lambda^{(\mM,\mH)} \boxtimes \Lambda^{(\mO,\mH')})}^{\mH'}(\mu!(\mM \times \mO))^\circledast_{\L\Enr_\kappa} $ the left $\mW$-enriched functor $\mu_!(\mM \times \mO)^\circledast \to (\mM \otimes^{\Enr_\kappa}_{\mH,\mH'} \mO)^\circledast $ induces an equivalence
$$\Enr\Fun^{\mH'}_{\mW, \emptyset}(\mM \otimes^{\Enr_\kappa}_\mH \mO, \mN) \simeq \Enr\Fun^{\mu_!(\Lambda^{(\mM,\mH)} \boxtimes \Lambda^{(\mO,\mH')})}_{\mW, \emptyset}(\mu_!(\mM \times \mO), \mN).$$

\end{proof}

We obtain the following corollary using Corollary \ref{psinho}:

\begin{corollary}\label{Qa}
Let $1 \leq \bk \leq \infty$ and $\kappa$ a small regular cardinal, $\mV^\boxtimes \to \bE_{\bk+1}$ an $\bE_{\bk+1}$-monoidal $\infty$-category compatible with $\kappa$-small colimits and $\mH$ a set of small left weights over $\mV$.
\begin{enumerate}
\item The $\bE_\bk$-monoidal structure on $^\kappa_\mV\L\Enr_\emptyset(\mH)$ of Theorem \ref{thqaz} is closed. 
\item If $\mV^\boxtimes \to \bE_{\bk+1}$ is a presentably $\bE_{\bk+1}$-monoidal $\infty$-category, the $\bE_\bk$-monoidal structure on $_\mV\L\Enr_\emptyset(\mH)$ of Theorem \ref{Thor} is closed. 
\end{enumerate}
	
\end{corollary}

\section{Applications}

\subsection{Monoidality of enriched presheaves}

In the following we prove that the functor of enriched presheaves is monoidal.
We use the notation of Example \ref{rfgp}:

\begin{corollary}\label{Preml}
Let $\mV^\boxtimes \to \bE_{\bk+1} $ be an $\bE_{\bk+1}$-monoidal $\infty$-category compatible with small colimits for $1 \leq \bk \leq \infty$.
The functor $$_\mV\L\Q\Enr_\emptyset \to {_\mV\rc\rc\LMod_\emptyset}, (\mM^\circledast \to \mV^\ot) \mapsto (\mP_\mV(\mM)^\circledast \to \mV^\ot)$$
is $\bE_\bk$-monoidal and restricts to an $\bE_\bk$-monoidal functor ${_\mV\L\Enr_\emptyset}\to {_\mV\Pr\LMod}_\emptyset.$

%The functors %$$ {^\kappa_\mV\L\Enr_{\mW}}\to {_\mV\kappa\LMod}_{\mW}, (\mM^\circledast \to \mV^\ot \times \mW^\ot) \mapsto (\mP\widetilde{\B\Env}_\kappa(\mM)_{\L\Enr}^\circledast \to \mV^\ot \times \mW^\ot),$$$${_\mV\R\Enr^\kappa_{\mW}}\to {_\mV \kappa\RMod_{\mW}}, (\mM^\circledast \to \mV^\ot \times \mW^\ot) \mapsto (\mP\widetilde{\B\Env}_\kappa(\mM)_{\R\Enr}^\circledast \to \mV^\ot \times \mW^\ot)$$
%$$ {^\kappa_\mV\L\Enr_{\mW}}\to {_\mV\kappa\LMod}_{\mW}, (\mM^\circledast \to \mV^\ot \times \mW^\ot) \mapsto (\mP\widetilde{\L\Env}_\kappa(\mM)_{\L\Enr}^\circledast \to \mV^\ot \times \mW^\ot),$$$$
%{_\mV\R\Enr^\kappa_{\mW}}\to {_\mV \kappa\RMod_{\mW}}, (\mM^\circledast \to \mV^\ot \times \mW^\ot) \mapsto (\mP\widetilde{\R\Env}_\kappa(\mM)_{\R\Enr}^\circledast \to \mV^\ot \times \mW^\ot)$$$$_\mV^\kappa\B\Enr^{\kappa}_{\mW}\to {_\mV\kappa\BMod}_{\mW}, (\mM^\circledast \to \mV^\ot \times \mW^\ot) \mapsto (\mP\B\Env_\kappa(\mM)_{\B\Enr}^\circledast \to \mV^\ot \times \mW^\ot)$$are $\bE_\bk$-monoidal.

\end{corollary}

\begin{proof}
By Corollary \ref{preee} the inclusion $${_\mV\rc\rc\LMod^{\ot}} = {_\mV\L\Q\Enr}(\L\rc\rc)^\ot \subset {_\mV\L\Q\Enr}(\emptyset)^\ot=: {_\mV\L\Q\Enr^{\ot}}$$
is a lax $\bE_\bk$-monoidal functor that admits a left adjoint relative to $\bE_\bk$.
The left adjoint sends a small left quasi-enriched $\infty$-category $\mM^\circledast \to \mV^\ot$ to $\mP_\mV(\mM)^\circledast= \mP\widetilde{\B\Env}^{\L\rc\rc}(\mM)_{\L\Enr}^\circledast.$	
\end{proof}

For the next corollary we fix the following notation:

\begin{notation}%Let $\kappa$ be a small regular cardinal.
Let $\mV^\ot \to \Ass,\mW^\ot \to \Ass$ be monoidal $\infty$-categories compatible with small colimits.
For every weakly left enriched $\infty$-categories $\mM^\circledast \to \mV^\ot, \mN^\circledast \to \mW^\ot$
let $$(\mM \boxtimes \mN)^\circledast \to (\mV \otimes \mW)^\ot $$
be the pushforward of $\mM^\circledast \times_{\Ass} \mN^\circledast \to \mV^\ot \times_\Ass \mW^\ot$ along the canonical monoidal functor $\mV^\ot \times_\Ass \mW^\ot \to (\mV \otimes \mW)^\ot $ to the tensor product of $\cc\cc\Mon.$
\end{notation}

\begin{corollary}\label{compat}

Let $\mV^\ot \to \Ass,\mW^\ot \to \Ass$ be monoidal $\infty$-categories compatible with small colimits and $\mM^\circledast \to \mV^\ot, \mN^\circledast \to \mW^\ot $ small left quasi-enriched $\infty$-categories. 
There is a canonical equivalence of $\infty$-categories left tensored over $\mV\ot\mW:$
$$ \mP_{\mV\ot\mW}(\mM \boxtimes \mN)^\circledast \simeq (\mP_\mV(\mM)\ot \mP_{\mW}(\mN))^\circledast.$$
\end{corollary}

\begin{proof} Let $\mU:\mO^\ot \to  \cc\cc\Mon^\ot$ be the identity.
By Corollary \ref{preee} and Example \ref{infi} the lax $\cc\cc\Mon$-monoidal inclusion $$\cc\cc\LMod^\ot\simeq {_\mU\cc\cc\LMod^\ot} \subset {_\mU\L\Q\Enr}^\ot\simeq  \L\Q\Enr^\otimes$$ admits a left adjoint $\theta$ relative to $\cc\cc\Mon^\ot$.
The relative left adjoint $\theta$ is a map of cocartesian fibrations over $\cc\cc\Mon^\ot$. %and thus especially a symmetric monoidal functor. 
Hence $\theta$ induces a commutative square
$$\begin{xy}
\xymatrix{
{_\mV\L\Q\Enr_\emptyset} \times {_{\mW}\L\Q\Enr_\emptyset} \ar[d]^{\boxtimes}
\ar[rrr]^{\theta_{\mV} \times \theta_{\mW}}
&&& {_\mV\cc\cc\LMod_\emptyset} \times {_{\mW}\cc\cc\LMod_\emptyset} \ar[d]^{\ot} 
\\
{_{\mV \otimes \mW}\L\Q\Enr_\emptyset} \ar[rrr]^{\theta_{\mV\ot\mW}} &&& _{\mV \ot \mW}\cc\cc\LMod_\emptyset.}
\end{xy}$$
This square gives the desired equivalence.
\end{proof}

%\subsection{The pre-relative tensor product}

%We will prove the following theorem:
%\begin{remark}For every $\U \in \mU, \W \in \mW$ and $[\n]\in \Ass$ the fiber of the cocartesian fibration $\mM^\circledast_\U \times_{\mV^\ot} \mN_\W^\circledast \to \mV^\ot \to \Ass$ is canonically equivalent to $\mM_\U \times \mV^{\times \n}\times \mN_\W$and the cocartesian fibration $\mM^\circledast_\U \times_{\mV^\ot} \mN_\W^\circledast \to \mV^\ot \to \Ass$ classifies the Bar-construction(see ... for details). By ... this implies that$(\mM \times_{\mV} \mN)^\circledast_{\U, \W}$, i.e. the $\infty$-category arising from $ \mM^\circledast_\U \times_{\mV^\ot} \mN_\W^\circledast$ by formally inverting the $\sigma$-cocartesian morphisms, canonically identifies with the relative tensor product of $\mM^\circledast_\U \to \mV^\ot$ and$\mN_\W^\circledast \to \mV^\ot$, and $\tau: (\mM \times_{\mV} \mN)^\circledast \to \mU^\ot \times \mW^\ot$ is a model for the $\mU, \mW$-biaction on the tensor product relative to $\mV.$\end{remark}
%There is a canonical functor $ \mM \times \mN \subset \mM^\circledast \times_{\mV^\ot} \mN^\circledast \to (\mM \times_{\mV} \mN)^\circledast.$

\begin{lemma}\label{retaco}
Let $\mM^\circledast \to \mU^\ot \times \mV^\ot, \mN^\circledast \to \mV^\ot \times \mW^\ot$ be bitensored $\infty$-categories compatible with small colimits.
There is a bitensored $\infty$-category $ (\mM \otimes_{\mV} \mN)^\circledast \to \mU^\ot \times \mW^\ot$ compatible with small colimits and a $\mU, \mW$-linear functor $$ (\mM \boxtimes_{\mV} \mN)^\circledast \to (\mM \otimes_{\mV} \mN)^\circledast$$ 
whose underlying functor preserves small colimits component-wise, that induces for every bitensored $\infty$-category $\mO^\circledast \to \mU^\ot \times \mW^\ot$ compatible with small colimits an embedding $$\LinFun^\L_{\mU,\mW}(\mM \otimes_{\mV} \mN, \mO) \to \LinFun_{\mU,\mW}(\mM \boxtimes_{\mV} \mN, \mO),$$
whose essential image are the $\mU,\mW$-linear functors whose underlying functor preserves small colimits component-wise.

\end{lemma}

\begin{proof}
We set $ (\mM \otimes_{\mV} \mN)^\circledast:=\mP\B\Env_\Lambda^\mH(\mM \boxtimes_{\mV} \mN)_{\L\Enr}^\circledast \to \mU^\ot \times \mW^\ot,$
where $\mH$ is the large collection of small $(\mU, \mW)$-weights %$*_{\mU,\mW}$
and $\Lambda$ is the large collection of all colimit diagrams weighted with respect to weights on $*_{\mU,\mW}$ together with the following conical diagrams:
$$ \K^{\triangleright} \to \mM \simeq \mM \times \{\Y\} \to \mM \times \mN \subset (\mM \times_{\mV} \mN)^\circledast, \T^{\triangleright} \to \mN \simeq \{\X\} \times \mN \to \mM \times \mN \subset \mM^\circledast \times_{\mV^\ot} \mN^\circledast \to (\mM \boxtimes_{\mV} \mN)^\circledast $$
for some $\X \in \mM, \Y \in \mN$ and colimit diagrams
$ \K^{\triangleright} \to \mM,  \T^{\triangleright} \to \mN$.

\end{proof}

\begin{remark}
%By construction the canonical $\mU, \mW$-linear functors $\mM \boxtimes_{\mV} \mN \to \mM \otimes_{\mV} \mN$ induces for every bitensored $\infty$-category $\mO^\circledast \to \mU^\ot \times \mW^\ot$ compatible with small colimits an equivalence $$ \LinFun^{\rc\rc}_{\mU, \mW}(\mM \otimes_{\mV} \mN, \mO) \to \LinFun^{\rc\rc}_{\mU, \mW}(\mM \boxtimes_{\mV} \mN, \mO).$$
The universal property of Lemma \ref{retaco} implies that the $\mU, \mW$-linear functor $ (\mM \boxtimes_{\mV} \mN)^\circledast \to (\mM \otimes_{\mV} \mN)^\circledast$  
exhibits $(\mM \otimes_{\mV} \mN)^\circledast \to \mU^\ot \times \mW^\ot$
as the relative tensor product of $\mM^\circledast \to \mU^\ot \times \mV^\ot $ and $\mN^\circledast \to \mV^\ot\times \mW^\ot $ induced by the tensor product of $\Cat_\infty^{\rc\rc}.$

\end{remark}

\subsection{A tensor product for presentable enriched $\infty$-categories}

Let  $\mV^\boxtimes \to \Comm$ be a presentably symmetric monoidal $\infty$-category.
By Corollary \ref{thqazz} there is a symmetric monoidal structure on the $\infty$-category $_\mV\L\Q\Enr_\emptyset(\L\rc\rc) $ of left $\mV$-quasi-enriched $\infty$-categories compatible with small colimits and a symmetric monoidal equivalence 
$$_\mV\L\Q\Enr_\emptyset(\L\rc\rc) \simeq {_\mV\rc\rc\LMod}_\emptyset$$
to the $\infty$-category of small $\infty$-categories left tensored over $\mV$ endowed with the relative tensor product.

We will prove the following theorem:

\begin{theorem}\label{prestensor} Let $\mV^\boxtimes \to \Comm$ be a $\kappa$-compactly generated symmetric monoidal $\infty$-category for some small regular cardinal $\kappa$ and $\mM^\circledast \to \mV^\ot, \mN^\circledast \to \mV^\ot$ presentably left tensored $\infty$-categories. There are small left $\mV$-enriched $\infty$-categories $\mC^\circledast \to \mV^\ot,\mD^\circledast \to \mV^\ot$, sets $\rS$ of morphisms of $\mP_\mV(\mC)$ and $\T$ of morphisms of $\mP_\mV(\mD)$ closed under the left actions and left $\mV$-enriched equivalences $$\mM^\circledast \simeq \rS^{-1}\mP_\mV(\mC)^\circledast, \mN^\circledast \simeq \T^{-1}\mP_\mV(\mD)^\circledast.$$
The tensor product of $_\mV\L\Q\Enr_\emptyset(\L\rc\rc) \simeq {_\mV\rc\rc\LMod_\emptyset}$ admits the following descriptions:
\begin{enumerate}
\item The relative tensor product $$(\mM \ot_{\mV} \mN)^\circledast \to \mV^\ot.$$
	
\item The full subcategory of the internal hom in $_\mV\L\Enr_\emptyset:$ $$ \Enr\Fun^\R_{\mV}(\mM^\op,\mN)^\circledast \to \mV^\ot.$$
	
\item The accessible $\mV$-enriched localization: $$(\rS \boxtimes \T)^{-1}\mP_\mV(\mC \ot \mD)^\circledast \to \mV^\ot,$$
where $\rS \boxtimes \T$ is the set of morphisms $\{\f \ot \W \ot \Y, \V \ot \X \ot \g \mid \V,\W \in \mV^\kappa, \X \in \mC, \Y \in \mD, \f \in \rS, \g \in \T\}.$
	
\end{enumerate}

\end{theorem}

To prove Theorem \ref{prestensor} we prove different characterizations of presentable enriched $\infty$-categories (Theorem \ref{Pree}) and prove a description of the relative tensor product as an enriched functor $\infty$-category (Theorem \ref{thei}).

%In this subsection we express the internal hom of the tensor producton $_\mV\L\Enr_\emptyset(\mH)$ for any $\bE_2$-monoidal $\infty$-category $\mV^\boxtimes \to \bE_2$ and set $\mH$ of small weights over $\mV$ as a relative tensor product over $\mV$ of the $\infty$-category of $\mV$-enriched presheaves (Theorem \ref{Days}). This leads to an enriched Day-convolution monoidal structure on the $\mV$-enriched $\infty$-category of $\mV$-enriched functors (Corollary \ref{Dayco2}).

\subsubsection{Presentable enriched $\infty$-categories}

\begin{proposition}\label{krye}
Let $\kappa$ be a small regular cardinal, $\mV^\ot \to \Ass$ a $\kappa$-compactly generated monoidal $\infty$-category and $\mM^\circledast \to \mV^\ot$ a $\kappa$-compactly generated $\infty$-category left tensored over $\mV.$
%small left$\kappa$-enriched $\infty$-category that admits left tensors and $\kappa$-small conical colimits.
%The weakly bienriched $\infty$-category $\Enr\Fun^{\kappa-\lim}_{\emptyset, \mV}(\mM^\op,\mV)^\circledast \to \mV^\ot$ is a presentably left tensored $\infty$-category
The restricted left $\mV$-enriched Yoneda-embedding $\mM^\circledast\to \mP_\mV(\mM^\kappa)^\circledast$ admits a $\mV$-enriched left adjoint,
preserves small $\kappa$-filtered conical colimits and induces a left $\mV$-enriched equivalence
%and is a left $\mV$-enriched embedding %$ (\mM^\kappa)^\circledast \to \Enr\Fun_{\emptyset, \mV}((\mM^\kappa)^\op,\mV)^\circledast$
to the full subcategory of $\mV$-enriched presheaves preserving $\kappa$-small conical limits and right cotensors with $\kappa$-compact objects of $\mV.$

\end{proposition}

\begin{proof}
%follows from ... since for every object $\X$ of $\mM$ the 
%By Corollary \ref{cosqa} (1) the linear embedding $(\mM_\kappa)_\kappa^\circledast \subset \mM^\circledast$ of Proposition \ref{cool} is the pullback of a unique left $\mV$-enriched embedding $(\mM^\kappa)^\circledast \subset \mM^\circledast$. 
%(1): Let $\mN^\circledast \to \mV^\ot \times \mW^\ot$ be a bienriched $\infty$-category that admits small $\kappa$-filtered conical colimits.By Proposition \ref{cool} the following composition is an equivalence: % since $\mN^\circledast \to \mV^\ot$ is compatible with small $\kappa$-filtered colimits:$$\Enr\Fun^{\kappa-\mathrm{fil}}_{\mV,\mW}(\Ind_{\kappa}(\mM_\kappa),\mN)\to \Enr\Fun_{\mV,\mW}(\mM,\mN) \to \Enr\Fun_{\mV^\kappa,\mW^\kappa}(\mM_\kappa,\mN_\kappa)$$The second functor in the composition is an equivalence by Corollary \ref{cosqa} (1).So the first functor in the composition is an equivalence and the statement follows from Proposition \ref{wooond}.
The left $\mV$-enriched embedding $\mM^\circledast%\simeq \Ind_\kappa(\mM^\kappa)^\circledast
\to \mP_\mV(\mM^\kappa)^\circledast$
is $\mV$-enriched right adjoint to the left $\mV$-enriched embedding
$ (\mM^\kappa)^\circledast \subset \mM^\circledast$ of the full left enriched subcategory of $\kappa$-compact objects. Since the embedding $\mM^\kappa \subset \mM$ preserves $\kappa$-compact objects, the right adjoint preserves small $\kappa$-filtered colimits.

Let $\mN^\circledast \to \mV^\ot $ be a presentably left tensored $\infty$-category.
By Proposition \ref{cool} the following composition is an equivalence: % since $\mN^\circledast \to \mV^\ot$ is compatible with small $\kappa$-filtered colimits:
$$\LinFun^{\L}_{\mV,\emptyset}(\mM,\mN)\to \Enr\Fun'_{\mV,\emptyset}(\mM^\kappa,\mN) \to \LinFun^\kappa_{\mV^\kappa}(\mM^\kappa_\kappa,\mN_\kappa),$$
where $\Enr\Fun'_{\mV,\emptyset}(\mM^\kappa,\mN) \subset \Enr\Fun_{\mV,\emptyset}(\mM^\kappa,\mN)$ is the full subcategory of left $\mV$-enriched functors preserving $\kappa$-small colimits and left tensors with $\kappa$-compact objects of $\mV$.
The second functor in the composition is an equivalence as a consequence of Corollary \ref{cosqa} (1). So the first functor in the composition is an equivalence and the statement follows from Theorem \ref{wooond}.
	
\end{proof}

\begin{definition}
Let $\kappa$ be a small regular cardinal and $\mM^\circledast \to \mV^\ot$ a left $\mV$-enriched $\infty$-category. An object $\X$ of $\mM$ is $\mV$-enriched $\kappa$-compact if the left $\mV$-enriched functor $\L\Mor_\mM(\X,-): \mM^\circledast \to \mV^\circledast$ preserves small $\kappa$-filtered conical colimits.
\end{definition}

\begin{definition}\label{presss}
Let $\kappa$ be a small regular cardinal. A left $\mV$-enriched $\infty$-category $\mM^\circledast \to \mV^\ot$ is $\mV$-enriched $\kappa$-presentable if it admits small weighted colimits and every object of $\mM$ is a small $\kappa$-filtered conical colimit of $\mV$-enriched $\kappa$-compact objects of $\mM.$
	
\end{definition}

\begin{theorem}\label{Pree}
Let $\tau\leq \kappa$ be small regular cardinals, $\mV^\ot\to \Ass $ a $\tau$-compactly generated monoidal $\infty$-category and $\mM^\circledast \to \mV^\ot$ a left $\mV$-enriched $\infty$-category. The following conditions are equivalent:

\begin{enumerate}
\item The left $\mV$-enriched $\infty$-category $\mM^\circledast \to \mV^\ot$ is a $\kappa$-compactly generated left tensored $\infty$-category.

\item The left $\mV$-enriched $\infty$-category $\mM^\circledast \to \mV^\ot$ is a $\kappa$-accessible $\mV$-enriched localization of $\mP_\mV(\mN)^\circledast \to \mV^\ot$
for some small left $\mV$-enriched $\infty$-category $\mN^\circledast \to \mV^\ot.$

\item The left $\mV$-enriched $\infty$-category $\mM^\circledast \to \mV^\ot$ is $\mV$-enriched $\kappa$-presentable.

\end{enumerate}

\end{theorem}

\begin{proof}
Condition (2) implies (1) since $\mP_\mV(\mN)^\circledast \to \mV^\ot $ is a $\kappa$-compactly generated left tensored $\infty$-category by Remark \ref{rembrako}. By Proposition \ref{krye} condition (1) implies (2).
We prove that (1) and (3) are equivalent:
in view of Proposition \ref{weeei} and the assumption that the tensor unit of $\mV$ is
$\tau$-compact and so $\kappa$-compact it suffices to see that if $\mM^\circledast \to \mV^\ot$ is a left tensored $\infty$-category,
the left $\mV$-action preserves $\kappa$-compact objects if and only if for every $\kappa$-compact object $\X$ of $\mM$ the functor $\L\Mor_\mM(\X,-):\mM \to \mV$ preserves small $\kappa$-filtered colimits. Since $\mV$ is $\kappa$-compactly generated,
an object $\X \in \mM$ has the property that $\L\Mor_\mM(\X,-):\mM \to \mV$ preserves small $\kappa$-filtered colimits
if and only if for every $\kappa$-compact object $\V \in \mV$ the left tensor $\V \ot \X$ is $\kappa$-compact because $\L\Mor_\mM(\V\ot\X,-)\simeq \L\Mor_\mV(\V,\L\Mor_\mM(\X,-)).$

\end{proof}

\subsubsection{The relative tensor product as an enriched functor $\infty$-category}

The next lemma is \cite[Lemma 5.88.]{heine2024bienriched}:

\begin{lemma}
Let $\mM^\circledast \to \mU^\ot \times \mV^\ot, \mN^\circledast \to \mV^\ot \times \mW^\ot$ be weakly bienriched $\infty$-categories.
\begin{enumerate}
\item Then $\mM^\circledast \times_{\mV^\ot} \mN^\circledast \to \mU^\ot \times \mW^\ot$ is a weakly bienriched $\infty$-category.

\vspace{1mm}
\item If $\mM^\circledast \to \mU^\ot \times \mV^\ot$ is a left tensored $\infty$-category and $ \mN^\circledast \to \mV^\ot \times \mW^\ot$ is a right tensored $\infty$-category, then $\mM^\circledast \times_{\mV^\ot} \mN^\circledast \to \mU^\ot \times \mW^\ot$ is a bitensored $\infty$-category.

\vspace{1mm}
\item If $\mM^\circledast \to \mU^\ot \times \mV^\ot$ is a left tensored $\infty$-category compatible with small colimits
and $ \mN^\circledast \to \mV^\ot \times \mW^\ot$ is a right tensored $\infty$-category compatible with small colimits, then $\mM^\circledast \times_{\mV^\ot} \mN^\circledast \to \mU^\ot \times \mW^\ot$ is a bitensored $\infty$-category compatible with small colimits.
\end{enumerate}

\end{lemma}

\begin{notation}
Let $\mM^\circledast \to \mU^\ot \times \mV^\ot, \mN^\circledast \to \mV^\ot \times \mW^\ot$ be bitensored $\infty$-categories %compatible with small colimits
and $\sigma : \mM^\circledast \times_{\mV^\ot} \mN^\circledast \to \mV^\ot \to \Ass$ the canonical cocartesian fibration.
We write $$ (\mM \boxtimes_{\mV} \mN)^\circledast $$ for
the $\infty$-category arising from $\mM^\circledast \times_{\mV^\ot} \mN^\circledast $ by formally inverting the $\sigma$-cocartesian morphisms.

\end{notation}

Since the functor $\mM^\circledast \times_{\mV^\ot} \mN^\circledast \to \mU^\ot \times \mW^\ot$ inverts $\sigma$-cocartesian morphisms, it induces a functor $$\tau: (\mM \boxtimes_{\mV} \mN)^\circledast \to \mU^\ot \times \mW^\ot.$$	

\begin{lemma}\label{reta}
Let $\mM^\circledast \to \mU^\ot \times \mV^\ot, \mN^\circledast \to \mV^\ot \times \mW^\ot$ be bitensored $\infty$-categories.
The functor $\tau: (\mM \boxtimes_{\mV} \mN)^\circledast \to \mU^\ot \times \mW^\ot$
is a bitensored $\infty$-category and the following canonical functor is $\mU, \mW$-linear:
$$\mM^\circledast \times_{\mV^\ot} \mN^\circledast \to (\mM \boxtimes_{\mV} \mN)^\circledast.$$

\end{lemma}

\begin{proof}Since the functor $\mM^\circledast \times_{\mV^\ot} \mN^\circledast \to \mU^\ot \times \mW^\ot$ inverts $\sigma$-cocartesian morphisms, every $\sigma$-cocartesian morphism belongs to some fiber
$\mM^\circledast_{[\n]} \times_{\mV^\ot} \mN_{[\m]}^\circledast $ for some $[\m], [\n] \in \Delta$. Every morphisms $ [\n] \to [\n'], [\m]\to [\m']$
in $\Delta$ induce a left $\mV$-linear functor
$\mM^\circledast_{[\n']} \times_{\mV^\ot} \mN_{[\n]}^\circledast\to \mM^\circledast_{[\m']} \times_{\mV^\ot} \mN_{[\m]}^\circledast$ that preserves $\sigma$-cocartesian morphisms being left $\mV$-linear.
Let $\mT$ be the collection of $\sigma$-cocartesian morphisms.
By \cite[Proposition 2.1.4.]{HinichDwyer} this implies that $\tau$ is a map of cocartesian fibrations over $\Ass \times\Ass$, the canonical functor 
$\mM^\circledast \times_{\mV^\ot} \mN^\circledast \to (\mM \times_{\mV} \mN)^\circledast $ is $\mU, \mW$-linear, and for any $[\n], [\m]\in \Delta$ the induced functor 
$$ (\mM \times_{\mV} \mN)_{[\n],[\m]}^\circledast \to (\mM_{[\n]}^\circledast \times_{\mV^\ot} \mN_{[\m]}^\circledast)[\mT] \simeq \mU^{\times \n} \times (\mM^\circledast_{[0]} \times_{\mV^\ot} \mN_{[0]}^\circledast)[\mT]\times \mW^{\times\m} $$%= \mU^{\times \n} \times (\mM \times_{\mV} \mN) \times \mW^{\times\m} $$
is an equivalence.
%= \times \mW^{\times\m} $$$$\simeq (\mM^\circledast_{[\n]} \times_{\mV^\ot} \mN_{[\m]}^\circledast)[\mT] \to \mU^\ot_{[\n]}\times \mW^\ot_{[\m]} \simeq \mU^{\times \n} \times \mW^{\times\m}.$$ %by formally inverting the $\sigma$-cocartesian morphisms on the left hand side.
Thus $\tau$ is a bitensored $\infty$-category.

\end{proof}

\begin{lemma}\label{retaco}	Let $\mM^\circledast \to \mU^\ot \times \mV^\ot, \mN^\circledast \to \mV^\ot \times \mW^\ot$ be bitensored $\infty$-categories compatible with small colimits.
There is a bitensored $\infty$-category $ (\mM \otimes_{\mV} \mN)^\circledast \to \mU^\ot \times \mW^\ot$ compatible with small colimits and a $\mU, \mW$-linear functor $$ (\mM \boxtimes_{\mV} \mN)^\circledast \to (\mM \otimes_{\mV} \mN)^\circledast$$ 
whose underlying functor preserves small colimits component-wise, that induces for every bitensored $\infty$-category $\mO^\circledast \to \mU^\ot \times \mW^\ot$ compatible with small colimits an embedding $$\LinFun^\L_{\mU,\mW}(\mM \otimes_{\mV} \mN, \mO) \to \LinFun_{\mU,\mW}(\mM \boxtimes_{\mV} \mN, \mO),$$
whose essential image are the $\mU,\mW$-linear functors whose underlying functor preserves small colimits component-wise.
	
\end{lemma}

\begin{proof}
We set $ (\mM \otimes_{\mV} \mN)^\circledast:=\mP\B\Env_\Lambda^\mH(\mM \boxtimes_{\mV} \mN)_{\L\Enr}^\circledast \to \mU^\ot \times \mW^\ot,$
where $\mH$ is the large collection of small $(\mU, \mW)$-weights %$*_{\mU,\mW}$
and $\Lambda$ is the large collection of all colimit diagrams weighted with respect to weights on $*_{\mU,\mW}$ together with the following conical diagrams:
$$ \K^{\triangleright} \to \mM \simeq \mM \times \{\Y\} \to \mM \times \mN \subset (\mM \times_{\mV} \mN)^\circledast, \T^{\triangleright} \to \mN \simeq \{\X\} \times \mN \to \mM \times \mN \subset \mM^\circledast \times_{\mV^\ot} \mN^\circledast \to (\mM \boxtimes_{\mV} \mN)^\circledast $$
for some $\X \in \mM, \Y \in \mN$ and colimit diagrams
$ \K^{\triangleright} \to \mM,  \T^{\triangleright} \to \mN$.
	
\end{proof}

\begin{remark}
Let $\mM^\circledast \to \mU^\ot \times \mV^\ot, \mN^\circledast \to \mV^\ot \times \mW^\ot$ be bitensored $\infty$-categories.
The cocartesian fibration $\mM^\circledast_{[0]} \times_{\mV^\ot} \mN_{[0]}^\circledast \to \mV^\ot \to \Ass$ 
classifies the Bar-construction $\Delta^\op \to \Cat_\infty, [\n] \mapsto \mM \times \mV^{\times \n} \times \mN $ of \cite[Construction 4.4.2.7.]{lurie.higheralgebra} and
$\mM \boxtimes_\mV \mN$ is the colimit of the Bar-construction, the relative 
tensor product of $\mM^\circledast_{[0]} \to \mV^\ot$ and $\mN_{[0]}^\circledast \to \mV^\ot$ \cite[Definition 4.4.2.10.]{lurie.higheralgebra}.

\end{remark}

\begin{notation}
Let $\mN^\circledast \to \mV^\ot \times \mW^\ot$ be a weakly bienriched $\infty$-category, $\mM^\circledast \to \mV^\ot $ a weakly left enriched $\infty$-category and $\Lambda$ a collection of right diagrams in $\mM^\op$. %and $\mH$ a collection of left weights over $\mV$.
Let $$ \Enr\Fun^{\Lambda-\lim}_{\mV, \emptyset}(\mM,\mN)^\circledast \subset \Enr\Fun_{\mV, \emptyset}(\mM,\mN)^\circledast$$
be the full weakly right enriched subcategory of left $\mV$-enriched functors
sending diagrams of $\Lambda$ to $\mH$-weighted limit diagrams.
	
\end{notation}

\begin{lemma}\label{Remrema}Let $\mN^\circledast \to \mV^\ot \times \mW^\ot$ be a presentably bitensored $\infty$-category, $\mM^\circledast \to \mV^\ot $ a small right enriched $\infty$-category and $\Lambda$ a collection of left diagrams in $\mM^\op$ such that $\mM $ admits limits weighted with respect to underlying weights of diagrams of $\Lambda.$	

\begin{enumerate}
\item The weakly right enriched $\infty$-category $ \Enr\Fun_{\mV, \emptyset}(\mM^\op,\mN)^\circledast \to \mW^\ot $ is presentably right tensored.

\item The full weakly right enriched subcategory $\Enr\Fun^{\Lambda-\lim}_{\mV, \emptyset}(\mM^\op,\mN)^\circledast \subset \Enr\Fun_{\mV, \emptyset}(\mM^\op,\mN)^\circledast$ is a right $\mW$-linear accessible localization and so presentably right tensored. 

\end{enumerate}

\end{lemma}

\begin{proof}
	
(1) follows from \cite[Lemma 2.55., Corollary 4.24., Corollary 5.14.]{heine2024bienriched}.

(2): For every $\Z \in \mN$ let $\psi_\Z$ be the composition $$\mM^\circledast \subset \Enr\Fun_{\mV, \emptyset}(\mM^\op,\mV)^\circledast \xrightarrow{((-)\ot \Z)_*} \Enr\Fun_{\mV, \emptyset}(\mM^\op,\mN)^\circledast,$$ where the first embedding is the right $\mV$-enriched Yoneda-embedding.
Let $\kappa$ be a small regular cardinal such that $\mN^\circledast \to \mV^\ot \times \mW^\ot$ is a $\kappa$-compactly generated bitensored $\infty$-category. The full weakly right enriched subcategory $\Enr\Fun^{\Lambda-\lim}_{\mV, \emptyset}(\mM^\op,\mN)^\circledast \subset \Enr\Fun_{\mV, \emptyset}(\mM^\op,\mN)^\circledast$ is the right $\mW$-linear accessible localization with respect to the set of morphisms $$\{\theta: \colim^{\rH}(\psi_\Z \circ \F) \to \psi_\Z(\Y) \mid \Z \in \mN^\kappa, (\F: \mJ^\circledast \to \mM^\circledast, \rH \in \mP\B\Env(\mJ), \Y \in \mM, \F_!(\rH)\to \Y ) \in \Lambda\} $$ in $\Enr\Fun_{\mV, \emptyset}(\mM^\op,\mN),$
where $\theta$ corresponds to the following morphism in $\mP\B\Env(\mJ):$ $$\rH \to \F^\ast(\Y) \to (\psi_\Z\circ \F)^\ast(\psi_\Z(\Y)) \simeq \F^*(\psi_\Z^\ast(\psi_\Z(\Y))).$$ 
	
\end{proof}

\begin{proposition}\label{leumorat}
Let $\mN^\circledast \to \mV^\ot \times \mW^\ot$ be a presentably bitensored $\infty$-category.

\begin{enumerate}
\item Let $\mM^\circledast \to \mV^\ot $ be a small left enriched $\infty$-category and $\Lambda$ a collection of right diagrams in $\mM^\op$ such that $\mM$ admits limits weighted with respect to underlying weights of diagrams of $\Lambda.$
There is a canonical equivalence of $\infty$-categories presentably right tensored over $\mW:$
$$ \Enr\Fun^{\Lambda-\lim}_{\mV, \emptyset}(\mM,\mN)^\circledast \simeq (\Enr\Fun^{\Lambda-\lim}_{\mV, \emptyset}(\mM,\mV) \ot_\mV \mN)^\circledast.$$

\item Let $\mM^\circledast \to \mW^\ot $ be a small right enriched $\infty$-category and $\Lambda$ a collection of left diagrams in $\mM^\op$ such that $\mM$ admits limits weighted with respect to underlying weights of diagrams of $\Lambda.$
There is a canonical equivalence of $\infty$-categories presentably left tensored over $\mV:$
$$ \Enr\Fun^{\Lambda-\lim}_{\emptyset, \mW}(\mM,\mN)^\circledast \simeq (\mN \ot_\mW \Enr\Fun^{\Lambda-\lim}_{\emptyset, \mW}(\mM,\mW))^\circledast.$$
\end{enumerate}	
\end{proposition}

\begin{proof}We prove (1). Statement (2) is dual.
Let $\alpha:\mV^\circledast \times_{\mV^\ot} \mN^\circledast \to (\mV \otimes_{\mV} \mN)^\circledast \simeq \mN^\circledast$ be the universal $\mV, \mW$-enriched functor.	
The $\mV, \mW$-enriched evaluation functor $\mM^\circledast \times \Enr\Fun_{\mV, \emptyset}(\mM,\mV)^\circledast \to \mV^\circledast$
gives rise to a $\mV, \mW$-enriched functor $$\mM^\circledast \times \Enr\Fun_{\mV, \emptyset}(\mM,\mV)^\circledast \times_{\mV^\ot} \mN^\circledast \to \mV^\circledast \times_{\mV^\ot} \mN^\circledast \xrightarrow{\alpha} \mN^\circledast $$
adjoint to a right $\mW$-enriched functor
$$\kappa: \Enr\Fun_{\mV, \emptyset}(\mM,\mV)^\circledast \times_{\mV^\ot} \mN^\circledast \to \Enr\Fun_{\mV, \emptyset}(\mM,\mN)^\circledast$$
that sends $\F \in \Enr\Fun_{\mV, \emptyset}(\mM,\mV), \Y \in \mN$ to $\X \mapsto \F(\X)\ot \Y.$
The functor $\kappa$ is right $\mW$-linear because for every $\F: \mM \to \mV$ and $\Y \in \mN$, $\W \in \mW$ the canonical morphism $\kappa(\F,\Y)\ot \W \to \kappa(\F, \Y \ot \W)$ identifies after evaluation at any $\Z \in \mM$ with the canonical equivalence $$(\F(\Z) \ot \Y) \ot \W \simeq \F(\Z) \ot (\Y \ot \W).$$

For every $\Y \in \mN$ the functor $\kappa(-, \Y): \Enr\Fun_{\mV, \emptyset}(\mM,\mV) \to \Enr\Fun_{\mV, \emptyset}(\mM,\mN), \F \mapsto \F(-) \ot \Y$ preserves weighted colimits because the left $\mV$-linear left adjoint functor $(-)\ot \Y: \mV^\circledast \to \mN^\circledast$ preserves weighted colimits using Proposition \ref{Enric}.
Since $\kappa(-, \Y): \Enr\Fun_{\mV, \emptyset}(\mM,\mV) \to \Enr\Fun_{\mV, \emptyset}(\mM,\mN)$ preserves weighted colimits, it sends generating local equivalences to local equivalences and so induces a right $\mW$-linear functor $\kappa'$
%$$\kappa': \Enr\Fun^{\mH}_{\mV, \emptyset}(\mM,\mV)^\circledast \boxtimes_{\mV^\ot} \mN^\circledast \to \Enr\Fun^{\mH}_{\mV, \emptyset}(\mM,\mN)^\circledast$$
that fits into a commutative square:
\begin{equation}\label{sq}
\begin{xy}
\xymatrix{
\Enr\Fun_{\mV, \emptyset}(\mM,\mV)^\circledast \times_{\mV^\ot} \mN^\circledast \ar[d]
\ar[rr]^{\kappa}
&& \Enr\Fun_{\mV, \emptyset}(\mM,\mN)^\circledast \ar[d]^{} 
\\
\Enr\Fun^{\Lambda-\lim}_{\mV, \emptyset}(\mM,\mV)^\circledast \times_{\mV^\ot} \mN^\circledast \ar[rr]^{\kappa'} && \Enr\Fun^{\Lambda-\lim}_{\mV, \emptyset}(\mM,\mN)^\circledast.}
\end{xy}\end{equation}

Let $\sigma: \Enr\Fun_{\mV, \emptyset}(\mM,\mV)^\circledast \times_{\mV^\ot} \mN^\circledast \to \mV^\ot \to \Ass$ be the canonical functor.
We prove next that $\kappa$ sends $\sigma$-cocartesian morphisms to equivalences.
This holds if and only if for every $\Z \in \mM$ the right $\mW$-linear functor $$ \Enr\Fun_{\mV, \emptyset}(\mM,\mV)^\circledast \times_{\mV^\ot} \mN^\circledast \xrightarrow{\kappa} \Enr\Fun_{\mV, \emptyset}(\mM,\mN)^\circledast \xrightarrow{\ev_\Z}\mN^\circledast$$
inverts $\sigma$-cocartesian morphisms. By construction the latter factors as 
$$ \Enr\Fun_{\mV, \emptyset}(\mM,\mV)^\circledast \times_{\mV^\ot} \mN^\circledast 
\xrightarrow{\ev_\Z \times_{\mV^\ot} \mN^\circledast} \mV^\circledast  \times_{\mV^\ot} \mN^\circledast \xrightarrow{\alpha} \mN^\circledast $$
%$$\rho_{[0]}: \Enr\Fun_{\mV, \emptyset}(\mM,\mV)^\circledast \times_{\mV^\ot} \mN^\circledast_{[0]} \xrightarrow{\ev_\Z \times_{\mV^\ot}  \mN^\circledast_{[0]}} \mV^\circledast \times_{\mV^\ot} \mN^\circledast_{[0]} \xrightarrow{\alpha_{[0]}} \mN^\circledast_{[0]}$$underlying $\rho$ sends $\sigma$-cocartesian morphisms to equivalences.
and $\ev_\Z: \Enr\Fun_{\mV, \emptyset}(\mM,\mV)^\circledast \to \mV^\circledast $ is left $\mV$-linear and $\alpha$ universally inverts morphisms cocartesian with respect to $\mV^\circledast \times_{\mV^\ot} \mN^\circledast \to \mV^\ot \to \Ass$.
Therefore $\kappa$ inverts $\sigma$-cocartesian morphisms. Let $\sigma': \Enr\Fun^{\Lambda-\lim}_{\mV, \emptyset}(\mM,\mV)^\circledast \times_{\mV^\ot} \mN^\circledast \to \mV^\ot \to \Ass$ be the canonical functor.
Since both vertical functors in square (\ref{sq}) send $\sigma$-cocartesian morphisms to $\sigma'$-cocartesian morphisms, $\kappa'$ 
%followed by $\Enr\Fun_{\mV, \emptyset}(\mM,\mN)^\circledast \to \Enr\Fun^{\mH}_{\mV, \emptyset}(\mM,\mN)^\circledast$ factors as $\Enr\Fun_{\mV, \emptyset}(\mM,\mV)^\circledast \times_{\mV^\ot} \mN^\circledast \to\Enr\Fun^{\mH}_{\mV, \emptyset}(\mM,\mV)^\circledast \times_{\mV^\ot} \mN^\circledast\xrightarrow{\kappa''} \Enr\Fun^{\mH}_{\mV, \emptyset}(\mM,\mN)^\circledast$and the localization functor $\Enr\Fun_{\mV, \emptyset}(\mM,\mV)^\circledast \to\Enr\Fun^{\mH}_{\mV, \emptyset}(\mM,\mV)^\circledast$ is $\mV$-linear,also $\kappa''$ 
inverts $\sigma'$-cocartesian morphisms. So by Lemma \ref{reta} the functor $\kappa'$ induces a right $\mW$-linear functor $$\rho: (\Enr\Fun^{\Lambda-\lim}_{\mV, \emptyset}(\mM,\mV) \boxtimes_{\mV} \mN)^\circledast \to \Enr\Fun^{\Lambda-\lim}_{\mV, \emptyset}(\mM,\mN)^\circledast.$$

The right $\mW$-linear functor $\kappa$ induces on underlying
$\infty$-categories the functor $$\Enr\Fun_{\mV, \emptyset}(\mM,\mV) \times \mN \to \Enr\Fun_{\mV, \emptyset}(\mM,\mN), (\F, \Y) \mapsto ((-)\ot \Y)\circ \F$$ that preserves small colimits component-wise.
Since both vertical functors in square (\ref{sq}) are left adjoints, 
the right $\mW$-linear functor $\kappa'$ induces on underlying
$\infty$-categories a functor preserving small colimits component-wise.
So by Lemma \ref{retaco} the functor $\rho$ induces a right $\mW$-linear small colimits preserving functor
$$\lambda: (\Enr\Fun^{\Lambda-\lim}_{\mV, \emptyset}(\mM,\mV) \otimes_{\mV} \mN)^\circledast \to \Enr\Fun^{\Lambda-\lim}_{\mV, \emptyset}(\mM,\mN)^\circledast$$
that fits into a commutative square of $\infty$-categories right tensored over $\mW:$
\begin{equation}\label{sqay}
\begin{xy}
\xymatrix{
(\Enr\Fun^{\Lambda-\lim}_{\mV, \emptyset}(\mM,\mV) \otimes_{\mV} \mN)^\circledast \ar[d]
\ar[rr]^{\lambda}
&& \Enr\Fun^{\Lambda-\lim}_{\mV, \emptyset}(\mM,\mN)^\circledast \ar[d]^{} 
\\
(\Fun(\mM^\simeq,\mV) \otimes_{\mV} \mN)^\circledast \ar[rr]^\simeq && \Fun(\mM^\simeq,\mN)^\circledast.}
\end{xy}
\end{equation}
The canonical conservative right $\mW$-linear functor $ \Enr\Fun_{\mV, \emptyset}(\mM,\mN)^\circledast \to (\mN^{\mM^\simeq})^\circledast$ preserves small colimits and admits a $\mW$-enriched left adjoint by \cite[Corollary 5.23.]{heine2024bienriched} and so is monadic by \cite[Theorem 4.7.3.5.]{lurie.higheralgebra}.
Therefore the right vertical functor in square (\ref{sqay}) is monadic, too, where
we write $\F$ for the left adjoint.
%and induces on underlying $\infty$-categories a conservative functor that preserves small colimits (Lemma \ref{colas}).
By \cite[Corollary 4.66.]{heine2024bienriched} also the left vertical functor in square (\ref{sqay}) is monadic, where the left adjoint is $\F' \ot_\mV \mN$
when $\F'$ is the left adjoint of $\Enr\Fun^{\Lambda-\lim}_{\mV, \emptyset}(\mM,\mV) \to \Fun(\mM^\simeq,\mV).$
Therefore by \cite[Corollary 4.7.3.16.]{lurie.higheralgebra} it is enough to check that $\lambda$ preserves the left adjoints,
i.e. that the canonical morphism $ \F \to \lambda \circ \F' \ot_\mV \mN$ is an equivalence.
By \cite[Lemma 5.21.]{heine2024bienriched} the $\infty$-category $\Fun(\mM^\simeq,\mN)$ is generated under small colimits by the functors of the form
$\mM^\simeq \xrightarrow{\mM^\simeq(\X,-)} \mS \xrightarrow{(-)\ot \Y} \mN$
for $\X \in \mM, \Y \in \mN.$
Therefore it is enough to check that for every $\X \in \mM, \Y \in \mN$ the canonical morphism $$\psi: \F(\mM^\simeq(\X,-)\ot \Y) \to \kappa'(\F'(\mM^\simeq(\X,-)), \Y) $$ is an equivalence.
By \cite[Lemma 5.21., Theorem 5.22.]{heine2024bienriched} there is a local equivalence $\Mor_\mM(\X,-)\to \F'(\mM^\simeq(\X,-)) $ and a commutative square in $\Enr\Fun^{\Lambda-\lim}_{\mV, \emptyset}(\mM,\mN) $:
\begin{equation*}
\begin{xy}
\xymatrix{
\Mor_\mM(\X,-) \ot \Y  \simeq  \kappa(\Mor_\mM(\X,-), \Y)\ar[d]^{} \ar[rr] && \kappa(\F'(\mM^\simeq(\X,-)), \Y)\ar[d]^{}
\\
\F(\mM^\simeq(\X,-)\ot \Y) \ar[rr]^\psi && \kappa'(\F'(\mM^\simeq(\X,-)), \Y),}
\end{xy}
\end{equation*}
where both vertical morphisms are local equivalences.
So $\psi$ is an equivalence if and only if the top morphism of the square is a local equivalence.
The top morphism is a local equivalence since $\kappa(-,\Y)$ preserves local equivalences.
\end{proof}

%Let $\mN^\circledast \to \mV^\ot \times \mW^\ot$ be a presentably bitensored $\infty$-category, $\mH$ a collection of weights over $\mV$ and $\mM^\circledast \to \mV^\ot $ a small $\infty$-category right enriched in $\mV$ that admits $\mH$-weighted colimits. There is a canonical equivalence of $\infty$-categories presentably right tensored over $\mW:$$$ \Enr\Fun^\mH_{\mV, \emptyset}(\mM^\op,\mN)^\circledast \simeq (\mP\R\Env^\mH(\mM)_{\R\Enr}\ot_\mV \mN)^\circledast.$$	

\begin{proposition}\label{Days}
Let $\mN^\circledast \to \mV^\ot \times \mW^\ot$ be a presentably bitensored $\infty$-category, $\mM^\circledast \to \mW^\ot $ a small left enriched $\infty$-category and $\Lambda$ a collection of right diagrams in $\mM^\op$ such that $\mM$ admits limits weighted with respect to underlying weights of diagrams of $\Lambda.$
There is a left $\mV$-linear equivalence: % of $\infty$-categories presentably left tensored over $\mV:$
$$ \Enr\Fun^{\Lambda-\lim}_{\emptyset,\mW}(\mM^\op,\mN)^\circledast \simeq (\mN\ot_\mW \mP\L\Env_\Lambda(\mM)_{\L\Enr})^\circledast.$$	
%$$ \Enr\Fun^{\Lambda-\lim}_{\emptyset,\mW}(\mM^\op,\mN)^\circledast \simeq (\mN\ot_\mW \mP_\Lambda(\mM))^\circledast.$$	
	
\end{proposition}

\begin{proof}
%For $\mH$ empty 
By Proposition \ref{leumorat} there is a left $\mV$-linear equivalence:
%an equivalence of $\infty$-categories presentably left tensored over $\mV:$
$$ \Enr\Fun^{\Lambda-\lim}_{\emptyset,\mW}(\mM^\op,\mN)^\circledast \simeq (\mN\ot_\mW\Enr\Fun^{\Lambda-\lim}_{\emptyset,\mW}(\mM^\op,\mW))^\circledast.$$	
By Theorem \ref{unitol} there is a % left $\mW$-enriched Yoneda-embedding $\mM^\circledast \to \Enr\Fun_{\emptyset,\mW}(\mM^\op,\mW)^\circledast$ uniquely extends to a 
left $\mW$-linear equivalence $\mP\L\Env(\mM)_{\L\Enr}^\circledast\simeq \Enr\Fun_{\emptyset,\mW}(\mM^\op,\mW)^\circledast$ that restricts to an equivalence %of $\infty$-categories presentably left tensored over $\mW$: $
%$\mP_\mW^\mH(\mM)^\circledast=
$\mP\L\Env_\Lambda(\mM)_{\L\Enr}^\circledast \simeq \Enr\Fun^{\Lambda-\lim}_{\emptyset,\mW}(\mM^\op,\mW)^\circledast$
by Remark \ref{umbes}.
%So the result follows.

\end{proof}

\begin{corollary}\label{Dayso}
Let $\mN^\circledast \to \mV^\ot \times \mW^\ot$ be a presentably bitensored $\infty$-category and $\mM^\circledast \to \mW^\ot $ a small left enriched $\infty$-category.
There is a left $\mV$-linear equivalence $$ \Enr\Fun_{\emptyset,\mW}(\mM^\op,\mN)^\circledast \simeq (\mN\ot_\mW \mP_\mW(\mM))^\circledast.$$	
\end{corollary}

\begin{lemma}\label{leums}
Let $\mN^\circledast \to \mV^\ot \times \mW^\ot$ be a weakly bienriched $\infty$-category and $\mM^\circledast \to \mV^\ot $ a weakly left enriched $\infty$-category. 
There is a canonical left $\mW$-enriched equivalence:
$$(\Enr\Fun_{\mV, \emptyset}(\mM,\mN)^\op)^\circledast \simeq \Enr\Fun_{\emptyset, \mV}(\mM^\op,\mN^\op)^\circledast.$$
\end{lemma}
\begin{proof}
The desired equivalence is represented by the canonical equivalence
$$\Enr\Fun_{\mW,\emptyset}(\mO^\op,\Enr\Fun_{\mV, \emptyset}(\mM,\mN)^\op) \simeq \Enr\Fun_{\emptyset, \mW}(\mO,\Enr\Fun_{\mV, \emptyset}(\mM,\mN)) $$$$\simeq \Enr\Fun_{\mV,\mW}(\mM \times \mO,\mN) \simeq \Enr\Fun_{\mW,\mV}(\mO^\op \times \mM^\op,\mN^\op) $$$$ \simeq \Enr\Fun_{\mW, \emptyset}(\mO^\op, \Enr\Fun_{\emptyset, \mV}(\mM^\op,\mN^\op))$$
for every weakly right enriched $\infty$-category $\mO^\circledast \to \mW^\ot.$	

\end{proof}

%\begin{notation}Let $\mN^\circledast \to \mV^\ot \times \mW^\ot$ be a weakly bienriched $\infty$-category and $\mM^\circledast \to \mV^\ot $ a weakly left enriched $\infty$-category. Let$$ \Enr\Fun^\R_{\mV, \emptyset}(\mM,\mN)^\circledast \subset \Enr\Fun_{\mV, \emptyset}(\mM,\mN)^\circledast $$be the full weakly right enriched subcategory spanned by the left $\mV$-enriched functors that admit a right adjoint relative to $\mV^\ot.$\end{notation}

\begin{theorem}\label{thei}
Let $\kappa$ be a small regular cardinal, $\mM^\circledast \to \mV^\ot$ a $\kappa$-compactly generated right tensored $\infty$-category and $\mN^\circledast \to \mV^\ot \times \mW^\ot$ a presentably bitensored $\infty$-category. % compatible with small colimits.
% and $\mM^\circledast \to \mV^\ot $ a small right tensored $\infty$-category compatible with $\kappa$-small colimits. 
There is a right $\mW$-linear equivalence 
$$ \Enr\Fun^\R_{\mV, \emptyset}(\mM^\op,\mN)^\circledast \simeq (\mM \ot_{\mV} \mN)^\circledast.$$
In particular, $(\mM \ot_{\mV} \mN)^\circledast \to \mW^\ot$ is a presentably right tensored
$\infty$-category by Lemma \ref{Remrema}.

\end{theorem}

%We obtain the following theorem:

\begin{proof}
There is a chain of right $\mW$-linear equivalences, where $\kappa$-$\colim$ refers to  right $\mV$-enriched functors preserving $\kappa$-small colimits and tensors with $\kappa$-compact objects of $\mV$ and $\kappa$-$\lim$ refers to left $\mV$-enriched functors preserving $\kappa$-small limits and cotensors with $\kappa$-compact objects of $\mV$:
$$ \Enr\Fun^\R_{\mV, \emptyset}(\mM^\op,\mN)^\circledast \simeq (\LinFun^\L_{\emptyset, \mV}(\mM,\mN^\op)^\op)^\circledast \simeq (\Enr\Fun^{\kappa-\colim}_{\emptyset, \mV}(\mM^\kappa,\mN^\op)^\op)^\circledast $$$$\simeq \Enr\Fun^{\kappa-\lim}_{\mV, \emptyset}((\mM^\kappa)^\op,\mN)^\circledast \simeq (\Enr\Fun^{\kappa-\lim}_{\mV, \emptyset}((\mM^\kappa)^\op,\mV)
\ot_{\mV} \mN)^\circledast \simeq (\Ind_\kappa(\mM^\kappa)
\ot_{\mV} \mN)^\circledast.$$
The first and third equivalence is by Lemma \ref{leums}, the fourth equivalence is by Proposition \ref{leumorat}, the second and fifth equivalence is by Proposition \ref{krye}.

\end{proof}
\begin{lemma}\label{relllo} Let $\kappa,\tau$ be small regular cardinals, $\mM^\circledast \to \mU^\ot \times \mV^\ot$ a $\kappa$-compactly generated bitensored $\infty$-category, $\mN^\circledast \to \mV^\ot \times \mW^\ot$ a $\tau$-compactly generated bitensored $\infty$-category,
$ \mM^\circledast \rightleftarrows \mM'^\circledast$ an accessible $\mU,\mV$-enriched localization and $ \mN^\circledast \rightleftarrows \mN'^\circledast$ an accessible $\mV,\mW$-enriched localization.

\begin{enumerate}
\item The induced left adjoint $\mU, \mW$-linear functor $(\mM \ot_\mV \mN)^\circledast \to (\mM' \ot_\mV \mN')^\circledast$ is an accessible $\mU,\mW$-enriched localization.
%admits an accessible fully faithful $\mU,\mW$-enriched right adjoint.
	
\item If the accessible $\mU,\mV$-enriched localization $ \mM^\circledast \rightleftarrows \mM'^\circledast$ is with respect to the set of morphisms $\rS$ and the accessible $\mV,\mW$-enriched localization $ \mN^\circledast \rightleftarrows \mN'^\circledast$ is with respect to the set of morphisms $\T$, then the accessible $\mU,\mW$-enriched localization $(\mM \ot_\mV \mN)^\circledast \rightleftarrows (\mM' \ot_\mV \mN')^\circledast$ is with respect to the set of morphisms $\rS \boxtimes \T:= \{\f \ot \Y, \X \ot \g \mid \X \in \mM^\kappa, \Y \in \mN^\tau, \f \in \rS, \g \in \T\}.$
\end{enumerate}
\end{lemma}

\begin{proof}
(1): The $\mU,\mV$-enriched localization $\L: \mM^\circledast \rightleftarrows \mM'^\circledast: \bj$ gives rise to a localization $$\bj^*: \Enr\Fun_{\mV, \emptyset}(\mM^\op,\mN) \rightleftarrows \Enr\Fun_{\mV, \emptyset}(\mM'^\op,\mN): \L^*$$ whose unit at any $\F \in \Enr\Fun_{\mV, \emptyset}(\mM^\op,\mN)$ is the map $ \F \to \F \circ \bj \circ \L.$
Consequently, the right adjoint of the localization restricts along accessible embeddings to an accessible embedding $\L^*: \Enr\Fun^\R_{\mV, \emptyset}(\mM'^\op,\mN) \to \Enr\Fun^\R_{\mV, \emptyset}(\mM^\op,\mN) $ whose essential image precisely consists of the left $\mV$-enriched right adjoint functors inverting local equivalences for the localization $\L: \mM \rightleftarrows \mM':\bj .$

For (1) we need to see that the right adjoint $\gamma$ of the underlying functor $\mM \ot_\mV \mN \to \mM' \ot_\mV \mN'$ is accessible and fully faithful. By Theorem \ref{thei} this right adjoint $\gamma$ factors as accessible embeddings $$\mM' \ot_\mV \mN' \simeq \Enr\Fun^\R_{\mV, \emptyset}(\mM'^\op,\mN') \to \Enr\Fun^\R_{\mV, \emptyset}(\mM'^\op,\mN) \xrightarrow{\L^*} \Enr\Fun^\R_{\mV, \emptyset}(\mM^\op,\mN) \simeq \mM \ot_\mV \mN, $$ where the first accessible embedding postcomposes with the underlying left $\mV$-enriched embedding of the $\mV,\mW$-enriched right adjoint accessible embedding $\mN'^\circledast \to \mN^\circledast$.

(2): Since $\mM$ is generated by $\mM^\kappa$ under small $\kappa$-filtered colimits
and similar for $\mN$ and $\tau$, an object of $\mM \ot_\mV \mN$ is $\rS \boxtimes \T$-local if and only if it is local with respect to the morphisms $\f \ot \Y, \X \ot \g$ for $ \X \in \mM, \Y \in \mN, \f \in \rS, \g \in \T.$ The functor $\mM \times \mN \to \mM \ot_\mV \mN \simeq \Enr\Fun^\R_{\mV, \emptyset}(\mM^\op,\mN)$ sends $(\A,\B)$ to $\R\Mor_\mM(-,\A) \ot \B.$ 
Hence an object $\Z$ of $\mM \ot_\mV \mN$ is $\rS \boxtimes \T$-local if and only if 
the corresponding object $\Z'$ in $ \Enr\Fun^\R_{\mV, \emptyset}(\mM^\op,\mN)$ is local with respect to the morphisms $\R\Mor_\mM(-,\f) \ot \Y, \R\Mor_\mM(-,\X) \ot \g $ for $\f \in \rS, \g \in \T, \X \in \mM, \Y \in \mN.$ By the enriched Yoneda-lemma \cite[Corollary 5.23.]{heine2024bienriched} the latter condition is equivalent to ask that $\Z'$ inverts $\rS$ and lands in $\mN',$ which by description of $\gamma$ is equivalent to say that $\Z'$ belongs to $\Enr\Fun^\R_{\mV, \emptyset}(\mM'^\op,\mN')$ or equivalently, that $\Z$ belongs to $\mM' \ot_\mV \mN'.$

\end{proof}

\begin{proof}[Proof of Theorem \ref{prestensor}]
	
The equivalence between (1) and (2) follows from Theorem \ref{thei}.
The equivalence between (1) and (3) follows from the monoidality of the functor of $\mV$-enriched presheaves (Corollary \ref{Preml}), Lemma \ref{relllo} and \cite[Theorem 4.5.2.1.]{lurie.higheralgebra}.	
	
\end{proof}

\subsection{Enriched Day-convolution}

In the following we equip the $\infty$-category of enriched functors between symmetric monoidal enriched $\infty$-categories with a symmetric monoidal structure, which is an enriched version of Day-convolution as constructed in \cite[Proposition 2.16.]{heine2023topological}.

\begin{definition}
Let $\mV^\boxtimes \to \Comm$ be a presentably symmetric monoidal $\infty$-category, $\mH$ a collection of left weights over $\mV$ and $\mO^\ot \to \Comm$ a symmetric $\infty$-operad.
%Let $1 \leq \bk \leq \infty$. 
An $\mO$-monoidal left $\mV$-enriched $\infty$-category compatible with $\mH$-weighted colimits is an $\mO$-algebra in $_\mV\L\Enr_\emptyset(\mH)$.
%\item Let $\mH$ be a collection of right weights over $\mW$. An $\bE_{\bk}$-monoidal right $\mW$-enriched $\infty$-category compatible with $\mH$-weighted colimits is an $\bE_\bk$-algebra in $\R\Enr_\mW(\mH)$.
%\item Let $\mH$ be a collection of weights over $\mW,\mW.$+An $\bE_{\bk}$-monoidal $\mV, \mW$-bienriched $\infty$-category compatible with $\mH$-weighted colimits is an $\bE_\bk$-algebra in $_\mV\B\Enr_\mW(\mH)$.We refer to $\bE_1$-monoidal as monoidal. 
We refer to $\Comm$-monoidal as symmetric monoidal.
	
\end{definition}

\begin{example}\label{moex} An $\mO$-monoidal left $\mV$-enriched $\infty$-category compatible with small weighted colimits is an $\mO$-algebra with respect to the relative tensor product on $_\mV\cc\cc\LMod_\emptyset$. In particular, a symmetric monoidal left $\mV$-enriched $\infty$-category compatible with small weighted colimits is a $\Comm$-algebra with respect to the relative tensor product on $_\mV\cc\cc\LMod_\emptyset$, which by \cite[Corollary 3.4.1.7.]{lurie.higheralgebra} is canonically identified with a symmetric monoidal $\infty$-category compatible with small colimits $\mW^\boxtimes \to \Comm$ equipped with a small colimits preserving symmetric monoidal functor $\mV^\boxtimes \to \mW^\boxtimes$. \end{example}

\begin{corollary}\label{Dayco}
Let $\mV^\boxtimes \to \Comm$ be a presentably symmetric monoidal $\infty$-category, $\mH$ a collection of left weights over $\mV$ and $\mO^\ot \to \Comm$ a symmetric $\infty$-operad.
Let $\mN^\circledast \to \mV^\ot $ be an $\mO$-monoidal left $\mV$-enriched $\infty$-category compatible with small weighted colimits such that $\mN$ is presentable and $\mM^\circledast \to \mV^\ot $ an $\mO$-monoidal left $\mV$-enriched $\infty$-category compatible with $\mH$-weighted colimits.
Then $$\Enr\Fun^{\Lambda^\mH-\lim}_{\emptyset,\mV}(\mM^\op,\mN)^\circledast \to \mV^\ot$$ refines to an $\mO$-monoidal left $\mV$-enriched $\infty$-category compatible with
small weighted colimits.

\end{corollary}

\begin{proof}
	
The functor $\mP\L\Env_{\Lambda^\mH}(-)_{\L\Enr}: {_\mV\L\Enr}_\emptyset(\mH) \to {_\mV\cc\cc\LMod_\emptyset}$ is symmetric monoidal by Theorem \ref{Thor} and so equips $\mP\L\Env_{\Lambda^\mH}(\mM)_{\L\Enr}^\circledast \to \mV^\ot$ with an $\mO$-algebra structure in $ {_\mV\cc\cc\LMod_\emptyset}$.
By Example \ref{moex} we have that $\mN^\circledast \to \mV^\ot $ carries the structure of an $\mO$-algebra in $ {_\mV\cc\cc\LMod_\emptyset}$. By \cite[Theorem 4.5.2.1.]{lurie.higheralgebra} the tensor product of the symmetric monoidal $\infty$-category $ {_\mV\cc\cc\LMod_\emptyset}$ is the relative tensor product, which therefore lifts to $\mO$-algebras. Hence $ \Enr\Fun^{\Lambda^\mH-\lim}_{\emptyset,\mV}(\mM^\op,\mN)^\circledast \simeq (\mN\ot_\mV \mP\L\Env_{\Lambda^\mH}(\mM)_{\L\Enr})^\circledast \to \mV^\ot$
carries the structure of an $\mO$-algebra in ${_\mV \cc\cc\LMod_\emptyset},$
where the latter left $\mV$-enriched equivalence is by Theorem \ref{Days} and we view $\mN$ as bitensored over $\mV,\mV$ by restricting along the tensor product $\mV \times \mV \to \mV$
and Theorem \ref{biii}. 
	
\end{proof}

\begin{corollary}\label{Dayco2}
Let $\mV^\boxtimes \to \Comm$ be a presentably symmetric monoidal $\infty$-category, $\mH$ a collection of left weights over $\mV$ and $\mO^\ot \to \Comm$ a symmetric $\infty$-operad.
Let $\mN^\circledast \to \mV^\ot $ be an $\mO$-monoidal left $\mV$-enriched $\infty$-category compatible with small weighted colimits such that $\mN$ is presentable and $\mM^\circledast \to \mV^\ot $ an $\mO$-monoidal left $\mV$-enriched $\infty$-category.
Then $$\Enr\Fun_{\emptyset,\mV}(\mM^\op,\mN)^\circledast \to \mV^\ot$$ refines to an $\mO$-monoidal left $\mV$-enriched $\infty$-category compatible with
small weighted colimits.
	
\end{corollary}

\begin{remark}Let the assumptions like in Corollary \ref{Dayco2}.
%Let $\mV^\boxtimes \to \bE_{\bk+1}, \mW^\boxtimes \to \bE_{\bk+1} $be presentably $\bE_{\bk+1}$-monoidal $\infty$-categories for $1 \leq \bk \leq \infty$ and $\mN^\circledast \to \mV^\ot \times \mW^\ot$ an  $\bE_{\bk}$-monoidal $\mV,\mW$-bienriched $\infty$-category compatible with small weighted colimitssuch that $\mN$ is presentable. Let $\mM^\circledast \to \mW^\ot $ be an $\bE_\bk$-monoidal left $\mW$-enriched $\infty$-category compatible with $\mH$-weighted colimits.
The tensor product $$\Enr\Fun_{\emptyset,\mW}(\mM^\op,\mN) \times \Enr\Fun_{\emptyset,\mW}(\mM^\op,\mN) \to \Enr\Fun_{\emptyset,\mW}(\mM^\op,\mN)$$ 
of the left $\mV$-enriched Day-convolution factors as
$$\Enr\Fun_{\emptyset,\mW}(\mM^\op,\mN) \times \Enr\Fun_{\emptyset,\mW}(\mM^\op,\mN) \to\Enr\Fun_{\emptyset,\mW}(\mM^\op,\mN) \otimes \Enr\Fun_{\emptyset,\mW}(\mM^\op,\mN) $$$$\simeq \mN\ot_\mW \mP\L\Env_{\Lambda^\mH}(\mM)_{\L\Enr} \otimes \mN\ot_\mW \mP\L\Env(\mM)_{\L\Enr} $$$$\simeq(\mN\ot \mN)\ot_{\mW\ot \mW} (\mP\L\Env(\mM)_{\L\Enr} \ot\mP\L\Env(\mM)_{\L\Enr}) $$
$$\simeq (\mN\ot \mN)\ot_{\mW\ot \mW} \mP\L\Env_{}(\mM\ot \mM)_{\L\Enr}$$
$$\to \mN\ot_\mW \mP\L\Env_{}(\mM)_{\L\Enr} \simeq \Enr\Fun_{\emptyset,\mW}(\mM^\op,\mN),$$
which identifies with the functor
$$\Enr\Fun_{\emptyset,\mW}(\mM^\op,\mN) \times \Enr\Fun_{\emptyset,\mW}(\mM^\op,\mN) \to \Enr\Fun_{\emptyset,\mW}(\mM^\op \otimes \mM^\op,\mN \otimes \mN)$$$$ \xrightarrow{(\ot_\mN)_*}\Enr\Fun_{\emptyset,\mW}(\mM^\op \otimes \mM^\op,\mN) \xrightarrow{(\ot_\mM)_!} \Enr\Fun_{\emptyset,\mW}(\mM^\op,\mN).$$
%where the last functor 
	
\end{remark}

\subsection{A tensor product for $(\infty, \n)$-categories with lax colimits}

%\subsubsection{Lax and oplax colimits}

In the following we apply our theory to $(\infty,\n)$-categories for $\n \geq 1$, which we define inductively as $\Cat_{(\infty,\n-1)}$-enriched $\infty$-categories, to define lax and oplax colimits
and deduce the existence of a presentably monoidal structure on the $\infty$-category of small
$(\infty, \n)$-categories that admit $\mK$-indexed (op)lax colimits for any set $\mK$ of $(\infty,\n-1)$-categories (Corollary \ref{laxtens}).
Moreover we construct a presentably monoidal structure on the $\infty$-category of small
$(\infty, 2)$-categories that admit (op)lax pushouts (Corollary \ref{pushten}).

\begin{definition}Let $\n \geq 0$.
Let $\Cat_{(\infty,0)}:= \mS, \Cat_{(\infty,\n+1)}:=  {_{\Cat_{(\infty,\n)}}\L\Enr},$
where we view $\Cat_{(\infty,\n)}$ as endowed with the cartesian structure and inductively assume that $\Cat_{(\infty,\n)}$ is a presentable cartesian closed $\infty$-category (Corollary \ref{Qa}).
\end{definition}
\begin{construction}Let $\n \geq 1$.
By \cite[Theorem 6.20.]{nuiten2023straightening} for every small $(\infty,\n)$-category $\mC$ there is an $(\infty,\n)$-categorical Grothendieck construction, which can be viewed as a functor
of $(\infty,\n)$-categories $$\theta: \Enr\Fun_{\Cat_{(\infty,\n-1)}, \emptyset}(\mC^\op, \Cat_{(\infty,\n-1)})^\circledast \to \Cat^\circledast_{(\infty,\n)}$$
that sends the final object to $\mC$ and sends a functor of $(\infty,\n)$-categories $\F: \mC^\op \to \Cat_{(\infty,\n-1)}$ to a functor of $(\infty,\n)$-categories $\mB \to \mC$ whose fiber over $\X \in \mC$ is $\F(\X).$
%where the left hand side is the internal hom in $\widehat{\Cat}_{(\infty,\n+1)}.$%\Mor_{\widehat{\Cat}_{(\infty,\n+1)}}(
%The $(\infty,\n+1)$-category $\mC$, i.e. $\Cat_{(\infty,\n)}$-enriched and so gives rise to a  $\Cat_{(\infty,\n)}$-enriched morphism object functor $\Mor_\mC: \mC^\op \times \mC \to \Cat_{(\infty,\n)} $ identified with a functor of $(\infty,\n+1)$-categories, adjoint to a functor of $(\infty,\n+1)$-categories$\mC^\op \to \Mor_{\widehat{\Cat}_{(\infty,\n+1)}}(\mC, \Cat_{(\infty,\n)}).$
%We obtain a functor of $(\infty,\n)$-categories: %, a left $\Cat_{(\infty,\n-1)}$-enriched functor:
Let $\iota$ be the $\Cat_{(\infty,\n-1)}$-enriched right adjoint of the embedding $\Cat_{(\infty,\n-1)} \subset \Cat_{(\infty,\n)}$, which discards all non-invertible cells of degree lower $\n.$
Let $$\kappa^\mC:= \iota\circ \theta \circ \y: \mC^\circledast \to \Enr\Fun_{\Cat_{(\infty,\n-1)}}(\mC^\op, \Cat_{(\infty,\n-1)})^\circledast \to \Cat_{(\infty,\n)}^\circledast \to \Cat_{(\infty,\n-1)}^\circledast$$
be the functor of $(\infty,\n)$-categories, where the first functor is the $\Cat_{(\infty,\n-1)}$-enriched Yoneda-embedding.

The functor $\kappa^\mC$ sends $\X \in \mC$ to a functor $\iota(\theta(\mC)) \to \iota(\mC)$ whose fiber over any $\Y \in \mC$ is $\Mor_\mC(\Y,\X).$ 

\end{construction}

%\begin{remark}Let $\n \geq 1.$The $\infty$-category $ \Cat_{(\infty,2)}$ carries $4$ canonical involutions:the $2$ involutions of $ \Cat_\infty$ give rise to $2$ involutions of $\Cat_{(\infty,2)}$ byapplying % Composing each of these involutions with theopposite $\Cat_{(\infty,\n-1)}$-enriched $\infty$-category involution we obtain $2^{\n-1}$-many further involutions of $\Cat_{(\infty,\n)}=  {_{\Cat_{(\infty,\n-1)}} \L\Enr}$ and so in total $2^\n$-many involutions of $ {_{\Cat_{(\infty,\n-1)}} \L\Enr} $.\end{remark}

\begin{construction}\label{tuti}
Let $\n \geq 1.$ By induction we construct $2^\n$-many involutions of the $\infty$-category $\Cat_{(\infty, \n)}.$
For $\n=0$ the identity is the only auto-equivalence of the $\infty$-category $\mS.$
Assume we have already constructed $2^\n$-involutions of the $\infty$-category $\Cat_{(\infty, \n)}$ and let $\tau$ be any of them. Then $\tau$ preserves finite products and so is symmetric monoidal for the cartesian structures and so induces two involutions $ \tau_!, (-)^\op \circ \tau_! \simeq \tau_! \circ (-)^\op$
of $\Cat_{(\infty,\n+1)}=  {_{\Cat_{(\infty,\n)}}\L\Enr},$
where $(-)^\op$ is the involution of $ {_{\Cat_{(\infty,\n)}}\L\Enr}$
that assigns the opposite $\Cat_{(\infty,\n)}$-enriched $\infty$-category.
In fact one can prove that these involutions of $\Cat_{(\infty, \n)}$ are the only auto-equivalences of $\Cat_{(\infty, \n)}$ but we don't need this here.

\end{construction}

\begin{definition}Let $\n \geq 1$ and $\mC \in \Cat_{(\infty,\n)}$. % and $\tau$ one of the involutions of $\Cat_{(\infty, \n-1)}$ of Construction \ref{tuti}.
The lax weight on $\mC$ is the functor of $(\infty,\n)$-categories $$\kappa^\mC: \mC \to \Cat_{(\infty,\n-1)}.$$

\end{definition}

\begin{definition}\label{Laax}Let $\n \geq 1$ and $\F: \mC \to \mD$ a functor of $(\infty,\n)$-categories and $\tau$ one of the involutions of $\Cat_{(\infty, \n-1)}$ of Construction \ref{tuti}.

\begin{enumerate}
\item The lax colimit of $\F$, denoted by $\colim^{\lax}(\F) $, is the colimit of $\F$ weighted at the lax weight $\kappa^{\mC^\op} $. % on $\mC^\op.$

\item The oplax colimit of $\F$, denoted by $\colim^{\op\lax}(\F) $, is the colimit of $\F$ weighted at $\tau \circ \kappa^{\mC^\op}$, where $\tau$ is the involution on $\Cat_{(\infty, \n-1)}$ that reverses cells of odd dimension.

\item The lax limit of $\F$, denoted by $\lim^{\lax}(\F) $, is the limit of $\F$ weighted at $\tau \circ \kappa^{\mC}$, where $\tau$ is the involution on $\Cat_{(\infty, \n-1)}$ that reverses cells of even dimension.

\item The oplax limit of $\F$, denoted by $\lim^{\op\lax}(\F) $, is the limit of $\F$ weighted at $\tau \circ \kappa^{\mC}$, where $\tau$ is the involution on $\Cat_{(\infty, \n-1)}$ that reverses cells of any dimension.	

%limit of $\F$ weighted at the oplax weight on $\mC$, i.e. the oplax colimit of $\F^\op$.

\end{enumerate}

%For $\tau$ the identity we call the $\tau$-lax (co)limit of $\F$ the lax (co)limit of $\F$.

\end{definition}

%\begin{definition}Let $\n \ge 2$ and $\tau$ one of the $2^{\n-1}$-many involutions of $ \Cat_{(\infty,\n-1)}$giving rise to an involution $\tau_*$ of $\Cat_{(\infty,\n)}= \Cat_\infty^{\Cat_{(\infty,\n-1)}}.$The $\tau$-lax (co)limit of $\F$ is the lax (co)limit of $\tau_*(\F).$\end{definition}

%\begin{example}We have 2-involutions on $\Cat_\infty$ denoted by $\id, (-)^\op$and set $(-)^\mathrm{co}:=(-)^\op_*: \Cat_{(\infty,2)}=\Cat_\infty^{\Cat_\infty} \to \Cat_{(\infty,2)}.$
%we denote the four involutions by $\id, (-)^\op, (-)^\mathrm{co}, (-)^{\mathrm{coop}}$, where $(-)^\op$ reverses the morphisms, $ (-)^\mathrm{co}$ reverses the 2-morphisms and $(-)^{\mathrm{coop}}$ reverses both.
%The oplax (co)limit of a 2-functor $\F: \mC \to \mD$ is the lax (co)limit of $\F^\mathrm{co}$.
%the lax limit of $\F$ is the lax colimit of $\F^\op$,the oplax limit of $\F$ is the lax colimit of $\F^\mathrm{coop}$.

%For general $\n \geq 2$ we obtain $2^{\n-1}$ many variants of colimits and$2^{\n-1}$ many variants of limits.\end{example}

\begin{remark}
The oplax colimit of $\F$ is likewise the lax colimit of $\tau(\F)$ for $\tau$ the involution on $\Cat_{(\infty, \n)}$ that reverses cells of even dimension.
The lax limit of $\F$ is likewise the lax colimit of $\tau(\F)$ for $\tau$ the involution on $\Cat_{(\infty, \n)}$ that reverses cells of odd dimension.
The oplax limit of $\F$ is likewise the lax colimit of $\tau(\F)$ for $\tau$ the involution on $\Cat_{(\infty, \n)}$ that reverses cells of any dimension.

\end{remark}

\begin{remark}

In \cite[Definition 2.9.]{articles} Gepner-Haugseng define the lax colimit and lax limit of a functor $\mC^\op \to \Cat_\infty $ under the assumption that $\mC$ is an $(\infty,1)$-category.

\end{remark}

%The authors show that the lax colimit and lax limit of $\phi$ is easily expressed in terms of the cartesian fibration $\alpha: \X \to \rS$ classifying $\phi.$They prove that the lax colimit of $\phi$ is $\X$ \cite[Theorem 7.4., Corollary 7.6.]{articles} and the oplax limit of $\phi$ is the $\infty$-category $\Fun_\rS(\rS,\X)$ of sections of $\alpha$ \cite[Proposition 7.1.]{articles}.By the next proposition this also holds for our definition of lax colimits and lax limits and so identifies our definition of lax (co)limits with the one of \cite{articles}.

\begin{notation}Let $\n \geq 1$ and $\mK \subset \Cat_{(\infty,\n)}$ a full subcategory
and $\tau$ the involution on $\Cat_{(\infty, \n-1)}$ that reverses cells of odd dimension.
Let
$$ \Cat_{(\infty,\n)}^{\mathrm{laxcolim}}(\mK):=  {_{\Cat_{(\infty,\n-1)}} \L\Enr}(\{\kappa^{\mC^\op} \mid \mC \in \Cat_{(\infty,\n-1)}\}) \subset \Cat_{(\infty,\n)} =  {_{\Cat_{(\infty,\n-1)}} \L\Enr} $$
$$ \Cat_{(\infty,\n)}^{\mathrm{oplaxcolim}}(\mK):=  {_{\Cat_{(\infty,\n-1)}} \L\Enr}(\{\tau \circ \kappa^{\mC^\op}\mid \mC \in \Cat_{(\infty,\n-1)}\}) \subset \Cat_{(\infty,\n)} =  {_{\Cat_{(\infty,\n-1)}} \L\Enr}$$
be the subcategories of small $(\infty,\n)$-categories that admit $\mK$-indexed (op)lax colimits and functors preserving $\mK$-indexed (op)lax colimits. \end{notation}

Theorem \ref{thqaz} specializes to the following:

\begin{corollary}\label{laxtens} Let $\n \geq 1$ and $\mK \subset \Cat_{(\infty,\n)}$ a full subcategory.
The $\infty$-categories \begin{equation}\label{Imk}
\Cat_{(\infty,\n)}^{\mathrm{laxcolim}}(\mK), \Cat_{(\infty,\n)}^{\op\mathrm{laxcolim}}(\mK)\subset \Cat_{(\infty,\n)}\end{equation}
carry presentably symmetric monoidal structures such that the inclusions (\ref{Imk}) are lax symmetric monoidal, where the right hand sides carry the cartesian structures. \end{corollary}

%Next we consider examples present in $(\infty,2)$-category theory.

\begin{example}
Let $\n \geq 1$. %Let $\Lambda^2_0$ be the category $\{0,1\}^\vartriangleleft.$
The functor $\kappa^{\{0,1\}^\vartriangleright}: \{0,1\}^\vartriangleright \to \Cat_{(\infty,\n-1)}$ corresponds to the diagram $ \{0\} \rightarrow \{0,1\}^\vartriangleright \leftarrow \{1\}.$ For every $(\infty, \n)$-category $\mD$, functor $\F: \{0,1\}^\vartriangleleft \to \mD$ corresponding to a diagram $ \X \leftarrow \Z \rightarrow \Y$ in $\mD$
and %the lax colimit of $\F$ %, i.e. the $\kappa$-weighted colimit of $\F$, 
every $\T \in \mD$ there is a canonical equivalence:
$$ \mD(\colim^\lax(\F), \T) \simeq \Fun((\{0,1\}^\vartriangleleft)^\op, \Cat_{(\infty,\n-1)})(\kappa^{\{0,1\}^\vartriangleright}, \Mor_\mD(\F(-),\T)) \simeq $$$$ \mD(\X,\T)\times_{\mD(\Z,\T)^{\{1\}}} \Cat_{(\infty,\n-1)}([1], \mD(\Z,\T))\times_{\mD(\Z,\T)^{\{0\}}} \Cat_{(\infty,\n-1)}([1], \mD(\Z,\T)) \times_{\mD(\Z,\T)^{\{1\}}} \mD(\Y,\T).$$

So if $\mD$ admits pushouts and tensors with $[1],$ the latter equivalence represents an equivalence $$\colim^{\lax}(\F) \simeq \X \coprod_{\{1\}\ot\Z} ([1] \ot \Z) \coprod_{\{0\}\ot\Z} ([1] \ot \Z) \coprod_{\{1\}\ot\Z} \Y.$$
Dually, if $\mD$ admits pullbacks and cotensors with $[1],$ the lax limit of a functor $\G: \{0,1\}^\vartriangleright \to \mD$ corresponding to a diagram $ \X \rightarrow \Z \leftarrow \Y$ in $\mD$
is $ \X \prod_{\Z^{\{1\}}} \Z^{[1]} \prod_{\Z^{\{0\}}} \Z^{[1]} \prod_{\Z^{\{1\}}} \Y.$

\end{example}

\begin{definition}\label{Puush}
Let $\n \geq 1$ and $\sigma: \{0,1\}^\vartriangleright \to \Cat_{(\infty,\n-1)}$ the functor
corresponding to the diagram $ \{0\} \rightarrow [1] \leftarrow \{1\}$
and $\mD$ an $(\infty, \n)$-category.
\begin{itemize}
\item Let $\F: \{0,1\}^\vartriangleleft \to \mD$ be a functor corresponding to a diagram $ \X \leftarrow \Z \rightarrow \Y$ in $\mD$. The lax pushout of the diagram $ \X \leftarrow \Z \rightarrow \Y$ in $\mD$, denoted by $ \X \coprod^\lax_\Z \Y $, is the $\sigma$-weighted colimit of $\F$.

\item Let $\F: \{0,1\}^\vartriangleleft \to \mD$ be a functor corresponding to a diagram $ \X \leftarrow \Z \rightarrow \Y$ in $\mD$. The oplax pushout of the diagram $ \X \leftarrow \Z \rightarrow \Y$ in $\mD$, denoted by $ \X \coprod^\oplax_\Z \Y $, is the $(-)^\op \circ \sigma$-weighted colimit of $\F$.

\item Let $\G: \{0,1\}^\vartriangleright \to \mD$ be a functor corresponding to a diagram $ \X \rightarrow \Z \leftarrow \Y$ in $\mD$.
The lax pullback of the diagram $ \X \rightarrow \Z \leftarrow \Y$ in $\mD$, denoted by $ \X \times^\lax_\Z \Y $, is the $\sigma$-weighted limit of $\G$.

\item Let $\G: \{0,1\}^\vartriangleright \to \mD$ be a functor corresponding to a diagram $ \X \rightarrow \Z \leftarrow \Y$ in $\mD$.
The oplax pullback of the diagram $ \X \rightarrow \Z \leftarrow \Y$ in $\mD$, denoted by $ \X \times^\oplax_\Z \Y $, is the $(-)^\op \circ \sigma$-weighted limit of $\G$.

\end{itemize}

\end{definition}

\begin{remark}\label{reö}

Let $\mB$ be an $\infty$-category and $\gamma, \gamma': \{0,1\}^\vartriangleright \to \mB$ be functors corresponding to diagrams $ \A \to \C \leftarrow \B, \A' \to \C' \leftarrow \B',$ respectively.
There is a canonical equivalence
$$\Fun(\{0,1\}^\vartriangleright, \mB)(\gamma, \gamma') \simeq  \mB(\A,\A')\times_{\mB(\A,\C')} \mB(\C, \C') \times_{\mB(\B,\C')} \mB(\B,\B'). $$

\end{remark}

\begin{example}Let $\n \geq1$ and $\mD$ an $(\infty, \n)$-category. Let $\F: \{0,1\}^\vartriangleleft \to \mD$ be a functor corresponding to a diagram $ \X \leftarrow \Z \rightarrow \Y$ in $\mD$.
By Remark \ref{reö} there is a canonical equivalence for every $\T \in \mD$:
$$ \mD(\X \coprod^\lax_\Z \Y, \T) \simeq \Fun(\{0,1\}^\vartriangleright, \Cat_{(\infty,\n-1)}(\sigma, \Mor_\mD(\F(-),\T))\simeq$$
$$ \Cat_{(\infty,\n-1)}([0],\Mor_\mD(\X,\T))\times_{\Cat_{(\infty,\n-1)}([0],\Mor_\mD(\Z,\T))} \Cat_{(\infty,\n-1)}([1], \Mor_\mD(\Z,\T))$$$$ \times_{\Cat_{(\infty,\n-1)}([0],\Mor_\mD(\Z,\T))} \Cat_{(\infty,\n-1)}([0],\Mor_\mD(\Y,\T)) \simeq $$
$$\mD(\X,\T)\times_{\mD(\Z,\T)}\Cat_{(\infty,\n-1)}([1], \Mor_\mD(\Z,\T)) \times_{\mD(\Z,\T)} \mD(\Y,\T).$$

So if $\mD$ admits pushouts and tensors with $[1]$, we find that $$\X \coprod^\lax_\Z \Y \simeq \X \coprod_{\{0\}\times\Z} ([1] \ot \Z) \coprod_{\{1\}\times\Z} \Y.$$

Similarly, if $\mD$ admits pushouts and tensors with $[1]$, we find that $$\X \coprod^\oplax_\Z \Y \simeq \X \coprod_{\{1\}\times\Z} ([1] \ot \Z) \coprod_{\{0\}\times\Z} \Y.$$
Dually, if $\mD$ admits pullbacks and cotensors with $[1],$ the
lax pullback of the diagram $ \X \rightarrow \Z \leftarrow \Y$ in $\mD$
is $\X \times^\lax_\Z \Y \simeq \X \prod_{\Z^{\{0\}}} \Z^{[1]} \prod_{\Z^{\{1\}}} \Y$
and the oplax pullback %of the diagram $ \X \rightarrow \Z \leftarrow \Y$ in $\mD$
is $\X \times^\oplax_\Z \Y \simeq \X \prod_{\Z^{\{1\}}} \Z^{[1]} \prod_{\Z^{\{0\}}} \Y.$

\end{example} 

\begin{notation}Let $\n \geq 1$.
Let $ \Cat_{(\infty,\n)}^{\lax \ \mathrm{push}}, \Cat_{(\infty,\n)}^{\oplax \ \mathrm{push}} \subset \Cat_{(\infty,\n)} $ be the subcategories of $(\infty, \n)$-categories that admit (op)lax pushouts and functors preserving (op)lax pushouts.

\end{notation}

\begin{corollary}\label{pushten}Let $\n \geq 1$. The $\infty$-categories $ \Cat_{(\infty,\n)}^{\lax \ \mathrm{push}}, \Cat_{(\infty,\n)}^{\oplax \ \mathrm{push}} $ carry presentably symmetric monoidal structures such that the inclusions $\Cat_{(\infty,\n)}^{\lax \ \mathrm{push}}\subset \Cat_{(\infty,\n)} , \Cat_{(\infty,\n)}^{\oplax \ \mathrm{push}} \subset \Cat_{(\infty,\n)} $ are lax symmetric monoidal, where the right hand sides carry the cartesian structures. \end{corollary}

\subsection{A tensor product for Cauchy-complete enriched $\infty$-categories}

%\subsection{Absolute weights and Cauchy completion}

\begin{definition}Let $\mV^\ot \to \Ass, \mW^\ot \to \Ass$ be presentably monoidal $\infty$-categories and $\mJ^\circledast \to \mV^\ot \times \mW^\ot$ a bienriched $\infty$-category.
An enriched weight on $\mJ$ is absolute if it is preserved by any $\mV,\mW$-enriched functor between $\infty$-categories bi-quasi-enriched in $\mV,\mW.$

\end{definition}

\begin{example}

For $\mV=\mS$ the $\infty$-category of spaces the $\infty$-category $\mathrm{Idem}$ classifying idempotents of \cite[Definition 4.4.5.2.]{lurie.HTT}
equipped with the trivial left $\mS$-enriched weight is an absolute left $\mS$-enriched weight.

%For $\mV=\Sp$ the $\infty$-category of spectra every finite $\infty$-category equipped with the trivial left $\Sp$-enriched weight is an absolute left $\Sp$-enriched weight:this holds since every left $\Sp$-enriched functor preserves finite colimits.

\end{example}

\begin{notation}\label{CC} Let $\mV^\ot \to \Ass, \mW^\ot \to \Ass$ be presentably monoidal $\infty$-categories and $\mJ^\circledast \to \mV^\ot \times \mW^\ot$ a bienriched $\infty$-category. The Cauchy-completion $\widehat{\mM}^\circledast \to \mV^\ot \times \mW^\ot $ of $\mM^\circledast \to \mV^\ot \times \mW^\ot$ is the full bienriched subcategory of $\mP\B\Env(\mM)^\circledast_{\B\Enr} \to \mV^\ot \times \mW^\ot$ 
spanned by the absolute enriched weights on $\mM.$

%$\mV^\ot \to \Ass, \mW^\ot \to \Ass$ be monoidal $\infty$-categories compatible with small colimits and $\mM^\circledast \to \mV^\ot \times \mW^\ot$ a (not necessarily small) bi-quasi-enriched $\infty$-category.The Cauchy-completion $\widehat{\mM}^\circledast \to \mV^\ot \times \mW^\ot $ of $\mM^\circledast \to \mV^\ot \times \mW^\ot$ is the full bi-quasi-enriched subcategory of $ \mP\B\Env(\mM)^\circledast_{\B\Enr} \to \mV^\ot \times \mW^\ot$ spanned by the absolute bi-quasi-enriched weights on $\mM.$
\end{notation}

\begin{remark}\label{rrro}

Let $\F: \mM^\circledast \to \mN^\circledast$ be a $\mV,\mW$-enriched functor between bienriched $\infty$-categories. The functor $\F_!: \mP\B\Env(\mM)_{\B\Enr} \to \mP\B\Env(\mN)_{\B\Enr} $ sends $\widehat{\mM}$ to $\widehat{\mN}$
since for every absolute enriched weight $\rH$ on $\mM$ the $\F_!(\rH)$-weighted colimit of a $\mV,\mW$-enriched functor $\phi: \mN^\circledast \to \mO^\circledast$ between bienriched $\infty$-categories is the $\rH$-weighted colimit of $ \phi \circ \F$ as a consequence of Corollary \ref{obsto}.
\end{remark}

\begin{definition}Let $\mV^\ot \to \Ass, \mW^\ot \to \Ass$ be presentably monoidal $\infty$-categories. A bienriched $\infty$-category $\mM^\circledast \to \mV^\ot \times \mW^\ot$ is Cauchy-complete if it admits the colimit of any $\mV,\mW$-enriched functor
weighted with respect to any absolute enriched weight. 

\end{definition}

\begin{proposition}\label{Chara}Let $\mV^\ot \to \Ass, \mW^\ot \to \Ass$ be presentably monoidal $\infty$-categories, $\mJ^\circledast \to \mV^\ot \times \mW^\ot$ a bienriched $\infty$-category and $\rH$ an enriched weight on $\mJ.$ The following conditions are equivalent:
\begin{enumerate}

\item The $\mV, \mW$-enriched functor
$$\Mor_{\mP\B\Env(\mJ)_{\B\Enr}}(\rH,-): \mP\B\Env(\mJ)_{\B\Enr}^\circledast \to (\mV\ot \mW)^\circledast$$ is linear and preserves small colimits.

\item The enriched weight $\rH$ on $\mJ$ is absolute.

\item The $\rH$-weighted colimit of the $\mV,\mW$-enriched embedding $\mJ^\circledast \subset \mM^\circledast:=\mP\B\Env(\mJ)_{\B\Enr}^\circledast $ is preserved by the $\mV,\mW$-enriched embedding $\mM^\circledast \subset \mP\B\Env(\mM)_{\B\Enr}^\circledast$.
% of large bi-quasi-enriched $\infty$-categories.

\end{enumerate}

\end{proposition}

\begin{proof}

Assume that (1) holds. Let $\F: \mJ^\circledast \to \mM^\circledast, \G: \mM^\circledast \to \mN^\circledast$ be $\mV,\mW$-enriched functors between bi-quasi-enriched $\infty$-categories. We like to see that $\G$ preserves the $\rH$-weighted colimit of $\F$ if the latter exists.
The universal morphism $\rH \to \F^*(\colim^\rH(\F))$ gives rise to a morphism
$$\theta: \mP\B\Env(\mM)_{\B\Enr}(\colim^\rH(\F),-) \to \mP\B\Env(\mJ)_{\B\Enr}(\rH,-) \circ \F^*$$
in $\LinFun^\L_{\mV,\mW}(\mP\B\Env(\mM)_{\B\Enr},\mV\ot \mW)$
since $\colim^\rH(\F) \in \mM$ and by \cite[Theorem 4.50.]{heine2024bienriched} and assumption (1) and Proposition \ref{pseuso}.
By universal property of the weighted colimit for every $\X \in \mM$ the morphism $\theta_\X$ 
is an equivalence so that $\theta$ is an equivalence by Corollary \ref{envvcor}.
To see that the canonical morphism $\colim^\rH(\G\circ \F) \to \G(\colim^\rH(\F))$ in $\mN$ is an equivalence, we need to see that for every $\Y \in \mN$ the canonical map $$\mP\B\Env(\mJ)_{\B\Enr}(\rH,\F^*(\G^*(\Y)))\simeq \mP\B\Env(\mM)_{\B\Enr}(\colim^\rH(\F),\G^*(\Y)) \simeq$$$$ \mP\B\Env(\mN)_{\B\Enr}(\G_!(\colim^\rH(\F)),\Y) \simeq \mN(\G(\colim^\rH(\F)),\Y) \to \mN(\colim^\rH(\G\circ \F), \Y) \simeq $$$$ \mP\B\Env(\mJ)_{\B\Enr}(\rH,(\G\circ\F)^*(\Y)) \simeq \mP\B\Env(\mJ)_{\B\Enr}(\rH,\F^*(\G^*(\Y)))$$ is an equivalence.
The latter morphism identifies with the identity. This proves (2).

Condition (2) trivially implies (3). Assume that (3) holds. 
The $\rH$-weighted colimit of the $\mV,\mW$-enriched embedding $\mJ^\circledast \subset \mM^\circledast:= \mP\B\Env(\mJ)_{\B\Enr}^\circledast $ is $\rH$. By assumption this $\rH$-weighted colimit is preserved by the $\mV,\mW$-enriched embedding $\mM^\circledast \subset  \mP\B\Env(\mM)_{\B\Enr}^\circledast.$
Thus the image of $\rH\in \mP\B\Env(\mJ)_{\B\Enr}=\mM$ in $\mP\B\Env(\mM)_{\B\Enr}$
is the $\rH$-weighted colimit of the composition $\mJ^\circledast \subset  \mP\B\Env(\mJ)_{\B\Enr}^\circledast=\mM^\circledast \subset  \mP\B\Env(\mM)_{\B\Enr}^\circledast.$
The latter composition factors as $\mJ^\circledast \subset  \mP\B\Env(\mJ)_{\B\Enr}^\circledast \xrightarrow{\theta_!}  \mP\B\Env(\mM)_{\B\Enr}^\circledast,$ where the latter $\mV,\mW$-enriched embedding is induced by the $\mV,\mW$-enriched embedding $\theta: \mJ^\circledast \subset  \mP\B\Env(\mJ)_{\B\Enr}^\circledast=\mM^\circledast.$
Since $\theta_!$ preserves weighted colimits, the image $\rH'$ of $\rH\in \mP\B\Env(\mJ)_{\B\Enr}=\mM$ in $ \mP\B\Env(\mM)_{\B\Enr}$
is the image of $\rH \in \mP\B\Env(\mJ)_{\B\Enr}$ under $\theta_!.$
Thus there is a canonical equivalence
$$ \Mor_{\mP\B\Env(\mJ)_{\B\Enr}}(\rH,-) \simeq \Mor_{ \mP\B\Env(\mM)_{\B\Enr}}(\theta_!(\rH),-) \circ \theta_! \simeq \Mor_{ \mP\B\Env(\mM)_{\B\Enr}}(\rH',-) \circ \theta_!  $$
of $\mV,\mW$-enriched functors
$ \mP\B\Env(\mJ)_{\B\Enr}^\circledast \to (\mV\ot \mW)^\circledast$.
%\rS^{-1}\mP\Env(\mV), \T^{-1}\mP\Env(\mW)where $\sigma$ is the strongly inaccessible cardinal corresponding to the small universe.
%Let $\F: \mJ^\circledast \subset \mM^\circledast:= \mP\B\Env(\mJ)^\circledast $ and $\iota: \mM^\circledast \subset \mN^\circledast:= \widehat{\mP}\mB\Env(\mM)^\circledast$ be the canonical embeddings. Then for every $\Y \in \mN$ the canonical map $$ \widehat{\mP}\B\Env(\mM)(\colim^\rH(\F), \Y) \simeq \widehat{\mP}\B\Env(\mM)(\colim^\rH(\F),\iota^*(\Y)) \simeq \mP\B\Env(\mN)(\iota_!(\colim^\rH(\F)),\Y) \simeq$$$$ \mN(\iota(\colim^\rH(\F)),\Y) \to \mN(\colim^\rH(\iota\circ \F), \Y) \simeq \mP\B\Env(\mJ)(\rH,(\iota\circ\F)^*(\Y)) \simeq$$$$ \mP\B\Env(\mJ)(\rH,\F^*(\iota^*(\Y)))\simeq \mP\B\Env(\mJ)(\rH,\F^*(\Y)) $$ is an equivalence.The resulting equivalence $$ \mP\B\Env(\mM)(\colim^\rH(\F),-) \simeq \mP\B\Env(\mJ)(\rH,-) \circ \F^* $$ gives rise to an equivalence $$ \mP\B\Env(\mM)(\colim^\rH(\F),-) \circ \F_! \simeq \mP\B\Env(\mJ)(\rH,-) \circ \F^*\circ \F_! \simeq \mP\B\Env(\mJ)(\rH,-).$$
Thus (1) follows.	

\end{proof}

\begin{corollary}
Let $\mV^\ot \to \Ass, \mW^\ot \to \Ass$ be presentably monoidal $\infty$-categories and $\mM^\circledast \to \mV^\ot \times \mW^\ot$ a small bienriched $\infty$-category. The Cauchy-completion $\widehat{\mM}^\circledast \to \mV^\ot \times \mW^\ot $ of $\mM^\circledast \to \mV^\ot \times \mW^\ot$ is small.
	
\end{corollary}

\begin{proof}Since $\mV^\ot \mW$ is presentable, the tensor unit of $\mV \ot \mW$ is $\kappa$-compact for some small regular cardinal $\kappa.$
Let $\rH \in \widehat{\mM}.$ By Proposition \ref{Chara} the $\mV, \mW$-enriched functor
$\Mor_{\mP\B\Env(\mJ)_{\B\Enr}}(\rH,-): \mP\B\Env(\mJ)_{\B\Enr}^\circledast \to (\mV\ot \mW)^\circledast$ preserves small colimits and therefore small $\kappa$-filtered colimits.
Hence the left adjoint of the latter functor, which sends $\V \ot \W$ to $\V \ot \rH \ot \W$, preserves $\kappa$-compact objects. Thus $\rH$ is $\kappa$-compact.
%The Cauchy-completion is the full weakly bienriched subcategory of $\mP\B\Env(\mM)^\circledast_{\B\Enr}$ spanned by the enriched weights $\rH$ such that the $\mV, \mW$-enriched functor$\Mor_{\mP\B\Env(\mJ)_{\B\Enr}}(\rH,-): \mP\B\Env(\mJ)_{\B\Enr}^\circledast \to (\mV\ot \mW)^\circledast$ is linear and preserves small colimits.
So $\widehat{\mM}$ is contained in the full subcategory of $\kappa$-compact objects of $\mP\B\Env(\mJ)_{\B\Enr}.$
Proposition \ref{rembrako} guarantees that $\mP\B\Env(\mJ)^\circledast_{\B\Enr} \to \mV^\ot \times \mW^\ot$ is a presentably bitensored $\infty$-category and so locally small. Presentability implies that the full subcategory of $\kappa$-compact objects of $\mP\B\Env(\mJ)_{\B\Enr}$ is small \cite[Remark 5.4.2.13.]{lurie.HTT}.
	
\end{proof}

\begin{proposition}Let $\mV^\ot \to \Ass, \mW^\ot \to \Ass$ be presentably monoidal $\infty$-categories. A bienriched $\infty$-category $\mM^\circledast \to \mV^\ot \times \mW^\ot$ is Cauchy-complete if and only if every absolute enriched weight on $\mM$ belongs to $\mM.$

\end{proposition}

\begin{proof}

Assume that every absolute enriched weight on $\mM$ belongs to $\mM$, let $\rH$ be an absolute enriched weight on a bienriched $\infty$-category $\mJ^\circledast \to \mV^\ot \times \mW^\ot$ and $\F: \mJ^\circledast \to \mM^\circledast$ a $\mV,\mW$-enriched functor.
The $\rH$-weighted colimit of the $\mV,\mW$-enriched functor $ \mJ^\circledast \xrightarrow{\F} \mM^\circledast \subset \mP\B\Env(\mM)_{\B\Enr}^\circledast$
is $\F_!(\rH)$ by Proposition \ref{weifunt}, which is an absolute enriched weight on $\mM$ by Remark \ref{rrro}.
So by assumption $\F_!(\rH) \in \mM.$ Thus the $\rH$-weighted colimit of the $\mV,\mW$-enriched functor $ \mJ^\circledast \xrightarrow{\F} \mM^\circledast \subset \mP\B\Env(\mM)_{\B\Enr}^\circledast$ lies in $\mM$ and so is the $\rH$-weighted colimit of $\F.$
Conversely, assume that $\mM^\circledast \to \mV^\ot \times \mW^\ot$ is
Cauchy-complete and let $\rH$ be an absolute enriched weight on $\mM$.
The $\rH$-weighted colimit of the identity of $\mM$, which exists by Cauchy-completeness,
is preserved by any $\mV,\mW$-enriched functor and so is sent by the $\mV,\mW$-enriched embedding $\mM^\circledast \subset \mP\B\Env(\mM)_{\B\Enr}^\circledast $ 
to the $\rH$-weighted colimit of the $\mV,\mW$-enriched embedding $\mM^\circledast \subset \mP\B\Env(\mM)_{\B\Enr}^\circledast $, which is $\rH$ by Proposition \ref{weifunt}.
So $\rH$ belongs to $\mM.$

\end{proof}

\begin{proposition}\label{Univop}Let $\mV^\ot \to \Ass, \mW^\ot \to \Ass$ be presentably monoidal $\infty$-categories, $\mM^\circledast \to \mV^\ot \times \mW^\ot$ a bienriched $\infty$-category and $\mN^\circledast \to \mV^\ot \times \mW^\ot$ a Cauchy-complete bienriched $\infty$-category.
The following induced functor is an equivalence: \begin{equation*}
\rho_\mN: \Enr\Fun_{\mV,\mW}(\widehat{\mM}, \mN) \to \Enr\Fun_{\mV,\mW}(\mM, \mN).\end{equation*}

\end{proposition}

\begin{proof}We prove first that the functor $\rho_\mN$ is conservative.
Let $\rH \in \widehat{\mM}$. Then $\rH \in  \mP\B\Env(\mM)_{\B\Enr}$ is the $\rH$-weighted colimit of the $\mV,\mW$-enriched embedding $\mM^\circledast \subset \widehat{\mM}^\circledast.$
Since $\rH$ is absolute, every $\mV,\mW$-enriched functor $\F: \widehat{\mM}^\circledast \to \mN^\circledast$ preserves the $\rH$-weighted colimit of the $\mV,\mW$-enriched embedding $\mM^\circledast \subset \widehat{\mM}^\circledast$ so that the canonical morphism $\colim^\rH(\F_{\mid \mM^\circledast}) \to \F(\rH)$ is an equivalence.
This implies that the functor $\rho_\mN$ is conservative
and that it is the pullback of the conservative functor
$\rho_{ \mP\B\Env(\mN)_{\B\Enr}}$ since $\mN$ is closed in $  \mP\B\Env(\mN)_{\B\Enr}$ under $\rH$-weighted colimits.
% this means that every $\mV,\mW$-enriched functor$\F: \widehat{\mM}^\circledast \to \mP\B\Env(\mN)^\circledast$ lands in $\mN^\circledast$if its restriction to $\mM^\circledast$ lands in $\mN^\circledast$.This follows from the fact that every $\rH \in \widehat{\mM}$ is the $\rH$-weighted colimit of the embedding $\mM^\circledast \subset \widehat{\mM}^\circledast$ so that $\F(\rH)$ is the $\rH$-weighted colimit of $\F_{\mid \mM^\circledast},$ and that $\mN$ admits $\rH$-weighted colimits, which are preserved by the embedding to $ \mP\B\Env(\mN).$
So it is enough to see that $\rho_{ \mP\B\Env(\mN)_{\B\Enr}}$ is an equivalence.
This functor admits a fully faithful left adjoint by \cite[Proposition 2.65.]{heine2024bienriched} and so is an equivalence being conservative.
\end{proof}

\begin{corollary}
Let $\mV^\ot \to \Ass, \mW^\ot \to \Ass$ be presentably monoidal $\infty$-categories and $\mM^\circledast \to \mV^\ot \times \mW^\ot$ a bienriched $\infty$-category.
The Cauchy-completion of $\mM^\circledast \to \mV^\ot \times \mW^\ot$ is Cauchy-complete.

\end{corollary}

\begin{proof}Let $\rH \in \mP\B\Env(\widehat{\mM})_{\B\Enr}$ be an absolute enriched weight on $\widehat{\mM}$. We like to see that $\rH \in \widehat{\mM}.$
By Proposition \ref{Univop} the embedding $\mM^\circledast \subset \widehat{\mM}^\circledast$
induces an equivalence $ \mP\B\Env(\mM)_{\B\Enr} \simeq \mP\B\Env(\widehat{\mM})_{\B\Enr}$.
So $\rH$ is the image of an enriched weight $\rH'$ on $\mM$ that is absolute by Proposition \ref{Chara}. Thus $\rH'\in \widehat{\mM}.$
By Proposition \ref{Univop} the $\mV,\mW$-enriched embedding $\widehat{\mM}^\circledast \subset \mP\B\Env(\mM)_{\B\Enr}^\circledast \simeq \mP\B\Env(\widehat{\mM})_{\B\Enr}^\circledast$
is the canonical embedding since both are equivalent after restriction to $\mM^\circledast.$
So $\rH$ is the image of $\rH'$ under the embedding $\widehat{\mM} \subset \mP\B\Env(\widehat{\mM})_{\B\Enr}$.
\end{proof}

Theorem \ref{wooond} gives the following corollary:

\begin{corollary}Let $\mV^\ot \to \Ass, \mW^\ot \to \Ass$ be presentably monoidal $\infty$-categories and $\mM^\circledast \to \mV^\ot \times \mW^\ot$ a small bienriched $\infty$-category.
The Cauchy-completion $\widehat{\mM}^\circledast \to \mV^\ot \times \mW^\ot$ is 
generated by $\mM$ under colimits weighted with respect to small absolute $\mV,\mW$-enriched weights.

\end{corollary}

\begin{notation}
Let $\mV^\ot \to \Ass$ be a presentably monoidal $\infty$-category. Let $ {_\mV\L\Enr^\wedge} \subset {_\mV\L\Enr}$ be the full subcategory of Cauchy-complete left $\mV$-enriched $\infty$-categories.

\end{notation}

\begin{corollary}\label{CaC} Let $\n \geq 1$ and $\mV^\boxtimes \to \bE_{\n+1}$ a presentably $\bE_{\n+1}$-monoidal $\infty$-category. % and $\mH$ a set of left enriched weights over $\mV.$
The $\infty$-category $ {_\mV\L\Enr^\wedge}$ of Cauchy-complete left $\mV$-enriched $\infty$-categories carries a canonical presentably $\bE_\n$-monoidal structure
such that Cauchy-completion ${_\mV\L\Enr} \to {_\mV\L\Enr^\wedge}$ refines to an $\bE_\n$-monoidal functor.
%such that the embedding $${_\mV\L\Enr}(\mH)^\ot \subset {_\mV\L\Q\Enr}(\mH)^\ot$$ is $\bE_\n$-monoidal.	

\end{corollary}

\begin{proof}
There is an identity $ {_\mV\L\Enr^\wedge} = {_\mV\L\Enr^\wedge}(\mH),$ 
where $\mH$ is the set of absolute left $\mV$-enriched weights.
We apply Theorem \ref{thqaz}.

\end{proof}

\subsection{A tensor product for $\n$-preadditive, $\n$-additive and $\n$-stable $(\infty, \n)$-categories}

In the following we construct presentably symmetric monoidal structures on the $\infty$-categories of small $\n$-preadditive, $\n$-additive and $\n$-stable $(\infty, \n)$-categories for any $\n \geq 1.$

\begin{definition}
A weakly bienriched $\infty$-category $\mM^\circledast \to \mV^\ot\times \mW^\ot$ is preadditive if it admits finite conical coproducts, i.e. conical colimits indexed by finite sets, and finite conical products, i.e. conical limits indexed by finite sets, and
the morphism $\emptyset \to *$ from the initial to the final object in $\mM$ is an equivalence, and for every $\X,\Y \in \mM$ the canonical morphism
$\X \coprod \Y \to \X \times \Y$ in $\mM$ is an equivalence. %$\mM$ is preadditive.

\end{definition}

\begin{remark}
	
Preadditive enriched $\infty$-categories were also discussed in \cite[Definition A.5.2.]{mazelgee2021universal}.
	
\end{remark}

\begin{remark}\label{simo}
A weakly bienriched $\infty$-category $\mM^\circledast \to \mV^\ot\times \mW^\ot$ is preadditive if it admits finite conical coproducts and finite conical products and $\mM$ is preadditive. A weakly bienriched $\infty$-category is preadditive if and only if the opposite 
weakly bienriched $\infty$-category is preadditive.
	
\end{remark}

\begin{example}\label{cotenso}
A weakly bienriched $\infty$-category $\mM^\circledast \to \mV^\ot\times \mW^\ot$ is preadditive if it admits left and right cotensors and $\mM$ is preadditive, and so dually
a weakly bienriched $\infty$-category $\mM^\circledast \to \mV^\ot\times \mW^\ot$ is preadditive if it admits left and right tensors and $\mM$ is preadditive.
In particular, for every monoidal $\infty$-category $\mV^\ot \to \Ass$ such that $\mV$ is preadditive we have that $\mV^\circledast \to \mV^\ot \times \mV^\ot$ is preadditive.
	
\end{example}

\begin{example}\label{prodo}
Every full weakly bienriched subcategory of a preadditive weakly bienriched $\infty$-category closed under finite products is again preadditive.
	
\end{example}

\begin{remark}\label{preadd}
	
Let	$\mM^\circledast \to \mV^\ot$ be a weakly left enriched $\infty$-category and $\mN^\circledast \to \mV^\ot \times \mW^\ot$ a preadditive weakly bienriched $\infty$-category.
The weakly right enriched $\infty$-category $ \Enr\Fun_{\mV,\emptyset}(\mM,\mN)^\circledast \to \mW^\ot $ is preadditive. This follows from Proposition \ref{Enric}. 
In particular, for every monoidal $\infty$-category $\mV^\ot \to \Ass$ such that $\mV$ is preadditive we have that $ \Enr\Fun_{\mV,\emptyset}(\mM,\mV)^\circledast \to \mV^\ot $ is preadditive. 

\end{remark}

\begin{notation}
Let $\mM^\circledast \to \mV^\ot\times \mW^\ot$	be a weakly bienriched $\infty$-category
that admits finite conical products, i.e. conical limits indexed by finite sets.
Let $\Cmon(\mM)^\circledast \subset (\mM^{\Comm})^\circledast \to \mV^\ot \times \mW^\ot$
be the full weakly bienriched subcategory spanned by the commutative monoid objects in $\mM$, i.e. the functors $\X: \Comm \to \mM$ such that for every $\n \geq 0$ the induced morphism
$\X(\langle \n\rangle) \to \X(\langle1\rangle)^{\times\n}$ is an equivalence.
\end{notation}

\begin{remark}
Evaluating at $\langle1\rangle \in \Comm$ gives a $\mV,\mW$-enriched functor
$\Cmon(\mM)^\circledast \to \mM^\circledast.$	
	
\end{remark}

\begin{remark}\label{fffoo}
Let $\mM^\circledast \to \mV^\ot\times \mW^\ot$	be a weakly bienriched $\infty$-category
that admits finite conical products and left (right) cotensors.	
Then $\Cmon(\mM)^\circledast \to \mV^\ot \times \mW^\ot$ admits left (right) cotensors,
which are preserved by the forgetful functor $\Cmon(\mM)^\circledast \to \mM^\circledast:$
by Proposition \ref{Enric} we find that $ (\mN^{\Comm})^\circledast \to \mV^\ot \times \mW^\ot$ admits left (right) cotensors, which are formed object-wise. Thus $\Cmon(\mN)^\circledast$ is closed in $(\mM^{\Comm})^\circledast$ under left (right) cotensors since forming left (right) cotensors are right adjoints.

\end{remark}

\begin{lemma}
Let $\mM^\circledast \to \mV^\ot\times \mW^\ot$	be a weakly bienriched $\infty$-category
that admits finite conical products.
Then $\Cmon(\mM)^\circledast \to \mV^\ot \times \mW^\ot$ is preadditive.
\end{lemma}

\begin{proof}
There is a $\mV,\mW$-enriched embedding $\mM^\circledast \subset \mN^\circledast $ preserving finite conical products into a weakly bienriched $\infty$-category that admits left and right cotensors.
For example take $\mN^\circledast := \mP\B\Env(\mM)^\circledast.$
The latter embedding yields an enriched embedding $\Cmon(\mM)^\circledast \subset \Cmon(\mM)^\circledast$ that preserves finite conical products.
By \cite[Proposition 2.3.]{Gepner_2015} the $\infty$-categories $\Cmon(\mM), \Cmon(\mN)$ are preadditive.
%The full subcategry $\Cmon(\mM) \subset \Fun(\Comm,\mC)$ is closed under finite products.
Thus by Example \ref{prodo} and \ref{cotenso} we find that $\Cmon(\mM)^\circledast \to \mV^\ot \times \mW^\ot$ is preadditive if $\Cmon(\mN)^\circledast \to \mV^\ot \times \mW^\ot$ 
admits left and right cotensors.
%So we can assume that $\mM^\circledast \to \mV^\ot\times \mW^\ot$ is a weakly bienriched $\infty$-category that admits finite conical products and left and right cotensors.
This holds by Remark \ref{fffoo}.

\end{proof}

\begin{lemma}\label{ahois}Let $\mV^\ot \to \Ass, \mW^\ot \to \Ass$ be $\infty$-operads such that $\mV$ and $\mW$ are the empty category or preadditive or $\mV^\ot =\mW^\ot= \mS^\times$. Let $\mM^\circledast \to \mV^\ot\times \mW^\ot$ be a preadditive weakly bienriched $\infty$-category such that for every $\X \in \mM, \Y \in \mP\B\Env(\mM)$ and $\V_1,...,\V_\n \in \mV, \W_1,...,\W_\m \in \mW$ for $\n,\m \geq 0$ the functors $$ \Mul_{\mP\B\Env(\mM)}(\V_1,..,(-),..., \V_\n,\X,\W_1,..., \W_\m;\Y): \mV^\op \to \mS, $$$$ \Mul_{\mP\B\Env(\mM)}(\V_1,.., \V_\n,\X,\W_1,...,(-),..., \W_\m;\Y): \mW^\op \to \mS$$ preserve finite products.
The $\mV,\mW$-enriched forgetful functor $\Cmon(\mM)^\circledast \to \mM^\circledast$ is an equivalence.
\end{lemma}

\begin{proof}
We first prove that there is an enriched embedding $\mM^\circledast \subset \mN^\circledast $ preserving finite conical products into a preadditive weakly bienriched $\infty$-category that admits left and right cotensors. Let $\mN^\circledast := \mP\B\Env(\mM)^\circledast.$
The weakly bienriched $\infty$-category $\mP\B\Env(\mM)^\circledast \to \mP\Env(\mV)^\ot \times \mP\Env(\mW)^\ot$ is preadditive by Example \ref{cotenso} since
$\mP\B\Env(\mM)$ is preadditive what we prove in the following: the jointly conservative functors $$ \Mul_{\mP\B\Env(\mM)}(\V_1,..,(-),..., \V_\n,\X,\W_1,..., \W_\m;-): \mP\B\Env(\mM) \to \Fun^{\prod}(\mV^\op,\mS), $$$$ \Mul_{\mP\B\Env(\mM)}(\V_1,.., \V_\n,\X,\W_1,...,(-),..., \W_\m;-): \mP\B\Env(\mM) \to \Fun^{\prod}(\mW^\op,\mS)$$ for $\V_1,...,\V_\n \in \mV, \W_1,...,\W_\m \in \mW$ for $\n,\m \geq 0$ and $\X \in \mM$ preserve small limits and small colimits by \cite[Theorem 4.50.]{heine2024bienriched} and $\Fun^{\prod}(\mV^\op,\mS), \Fun^{\prod}(\mW^\op,\mS)$ are preadditive if $\mV,\mW$ are empty or preadditive by  \cite[Corollary 2.4.]{Gepner_2015}.
%the $\infty$-categories $\Fun(\mW^\op,\mV), \Fun(\mV^\op,\mW)$ are preadditive if $\mV,\mW$ are preadditive, respectively, and the full subcategories $\Fun^\R(\mW^\op,\mV) \subset \Fun(\mW^\op,\mV), \Fun^\R(\mV^\op,\mW) \subset \Fun(\mV^\op,\mW)$ spanned by the right adjoint functors are closed under finite products if $\mV,\mW$ admit finite products, respectively.
The $\mV,\mW$-enriched forgetful functor $\Cmon(\mM)^\circledast \to \mM^\circledast$ is the pullback of the $\mV,\mW$-enriched forgetful functor $\Cmon(\mN)^\circledast \to \mN^\circledast$ that preserves left and right cotensors by Remark \ref{fffoo}.
The latter is an equivalence since it preserves left and right cotensors
and induces an equivalence on underlying $\infty$-categories by \cite[Proposition 2.3.]{Gepner_2015}.

\end{proof}

\begin{proposition}\label{rigz} Let $\mV^\ot \to \Ass, \mW^\ot \to \Ass$ be $\infty$-operads such that $\mV,\mW$ are empty or preadditive or $\mV^\ot =\mW^\ot= \mS^\times$. Let $\mM^\circledast \to \mV^\ot\times \mW^\ot$ be a preadditive weakly bienriched $\infty$-category such that for every $\X \in \mM, \Y \in \mP\B\Env(\mM)$ and $\V_1,...,\V_\n \in \mV, \W_1,...,\W_\m \in \mW$ for $\n,\m \geq 0$ the functors $$ \Mul_{\mP\B\Env(\mM)}(\V_1,..,(-),..., \V_\n,\X,\W_1,..., \W_\m;\Y): \mV^\op \to \mS, $$$$ \Mul_{\mP\B\Env(\mM)}(\V_1,.., \V_\n,\X,\W_1,...,(-),..., \W_\m;\Y): \mW^\op \to \mS$$ preserve finite products and $\mN^\circledast \to \mV^\ot\times \mW^\ot$ a weakly bienriched $\infty$-category that admits finite conical products. The following induced functor is an equivalence: $$\theta_\mN: \Enr\Fun^{\prod}_{\mV,\mW}(\mM,\Cmon(\mN)) \to \Enr\Fun^{\prod}_{\mV,\mW}(\mM,\mN).$$
	
\end{proposition}

\begin{proof}

For every $\infty$-category $\K$ there is a $\mV,\mW $-enriched equivalence
$\Cmon(\mN^\K)^\circledast \simeq (\Cmon(\mN)^\K)^\circledast$ and the functor
$\theta_\mN^\K $ identifies with the functor $\theta_{\mN^\K}.$
Consequently, it is enough to prove that $\theta_\mN$ is a bijection on equivalence classes.
Since $\mM^\circledast \to \mV^\ot \times \mW^\ot$ is preadditive, by Lemma \ref{ahois}
the $\mV,\mW$-enriched forgetful functor $\Cmon(\mM)^\circledast \to \mM^\circledast$ is an equivalence. The functor $\theta_\mN$ is essentially surjective since every $\mV,\mW$-enriched functor
$\mM^\circledast \to \mN^\circledast$ is the image under $\theta_\mN$ of the $\mV,\mW$-enriched functor $\mM^\circledast \simeq \Cmon(\mM)^\circledast \to \Cmon(\mN)^\circledast$.
The functor $\theta_\mN$ is essentially injective: every $\mV,\mW$-enriched functor
$\phi: \mM^\circledast \to \Cmon(\mN)^\circledast$ factors as 
$\mM^\circledast \simeq \Cmon(\mM)^\circledast \xrightarrow{\Cmon(\theta_\mN(\phi))} \Cmon(\Cmon(\mN))^\circledast \to \Cmon(\mN)^\circledast$.
%Thus $\theta_\mN$ is an equivalence if $\mM^\circledast \to \mV^\ot$ admits cotensors.	
		
\end{proof}

%\begin{notation}Let $\mV^\ot \to \Ass$ be a monoidal $\infty$-category.

%\begin{enumerate}\item Let $^{\coprod}_\mV\L\Enr := {_\mV\L\Enr(\Fin)}$ be the subcategory of left $\mV$-enriched $\infty$-categories that admit finite conical coproducts and left $\mV$-enriched functors preserving finite conical coproducts.\item Let $_\mV\mathrm{Preadd} \subset {^{\coprod}_\mV\L\Enr}$ be the full subcategory of preadditive left $\mV$-enriched $\infty$-categories.\end{enumerate}\end{notation}

Let $\Fin$ be the category of finite sets.
\begin{notation}
Let $\mV^\ot \to \Ass, \mW^\ot \to \Ass$ be $\infty$-operads.
Let $${^{\coprod}_\mV\omega\B\Enr_\mW}:= {_\mV\omega\B\Enr_\mW(\Fin)},\ {^{\coprod}_\mV\B\Enr_\mW}:= {_\mV\B\Enr_\mW(\Fin)}.$$

\end{notation}

\begin{notation}
Let $\mV^\ot \to \Ass, \mW^\ot \to \Ass$ be $\infty$-operads.
Let $$_\mV\mathrm{Preadd}_\mW \subset {^{\coprod}_\mV\B\Enr_\mW}$$ be the full subcategory of preadditive $\mV,\mW$-bienriched $\infty$-categories. For $\mV^\ot =\mW^\ot= \mS^\times$ we drop $\mV,\mW$ from the notation.

\end{notation}

By definition $^{\coprod}_\mV\omega\B\Enr_\mW, {^{\coprod}_\mV\B\Enr_\mW}$ are the subcategories of (weakly) $\mV,\mW$-bienriched $\infty$-categories that admit finite conical coproducts and $\mV,\mW$-enriched functors preserving finite conical coproducts.

%By definition $^{\coprod}_\mV\L\Enr := {_\mV\L\Enr(\Fin)}$ is the subcategory of left $\mV$-enriched $\infty$-categories that admit finite conical coproducts and left $\mV$-enriched functors preserving finite conical coproducts.

%\begin{corollary}\label{ihis}Let $\mV^\ot \to \Ass$ be a presentably monoidal $\infty$-category such that $\mV$ is preadditive.The embedding $_\mV\mathrm{Preadd} \subset {^{\coprod}_\mV\L\Enr}$ admits a right adjointthat sends $\mM^\circledast \to \mV^\ot $ to $\Cmon(\mM)^\circledast \to \mV^\ot.$\end{corollary}

\begin{proposition}\label{Pread}
Let $\mV^\ot \to \Ass, \mW^\ot \to \Ass$ be $\infty$-operads.
Then ${^{\coprod}_\mV\omega\B\Enr_\mW}, {^{\coprod}_\mV\B\Enr_\mW}, {_\mV\mathrm{Preadd}_\mW} $ are preadditive.
%Since $_\mV\mathrm{Preadd}_\mW \subset {^{\coprod}_\mV\L\Enr_\mW}$ is closed under finite products, also $_\mV\mathrm{Preadd}_\mW$ is preadditive.

\end{proposition}

\begin{proof}
The $\infty$-category $_\mV\omega\B\Enr_\mW$ admits finite products: for every small weakly bienriched $\infty$-categories $\mM^\circledast \to \mV^\ot \times \mW^\ot, \mN^\circledast \to \mV^\ot\times \mW^\ot$
the weakly bienriched $\infty$-category $\mM^\circledast \times_{(\mV^\ot \times \mW^\ot)} \mN^\circledast \to \mV^\ot \times \mW^\ot$ is the product in $_\mV\omega\B\Enr_\mW$, and $\id: \mV^\ot\times \mW^\ot \to \mV^\ot\times \mW^\ot$ is the final object.
By description of the morphism objects of a product a $\mV,\mW$-enriched functor to a product preserves finite conical coproducts if it components do. Thus the subcategory ${^{\coprod}_\mV\omega\B\Enr_\mW}$ admits finite products and the inclusion ${^{\coprod}_\mV\omega\B\Enr_\mW} \subset {_\mV\omega\B\Enr_\mW}$ preserves finite products.	
Moreover the final object in ${^{\coprod}_\mV\omega\B\Enr_\mW} $ is an initial object. Let $\mM^\circledast \to \mV^\ot\times \mW^\ot$ be a weakly bienriched $\infty$-category that admits finite conical coproducts. Then the $\mV,\mW$-enriched diagonal functor $\mM^\circledast \to \mM^\circledast \times_{(\mV^\ot \times \mW^\ot)} \mM^\circledast$ admits a $\mV,\mW$-enriched left adjoint $\mu: \mM^\circledast \times_{(\mV^\ot \times \mW^\ot)} \mM^\circledast \to \mM^\circledast$ assigning the coproduct.
Moreover the unique $\mV,\mW$-enriched functor $\mM^\circledast \to \mV^\ot\times \mW^\ot$ admits a $\mV,\mW$-enriched left adjoint $\alpha: \mV^\ot \times \mW^\ot \to \mM^\circledast$ assigning the initial object.
We apply \cite[Proposition 2.4.3.19.]{lurie.higheralgebra} and have to verify that 
the composition $\mM^\circledast \xrightarrow{\mM^\circledast \times_{(\mV^\ot \times \mW^\ot)} \alpha} \mM^\circledast \times_{(\mV^\ot \times \mW^\ot)} \mM^\circledast \xrightarrow{\mu} \mM^\circledast$ is the identity
and the composition $ \mM^\circledast \times_{(\mV^\ot \times \mW^\ot)} \mN^\circledast \times_{(\mV^\ot \times \mW^\ot)} \mM^\circledast \times_{(\mV^\ot \times \mW^\ot)} \mN^\circledast \simeq \mM^\circledast \times_{(\mV^\ot \times \mW^\ot)} \mM^\circledast \times_{(\mV^\ot \times \mW^\ot)} \mN^\circledast \times_{(\mV^\ot \times \mW^\ot)} \mN^\circledast \xrightarrow{\mu \times_{(\mV^\ot \times \mW^\ot)}\mu} \mM^\circledast\times_{(\mV^\ot \times \mW^\ot)} \mN^\circledast$ is $\mu.$ This holds by adjointness.
So ${^{\coprod}_\mV\omega\B\Enr_\mW}$ is preadditive. The full subcategories ${_\mV\mathrm{Preadd}_\mW}, {^{\coprod}_\mV\B\Enr_\mW}$ are closed under finite products and so preadditive, too.
	
\end{proof}

\begin{lemma}\label{gtr}
Let $\mD$ be a presentable $\infty$-category and  $\mC \subset \mD $ a full subcategory such that the embedding $\iota: \mC \subset \mD $ admits a right adjoint $\G: \mD \to \mC. $ 
The following are equivalent:
	
\begin{enumerate}
		
\item The functor $\iota: \mC \subset \mD $ admits a left adjoint 
$\F: \mD \to \mC  $ such that the composition $\iota \circ \F: \mD \to \mC \subset \mD $ is an accessible functor.
%In this case $\mC$ is presentable. 
		
\item The full subcategory $\mC $ is closed in $\mD$ under small limits and $\G: \mD \to \mC  $ is an accessible functor.

\item The functor $ \iota \circ \G : \mD \to \mC \subset \mD $ is accessible and preserves small limits.
		
\end{enumerate}

In particular, if (1) or (2) holds, $\mC$ is presentable.
	
\end{lemma}

\begin{proof}
The $\infty$-category $\mC $ admits small limits and small colimits being a colocalization of an $\infty$-category that admits small limits and small colimits.
Moreover small limits in $\mC$ are coreflected via $\G$ from small limits in $\mD$.
So (2) and (3) are equivalent.

If (1) holds, %$\iota: \mC \subset \mD $ preserves small limits so that 
$\mC $ is closed in $\mD$ under small limits. 
If (1) holds, $\mC \subset \mD $ is an accessible localization so that 
$\mC$ is presentable. 
Hence $\G: \mD \to \mC  $ is an accessible functor as a right adjoint functor between presentable $\infty$-categories. 
	
Assume that (2) holds. The composition $\R:= \iota \circ \G : \mD \to \mC \subset \mD $ is an accessible functor that preserves small limits.  
Since $\mD$ is presentable, $\R : \mD \to \mD $ admits a left adjoint $ \T: \mD \to \mD $ by the adjoint functor theorem \cite[Corollary 5.5.2.9]{lurie.HTT}. 
Let $ \varepsilon : \R = \iota \circ \G  \to  \id_\mD $ be the counit. 
Because $\iota: \mC \subset \mD $ is fully faithful, the unit $\id_\mC \to \G \circ \iota $ is an equivalence so that the composition 
$  \G \circ \epsilon :   \G \circ \iota \circ \G \to  \G $ and thus 
$ \R \circ \varepsilon : \R  \circ \R   \to  \R  $
are equivalences by the triangle identities. 
	
Let $ \X, \Y $ be objects of $\mD. $  The map $ \mD(\T(\X), \R(\Y)) \to \mD(\T(\X), \Y) $ induced by the counit $ \varepsilon(\Y) : \R(\Y) \to \Y $ is equivalent to the map $$ \mD(\T(\X), \R(\Y)) \simeq \mD(\X, \R  \circ \R(\Y )) \to \mD(\X, \R(\Y) ) \simeq \mD(\T(\X), \Y) $$ induced by the counit 
$ \R(\varepsilon(\Y)) : \R  \circ \R(\Y)  \to  \R(\Y) $ and thus is an equivalence.
Thus $\T(\X) $ is a colocal object of $\mD$ and so belongs to $\mC. $
So $ \T: \mD \to \mD $ factors as $ \F: \mD \to \mC \subset \mD. $
Since $ \T: \mD \to \mD $  is left adjoint to $ \R: \mD \to \mD,  $ the functor 
$ \F: \mD \to \mC $ is left adjoint to $ \iota \simeq \R \circ \iota: \mC \to \mD.$  
	
\end{proof}

\begin{proposition}\label{ihis}Let $\mV^\ot \to \Ass, \mW^\ot \to \Ass$ be $\infty$-operads such that $\mV,\mW$ are empty or preadditive or $\mV^\ot =\mW^\ot= \mS^\times$.
The embeddings $_\mV\mathrm{Preadd}_\mW \subset {^{\coprod}_\mV\B\Enr_\mW}, {_\mV\mathrm{Preadd}^\wedge_\mW} \subset {{^{\coprod}_\mV\B\Enr^\wedge_\mW}}$ admit a left and right adjoint.
	
\end{proposition}

\begin{proof}
By Proposition \ref{rigz} the embedding $_\mV\mathrm{Preadd}_\mW \subset {^{\coprod}_\mV\B\Enr_\mW}$ admits a right adjoint $\G$ such that ${^{\coprod}_\mV\B\Enr_\mW} \xrightarrow{\G} {_\mV\mathrm{Preadd}_\mW} \subset {^{\coprod}_\mV\B\Enr_\mW}$ preserves small limits and filtered colimits. By Corollary \ref{presenta} the $\infty$-category ${^{\coprod}_\mV\B\Enr_\mW}$ is presentable. We apply Lemma \ref{gtr} to deduce the existence of a left adjoint of the embedding $_\mV\mathrm{Preadd}_\mW \subset {^{\coprod}_\mV\B\Enr_\mW}$.
By the same proposition the embedding $_\mV\mathrm{Preadd}^\wedge_\mW \subset {^{\coprod}_\mV\B\Enr^\wedge_\mW}$ admits a right adjoint $\G$ such that ${^{\coprod}_\mV\B\Enr^\wedge_\mW} \xrightarrow{\G} {_\mV\mathrm{Preadd}^\wedge_\mW} \subset {^{\coprod}_\mV\B\Enr^\wedge_\mW}$ preserves small limits and filtered colimits
since for every Cauchy-complete bienriched $\infty$-category $ \mM^\circledast \to \mV^\ot \times \mW^\ot$ the bienriched $\infty$-category $ (\mM^{\Comm})^\circledast \to \mV^\ot \times \mW^\ot$ is Cauchy-complete (Proposition \ref{Enric}) and the full bienriched subcategory  $\Cmon(\mM)^\circledast \to \mV^\ot \times \mW^\ot $ of the latter is closed under abosulte colimits and so Cauchy-complete, too.
By Corollary \ref{presenta} the $\infty$-category ${^{\coprod}_\mV\B\Enr^\wedge_\mW}$ is presentable. We apply Lemma \ref{gtr} to deduce the existence of a left adjoint of the embedding $_\mV\mathrm{Preadd}^\wedge_\mW \subset {^{\coprod}_\mV\B\Enr^\wedge_\mW}$.

\end{proof}

\begin{notation}
Let $\mV^\boxtimes \to \bE_{\n+1}$ be a presentably $\bE_{\n+1}$-monoidal $\infty$-category corresponding to an $\bE_\n$-algebra in $\Pr\Mon.$
	
\begin{enumerate}
\item Let $^{\coprod}_\mV\L\Enr^\ot := {_\mV\L\Enr(\Fin)^\ot}.$

\item Let $^{\coprod}_\mV\L\Enr^{\wedge,\ot} := {_\mV\L\Enr(\Fin \cup \{ \mathrm{absolute} \ \mV\mathrm{-enriched} \ \mathrm{weights}\})^\ot}.$
		
\item Let $_\mV\mathrm{Preadd}^\ot \subset {^{\coprod}_\mV\L\Enr^\ot}$ be the full suboperad spanned by the preadditive left $\mV$-enriched $\infty$-categories.

\item Let $_\mV\mathrm{Preadd}^{\wedge,\ot} \subset {^{\coprod}_\mV\L\Enr^{\wedge,\ot}}$ be the full suboperad spanned by the preadditive left $\mV$-enriched $\infty$-categories.

\item For $\mV^\ot =\mS^\times$ we drop $\mV$ from the notation.
		
\end{enumerate}
	
\end{notation}

%By definition $^{\coprod}_\mV\L\Enr := {_\mV\L\Enr(\Fin)}$ is the subcategory of left $\mV$-enriched $\infty$-categories that admit finite conical coproducts and left $\mV$-enriched functors preserving finite conical coproducts.

\begin{proposition}\label{ihist}
Let $\n \geq 0$ and $\mV^\boxtimes \to \bE_{\n+1}$ a preadditive presentably $\bE_{\n+1}$-monoidal $\infty$-category corresponding to an $\bE_\n$-algebra in $\Pr\Mon$
or $\mV^\ot = \mS^\times$.

\begin{enumerate}
\item The embedding $_\mV\mathrm{Preadd}^\ot \subset {^{\coprod}_\mV\L\Enr^\ot}$ admits a left adjoint relative to $\bE_\n$ and the embedding $_\mV\mathrm{Preadd} \subset {^{\coprod}_\mV\L\Enr}$ admits a right adjoint.

\item The embedding $_\mV\mathrm{Preadd}^{\wedge,\ot} \subset {^{\coprod}_\mV\L\Enr^{\wedge,\ot}}$ admits a left adjoint relative to $\bE_\n$ and the embedding $_\mV\mathrm{Preadd}^\wedge \subset {^{\coprod}_\mV\L\Enr^\wedge}$ admits a right adjoint.
\end{enumerate}

\end{proposition}

\begin{proof}
	
In view of Proposition \ref{ihis} it is enough to see that for every (Cauchy-complete) left $\mV$-enriched $\infty$-categories $\mM^\circledast \to \mV^\ot, \mN^\circledast \to \mV^\ot$ having finite conical coproducts the internal hom from $\mM^\circledast \to \mV^\ot$ to $ \mN^\circledast \to \mV^\ot$ in ${^{\coprod}_\mV\L\Enr}, {^{\coprod}_\mV\L\Enr^\wedge}$, respectively, is preadditive if $ \mN^\circledast \to \mV^\ot$ is preadditive.
This follows from the description of the internal hom (Theorem \ref{Thhs}) and Remark \ref{preadd}.
	
\end{proof}

%Since ${^{\coprod}_\mV\L\Enr^\ot} \to \Comm$ is presentably symmetric monoidal, Proposition \ref{ihis} implies the following corollary:

\begin{corollary}
Let $\n \geq 0$ and $\mV^\boxtimes \to \bE_{\n+1}$ a preadditive presentably $\bE_{\n+1}$-monoidal $\infty$-category corresponding to an $\bE_\n$-algebra in $\Pr\Mon$ or $\mV^\ot= \mS^\times$. Then $_\mV\mathrm{Preadd}^\ot \to \bE_\n, {_\mV\mathrm{Preadd}^{\wedge,\ot}} \to \bE_\n$ are preadditive presentably $\bE_\n$-monoidal $\infty$-categories. %$_\mV\mathrm{Preadd}, {_\mV\mathrm{Preadd}^\wedge} $ are preadditive.

\end{corollary}

\begin{proof}
	
This follows from Proposition \ref{ihis} since ${^{\coprod}_\mV\L\Enr^\ot} \to \bE_\n,{^{\coprod}_\mV\L\Enr^{\wedge,\ot}} \to \bE_\n $ are presentably $\bE_\n$-monoidal $\infty$-categories by Theorem \ref{thqaz}, which are preadditive by Proposition \ref{Pread}
and the fact that $_\mV\mathrm{Preadd}^\wedge$ is closed in $_\mV\mathrm{Preadd}$ under finite products.
	
\end{proof}

\begin{definition}\label{iter} Let $\mV^\boxtimes \to \Comm$ be a preadditive presentably symmetric monoidal $\infty$-category or $\mV^\ot = \mS^\times$. Let $\n \geq 1.$
Let $$ {_\mV\mathrm{Preadd}(1)^\ot}:=  {_\mV\mathrm{Preadd}^\ot},\ {_\mV\mathrm{Preadd}(\n+1)^\ot}:= {_{{_\mV\mathrm{Preadd}(\n)} }\mathrm{Preadd}^\ot},$$
$$ {_\mV\mathrm{Preadd}^\wedge(1)^{\ot}}:=  {_\mV\mathrm{Preadd}^{\wedge,\ot}},\ {_\mV\mathrm{Preadd}^\wedge(\n+1)^\ot}:= {_{{_\mV\mathrm{Preadd^\wedge}(\n)} }\mathrm{Preadd}^{\wedge,\ot}}.$$

For $\mV^\ot =\mS^\times$ we drop $\mV$ from the notation and call $$ {\mathrm{Preadd}(\n)}:={_\mS\mathrm{Preadd}(\n)}$$ the $\infty$-category of $\n$-preadditive $\infty$-categories and $$ {\mathrm{Preadd}^\wedge(\n)}:={_\mS\mathrm{Preadd}^\wedge(\n)}$$ the $\infty$-category of Cauchy-complete $\n$-preadditive $\infty$-categories.
		
\end{definition}

\begin{corollary}\label{kiop} Let $\mV^\boxtimes \to \Comm$ be a preadditive presentably symmetric monoidal $\infty$-category or $\mV^\ot = \mS^\times$ and $\n \geq 1.$
Then ${_\mV\mathrm{Preadd}(\n)^\ot} \to \Comm, {_\mV\mathrm{Preadd}^\wedge(\n)^\ot} \to \Comm $ are preadditive presentably symmetric monoidal $\infty$-categories.
In particular, ${_\mV\mathrm{Preadd}(\n)}$ is enriched in itself and so belongs to
${_\mV\widehat{\mathrm{Preadd}}(\n+1)}$ and ${_\mV\mathrm{Preadd}^\wedge(\n)}$ is enriched in itself and so belongs to ${_\mV\widehat{\mathrm{Preadd}}^\wedge(\n+1)}.$

\end{corollary}

%\begin{corollary}The $(\infty,\n+1)$-category ${\mathrm{Preadd}(\n)}$ is $\n+1$-preadditive.The $(\infty,\n+1)$-category ${\mathrm{Preadd}^\wedge(\n)} $ is Cauchy-complete $\n+1$-preadditive\end{corollary}

%\begin{proof}By Corollary \ref{kiop} and Remark \ref{simo} the $(\infty,\n+1)$-categories ${\mathrm{Preadd}(\n)}, {\mathrm{Preadd}^\wedge(\n)} $ are preadditive and Cauchy-complete.By Corollary \ref{kiop} the $\infty$-categories ${\mathrm{Preadd}(\n)}, {\mathrm{Preadd}^\wedge(\n)} $ are enriched in themselves, which gives the canonical enrichment, and so are (Cauchy-complete) $\n+1$-preadditive.\end{proof}

\begin{notation}
Let $\Add^{\wedge} \subset \Add \subset \Cat_\infty \simeq {_\mS \L\Enr}$ be the subcategories of (idempotent complete) additive $\infty$-categories and finite coproducts preserving functors.
	
\end{notation}

\begin{remark}
	
The $\infty$-categories $\Add^{\wedge}, \Add $ are preadditive because they are full subcategories closed under finite products of the $\infty$-category ${_\mS \L\Enr}(\{\mathrm{finite}\ \mathrm{sets}\})$
of small $\infty$-categories having finite coproducts and functors preserving finite coproducts, which is a preadditive $\infty$-category.
\end{remark}

\begin{notation}
Let $(\Add^\wedge)^\ot \subset \Add^\ot \subset {_\mS \L\Enr}(\{\mathrm{finite}\ \mathrm{sets}\})^\ot$ be the full symmetric suboperads spanned by the (idempotent complete) additive $\infty$-categories. %So we can apply Definition \ref{iter}:
	
\end{notation}

By \cite[Lemma 8.15.]{HEINE2023108941} the embeddings $(\Add^\wedge)^\ot \subset \Add^\ot \subset {_\mS \L\Enr}(\{\mathrm{finite}\ \mathrm{sets}\})^\ot$ admit a left adjoint relative to $\Comm$ so that the symmetric $\infty$-operads $(\Add^\wedge)^\ot \to \Comm, \Add^\ot \to \Comm$ are presentably symmetric monoidal $\infty$-categories.
%\begin{definition}Let $\n \geq 2$.The $\infty$-category of $\n$-preadditive $\infty$-categories is $$ {\mathrm{Preadd}(\n)}:={_{\mathrm{Preadd}}\mathrm{Preadd}(\n-1)}.$The $\infty$-category of Cauchy-complete $\n$-preadditive $\infty$-categories is $$ {\mathrm{Preadd}^\wedge(\n)}:={_{\mathrm{Preadd}^\wedge}\mathrm{Preadd}^\wedge(\n-1)}.$$\end{definition}
So we can make the following definition:

\begin{definition}Let $\n \geq 2$.
The $\infty$-category of $\n$-additive $\infty$-categories is 	
$$ \Add(\n):={_{\Add}\mathrm{Preadd}(\n-1)}.$$
The $\infty$-category of Cauchy-complete $\n$-additive $\infty$-categories is 	
$$ \Add(\n)^{\wedge}:={_{\Add^\wedge}\mathrm{Preadd}^\wedge(\n-1)}.$$
Moreover we set $$\Add(1):=\Add, \Add^\wedge(1):=\Add^\wedge.$$
	
\end{definition}

For $\mV=\Add, \Add^\wedge$ Corollary \ref{kiop} gives the following corollary:

\begin{corollary}\label{jlmmozzi} Let $\n \geq 1$.
Then $ \Add(\n)$ is $\n+1$-additive and $\Add(\n)^{\wedge}$ is Cauchy-complete $\n+1$-additive.
	
\end{corollary}

\begin{remark}
Let $\n \geq 1$. By induction the inclusions $ \Add \subset \Cat_\infty, \Add^\wedge \subset \Cat_\infty$ give rise to inclusions
\begin{equation}\label{ijjj}
\Add(\n), \Add(\n)^\wedge \subset \Cat_{(\infty,\n)}
\end{equation}	
so that (Cauchy-complete) $\n$-additive $\infty$-categories are $(\infty,\n)$-categories.
\end{remark}

Next we define $\n$-stable $\infty$-categories.

\begin{notation}
Let $\Rex \subset \Cat_\infty$ be the full subcategory spanned by the finite $\infty$-categories.
\end{notation}

\begin{notation}
Let $\St^\wedge \subset \St \subset {_\mS \L\Enr}(\Rex) \subset \Cat_\infty$ be the full subcategories of (idempotent complete) stable $\infty$-categories. % and finite colimits preserving functors.
	
\end{notation}

\begin{remark}
	
The $\infty$-categories $\St, \St^\wedge $ are preadditive because they are full subcategories closed under finite products of the $\infty$-category ${_\mS \L\Enr}(\Rex)$
of small $\infty$-categories having finite colimits and functors preserving finite colimits, which is a preadditive $\infty$-category.
\end{remark}

By Theorem \ref{Thor} there is a symmetric monoidal $\infty$-category ${_\Sp\L\Enr}(\Rex)^\ot \to \Comm$.

\begin{notation}
Let $(\St^\wedge)^\ot \subset \St^\ot \subset {_\mS \L\Enr}(\Rex)^\ot$ be the full symmetric suboperads spanned by the (idempotent complete) stable $\infty$-categories. %So we can apply Definition \ref{iter}:
	
\end{notation}

By \cite[Lemma 8.15.]{HEINE2023108941} the embeddings $(\St^\wedge)^\ot \subset \St^\ot \subset {_\mS \L\Enr}(\Rex)^\ot$ admit a left adjoint relative to $\Comm$ so that the
symmetric $\infty$-operads $\St^\ot \to \Comm, (\St^\wedge)^\ot\to \Comm$ are presentably symmetric monoidal $\infty$-categories.
%\begin{definition}Let $\n \geq 2$.The $\infty$-category of $\n$-preadditive $\infty$-categories is $$ {\mathrm{Preadd}(\n)}:={_{\mathrm{Preadd}}\mathrm{Preadd}(\n-1)}.$The $\infty$-category of Cauchy-complete $\n$-preadditive $\infty$-categories is $$ {\mathrm{Preadd}^\wedge(\n)}:={_{\mathrm{Preadd}^\wedge}\mathrm{Preadd}^\wedge(\n-1)}.$$\end{definition}
The next definition for $\n=2$ is due to \cite{mazelgee2021universal}:

\begin{definition}Let $\n \geq 2$.
The $\infty$-category of $\n$-stable $\infty$-categories is 	
$$ \St(\n):={_\St\mathrm{Preadd}(\n-1)}.$$
The $\infty$-category of Cauchy-complete $\n$-stable $\infty$-categories is 	
$$ \St^\wedge(\n):={_{\St^\wedge}\mathrm{Preadd}^\wedge(\n-1)}.$$
Moreover we set $$\St(1):=\St, \St^\wedge(1):=\St^\wedge.$$
	
\end{definition}

For $\mV= \St, \St^\wedge$ Corollary \ref{kiop} gives the following corollary:

\begin{corollary}\label{jlmmozz} Let $\n \geq 1$.
Then $\St(\n)$ is $\n+1$-stable and $\St^\wedge(\n)$ is Cauchy-complete $\n+1$-stable.
	
\end{corollary}

%\begin{proof}We prove the case of $\Cat_{(\infty,\n)}^\ex.$ The other case is similar.By ... the $(\infty,\n+1)$-category $\Cat_{(\infty,\n)}^\ex$ is preadditive.So we need to see that for every $\mB, \mC,\mD \in \Cat_{(\infty,\n)}^\ex$ the morphism $(\infty,\n)$-category $\Mor_{\Cat_{(\infty,\n)}^\ex}(\mC,\mD)$ is $\n$-stableand the composition $\Mor_{\Cat_{(\infty,\n)}^\ex}(\mC,\mD) \ot \Mor_{\Cat_{(\infty,\n)}^\ex}(\mB,\mC) \to \Mor_{\Cat_{(\infty,\n)}^\ex}(\mB,\mD)$is a morphism in $\Cat_{(\infty,\n)}^\ex.$ We prove the first, the proof of the latter is similar. To see the first, we need to see that for every $\F,\G,\rH \in \Theta:=\Mor_{\Cat_{(\infty,\n)}^\ex}(\mC,\mD)$ the morphism object $\Mor_{\Theta}(\F,\G)$is $\n-1$-stable and the composition $\Mor_{\Theta}(\G,\rH) \ot \Mor_{\Theta}(\F,\G) \to \Mor_{\Theta}(\F,\rH)$ is a morphism in $\Cat_{(\infty,\n-1)}^\ex.$\end{proof}

\begin{remark}
Let $\n \geq 1$. By induction the inclusions $ \St \subset \St^\wedge \subset \Cat_\infty$ give rise to inclusions
\begin{equation}\label{ijjjj}
\St(\n), \St^\wedge(\n) \subset \Cat_{(\infty,\n)}
\end{equation}	
so that (Cauchy-complete) $\n$-stable $\infty$-categories are $(\infty,\n)$-categories.
\end{remark}

Theorem \ref{thqaz} gives the following corollary:

\begin{corollary}\label{Stab} Let $\n \geq 1$.
The $\infty$-categories
$${\mathrm{Preadd}(\n)},\ {\mathrm{Preadd}^\wedge(\n)}, \ {\mathrm{Add}(\n)},\ {\mathrm{Add}^\wedge(\n)}, \ \St(\n),\ \St^\wedge(\n) $$ carry presentably symmetric monoidal structures such that the inclusions (\ref{ijjj}), (\ref{ijjjj}) are lax symmetric monoidal, where the right hand sides carry the cartesian structures.
	
\end{corollary}

\bibliographystyle{plain}
\bibliography{ma}
\end{document}